\documentclass[11pt]{article}
\usepackage{bbm}
 \usepackage{amssymb}
\usepackage{amssymb, amsthm, amsmath, amscd}
\usepackage[usenames,dvipsnames]{color}

\newtheorem{thm}{Theorem}[section]
\newtheorem{lem}[thm]{Lemma}
\newtheorem{prop}[thm]{Proposition}
\newtheorem{cor}[thm]{Corollary}
\theoremstyle{definition}

\theoremstyle{definition}
\newtheorem{df}[thm]{Definition}
\theoremstyle{definition}
\newtheorem{rem}[thm]{Remark}

\theoremstyle{definition}

\renewcommand{\phi}{\varphi}

\newcommand{\N}{\mathbb{N}}
\newcommand{\Z}{\mathbb{Z}}
\newcommand{\Q}{\mathbb{Q}}
\newcommand{\R}{\mathbb{R}}
\newcommand{\C}{\mathbb{C}}
\newcommand{\T}{\mathbb{T}}

\numberwithin{equation}{section}

\newcommand{\Aff}{\operatorname{Aff}}

\newcommand{\id}{\operatorname{id}}

\newcommand{\aff}{\rm aff}

\newcommand{\cpc}{completely positive contractive linear map}
\newcommand{\morp}{contractive completely positive linear map}

\newcommand{\hm}{homomorphism}
\newcommand{\dt}{\delta}
\newcommand{\ep}{\varepsilon}

\newcommand{\la}{\langle}
\newcommand{\ra}{\rangle}
\newcommand{\andeqn}{\,\,\,{\rm and}\,\,\,}
\newcommand{\rforal}{\,\,\,{\rm for\,\,\,all}\,\,\,}
\newcommand{\CA}{$C^*$-algebra}
\newcommand{\SCA}{$C^*$-subalgebra}

\newcommand{\af}{{\alpha}}
\newcommand{\bt}{{\beta}}

\newcommand{\diag}{{\rm diag}}

\newcommand{\wilog}{without loss of generality}
\newcommand{\Wlog}{Without loss of generality}

\newcommand{\beq}{\begin{eqnarray}}
\newcommand{\eneq}{\end{eqnarray}}
\newcommand{\tforal}{\,\,\,\text{for\,\,\,all}\,\,\,}
\newcommand{\tand}{\,\,\,\text{and}\,\,\,}

\newcommand{\p}{\mathfrak{p}}
\newcommand{\q}{\mathfrak{q}}

\usepackage{amsfonts}
\usepackage{mathrsfs}
\usepackage{textcomp}
\usepackage[all]{xy}

\title{On classification of  non-unital  simple amenable C*-algebras,
I}
\author{Guihua Gong and  Huaxin Lin
 }
\date{}

\begin{document}

\maketitle

\begin{abstract}
We present a stable uniqueness theorem for non-unital \CA s.
Generalized  tracial rank one is defined for stably projectionless simple \CA s.
Let $A$ and $B$ be two  stably projectionless separable  simple amenable \CA s
 with $gTR(A)\le 1$
and $gTR(B)\le 1.$
Suppose also that  $KK(A, D)=KK(B,D)=\{0\}$ for all \CA s $D.$
Then $A\cong B$ if and only if they have the same tracial cones with scales.
We also show that every separable simple \CA\, $A$ with finite nuclear dimension
which satisfies the UCT with non-zero traces must have  $gTR(A)\le 1$ if $K_0(A)$ is torsion.
In the next part of this research, we show similar results without the restriction on $K$-theory.
\end{abstract}
{\tiny{\footnote{{\tiny (This is a revised version  of  Arxiv:1611.04440, Nov. 14, 2016)}}}}

\section{Introduction}

Recently some  sweeping progresses have been made in the Elliott  program (\cite{Ellicm}),
the program of classification of  separable amenable \CA s by the Elliott invariant
(a $K$-theoretical set of invariant) (see \cite{GLN}, \cite{TWW} and \cite{EGLN}).
These are the results of decades of work by many  mathematicians (see also \cite{GLN},
\cite{TWW} and \cite{EGLN} for the historical discussion there).
These progresses   could be summarized  briefly as the following:
Two unital  finite separable simple \CA s $A$ and $B$ with finite nuclear dimension which satisfy
the UCT are isomorphic if and only if  their Elliott invariant ${\rm Ell}(A)$ and ${\rm Ell}(B)$ are
isomorphic.  Moreover, all  weakly unperforated Elliott invariant
can be achieved by a finite separable simple \CA s in UCT class with finite nuclear dimension
(In fact these can be constructed as so-called ASH-algebras--see \cite{GLN}).
Combining with the previous classification of purely infinite simple \CA s, results
of Kirchberg and Phillips (\cite{Pclass} and \cite{KP}), now all unital separable simple \CA s in the UCT class with
finite nuclear dimension are classified by the Elliott invariant.

This research  studies the non-unital cases.

Suppose that $A$ is a separable simple \CA.
In the case that $K_0(A)_+\not=\{0\},$ then $A\otimes {\cal K}$ has a non-zero projection, say
$p.$ Then $p(A\otimes K)p$ is unital.  Therefore if $A$ is in the UCT class and has finite
nuclear dimension, then $p(A\otimes {\cal K})p$ falls into the class of \CA s which has been classified.
Therefore isomorphism theorem  for  these \CA s  is an immediate
consequence of that in \cite{GLN} (see section 8.4 of \cite{Lncbms}).

Therefore this paper considers the case that $K_0(A)_+=\{0\}.$
Simple \CA s with $K_0(A)_+=\{0\}$ are stably projectionless in the sense
that not only $A$ has no non-zero projections but $M_n(A)$ also has no non-zero
projections for every integer $n\ge 1.$  In particular, the results in \cite{GLN} cannot be applied
in the stably projectionless case. It is entirely new situation.

One of the central issues  of the establishment of  the isomorphism theorem is the  uniqueness theorem.
The uniqueness theorem for  unital simple \CA s
is based on a stable uniqueness theorem first established in \cite{Lnsuniq}.   Stable uniqueness
theorem was established aimed at the introduction of \CA s of tracial rank zero (and later for
tracial rank one and the generalization of them). One of the main constrains
of the stably uniqueness theorem is that the stabilized maps are required to be full (see also \cite{ED}).
To deal with non-unital cases, the first thing one might do is the unitization.
However, unitization would immediately lose  the fullness condition since
$A$ is always an ideal of its unitization.  Therefore stable uniqueness theorem
for non-unital \CA s has to be established  without the usage of the known unital version.
The first task of this research is to do just that.
It is also important to establish a useful one, i.e., it should be able to  be applied
to  the  classification process, in particular,  stable uniqueness theorem  should work
for \cpc s that are only approximately multiplicative.  In section 6 and 7, we present
the needed  stable uniqueness theorems.

One of the other important issues of  this research is to introduce an appropriate notion of
generalized tracial rank one for stably projectionless simple \CA s.
We use stably projectionless 1-dimensional non-commutative finite complices as the models.  In section
12, we introduce a class of simple \CA s which we will called ${\cal D}$ (and
${\cal D}_0$ for the variation). We show that these \CA s have the regularities
that required for the classification purpose.  In particular,  we show that these \CA s have strict comparison
for positive elements, approximate divisibility as well as stable rank one.
\CA s with  generalize tracial rank one are defined based on the class of \CA s in ${\cal D}.$

Our goal is to give a classification for general stably projectionless simple \CA s using a modified version of Elliott invariant.
This first part of the research, including the stable uniqueness theorem and study of
generalized tracial rank,  serves as the foundation for the general classification.
However, since the  case that \CA s have trivial $K$-theory is  relatively  less involved.
The classification of the  stably projectionless simple \CA s with trivial $K$-theory
(but with arbitrary tracial cones) is also presented.
 It was first presented by S. Razak (\cite{Raz}) that  certain
stably projectionless
\CA s can be constructed as inductive limits of 1-dimensional non-commutative finite complices.
Razak showed that these special inductive limits can be classified by
their cone of lower semi-continuous traces.  One of them  (see \cite{Tsang} and \cite{Jb})
is called $W.$ It is a stably projectionless  simple \CA\, with a unique tracial state and with
$K_0(W)=K_1(W)=\{0\}.$

The main isomorphism theorem in this article  is the following statement:
\begin{thm}\label{TTCLM}
Let $A$ and $B$ be two non-unital separable amenable simple \CA s which satisfy the UCT.
Suppose that $K_0(A)=K_0(B)=K_1(A)=K_1(B),$  $gTR(A)\le 1$ and $gT(B)\le 1.$
Then $A\cong B$ if and only if  there is an isomorphism $\Gamma:$
\beq\label{TTCLM-1}
({\tilde T}(A), \Sigma_A)\to ({\tilde T}(B), \Sigma_B).
\eneq
Moreover, the isomorphism $\psi: B\to A$ can be chosen so that
it induces $\Gamma.$
\end{thm}

Here ${\tilde T}(A)$ and ${\tilde T}(B)$ are cones of lower semi-continuous traces and
$\Sigma_A$ and $\Sigma_B$ are the scales (see  \ref{Ddimf} bellow).
This is first proved for \CA s in ${\cal D}$ with continuous scale using tracial state space $T(A)$ and
$T(B)$ instead of ${\tilde T}(A)$ and ${\tilde T}(B)$ to avoid the difficulties  created by
those simple \CA s which have both bounded  and unbounded traces as well as non-compact
tracial state spaces.  One particular
consequence is that $W\cong W\otimes W.$
In this first part of the research, at least for the statement \ref{TTCLM} above,
we assume that \CA s have trivial $K$-theory. Therefore UCT assumption can be replaced
by the assumption that $KK(A,D)=KK(B,D)=\{0\}$ for all \CA s $D.$

One immediate question
is whether Theorem \ref{TTCLM} applies to all   stably projectionless simple \CA s with finite nuclear
dimension. By studying the so-called $W$-trace,  we show, in this first part of the research, that
within UCT class,  under the assumption that $K_0(A)$ is torsion,
a finite separable simple amenable \CA\, $A$ has $gTR(A)\le 1$ if and only if it has finite nuclear dimension
(this restriction on $K$-theory  will be removed in the later parts of this research).
Therefore the above theorem holds when we remove the condition
$gTR(A)\le 1$ and $gTR(B)\le 1$ and replaced it by the condition
that both $A$ and $B$ have finite nuclear dimension.
In other words, we have the following:

\begin{thm}\label{TMpartI}
Let $A$ and $B$ be two separable simple \CA\, with finite nuclear dimension and which satisfies the UCT.
Suppose that ${\tilde T}(A)\not=\{0\},$ ${\tilde T}(B)\not=\{0\}$ and $K_i(A)=K_i(B)=\{0\},$ $i=0,1.$
Then $A\cong B$ if and only if
$$
({\tilde T}(A), \Sigma_A)\cong ({\tilde T}(B), \Sigma_B).
$$

\end{thm}

In particular, for any finite separable simple \CA\, $A$ with finite nuclear dimension, one has that
$gTR(A\otimes W)\le 1.$  Therefore we have the following corollary:

\begin{cor}\label{CWiso}
Let $A$ and $B$ be two separable simple \CA s with finite nuclear dimension which
satisfy the UCT.
Then, in the case ${\tilde T}(A)\not=\{0\},$  $A\otimes W\cong B\otimes W$ if and only if
$$
({\tilde T}(A), \Sigma_A)\cong ({\tilde T}(B), \Sigma_B).
$$
In the case that ${\tilde T}(A)=\{0\},$ then $A\otimes W\cong {\cal O}_2\otimes {\cal K}.$
\end{cor}


This paper  is the first part of a series research.  Some of the results  in this
part will play more prominent role in the  later parts of this research.
In particular, the stable uniqueness theorem established here only has limited usage in
this part.
In the second part of this research,  using  also some of the results established in this part, we show the following:

\begin{thm}
Let $A$ and $B$ be two non-unital  finite separable simple  \CA s with finite nuclear dimension which satisfy the UCT.
Suppose that $K_0(A)={\rm ker}\rho_A$ and $K_0(B)={\rm ker}\rho_B.$
Then $A\cong B$ if and only if
$$
{\rm {\tilde{Ell}}}(A)\cong {\rm {\tilde{Ell}}}(B)
$$
\end{thm}
Here ${\rm {\tilde{Ell}}}(A)$ is a modified Elliott invariant for stably  projectionless  simple \CA s.
A range theorem will also presented in the second part of this research.

{\bf Acknowledgement}:   Most of the first part of this  research
was done when both authors stayed
in the Research Center for Operator Algebras in East China Normal University
in the summer of 2016.  The second part of this research was also initiated during
the same stay.  Both authors acknowledge the support by the Center.
The first named author was also partially supported by by NNSF of China (11531003) and
the second named author was also supported by a  NSF grant.

\section{Notation}


\begin{df}\label{DTtilde}
Let $A$ be a \CA. Denote by $P(A)$ the Pedersen ideal.

Denote by ${\tilde T}(A)$ the space  of all  densely defined  lower semi-continuous  positive traces
equipped   with the weak *-topology induced by elements in $P(A)$ as a topological convex cone.
In particular, if $\tau\in {\tilde T}(A)$ and $b\in P(A)_+$ then
$\tau$ is a finite trace on $\overline{bAb}.$

Suppose that $A=P(A).$ Let $T(A)$ be those $\tau\in {\tilde T}(A)$ such that
they are also states of $A.$ These are called tracial states.
In this case
define $T_0(A)=\{\af \cdot \tau: \tau\in T(A),\,\, \, 0\le \af\le 1\}.$

Suppose that $A$ is $\sigma$-unital.
In the case that $P(A)_+$ contains a full element $a$ of $A$ (in particular when
$A$ is simple),  let us clarify the structure of ${\tilde T}(A).$
Put $A_1=\overline{aAa}.$   Then we may identify $A$ with a $\sigma$-unital hereditary
\SCA\, of $A_1\otimes {\cal K}$ by
the Brown's theorem (\cite{Br1}).
Then $T_0(A_1)$ is a  weak*-closed convex subset of all positive linear functionals of $A_1$ with norm no more than
1.  Thus  $T_0(A_1)$  has the usual  structure of a topological  convex set and it is a Choquet simplex.

Hence ${\tilde T}(A)=\{\af \tau: \tau\in T_0(A_1)\andeqn \af\in \R_+\}.$
Therefore we view that ${\tilde T}(A)$ is the cone generated by $T_0(A_1)$ and equipped with the topology
induced by that of $T_0(A_1).$

Let $T_a=\{\tau\in {\tilde T}(A): \tau(a)=1\}.$ Then $T$ is a weak *-compact. In fact
$T_a$ is a Choquet simplex.  Moreover ${\tilde T}(A)$ is a topological cone
with the Choquet simplex $T_a$ as a base.

\end{df}

\begin{df}\label{Dfep}
Let $1>\ep>0.$  Define
\beq
f_{\ep}(t)=\begin{cases} f_{\ep}(t)=0, &\text{if}\, t\in [0,\ep/2];\\
                                      f_{\ep}(t)={t-\ep/2\over{\ep/2}}, &\text{if}\, t\in (\ep/2, \ep];\\
                                       f_{\ep}(t)=1 & \text{if}\, t\in (\ep, \infty).\end{cases}
                                       \eneq
\end{df}

\begin{df}\label{Ddimf}
Let $A$ be a \CA\, and let $a\in A_+.$  Suppose that ${\tilde T}(A)\not=\emptyset.$
Define
$$
d_\tau(a)=\lim_{\ep\to 0} \tau(f_\ep(a))
$$
with possible infinite value.  Note that $f_\ep(a)\in P(A)_+.$ Therefore
$\tau\mapsto d_\tau(a)$ is a lower semi-continuous affine function  on ${\tilde T}(A)$
(to $[0,\infty]$).  One also notices
that
$$
d_\tau(a)=\lim_{n\to\infty}\tau(a^{1/n})\rforal \tau\in {\tilde T}(A).
$$

Suppose that $A$ is non-unital. Let $a\in A_+$ be a strictly positive element.
Define
$$
\Sigma_A(\tau)=d_\tau(a)\rforal \tau\in {\tilde T}(A).
$$
It is standard and routine to check that $\Sigma_A$ is independent of the choice
of $a.$
The lower semi-continuous affine function $\Sigma_A$ is called the scale function of $A.$
\end{df}

\begin{df}\label{Dcu}
Let $A$ be a \CA\, and let $a, b\in A_+.$
We write $a\lesssim b,$ if there exists a sequence $\{x_n\}$ in $A$
such that $x_n^*bx_n\to a$ in norm. If $a\lesssim b$ and $b\lesssim a,$ we
write $a\sim b$ and say that $a$ and $b$ are Cuntz equivalent.
 It is know that $\sim$ is an equivalence relation.  Let $W(A)$ be the equivalence class
 of positive elements in $M_n(A)$ for all $n$ with the usual embedding
 from $M_n$ into $M_{n+1}.$  Denote by
 $Cu(A)$ the Cuntz equivalence classes of positive elements in $A\otimes {\cal K}.$
 It is an ordered  semigroup.
 We use $\la a \ra$ for the equivalence class represented by $a.$
 Thus $a\lesssim b$ will be also written as $\la a\ra \le \la b\ra.$
 Recall that we write $a\ll b$ if  the following holds: for any increasing sequence $\{y_n\},$
 and if $\la b\ra \le  \sup\{\la y_n \ra \}$ then  there exists $n_0\ge 1$ such that $a\lesssim y_{n_0}.$

In what follows we will also use the notation $Cu^{\sim}(A)$ and $Cu^{\sim}(\phi)$ as introduced
in \cite{Rl}.
\end{df}

\begin{df}\label{Dstrictcom}

Let $A$ be a \CA. We will use $QT(A)$ for the quasi-trace $\tau$ with $\|\tau\|=1.$
Suppose that ${\tilde T}(A)\not=\{0\}.$ Suppose also that every quasi-trace is a trace.

If $\tau\in {\tilde T}(A),$ we will extend it to $(A\otimes {\cal K})_+$ by
$\tau(a\otimes b)=\tau(a)Tr(b),$  for all $a\in A$ and $b\in {\cal K}.$ where $Tr$ is the densely defined trace on ${\cal K}.$

We say $A$ has the property of  strict comparison for positive elements, if for any two elements $a, \, b\in M_n(A)_+$
(for any integer $n\ge 1$)
with property that $d_\tau(a)<d_\tau(b)<\infty$ for all $\tau\in {\tilde T}(A)\setminus \{0\},$
then $a\lesssim b.$

Let  $A$ be a  $\sigma$-unital and let $a\in P(A)_+.$ Suppose that $a$ is full in
$A.$   We view $A$ as a $\sigma$-unital hereditary \SCA\, of $A_1\otimes {\cal K}$ as in \ref{DTtilde}.
View $T_0(A_1)$ as a convex subset of ${\tilde T}(A).$

Then $A$ has strictly comparison for positive elements if and only if,
for any two positive elements $a, \, b\in  M_n(A)_+$
(for any integer $n\ge 1$) with
with property that $d_\tau(a)<d_\tau(b)<\infty$ for all $\tau\in T_0(A_1)\setminus \{0\},$
then $a\lesssim b.$

Suppose that $T(A)\not=\emptyset.$ Denote by
$\overline{T(A)}^w$ the weak*-closure of $T(A)$ in ${\tilde T}(A).$
Suppose also that $0\not\in \overline{T(A)}^w$ (see \ref{compactrace} below). We say $A$ has  the  property of {\it strong} strict comparison for positive elements, if for any two elements $a, \, b\in M_n(A)_+$
(for any integer $n\ge 1$)
with property that $d_\tau(a)<d_\tau(b)$ for all $\tau\in \overline{T(A)}^w,$
then $a\lesssim b.$

\end{df}

\begin{df}\label{DAq}
Let $A$ be a \CA\, with $T(A)\not=\{0\}$ such that $0\not\in \overline{T(A)}^w.$
There is an affine  map
$r_{\aff}: A_{s.a.}\to \Aff(\overline{T(A)}^w)$ defined by
$$
r_{\aff}(a)(\tau)=\hat{a}(\tau)=\tau(a)\tforal \tau\in \overline{T(A)}^w
$$
and for all $a\in A_{s.a.}.$ Denote by $A_{s.a.}^q$ the space  $r_{\aff}(A_{s.a.})$ and
$A_+^q=r_{\aff}(A_+).$

\end{df}

\begin{df}\label{Dcpcamen}
Let $A$ and $B$ be two \CA s. Let $\phi: A\to B$ be a \cpc.
We say that $\phi$ is amenable, if, for any $\ep>0,$ and any finite subset ${\cal F}\subset A,$
there exists an integer $k\ge 1$ and there exist  \cpc s $\psi_0: A\to M_k$ and
$\psi_1: M_k\to B$ such that
$$
\|\psi_1\circ\psi_0(a)-\phi(a)\|<\ep\rforal a\in A.
$$
\end{df}

\begin{df}[cf. \cite{Rlz}]\label{Dalst1}
Let $A$ be a non-unital \CA. We say $A$ almost has stable rank one
if for any integer $m\ge 1$ and any hereditary \SCA\, $B\subset M_m(A),$
$B\subset \overline{GL({\tilde B})}, $ where  $GL({\tilde B})$ is the group of invertible elements of
${\tilde B}.$  This definition is slightly different from that in \cite{Rlz}.
\end{df}

\begin{df}\label{Dappcpc}
Let $A$ and $B$ be \CA s and let $\phi_n: A\to B$ be
\cpc s. We say $\{\phi_n\}$ is a sequence of approximately multiplicative \cpc s
if
$$
\lim_{n\to\infty}\|\phi_n(a)\phi_n(b)-\phi_n(ab)\|=0\rforal a, \, b\in A.
$$
\end{df}

\begin{df}
Let $A$ be a \CA.  Denote by $A^{\bf 1}$ the unit ball of $A.$
$A_+^{q, {\bf 1}}$ is the image of the intersection of $A_+\cap A^{{\bf 1}}$ in $A_+^q.$
\end{df}

\begin{df}\label{DKLtriple}
Let $A$ be a unital C*-algebra. Recall that, following D\u{a}d\u{a}rlat and Loring (\cite{DL}), one defines
\begin{equation}\label{Dbeta-5}
\underline{K}(A)=\bigoplus_{i=0,1}K_i(A)\oplus\bigoplus_{i=0,1}
\bigoplus_{k\ge 2}K_i(A,\Z/k\Z).
\end{equation}
There is a commutative \CA\, $C_k$ such that one may identify
$K_i(A\otimes C_k)$ with $K_i(A, \Z/k\Z).$
Let $A$ be a unital separable amenable C*-algebra, and let $B$ be a $\sigma$-unital C*-algebra.
Following R\o rdam (\cite{Ror-KL-I}),  $KL(A,B)$ is the quotient of $KK(A,B)$ by those elements
represented by limits of trivial extensions (see \cite{Lnauct}). In the case that $A$ satisfies the UCT,
R\o rdam defines $KL(A,B)=KK(A,B)/{\cal P},$ where
${\cal P}$ is the subgroup corresponding to the pure extensions of the $K_*(A)$ by $K_*(B).$
In \cite{DL}, D\u{a}d\u{a}rlat and Loring  proved that
\beq\label{DKL-2}
KL(A,B)={\rm Hom}_{\Lambda}(\underline{K}(A), \underline{K}(B)).
\eneq
\end{df}

\begin{df}\label{Dceil}
Let $A$ be a unital separable amenable \CA\, and let  $x\in A.$ Suppose
that
$\|xx^*-1\|<1$ and $\|x^*x-1\|<1.$ Then $x|x|^{-1}$ is a unitary.
Let us use $\lceil x \rceil $ to denote $x|x|^{-1}.$

Let $C$ be a separable \CA\, and $B$ be another \CA.
Let ${\cal F}\subset C$ be a finite subset and $\ep>0$ be a positive number.
We say a map $L: C\to B$ is ${\cal F}$-$\ep$-multiplicative if
$$
\|L(xy)-L(x)L(y)\|<\ep\rforal x,\, y\in {\cal F}.
$$
We now assume that $L: C\to B$ is a \cpc.
Denote by ${\tilde L}: {\tilde C}\to {\tilde B}$  the  unital extension of $L.$
Let ${\cal
P}\subset \underline{K}(C)$ be a finite subset.
Let\linebreak $\{p_1, p_2,...,p_{m_1}, p_1', p_2',...,p_{m_1}'\}\subset M_N({\tilde{C}}\otimes {\tilde{C_k}})$ be a finite subset of projections
and \linebreak
 $\{u_1, u_2,...,u_{m_2}\}\subset M_N({\tilde{C}}\otimes {\tilde{C_k}})$ (for some
integer $N\ge 1$) be a finite subset of
unitaries so that  $\{[p_i]-[p_i'], [u_j]: 1\le i\le m_1\andeqn 1\le j\le m_2\}={\cal P}.$

There is $\ep>0$ and a finite subset
${\cal F}$ satisfying the following: for any  \CA\, $B$ and
any unital ${\cal F}$-$\ep$-multiplicative \morp\, $L : C\to B,$ $L$
induces a \hm\, $[L]$ defined on  $G({\cal P}),$ where $G({\cal
P})$ is the subgroup generated by ${\cal P},$ to $\underline{K}(B)$
such that,
\beq\label{KLtriple-1}
\|{\tilde L}\otimes {\rm id}_{M_N}(p_i)-q_i\|<1,\,\,\,\|{\tilde L}\otimes {\rm id}_{M_N}(p_i')-q_i'\|<1\andeqn \|\lceil
{\tilde L}\otimes {\rm id}_{M_N}(u_j)\rceil -v_j\|<1
\eneq
for some  projection  $q_i, q_i'\in M_N({\tilde B}\otimes {\tilde C}_k)$  such that $[q_i]=[L]([p_i])$
and $[q_i']=[L]([p_i'])$ in
$K_0({\tilde B}\otimes {\tilde C}_k)$
  and some  unitary
$v_j\in M_N({\tilde B}\otimes {\tilde{C_k}})$  such that
$[v_j]=[L]([u_j]),$ $1\le i\le m_1$ and $1\le m_2.$  for
Such a triple $(\ep, {\cal F}, {\cal P})$ is called a
$KL$-triple for $C.$

\end{df}

\begin{df}\label{Dsb}
Let $A$ be a \CA.
Denote by $SA$ the suspension of $A:$
$SA=C_0((0,1], A).$
\end{df}

\section{ Some results of R\o rdam}

For the convenience, we  would like to use the following version
of  a lemma of R\o rdam:

\begin{lem}[R\o rdam, Lemma 2.2 of \cite{Rr11}]\label{Lrorm}
Let $a, \, b\in A$ with $0\le a,\, b\le 1$ such that
$\|a-b\|<\dt/2.$
Then there exists $z\in A$ with $\|z\|\le 1$ such that
$$
(a-\dt)_+=z^*bz.
$$
\end{lem}

\begin{proof}
By Lemma 2.2 of \cite{Rr11}, let $\dt_0=\|a-b\|,$
$$
f_\dt(a)^{1/2} (a-\dt_0\cdot 1)f_\dt(a)^{1/2}\le f_\dt(a)^{1/2}b f_\dt(a)^{1/2}.
$$
Therefore
$$
f_\dt(a)^{1/2} (a-\dt_0\cdot 1)_+f_\dt(a)^{1/2}\le  f_\dt(a)^{1/2}b f_\dt(a)^{1/2}.
$$
Thus
$$
(a-\dt)_+\le f_\dt(a)^{1/2} (a-\dt_0\cdot 1)_+f_\dt(a)^{1/2}\le  f_\dt(a)^{1/2}b f_\dt(a)^{1/2}.
$$
The lemma then follows from 2.3 of \cite{Rr11}.
\end{proof}

\begin{lem}[Proposition 1  of \cite{CEI}]\label{Lalmstr1}
Suppose that $A$ is a non-unital \CA\, which has almost stable rank one.
Suppose that $a, b\in M_m(A)_+$ (for some $m\ge 1$) are two elements
such that $a\lesssim b.$
Then, for any $0< \dt<1,$  there exists a unitary $u\in {\widetilde M_m(A)}$ such that
$$
u^*f_\dt(a)u\in \overline{bAb}.
$$
Moreover, there exists $x\in M_m(A)$ such that
$$
x^*x=a\andeqn xx^*\in \overline{bM_m(A)b}.
$$
\end{lem}

We would also like to include the following variation of \ref{Lrorm}  which is also known.

\begin{lem}\label{LRL}
Let $1>\ep>0$ and $1/4>\sigma>0.$  There exists $1/4>\dt>0$ satisfying the following:
If $A$ is a \CA, $x, y\in A_+$ with $0\le x,\, y\le 1$ and
$$
\|x-y\|<\dt,
$$
then there exists a partial isometry $w\in A^{**}$ with
\beq\label{LRL-1}
ww^*f_{\sigma}(x)=f_{\sigma}(x)ww^*=f_{\sigma}(x),\\
w^*aw\in \overline{yAy}\rforal a\in \overline{f_\sigma(x)Af_\sigma(x)}\andeqn\\
\|w^*cw-c\|<\ep\rforal c\in \overline{f_{\sigma}(x)Af_{\sigma}(x)}\,\,\,{\rm with}\,\,\,\|c\|\le 1.
\eneq

If $A$ has almost stable rank one, then there exists a unitary $u\in {\tilde A}$
to replace $w$ above.
\end{lem}

\begin{proof}
Let $1/4>\sigma\cdot \ep/128>\dt_1>0$ be such that, for any \CA\, $B,$ and any pair of positive elements
$x',\, y'\in B$ with $0\le x',\, y'\le 1$ such that
$$
\|x'-y'\|<\dt_1,
$$
then
\beq\label{LRL-n1}
\|f_{\sigma/2}(x')-f_{\sigma/2}(y')\|<\ep\cdot \sigma/64\andeqn \|(x')^{1/2}-(y')^{1/2}\|<\ep\cdot \sigma/64
\eneq

Put $\eta=(\min\{\ep/64, \sigma/32, \dt_1\cdot \sigma/32\})^2.$
Define
$g(t)=f_{\sigma/2}(t)/t$ for all $0<t\le 1$ and $g(0)=0.$ Then $g(t)\in C_0((0,1]).$
Note that $\|g\|\le 2/\sigma.$
Choose $\dt=\eta\cdot \dt_1/32.$

Now let $A$ be a \CA\, and $x,\, y\in A$ with
$0\le x,\, y\le 1$ and $\|x-y\|<\dt.$

Then
\beq\label{LRL-n2}
\|x^2-y^2\|<2\dt.
\eneq
Set $z=yf_{\eta}(x^2)^{1/2}.$
Then
\beq\label{LRL-n3}
\|(z^*z)^{1/2}-x\| &=&\|(f_{\eta}(x^2)^{1/2}y^2f_{\eta}(x^2)^{1/2})^{1/2}-x\|\\
&<& 2\dt+\|(f_{\eta}(x^2)^{1/2}x^2f_{\eta}(x^2)^{1/2})^{1/2}-x\|\\
&<& 2\dt+2\sqrt{\eta}<\sigma\cdot \dt_1/8.
\eneq
Also
\beq\label{LRL-n4}
\|(z^*z)^{1/2}-z\|&<&\sigma\cdot\dt_1/8+\|x-yf_{\eta}(x^2)^{1/2}\|\\
&<& \sigma\cdot \dt_1/8+\dt_2+\|x-xf_{\eta}(x^2)^{1/2}\|\\\label{LRL-n4+}
&<& \sigma\cdot \dt_1/8+\dt_2+\sqrt{\eta}<\sigma\cdot \dt_1/4.
\eneq
Write $z=v(z^*z)^{1/2}$ as polar decomposition in $A^{**}.$
Then
\beq\nonumber
\hspace{-0.4in}\|vf_{\sigma/2}(x)-f_{\sigma/2}(x)\| &\le & \|vf_{\sigma/2}(x)-vf_{\sigma/2}((z^*z)^{1/2})\|+\|vf_{\sigma/2}((z^*z)^{1/2})-f_{\sigma/2}(x)\|\\\nonumber
&<&\sigma\cdot \ep/64+\|v(z^*z)^{1/2} g((z^*z)^{1/2})-f_{\sigma/2}(x)\|\hspace{0.8in}\, ({\rm using \eqref{LRL-n1}})\\\nonumber
&=&\sigma\cdot \ep/64+\|zg((z^*z)^{1/2})-f_{\sigma/2}(x)\|\\\nonumber
&<&\ep/64 +\dt_1/4+\|(z^*z)^{1/2}g((z^*z)^{1/2})-f_{\sigma/2}(x)\|\hspace{0.5in} ({\rm using \eqref{LRL-n4+}}) \\\nonumber
&=&3\ep/64+\|f_{\sigma/2}((z^*z)^{1/2})-f_{\sigma/2}(x)\|\\\nonumber
&<&3\ep/64+\ep/64=\ep/16.\hspace{2in} ({\rm using \eqref{LRL-n1}})
\eneq
Thus, for any  $c\in \overline{f_{\sigma}(x)Af_{\sigma}(x)}$ with $\|c\|\le 1,$
\beq\label{LRL-12}
\|vcv^*-c\| &=&\|vf_{\sigma/2}(x)cf_{\sigma/2}(x)v^*-c\|\\
&<& \ep/16+\|f_{\sigma/2}(x)cf_{\sigma/2}(x)v^*-c\|\\
&<& \ep/16+\ep/16=\ep/8.
\eneq

Let $v^*v$ be  the open projection  as the range projection of
the positive element $f_\eta(x^2)^{1/2}y^2f_\eta(x^2)^{1/2}.$
It follows from 2.2 of \cite{Rr11} that
$$
(\eta-\|x^2-y^2\|)f_\eta(x^2)\le f_\eta(x^2)^{1/2} y^2 f_\eta(x^2)^{1/2}\le f_\eta(x^2).
$$
Therefore $v^*v$ is also the range projection of
$f_\eta(x^2).$
It follows
that
$$
vcv^*\in \overline{y^{1/2}Ay^{1/2}}
$$
for any $c\in \overline{f_\eta(x^2)Af_\eta(x^2)}.$

Choose $w=v^*.$  Then, since $\sqrt{\eta}\le \sigma/2,$
\beq\label{LRL-14}
f_{\sigma/2}(x)f_\eta(x^2)=f_{\sigma/2}(x)\andeqn
ww^*f_\sigma(x)=f_\sigma(x)ww^*=f_\sigma(x).
\eneq

\end{proof}

The following lemma is known.

\begin{lem}\label{Lfullep}
Let $A$ be a \CA\, and $a\in A_+$ be a full element.
Then, for any $b\in A_+,$ any $1>\ep>0$ and
any $g\in C_0((0,\infty))$ whose support is
in $[\ep, N]$ for some $N\ge 1,$  there are
$x_1, x_2,..,x_m\in A$ such that
$$
g(b)=\sum_{i=1}^m x_i^*ax_i.
$$
\end{lem}

\begin{proof}
Fix $1>\ep>0.$ There are $z_1,z_2,...,z_m\in A$
such that
$$
\|\sum_{i=1}^m z_i^*az_i-b\|<\ep/2,
$$
Therefore, by 2.2 and 2.3 of \cite{Rr11},   there are $y\in B$  with $\|y\|\le 1/\ep$ such that
$$
f_{\ep}(b)=y^*(\sum_{i=1}^m x_i^*ax_i)y.
$$
Therefore
$$
g(b)=g(b)^{1/2}y^*(\sum_{i=1}^m x_i^*ax_i)yg(b)^{1/2}.
$$

\end{proof}

We would also like to include the following theorem of R\o rdam:

\begin{thm}[\cite{Rrzstable}]\label{TRozs}
Let $A$ be an exact   simple \CA\, with is ${\cal Z}$-stable.
Then $A$ has  the  strict comparison property  for positive elements:
Let $a, \, b\in M_n(A)_+$ (for some $n\ge 1$) be two elements
such that
\beq\label{Trozs-1}
d_\tau(a)<d_\tau(b)<\infty\rforal \tau\in {\tilde T}(A)\setminus \{0\},
\eneq
then $a\lesssim b.$
\end{thm}

\begin{proof}

Let
$$
S=\{\tau\in {\tilde T}(A): d_\tau(b)=1\}.
$$
The assumption \eqref{Trozs-1} implies that
\beq\label{Trozs-2}
d_\tau(a)<d_\tau(b)\rforal \tau\in S.
\eneq
Since $A$ is simple, for every $\ep>0,$ $\la (a-\ep)_+\ra \le K \la b\ra $ for some integer $K\ge 1.$
The above implies that
$$
f(a)<f(b)\rforal f\in S(W, b)
$$
(see \cite{Rrzstable} for the notation and other details). Since, by Theorem 4.5 of \cite{Rrzstable}, $W(A)$ is almost unperforated, and by
3.2 of \cite{Rrzstable},
$a\lesssim b.$

\end{proof}

\begin{cor}\label{{CRozs}}
Let $A$ be an exact  separable  simple \CA\, with is ${\cal Z}$-stable
Then $A$ has  the  following strict comparison property  for positive elements:
Let $a, \, b\in M_n(A)_+$ (for some $n\ge 1$) be two elements
such that
\beq\nonumber
\label{Trozs-3}
d_\tau(a)<d_\tau(b)<\infty\rforal \tau\in
\overline{T(B)}^w,
\eneq
then $a\lesssim b.$
where
$B=\overline{cAc}$ for some $c\in P(A)_+\setminus \{0\}.$
\end{cor}

\section{Completely positive linear maps from non-unital \CA s to non-unital \CA s}

The following is a version of  Proposition 2.2 of \cite{AAP}.

\begin{lem}\label{532}
Let $A$ be a separable \CA\, and let $\phi: A\to \C$ be a pure state.
For any $\ep>0$ and any finite subset ${\cal F}\subset A,$
there exist $z_1, z_2\in A$ with $z_2\ge 0$ and $\|z_1\|=\|z_2\|=1$ such that
\beq\nonumber
&&\|\phi(a)z_1^*z_1-z_1^*az_1\|<\ep\andeqn
\|\phi(a)z_2-z_2^{1/2}z_1^*az_1z_2^{1/2}\|<\ep\rforal a\in {\cal F},\,\,\andeqn\\\nonumber
&& z_2z_1^*z_1=z_1^*z_1z_2=z_2.
\eneq
\end{lem}

\begin{proof}
It follows from Proposition 2.2 of \cite{AAP} that there is an element $x_0\in A_+$ with
$\|x_0\|=1$ such that
$$
\|\phi(a)x_0^2-x_0ax_0\|<\ep/8\rforal a\in {\cal F}.
$$
Let $f(t)\in C_0((0,1])_+$ be defined by
\beq
f(t)=\begin{cases}  
                                      f(t)={t\over{1-\ep/5}}, &\text{if}\, t\in [0, 1-\ep/5);\\
                                       f(t)=1 & \text{if}\, t\in [1-\ep/5, 1].\end{cases}
                                       \eneq
Then
$$
\|f(x_0)-x_0\|<\ep/4.
$$
Put $z_1=f(x_0).$
Then
$$
\|\phi(a)z_1^2-z_1az_1\|<\ep\rforal a\in {\cal F}.
$$
Let $g(t)\in C_0((0,1])_+$ such that
\beq
g(t)=\begin{cases}  f(t)=0, &\text{if}\, t\in [0,1-\ep/4];\\
                                      f(t)={t-1+\ep/4\over{\ep/4}}, &\text{if}\, t\in [1-\ep/4, 1]
                                      \end{cases}
                                       \eneq

Let $z_2=g(x_0).$

Then $\|z_2\|=1$ and
$$
z_2z_1^2=z_1^2z_2=z_2\andeqn z_2z_1=z_1z_2=z_2.
$$
Then
$$
\|\phi(a)z_2-z_2^{1/2}z_1^*az_1z_2^{1/2}\|=\|(z_2^{1/2}(\phi(a)z_1^2-z_1az_1)z_2^{1/2}\|<\ep
$$
for all $a\in {\cal F}.$


\end{proof}

\begin{def}\label{Ddn}
Let $A$ and $B$ be two \CA\, and let $j: A\to B$ be an embedding.
For each integer $n\ge 1,$ denote by
$d_n: A\to M_n(B)$ defined by
$$
d_n(a)={\diag}(\overbrace{j(a), j(a),...,j(a)}^n)\rforal a\in A.
$$

\end{def}

\begin{def}\label{Dlocalapp}
Let $A$ be a \CA\, and $S\subset A$ be a subset.
We say $S$ has a countable local approximate identity $\{e_n\}$ in $A,$ if
each $e_n$ is a positive element in $A$ with
$0\le e_n\le 1$  such that
$e_ne_{n+1}=e_ne_{n+1}=e_n$ and $\lim_{n\to\infty}\|e_ns-s\|=\lim_{n\to\infty}\|se_n-s\|=0.$
\end{def}

\begin{lem}\label{535}
Let  $B$ be a \CA, $A$ be a separable \SCA\, of $B$ and  let $j: A\to B$ be a full embedding.
Let   $\{e_n\}\subset B$ be a countable local  approximate identity for some subset $S.$
Then for any state $\phi: A\to \C,$ any $\ep>0,$  any finite subset ${\cal F}\subset A,$
there exists
an integer $k\ge 1$ and $0<r<1,$
a contraction $V\in M_n(B)$ for some integer $n\ge 1$ such that
\beq\nonumber
\|\phi(a)e_k-V^*d_n(a)V\|<\ep \tforal a\in {\cal F},\\\nonumber
\andeqn  V^*V\le e_k^r\andeqn \|V^*V-e_k\|<\ep.
\eneq
(Here we identify $B$ with the upper left corner of $M_n(B).$)
Moreover, if $B$ is unital, then one can choose $e_k=1_B$ for all $k,$  and $V$ can be chosen so that $V^*V=1_B.$
Furthermore, if $\phi$ is only assumed to be contractive positive linear functional ($\|\phi\|\le 1$), then
the statement still holds (with $V$ is a contractive).
\end{lem}

\begin{proof}
To simplify the notation, \wilog, we may assume that
$\|a\|\le 1$ for all $a\in {\cal F}.$

We first assume that $\phi$ is pure. By \ref{532}, there are $z_1, z_2\in A$ with
$z_2\ge 0,$
$z_1^*z_1z_2=z_2z_1^*z_1=z_2$ and $\|z_i\|=1$ ($i=1,2$)
such that
\beq\label{535-nl}
\|\phi(a)z_1^*z_1-z_1^*az_1\|<\ep/2\andeqn \|\phi(a)z_2-z_2^{1/2}z_1^*az_1z_2^{1/2}\|<\ep/2\rforal a\in {\cal F}.
\eneq

Since $j: A\to B$ is full, the ideal generated by $A$ is $B.$
Therefore,  there are $x_1, x_2,...,x_n\in B$ such that
\beq\label{535-2}
\|\sum_{i=1}^nx_i^*z_2x_i-(1-\ep/4)e_k^{1-r}\|<\ep/4.
\eneq
Therefore
\beq\label{532-2+}
\|\sum_{i=1}^n e_k^{r/2} x_i^*z_2x_ie_k^{r/2}-(1-\ep/4)e_k\|<\ep/4.
\eneq
Let $W_1^*=(x_1^*z_2^{1/2}, x_2^*z_2^{1/2},...,x_k^*z_2^{1/2})$ (so $W_1$ is a column).
Note that
$$
W_1^*W_1=\sum_{i=1}^nx_i^*z_2x_i\andeqn \|W_1^*W_1\|\le \|(1-\ep/4)e_k^{1-r}\|+\ep/4\le 1.
$$
Define $W^*=(e_k^{r/2}x_1^*z_2^{1/2}, e_k^{r/2}x_2^*z_2^{1/2},...,e_k^{r/2}x_k^*z_2^{1/2})$ (so $W$ is also a column).
Note that
\beq\label{535-3}
W^*W=\sum_{i=1}^ke_k^{r/2}x_i^*(z_2)x_ie_k^{r/2}\le e_k^{r/2}W_1^*W_1e_k^{r/2}\le e_k^r\andeqn\\
\|W^*W\|\le \|(1-\ep/4)e_k\|+\ep/4\le 1.
\eneq

Define $V=Wz_1.$
Since $z_2z_1^*z_1=z_2=z_1^*z_1z_2,$
$$
W^*W=V^*V.
$$
In particular, $\|V^*V\|\le 1,$  $V^*V\le e_k^r$ and
\beq\label{535-nl2}
\|V^*V-e_k\|=\|W^*W-e_k\|<\ep/2.
\eneq
Note that
$$
W^*d_k(z_1^*z)W=\sum_{i=1}^ke_k^{r/2}x_i^*(z_2)^{1/2}z_1^*z_1(z_2)^{1/2}x_ie_k^{r/2}=W^*W.
$$
Then,   by \eqref{535-nl2} and \eqref{535-nl},
\beq\nonumber
\|\phi(a)e_k-V^*d_k(a)V\|&=&\|\phi(a)e_k-\phi(a)W^*d_k(z_1^*z_1)W\|\\
&&+\|\phi(a)W^*d_k(z_1^*z_1)W-V^*d_k(a)V\|\\
&\le & \ep/2+\|\phi(a)W^*d_k(z_1^*z_1)W- W^*d_k(z_1^*az_1)W\|\\
&<&\ep/2+\|\phi(a)d_k(z_1^*z_1)-d_k(z_1^*az_1)\|\\
&<& \ep/4+\ep/2<\ep
\eneq
for all $a\in {\cal F}.$

For a general  state $\phi,$ by the Krein-Milman theorem, we have positive numbers
$\af_1, \af_2,...,\af_m$ with $\sum_{i=1}^m \af_i=1$ and
pure states $\phi_1, \phi_2,...,\phi_m$ such that
$$
\|\phi(a)-\sum_{i=1}^m \af_i\phi_i(a)\|<\ep/2\rforal a\in {\cal F}.
$$

Let $k(i)$ be the integer $k$ in the first part of the proof corresponding to
$\phi_i,$ $i=1,2,...,m.$ Set $n(i)=\sum_{j=1}^ik(j)$ with $n(0)=0.$
Let $V_i\in M_{n(m)}(B)$ be given in the first part of the proof such that
\beq\nonumber
&&\|V_i^*V_i-e_k\|<\ep/2,\,\,\, V_i^*V_i\le e_k^r, \\
&& V_iV_i^*\le \diag(\overbrace{0,0,...,0}^{n(i-1)}, d_{k(i)}(1), 0,...,0)\andeqn
\|\phi(a)e_k-V_i^*d_{k(i)}(a)V_i\|<\ep/2,
\eneq
$i=1,2,...,m.$
Set $V=\sum_{i=1}^m\sqrt{\af_i}V_i\in M_{n(m)}(B).$ Then
$$
V^*V=\sum_{i=1}^m \af_iV_i^*V_i\le e_k^r.
$$
Moreover,
\beq\nonumber
\|\phi(a)e_k-V^*d_{n(m)}(a)V\| &=&\|\phi(a)e_k-\sum_{i=1}^m \af_iV_i^*d_{k(i)}(a)V_i\|\\\nonumber
&\le & \|\phi(a)e_k-\sum_{i=1}^m\af_j\phi_i(a)e_k\|+\sum_{i=1}^m\af_i\|\phi_i(a)e_k-V_i^*d_{k(i)}V_i\|\\\nonumber
&<& \ep/2+\ep/2=\ep
\eneq
for all $a\in {\cal F}.$

If $B$ is unital, we can assume that $e_k=1_B.$

Moreover, \eqref{535-2} can be replaced by
$$
\sum_{i=1}^nx_i^*z_2x_i=1_B.
$$
For the last part of the statement, we note that, there is $0\le \lambda\le 1$ such that  $\phi=\lambda\psi$
for some state $\psi.$
\end{proof}

\begin{lem}\label{535A}
Let $A$ be a separable \CA\, and let $j: A\to B$ be a full embedding.
Let $\{e_n\}$ be an approximate identity for $B.$
Then for any state $\phi: A\to \C,$ any $\ep>0,$  any finite subset ${\cal F}\subset A,$
there exists an integer $k_0\ge 1$ such that,
any integer $k\ge k_0,$
there exists
a partial isometry  $V\in {\widetilde{M_n(B)}}$ for some integer $n\ge 1$ such that
\beq\nonumber
\|\phi(a)e_k-V^*d_n(a)V\|<\ep \tforal a\in {\cal F}
\andeqn  V^*V=1_{{\widetilde B}}.
\eneq
(Here we identify $B$ with the first corner of $M_n(B).$)
\end{lem}

\begin{proof}
Without loss of generality, we may assume
that
$e_{n+1}e_n=e_ne_{n+1}=e_n$ for all $n.$
We may also assume, \wilog, $1>\ep>0$ and
$\|a\|\le 1$ for all $a\in {\cal F}.$

Let $\{a_n\}$  be  an approximate identity  for $A.$
There exists $m_0\ge 1$ such that, for any $m\ge m_0,$
\beq\label{536-1}
\phi(a_m)>1-\ep/32\andeqn \|a_ma-a\|<\ep/16\rforal a\in {\cal F}.
\eneq
Recall that we identify $B$ with the first corner of $M_n(B).$
There exists an integer $k_0\ge 2$ such that, for all $k\ge k_0,$
\beq\label{536-2}
\|a(1-e_k)^{1/2}\|<\ep/32\rforal a\in {\cal F}\cup\{a_{m_0}^{1/2}\}.
\eneq
It follows from \ref{535} that there is a contraction $V_1\in M_n(B)$ such that
\beq\label{536-3}
\|\phi(a)e_k-V_1^*d_n(a)V_1\|<\ep/32\rforal a\in {\cal F}.
\eneq
Define
$V_2=(1-e_k)^{1/2}+d_n(a_{m_0}^{1/2})V_1\in {\widetilde{M_n(B)}}.$
Note, by our notation,  that
$$
(1-e_k)^{1/2}=\diag((1-e_k)^{1/2},\overbrace{0,...,0}^{n-1}).
$$
Then, by \eqref{536-2} and \eqref{536-3},
\beq\nonumber
V_2^*V_2&=&((1-e_k)^{1/2}+V_1^*d_n(a_{m_0})^{1/2})((1-e_k)^{1/2}+d_n(a_{m_0}^{/12})V_1)\\\nonumber
&=& (1-e_k)+(1-e_k)^{1/2} d_n(a_{m_0})V_1+V_1^*d_n(a_{m_0}^{1/2})(1-e_k)^{1/2}+V_1^*d_n(a_{m_0})V_1\\\nonumber
&\approx_{\ep/16}&(1-e_k)+V_1^*d_n(a_{m_0})V_1\\\nonumber
&\approx_{\ep/32}& (1-e_k)+\phi(e_{m_0})e_k\\\nonumber
&\approx_{\ep/32}& (1-e_k)+e_k=1.
\eneq
There is $T\in B_+$ such that
$$
TV_2^*V_2T=1\andeqn \|T-1\|<\ep/8.
$$
Define $V=V_2T.$  Then $V^*V=TV_2^*V_2T=1.$
We estimate that, for all $a\in {\cal F},$   applying also \eqref{536-2},
\beq\label{536-5}
V^*d_n(a)V&\approx_{\ep/4}& V_2^*d_n(a)V_2\\
&\approx_{\ep/16}&(1-e_k)^{1/2}d_n(a)V_2+V_1^*d_n(a)(1-e_k)^{1/2}+ V_1^*d_n(a)V_1\\
&\approx_{\ep/16}  & V_1^*d_n(a)V_1.
\eneq
We also have that, for all $a\in {\cal F},$ by the above,
\beq\label{536-6}
\|\phi(a)e_k-V^*d_n(a)V\| <
\|\phi(a)e_k-V_1^*d_n(a)V_1\|+2\ep/16+\ep/4<\ep.
\eneq

\end{proof}

\begin{lem}\label{538}
Let $A$ be a separable \CA, $B$ be a $\sigma$-unital \CA\, and let $j: A\to B$ be a full embedding.
Let $\{e_n\}$ be an approximate identity of $B.$
Then, for any non-zero contractive completely positive linear map $\phi: A\to M_n(\C),$ any finite subset
${\cal F}\subset A,$  any $\ep>0,$
and any
integer $k\ge 1$ and $0<r<1,$ there exists a contraction  $V\in M_{Kn}(B)$
(for some large integer $K\ge 1$) such that
\beq\nonumber
&&\|\phi(a)d_n(e_k)-V^*d_K(a)V\|<\ep\tforal a\in {\cal F}\tand\\
&&V^*V\le d_n(e_k^r),
\eneq
where $M_n(\C)$ is viewed as the scalar matrix subalgebra of $M_n({\tilde B}).$
\end{lem}

\begin{proof}
Let $\{a_n\}$ be an approximate identity of $A.$ We will identify $A$ with $j(A).$

Write
$\phi(a)=\sum_{i=1}^n\phi_{ij}(a)\otimes e_{i,j}$ ($a\in A$), where
$\{e_{i,j}\}$ is a system of matrix units for $M_n$ and
$\phi_{i,j}: A\to\C$ is linear.

Define $\psi: M_n(A)\to \C$ by
$\psi((a_{i,j})_{n\times n})=(1/n)\sum_{i,j=1} \phi_{i,j}(a_{i,j}),$ where $a_{i,j}\in A.$
Since $\phi$ is completely positive, $\psi/\|\psi\|$ is a state.

Note that $M_n(A)\to M_n(B)$ is full.
\Wlog, we may assume that ${\cal F}$ is in the unit ball of $A.$
Fix $m_0\ge 1.$ Choose $m_1\ge 1$ such that,
for all $m\ge m_1,$
\beq\label{538-1-1}
\|a_m^{1/2}ae_m^{1/2}-a\|<\ep/16n^2\rforal a\in {\cal F}\cup\{e_{m_0}\}.
\eneq
Set
$${\cal F}_1={\cal F}\cup\{e_j: 1\le j\le m_1\}\cup\{a_j^{1/2}aa_j^{1/2}: a\in {\cal F},\, 1\le j\le m_1\}$$ and
set ${\cal G}=\{(a_{i,j})_{n\times n}: a_{i,j}\in {\cal F}_1\}.$

Note that $\{e_k^{1-r}\}$ is also an approximate identity for $A.$
Thus, by applying \ref{535}, there is $W\in M_{Kn}(B)$ with
$W^*W\le \|\psi\|d_n(e_k^r)$ and
$$
\|W^*W-\|\psi\|d_n(e_k^{1-r})\|<\ep/(16n)^3\andeqn \|\psi(b)d_n(e_k^{1-r})-W^*d_{Kn}(b)W\|<\ep/(16n)^3
$$
for all $b\in {\cal G}.$
So $W$ is a $K\times 1$-column of $n\times n$ matrix of $B.$
Denote by $\{e_{i,j}\}$ the matrix unit for $n\times n$ matrix.
Define $W_1=W(1\otimes e_{1,1})$ and identify $A$ with the first corner of $M_n(A).$
So now we view $W_1$ as a $Kn$-column.
We have $W_1^*W_1\le \|\psi\|e_k^r,$
$$
\|W_1^*W_1-\|\psi\|e_k^{1-r}\|<\ep/(16n)^3\andeqn  \|\psi(b)e_k^{1-r}-W_1^*d_K(b)W_1\|<\ep/(16n)^3
$$
for all $b\in {\cal G}.$

Let $v_i=(c_{k,j})$ be a $n\times n$ with $c_{k,i}=1,$ $k=1,2,...,n$ and $c_{k,j}=0$
if $j\not=i.$ Define ${\tilde v}_i=d_K(v_i),$ $i=1,2,...,n.$
 Note that, for $a\in A,$
${\tilde v_i}^*d_K(a){\tilde v}_j=d_K((b_{s,t})_{n\times n}),$
where $b_{i,j}=na$ and $b_{s,t}=0$ if $(s,t)\not=(i,j).$  Therefore
$$
\|\phi_{i,j}(a)e_k^{1-r}-W_1^*{\tilde v_i}^*d_{Kn}(a){\tilde v}_jW_1\|<\ep/(16n)^2
$$
for all $a\in {\cal F}.$

Define
$$
V_0=({\tilde v}_1W_1, {\tilde v}_2W_1,...,{\tilde v}_nW_1).
$$
Then
$$
V_0^*V_0=(W_1^*{\tilde v}_i^*{\tilde v}_jW_1)_{n\times n}=\sum_{i,j}(W_1^*{\tilde v}_i^*{\tilde v}_jW_1)\otimes e_{i,j},
$$
$$
 V_0^*d_K(a)V_0 =(W_1^*{\tilde v}_i^*d_K(a){\tilde v}_jW_1)_{n\times n}.
 $$
 Therefore
 \beq\nonumber
 \|\phi(a)d_n(e_k^{1-r})-V_0^*d_{Kn}(a)V_0\| &=&\|\sum_{i,j}(\phi_{i,j}(a)e_k^{1-r})\otimes e_{i,j}-\sum_{i,j}(W^*{\tilde v}_i^*d_K(a){\tilde v}_jW_1)\otimes e_{j,j}\|\\\label{538-10}
 &<&n^2\ep/16n^2=\ep/16\rforal a\in {\cal F}_1.
 \eneq
 Choose $V_1=d_{Kn}(a_{m_1}^{1/2})V_0d_n(e_k^{r/2}).$
 In  \eqref{538-10}, replacing $a$ by $b\hspace{-0.03in}:=a_{m_1}^{1/2}aa_{m_1}^{1/2},$
 we have, by \eqref{538-1-1} and \eqref{538-10},
 \beq\nonumber
 \|\phi(a)d_n(e_k)-V_1^*d_K(a)V_1\| &=& \|\phi(a)d_n(e_k)-\phi(a_{m_1}^{1/2}aa_{m_1}^{1/2})d_n(e_k)\|\\\nonumber
   &&\hspace{0.2in}+\|\phi(b)d_n(e_k)-d_n(e_k^{r/2})V_0^*d_{Kn}(b)V_0d_n(e_k^{r/2})\|\\\nonumber
   &&\hspace{0.6in}+\|d_n(e_k^{r/2})V_0^*d_K(b)V_0d_n(e_k^{r/2})-V_1^*d_{Kn}(a)V_1\|\\\nonumber
   &<& \ep/16n^2 +\ep/16+0<\ep/8
 \eneq
for all $a\in {\cal F}.$
We also have, by \eqref{538-10},
\beq\nonumber
\|\phi(a_{m_1})d_n(e_k)-V_1^*V_1\| &= &
\|\phi(a_{m_1})d_n(e_k)-d_n(e_k^{r/2})V_0^*d_K(a_{m_1})V_0d_n(e_k^{r/2})\|\\
&\le &\|\phi(a_{m_1})d_n(e_k^{1-r})-V_0^*d_{Kn}(a_{m_1})V_0\|\\
&<&
\ep/16.
\eneq
It follows that
$$
\|V_1^*V_1\|\le 1+\ep/16.
$$
Put $V=V_1/\|V_1\|.$ Then
$$
V^*V\le d_n(e_k^r).
$$
Moreover,
\beq\nonumber
\|\phi(a)d_n(e_k)-V^*d_{Kn}(a)V\| &\le &\|\phi(a)d_n(e_k)-V_1^*d_{Kn}(a)V_1\|\\
 &&\hspace{0.02in}\|V_1^*d_{Kn}(a)V_1-V^*d_{Kn}(a)V\|\\
 &<&\ep/16 +|1-{1\over{1+\ep/16}}|<\ep/8
 \eneq
 for all $a\in {\cal F}.$

\end{proof}

\begin{lem}\label{538A}
Let $A$ be a separable \CA, $B$ be a $\sigma$-unital \CA\, and let $j: A\to B$ be a full embedding.
Let $\{e_n\}$ be an approximate identity of $B.$
Then, for any non-zero contractive completely positive linear map $\phi: A\to M_n(\C),$ any finite subset
${\cal F}\subset A$ and any  $\ep>0,$
there exists an integer $k_0\ge 1$ such that, for any
integer $k\ge k_0,$  there exists a contraction  $V\in M_{Kn}(B)$
(for some large integer $K\ge 1$) such that
\beq\nonumber
\|\phi(a)d_n(e_k)-V^*d_{Kn}(a)V\|<\ep\tforal a\in {\cal F},
\eneq
where $M_n(\C)$ is viewed as scalar matrix subalgebra of $M_n({\tilde B}).$

If furthermore,  $e_{n+1}e_n=e_n=e_ne_{n+1}$ and $e_{n+1}-e_n$ is full in $B$
for all $n,$  then $V$ can be chosen  in $M_{Kn}({\tilde B})$
so that $V^*V=1_{M_n}.$

\end{lem}

\begin{proof}
The proof is a modification of that of \ref{538}.
Again we assume that $\|\phi\|=1$ and identify $j(A)$ with $A.$
We may assume that $0<\ep<1$ and   $\|a\|\le 1$ for all $a\in {\cal F}.$
Let $\{a_n\}$ be an approximate identity of $A.$ Let $C=\overline{\cup_{n=1}^{\infty} \phi(a_n)M_n\phi(a_n)}.$
Then $C$ is a hereditary \SCA\, of $M_n.$ Therefore $C\cong M_{n_1}.$
\Wlog, by replacing $\phi$ by ${\rm Ad}\, U\circ \phi$ for some
unitary $U\in M_n,$  we may assume that $C=M_n$ and $n_1=n.$

Note that $\{\phi(a_n)\}$ is increasing and strongly converges to a positive element in $M_{n}.$
Since $M_n$ is finite dimensional,  $\{\phi(a_n)\}$ converges in norm.
We will denote the limit by $\phi(1).$ As we assumed,  $\phi(1)$ is  an invertible positive element.
\Wlog, we may write
$$
\phi(1)=\diag(\lambda_1, \lambda_2,...,\lambda_n),
$$
where $0< \lambda_i\le 1,$ $i=1,2,...,n.$

Choose $m_0\ge 1$ such that, for all $m\ge m_0,$
\beq\label{538A-1}
\|\phi(1)-\phi(a_m)\|<\ep/32\andeqn \|a_m^{1/4}aa_m^{1/4}-a\|<\ep/32 \rforal a\in {\cal F}.
\eneq
Choose $k_0\ge 1$ such that, for all $k\ge k_0,$
\beq\label{538A-2}
\|e_kx-x\|<(\ep/32)^2\rforal x\in {\cal F}\cup\{a_{m_0}^{1/4}\}\cup \{\phi(a): a\in {\cal F}\}.
\eneq
Therefore, for $k\ge k_0,$
\beq\label{538A-2+}
\|(e_{k+1}-e_k)^{1/2}a\|^2 = \|a(e_{k+1}-e_k)a\|\le \|a(1-e_k)a\|<(\ep/32)^2.
\eneq
Fix $k\ge k_0.$
It follows from \ref{538} that there exists  an integer $K_1\ge 1$ and a contraction $V_1\in M_{K_1n}(B)$ such that
\beq\label{538A-3}
\|\phi(a)d_n(e_{k})-V_1^*d_{K_1n}(a)V_1\|<\ep/32\rforal a\in {\cal F}\cup\{a_{m_0}^{1/2}\}\andeqn V_1^*V_1\le e_{k}^{1/2}
\eneq
Now we assume that $e_{n+1}e_n=e_ne_{n+1}$ and $e_{n+1}-e_n$ is full in $B.$
Since $e_{k+1}-e_{k}$ is full, there are $x_1, x_2,...,x_{K_2}\in M_n(B)$ such that
\beq\label{538A-3+}
\|\sum_{i=1}^{K_2}x_i^*d_n(e_{k+1}-e_{k})x_i-(1_{M_n}-\phi(1))d_n(e_{k})\|<\ep/32.
\eneq
Define $V_2= (1-d_n(e_{k}))^{1/2}+d_n(a_{m_0}^{1/4})V_1.$
Then, by \eqref{538A-2}, \eqref{538A-3} and \eqref{538A-1},
\beq\label{538A-4}
V_2^*V_2 &=&(1-d_n(e_{k}))+(1-d_n(e_{k}))^{1/2}d_n(a_{m_0}^{1/4})V_1\\
&& +V_1^*d_n(a_{m_0}^{1/4})(1-d_n(e_{k}))^{1/2}+V_1^*d_n(a_{m_0})^{1/2}V_1\\
&\approx_{\ep/16}&(1-d_n(e_{k}))+V_1^*d_n(a_{m_0})^{1/2}V_1\\
&\approx_{\ep/32}& (1-d_n(e_{k}))+\phi(a_{m_0}^{1/2})d_n(e_{k})\\
&\approx_{\ep/32}& (1-d_n(e_{k}))+\phi(1)d_n(e_{k}).
\eneq
Define
$$
W^*=(V_2^*, x_1^*d_n(e_{k+1}-e_{k})^{1/2}, x_2^*d_n(e_{k+1}-e_{k})^{1/2},...,x_{K_1}^*d_n(e_{k+1}-e_{k})^{1/2})
$$
which will be viewed as an element in $M_{(K_1+K_2)n}({\tilde B}).$
Then
\beq\label{538A-5}
W^*W &=&V_2^*V_2+\sum_{i=1}^{K_2}x_i^*d_n(e_{k+1}-e_{k})x_i\\
&\approx_{\ep/8} & (1_{M_n}-d_n(e_{k}))+\phi(1)d_n(e_{k}) +\sum_{i=1}^{K_1}x_i^*d_n(e_{k+1}-e_{k})x_i\\
&\approx_{\ep/32} &(1_{M_n}-d_n(e_{k}))+\phi(1)d_n(e_{k}) +(1-\phi(1))d_n(e_{k})=1_{M_n}
\eneq
In particular,
\beq\label{538A-6}
\|W^*W\|\le 1+5\ep/32.
\eneq
Put $K=K_1+K_2$ and
$$
W_1^*=(x_1^*d_n(e_{k+1}-e_{k})^{1/2}, x_2^*d_n(e_{k+1}-e_{k})^{1/2},...,x_{K_1}^*d_n(e_{k+1}-e_{k})^{1/2}).
$$
By \eqref{538A-3+},
\beq\label{538A-7}
\|W_1^*W_1\|\le 1+\ep/32.
\eneq
We may write $W^*=(V_2^*, W_1^*).$ We note that $(e_{k+1}-e_k)e_{k_0}=0.$
Then, for  $a\in {\cal F},$ by \eqref{538A-3}
and \eqref{538A-2},
\beq\nonumber
\|\phi(a)d_n(e_k)-W^* d_{Kn}(a)W\| &\le & \|\phi(a)d_n(e_k)-V_2^*d_{K_1n}(a)V_2\|+\|W_1^*d_{K_2n}(a)W_1\|\\\nonumber
&<&  \ep/32 +\|W_1^*d_{K_2n}(e_{k_0}a)W_1\|+\ep/32\|W_1^*W_1\|\\
&<&\ep/32+0+\ep/32(1+\ep/32)<\ep/16.
\eneq
There exists $W_2\in M_{Kn}({\tilde B})$ such that
$$
\|W_2-1_{M_n}\|<\ep/4\andeqn W_2^*W^*WW_2=1_{M_n}.
$$
Set $V=WW_2.$ Then $V^*V=1_{M_n}$ and
$$
\|V-W\|<\ep/2.
$$
One verifies that $V$ meets the requirements.
\end{proof}

The following is a  non-unital version of  a result of Kirchberg.

\begin{lem}\label{542}
Let $A$ be a separable \SCA\, of a \CA\, $B$ such that the embedding $j: A\to B$ is full and let $\phi: A\to B$
be a \morp\, which is amenable.
Then, for any $\ep>0$ and any finite subset ${\cal F}\subset A,$ there exists a contraction   $V\in M_N({\tilde B})$
for some integer $N\ge 1$ such that
$$
\|\phi(a)-V^*d_N(a)V\|<\ep\rforal a\in {\cal F}.
$$
If we write $V^*=(x_1^*,x_2^*,...,x_N^*)$ with $x_i\in {\tilde B}, $ then
$\sum_{i=1}^Nx_i^*x_i\le 1$
and
$$
\|\phi(a)-\sum_{i=1}^N x_i^*ax_i\|<\ep\rforal a\in {\cal F}.
$$
Moreover, if we assume that $B$ contains an approximate identity  $\{e_n\}$ such that $e_{n+1}-e_n$ is full in $B,$
we may further require that $V^*V=1_{{\tilde B}}$ and
$\sum_{i=1}^Nx_i^*x_i=1_{\tilde B}.$
\end{lem}

\begin{proof}
Fix a finite subset ${\cal F}\subset A$ and $\ep>0.$ Since $\phi$ is assumed to be amenable,
\wilog, we may assume that $\phi=\psi\circ \sigma,$ where
$\sigma: A\to M_n$ and $\psi: M_n\to B$ are \morp s.
As in the proof of \ref{538A}, we may also assume that $C=M_n,$
where $C=\overline{\cup_{n=1} \sigma(a_n)M_n\sigma(a_n)}$ and where
$\{a_n\}$ is an approximate identity for $A.$

We identify $M_n\oplus \C$ with a unital \SCA\,  and define  a unital \morp\, $s: M_{n+1}\to M_n\oplus \C$
defined by
$$
s(a)=e_{1,1}ae_{1,1}+(1-e_{1,1})a(1-e_{1,1})\rforal a\in M_{n+1},
$$
where $e_{1,1}$ is the rank one minimal projection corresponding to the first corner of $(n+1)\times (n+1)$
matrices.
Define ${\tilde \psi}: M_n\oplus \C\to {\tilde B}$ by
$$
{\tilde \psi}(a+\lambda\cdot 1)=\psi(a)+\lambda\cdot (1_{{\tilde B}}-\psi(1_{M_n}))\rforal a\in M_n\andeqn \lambda\in \C.
$$
It is ready to verify that ${\tilde \psi}$ is a \morp\, (see, for example the proof of 2.27 of \cite{Lnbk}).
Define ${\tilde{\psi}}_1: M_{n+1}\to {\tilde B}$ by
$$
{\tilde \psi}_1={\tilde \psi}\circ s.
$$

Let $\{e_{i,j}\}$ be a system of matrix units for $M_n.$
Consider $(n+1)^2\times (n+1)^2$ matrix $(e_{i,j})_{n\times n}.$ It is a positive matrix.
Since ${\tilde \psi}_1$ is completely positive, $({\tilde \psi}_1(e_{i,j}))_{(n+1)\times (n+1)}$ is positive in $M_{n+1}({\tilde B}).$
Let $(r_{i,j})_{(n+1)\times (n+1)}\in M_{n+1}({\tilde B})$ be the square root of $({\tilde \psi}_1(e_{i,j}))_{(n+1)\times (n+1)}.$
Note that $r_{i,j}\in B.$ Then
\beq\label{542-1-1}
{\tilde \psi}_1(e_{i,j})=\sum_{k=1}^{n+1}r_{k,i}^*r_{k,j}.
\eneq
Let
\beq\label{542-1-2}
R_k^*=(r_{k,1}^*,r_{k,2}^*,...,r_{k,n+1}^*),\,\,\,k=1,2,...,n+1,\andeqn
R^*=(R_1^*, R_2^*,...,R_{n+1}^*).
\eneq
Then
\beq
\label{542-1-3}
{\tilde \psi}_1((\lambda_{i,j})_{(n+1)\times (n+1)})&=&\sum_{i,j}\lambda_{i,j}{\tilde \psi}_1(e_{i,j})
=\sum_{i,j,k}\lambda_{i,j}r_{k,i}^*r_{k,j}\\\label{542-1-3+}
&=&R^*d_{n+1}((\lambda_{i,j})_{(n+1)\times (n+1)})R
\eneq
for all $(\lambda_{i,j})_{(n+1)\times (n+1)}\in M_{n+1}.$
Note that
$$
{\tilde \psi}_1(1_{M_{n+1}})=1_{\tilde B}.
$$
Thus
\beq\label{542-1-4}
R^*R=1_{\tilde B}.
\eneq
Then, by \eqref{542-1-3}, for all $a\in A,$
\beq\label{542-2+1}
\phi(a)=\sum_{i,j}\sigma_{i,j}(a)\psi(e_{i,j})=\sum_{i,j,s}\sigma_{i,j}(a)r_{s,i}^*r_{s,j}=R^*d_{n+1}(\sigma(a))R.
\eneq

Suppose that $\{e_m\}$ is an approximate identity for $B.$
Choose $m_0\ge 1$ such that
\beq\label{542-1}
\|\sum_{k=1}^{n+1} r_{k,i}^*e_mr_{k,j}-\sum_{k=1}^{n+1}r_{k,i}^*r_{k,j}\|<\ep/2(n+1)^3.
\eneq
for any $m\ge m_0.$
Consequently,
\beq\label{542-10}
\|\phi(a)-R^*d_{n+1}(\sigma(a))d_{(n+1)^2}(e_{m_0})R\|<\ep/2.
\eneq
By \ref{538}, there exists a partial isometry  $W\in M_{Kn}({\tilde B})$
for some integer $K\ge 1$ such that
\beq\label{542-2}
W^*W=1_{M_n}\andeqn
\|\sigma(a)d_n(e_{m_0})-W^*d_{Kn}(a)W\|<\ep/2\rforal a\in {\cal F}.
\eneq
Define $W_1\in M_{K(n+1)}({\tilde B})$ by
$$
W_1=W\oplus e_{n+1, n+1}.
$$
Set $V=d_{n+1}(W_1)R.$
Then
$$
V^*V=R^*d_{n+1}(W_1^*)d_{n+1}(W_1)R=R^*d_{n+1}(1_{M_{n+1}})R=R^*R=1_{\tilde B}.
$$
Combining \eqref{542-10} and \eqref{542-2}, we have
\beq\nonumber
\|\psi(a)-V^*d_{(n+1)Kn}(a)V\|&=&(R^*d_{n+1}(W^*)d_{(n+1)Kn}(a)d_{n+1}(W)R)\|\\\nonumber
&\le & \|\phi(a)-R^*d_{n+1}(\sigma(a)d_n(e_{m_0}))R\|\\\nonumber
&&\hspace{-0.6in}+\|R^*(d_{n+1}(\sigma(a)d_n(e_m))-d_{n+1}(W^*d_{Kn}(a)W))R\|\\\nonumber
&<&\ep/2+\ep/2=\ep
\eneq
for all  $a\in {\cal F}.$ Choose $N=(n+1)Kn.$

\end{proof}

\section{Absorbing extensions}

\begin{lem}\label{L551}
Let $A$ be a separable \CA\, of a \CA\, $B$ so that the embedding $j: A\to B$ is full.
Suppose that $B$ is non-unital .
Identifying the constant functions in $IB=C([0,1],B)$ with $B.$ Suppose that
$\sigma: A\to M_n(SB)$ is an amenable \morp\, for some integer $n\ge 1.$
Then, for any $\ep>0$ and any finite subset ${\cal F}\subset A,$ there exists an integer
$K\ge 1$  and a partial isometry $V\in M_{(K+1)n}({\widetilde{SB}})$ such that
\beq\label{L551-1}
&&\|\phi(a)-V^*d_{Kn}(a)V\|<\ep,\\
&&V^*d_{Kn}(a)V\in SB\rforal a\in {\cal F}\andeqn\\
&&V^*V=1_{M_{n}({\tilde{IB}})}.
\eneq
\end{lem}

\begin{proof}
We will repeat some of arguments in the proof of \ref{538A}.
Let $\{a_m\}$ be an approximate identity for $A.$
Choose $m_0\ge 1$ such that, for all $m\ge m_0,$
\beq\label{L551-2}
\|a_m^{1/4}aa_m^{1/4}-a\|<\ep/3\rforal a\in {\cal F}.
\eneq
Choose $k_0\ge 1$ such that, for all $k\ge k_0,$
\beq\label{L551-3}
\|e_kx-x\|<(\ep/3n)^2\rforal x\in {\cal F}\cup\{a_{m_0}^{1/4}\}\cup \{\phi(a): a\in {\cal F}\}.
\eneq
It follows from \ref{542} that there exists  an integer $K_1\ge 1$ and a contraction $V_1\in M_{K_1n}({\tilde B})$ such that
\beq\label{L551-4}
\|\phi(a)-V_1^*d_{K_1n}(a)V_1\|<\ep/3\rforal a\in {\cal F}\cup\{a_{m_0}^{1/2}\}
\eneq
Define
$$
V^*=(d_n(e_k)V_1^*d_{Kn}(\phi(a_{m_0}^{1/4}), (1_{M_n}-(d_n(e_k)V_1^*d_{K_1n}(\phi(a_{m_0}^{1/2}))V_1d_n(e_k)))^{1/2})\in M_{(K+1)n}({\widetilde{SB}}).
$$
Then
\beq\label{L551-6}
V^*V &=&
d_n(e_k)V_1^*d_n(\phi(a_{m_0})^{1/2})V_1d_n(e_k)+\\
&&1_{M_n}-(d_n(e_k)V_1^*d_{K_1n}(\phi(a_{m_0}^{1/2}))V_1d_n(e_k))\\
&=& 1_{M_n}.\\
\eneq
Moreover,  since $d_{Kn}(a_{m_0}^{1/4})\in M_{Kn}(SB),$
$$
V^*d_n(a)V=d_n(e_k)V_1^*d_{Kn}(a_{m_0}^{1/4}aa_{m_0}^{1/4})V_1d_n(e_k)\in M_{(K+1)n}(SB)
$$
for all $a\in A.$
One also has that, for each $x\in {\cal F},$ by \eqref{L551-2}, \eqref{L551-3} and \eqref{L551-4},
\beq\label{L551-7}
V^*d_n(a)V &=& d_n(e_k)V_1^*d_{Kn}(a_{m_0}^{1/4}aa_{m_0}^{1/4})V_1d_n(e_k)\\
&\approx_{\ep/3} & V_1^*d_{Kn}(a_{m_0}^{1/4}aa_{m_0}^{1/4})V_1\\
&\approx_{\ep/3}& V_1^*d_{Kn}(a)V_1\\
&\approx_{\ep/3}& \phi(a).
\eneq
\end{proof}


\begin{lem}\label{551}
Let $A$ be a separable \SCA\, of a \CA\, $B$ so that the embedding $j: A\to B$ is full.
Identifying the constant functions in $IB=C([0,1],B)$ with $B.$ Suppose that
$\sigma_n: A\to M_{k(n)}(SB)$ are amenable \morp s and
$\sigma: A\to M(SB\otimes {\cal K})$ is defined by
$$
\sigma(a)=\diag(\sigma_1(a), \sigma_2(a),...,\sigma_n(a),...)\rforal a\in A,
$$
where the convergence is in the strict topology. Then there exists a sequence
$V_n\in M_2(M( {\widetilde{IB}}\otimes {\cal K}))$ such that
$$
\sigma(a)-V_n^*d_{\infty}(a)V_n\in SB\otimes {\cal K}\andeqn
\lim_{n\to\infty}\|\sigma(a)-V_n^*d_{\infty}(a)V_n\|=0
$$
for all $a\in A.$  Moreover
$$
V^*V=
\diag(p_{k(1)}, p_{k(2)},...,p_{k(n)},...),
$$
where $p_{k(n)}=1_{M_{k(n)}(C[0,1], {\tilde B})},$ $n=1,2,....$
\end{lem}

\begin{proof} The proof is almost identical to that of Lemma 5.5.1 of \cite{Lnbk}.
Since  $A$ is not assumed to be unital, we will present a full proof here using the results in
previous section.

Fix a finite subset ${\cal F}$ and $\ep>0.$
Let ${\cal F}_1={\cal F},...,{\cal F}_n\subset {\cal F}_{n+1},$ $n=1,2,...,$
be a sequence of increasing finite subsets of $A$ such that
$\cup_{n=1}^{\infty} {\cal F}_n$ is dense in $A.$
By \ref{542}, there exists a sequence of partial isometries $V_n\in M_{K(n)+k(n)}({\widetilde{IB}})$
($K(n)>k(n)$) such that
$V_n^*V_n=1_{{\widetilde{M_{k(n)}(IB)}}},$
\beq\label{551-1}
&&\sigma_n(a)-V_n^*d_{K(n)}(a)V_n\in M_{k(n)}(SB)\rforal a\in A\andeqn\\\label{551-2}
&&\|\sigma_n(a)-V_n^*d_{K(n)}(a)V_n\|<\ep/2^{n+3}
\eneq
for all $a\in {\cal F}_n,$ $n=1,2,....$  We identify $1_{\widetilde{M_{k(n)}(IB)}}$ with
$1_{M_{k(n)}({\widetilde{IB}})}$ and denote it by
$p_{k(n)}.$ We may further write $p_{k(n)}=1_{M_{k(n)}({\widetilde{IB}}))}.$

Write
$$
d_{\infty}(a)=\diag(d_{K(1)}(a),d_{K(2)}(a),..., d_{K(n)}(a),...)\rforal a\in A.
$$
Put
$V=\diag(V_1, V_2,..., V_n,...),$ where the convergence is in the strict topology
in $M({\widetilde{SB}}\otimes {\cal K}).$ So $V\in M({\widetilde{SB}}\otimes {\cal K}).$
Then
$$
V^*V=\diag(p_{k(1)}, p_{k(2)},...,p_{k(n)},...),
$$
where convergence in the strict topology. Therefore $V^*V=1_{M({\widetilde{SB}}\otimes {\cal K})}.$
Put
$$
L_n(a)=V^*\diag(\overbrace{0,...,0}^{n-1},d_{K(n)}(a),0,...)V,\,\,\,
\sigma(a)=\sum_{n=1}^{\infty}\sigma_n(a)
$$
for all $a\in A.$

Then
$\sum_{n=1}^NL_n(a)$ converges to $V^*d_{\infty}(a)V$ and $\sum_{n=1}^N\sigma(a)$
converges to $\sigma(a)$ in the strict topology  (in $M(SB\otimes {\cal K})$) for all $a\in A,$  as $N\to\infty,$
respectively.
Therefore, as $N\to\infty,$
$$
\sum_{n=1}^NL_n(a)-\sum_{n=1}^N\sigma_n(a)\to V^*d_{\infty}(a)V-\sigma(a)\rforal a\in A
$$
in the strict topology.
By   \eqref{551-2},  for any fixed $k$ and any $a\in {\cal F}_k,$\\
$\sum_{n=1}^NL_n(a)-\sum_{n=1}^N\sigma(a)$ actually converges in norm,
as $N\to \infty.$  By \eqref{551-1} and \eqref{551-2}, for each $k,$
$$
V^*d_{\infty}(a)V-\sigma(a)\in SB\otimes {\cal K}\rforal a\in {\cal F}_k.
$$
Since $\cup_{n=1}^{\infty}{\cal F}_n$ is dense in $A,$  the above holds for all $a\in A.$
We also compute that
$$
\|V^*d_{\infty}(a)V-\sigma(a)\|< \ep\rforal a\in {\cal F}.
$$
\end{proof}



\begin{df}\label{tilde phi}
Let $A$ be a non-unital separable \CA\, and $B$ be a $\sigma$-unital \CA.
Since $SB\otimes {\cal K}$ is a \SCA\, of $S{\tilde B}\otimes {\cal K},$
we may view $SB\otimes {\cal K}\subset M(S{\tilde B}\otimes {\cal K}).$
Let $x\in M(S{\tilde B}\otimes {\cal K}).$ Then
$xa\in S{\tilde B}\otimes {\cal K}$ if $a\in  SB\otimes {\cal K}.$
Let $\{e_n\}$ be an approximate identity for $SB\otimes {\cal K},$
Then $xae_n\in SB\otimes {\cal K}$ for all $n.$
Since $xae_n\to xa,$ $xa\in SB\otimes {\cal K}.$ This implies that
we may write that $M(S{\tilde B}\otimes {\cal K})\subset M(SB\otimes {\cal K}).$

\end{df}

\begin{lem}\label{553}
Let $A$ be a separable \SCA\, of a $\sigma$-unital \CA\, $B$ so that the embedding
$j: A\to B$ is full. Identifying the constant functions in
$C([0,1],B)$ with $B.$ Suppose that $\phi: A\to M(SB\otimes {\cal K})$ is an amenable
\morp. Then there is a sequence of   isometries $V_n\in M_2(M({\widetilde{SB}}\otimes {\cal K}))$ such that
$$
\phi(a)-V_n^*d_{\infty}(a)V_n\in SB\otimes {\cal K}\andeqn
\lim_{n\to\infty}\|\phi(a)-V_n^*d_{\infty}(a)V_n\|=0
$$
for all $a\in A.$


\end{lem}

\begin{proof}
Fix a finite subset ${\cal F}\subset A$ and $\ep>0.$ Let ${\cal F}\subset {\cal F}_1\subset {\cal F}_2,...$
be an increasing sequence  of finite subsets of $A$ such that $\cup_{n=1}^{\infty} {\cal F}_n$ is dense in
$A.$ One can find an approximate identity $\{e_n\}$ for $SB\otimes {\cal K}$ such that
$e_n\in M_{L(n)}(SB).$  A subsequence of convex combination of $\{e_n\}$ forms a quasi-central
approximate identity for $SB\otimes {\cal K}.$ \Wlog, we may further assume that $\{e_n\}$ is quasi-central
and $e_n\in M_{k(n)}(SB)$ for a subsequence $\{k(n)\},$ $n=1,2,....$

 By passing to a subsequence, if necessary, we may assume that (with $e_0=0$),
 for all $a\in {\cal F}_n,$
 \beq\label{553-2}
 &&\|(e_n-e_{n-1})\phi(a)-\phi(a)(e_n-e_{n-1})\|<\ep/2^{n+2},\\\label{553-2+}
 &&\|(e_n-e_{n-1})^{1/2}\phi(a)-\phi(a)(e_n-e_{n-1})^{1/2}\|<\ep/2^{n+2}\andeqn\\
 &&\|(1-e_{n+1})e_n\|<\ep/2^{(n+4)},
 \eneq
$n=1,2,....$ Let $E_1=e_1$ and $E_n=(e_{n+1}-e_n)^{1/2}$ for $n>1.$
For each $n,$ there exists $m(n)\ge 1$ such that
\beq\label{553-d-1}
\|E_n^{1/m(n)}E_n-E_n\|<\ep/2^{2(n+4)},\,\,\, n=1,2,....
\eneq
Put
\beq\nonumber
E_n'=E_n^{1/m(n)},\,\,\,
\sigma_n'(a)=E_n'\phi(a)E_n'\andeqn \sigma_n(a)=E_n\phi(a)E_n
\eneq
for all $a\in A,$ $n=1,2,....$
 We also define
${\tilde \sigma}_n(a)=E_n{\tilde \phi}(a)E_n$  for all $a\in A,$
 and  $n=1,2,....$

Note
that $\sigma_n: A\to M_{k(n)}(SB)$ is also amenable.
Set
\beq\nonumber
&&\sigma'(a)=\diag(\sigma_1'(a), \sigma_2'(a),...,\sigma_n'(a),...)
\rforal a\in A.
\eneq
Moreover, one easily verifies that
$\psi(a)=\sum_{n=1}^{\infty}\sigma_n(a)$ converges for every
$a\in A$ (this also holds for ${\tilde \psi}(a)=\sum_{n=1}^{\infty}{\tilde \sigma}_n(a)$)
in the strict topology.
Define ${\tilde \psi}(1)=1_{M({\widetilde{SB}}\otimes {\cal K})}.$

Put
$$
W^*=(e_1^{1/2}, (e_2-e_1)^{1/2},...,(e_n-e_{n-1})^{1/2},....).
$$
Note that $\{(e_1^{1/2}, (e_2-e_1)^{1/2},...,(e_{n+1}-e_{n})^{1/2},0,0,...)\}$
converges strictly not only as elements in $M(SB\otimes {\cal K})$
but also as elements in $M({\widetilde{SB}}\otimes {\cal K}).$ Therefore
 $W\in M({\widetilde{SB}}\otimes {\cal K}).$ With $e_0=0,$
one has that
$$
W^*W=\sum_{n=1}^{\infty} (e_{n}-e_{n-1})=e_1+(e_2-e_1)+(e_3-e_2)+\cdots=1_{M(SB\otimes {\cal K})}
$$
and it converges strictly in both $M({\widetilde{SB}}\otimes {\cal K})$ and   $M(SB\otimes {\cal K}).$
Therefore (with $e_{-1}=0$), for all $a\in A,$
\beq\nonumber
&&\hspace{-0.3in}W^*\diag(\sigma_1'(a), \sigma_2'(a),...,\sigma_n'(a),0,...)W\\\nonumber
&=&\sum_{k=0}^n(e_{k}-e_{k-1})^{1/2}\sigma_{k+1}'(a)(e_{k}-e_{k-1})^{1/2}.
\eneq
Combining  this identity with \eqref{553-d-1}, one obtains that
\beq\label{553-nn+2}
W^*\sigma'(x)W-\sum_{n=1}^{\infty}\sigma_n(x)\in SB\otimes {\cal K}\rforal x\in A\andeqn\\
\|W^*\sigma'(a)W-\sum_{n=1}^{\infty}\sigma_n(a)\|<\ep/4\rforal a\in {\cal F}.
\eneq

Furthermore, $\psi: A\to M(SB\otimes {\cal K})$ is an amenable \morp.
For each $a\in {\cal F},$ by \eqref{553-2+},
\beq\label{553-3}
\|\phi(a)-\psi(a)\|&=&\|\sum_{n=1}^{\infty}(e_n-e_{n-1})\phi(a)-\sum_{n=1}^{\infty}E_n\phi(a)E_n\|\\
&\le &\sum_{n=1}^{\infty}\|E_n(E_n\phi(a)-\phi(a)E_n)\|<\sum_{n=1}^{\infty} \ep/2^{n+2}<\ep/4.
\eneq
Since
\beq\label{553-3n+1}
&&\sum_{n=1}^N((e_n-e_{n-1})\phi(a)-E_n\phi(a)E_n)\in M_{\sum_{i=1}^{N}(k(i)}(SB)
\eneq
for all $a\in A$ (or $b\in {\tilde A}$), and  (by  \eqref{553-2+} again)
\beq\nonumber
&&\|\sum_{n=N+1}^{\infty}(e_n-e_{n-1})\phi(a)-\sum_{n=N+1}^{\infty}\sigma_n(a)\|<\sum_{N+1}^{\infty}\ep/2^{n+2}\to 0,
\eneq
as $N\to\infty$ for all $a\in {\cal F}_k$ (for all $k$),
we conclude that
$\phi(a)-\psi(a)\in SB\otimes {\cal K}$
for all $a\in {\cal F}_k$ and for all $k.$ Therefore
\beq\label{553-nnn}
&&\phi(a)-\psi(a)\in SB\otimes {\cal K}\rforal a\in A
\eneq
It follows from \ref{551} that there exists an isometry $Z\in M({\widetilde{SB}}\otimes {\cal K})$ such that
\beq\label{553-4}
&&\sigma(a)-Z^*d_{\infty}(a)Z\in SB\otimes {\cal K}\rforal a\in A\andeqn\\\label{553-5}
&&\|\sigma(x)-Z^*\sigma_{\infty}(x)Z\|<\ep/4\rforal x\in  {\cal F}.
\eneq
Set $V=ZW.$ Then
$V\in M({\widetilde{SB}}\otimes {\cal K}).$
Moreover, for all $x\in A,$  by \eqref{553-nn+2}, \eqref{553-nnn} and by \eqref{553-4},
$$
\phi(a)-V^*d_{\infty}(x)V\in SB\otimes {\cal K}.
$$
Thus, for  $a\in {\cal F},$
\beq
\|\phi(a)-V^*d_{\infty}(a)V\| &\le& \|\phi(a)-\psi(a)\|+\|\psi(a)-W^*\sigma(a)W\|\\
&+&\|W^*\sigma(a)W-V^*\sigma_{\infty}(a)V\|<\ep.
\eneq

\end{proof}

\begin{df}
Let $A$ be a separable \CA\, and $B$ be a $\sigma$-unital  \CA\,
Suppose that $\phi, \psi: A\to M(B\otimes {\cal K}).$
There is an isometry $W\in M_2(M(B\otimes {\cal K}))$
with $W^*W=1_{M(B\otimes {\cal K})}$ and $WW^*=1_{M_2(M(B\otimes {\cal K}))}.$
Then we will identity $\phi\oplus \psi$ with $W^*(\phi\oplus \psi)W$ in the next statement.

\end{df}

\begin{thm}\label{Abs}
Let $A$ be a separable \SCA\, of $B$ so that
the embedding $j: A\to B$ is full and amenable.
Identify the constant functions in $C([0,1], B)$ with $B.$
Then $d_{\infty}: A\to M(SB\otimes {\cal K})$ gives an amenable
absorbing trivial extension of $A$ by $SB\otimes {\cal K}.$
Moreover, given any amenable monomorphism $\phi: A\to M(SB\otimes {\cal K})$ and any
finite subset ${\cal F}\subset A,$ there is a unitary  $V\in M({\widetilde{SB}}\otimes {\cal K})$ such that
\beq\nonumber
&&V^*(\phi(x)\oplus d_{\infty}(x))V-d_{\infty}(x)\in SB\otimes {\cal K}\rforal x\in A\andeqn\\
&&\|V^*(\phi(a)\oplus d_{\infty}(a))V-d_{\infty}(a)\|<\ep\rforal a\in {\cal F}.
\eneq

%

\end{thm}

\begin{proof}
Let $\phi: A\to M(SB\otimes {\cal K})$ be an  amenable monomorphism.
One defines  an amenable monomorphism
$\phi_{\infty}=\oplus_{n=1}^{\infty} \phi: A\to M(SB\otimes {\cal K})$ by
dividing the identity of $M(SB\otimes {\cal K})$ into countably many
copies of equivalent projections in $M(SB\otimes {\cal K})$ each of which is equivalent to
the identity of $M(SB\otimes {\cal K}).$

Let $\ep>0$ and
${\cal F}\subset A$ be a finite subset.
By \ref{553}, there is an isometry $s\in M_2(M({\widetilde{SB}}\otimes {\cal K}))$ such
that
\beq\label{Abs-1}
&&\phi_{\infty}(x)-s^*d_{\infty}(x)s\in SB\otimes {\cal K}\rforal x\in A \andeqn\\
 &&\|\phi_{\infty}(a)-s^*d_{\infty}(a)s\|<\ep/12\rforal a\in {\cal F}.
 \eneq
Put $ss^*=p$ which is a projection in $M_2(M({\widetilde{SB}}\otimes {\cal K})).$
Since $\phi_{\infty}$ is a \hm,
we have, for any $a\in A_{s.a.},$
$$
pd_{\infty}(a)pd_{\infty} (a)p-pd_{\infty}(a)p\in M_2(SB\otimes {\cal K})
\rforal a\in A.
$$
Therefore
$$
pd_{\infty}(a)(1-p)d_{\infty}(a)p\in M_2(SB\otimes {\cal K})\rforal a\in A.
$$
Since $SB\otimes {\cal K}$ is an ideal of $M(SB\otimes {\cal K}),$
this implies that
$$
pd_{\infty}(a)(1-p)\in M_2(SB\otimes {\cal K})\rforal a \in A.
$$
Therefore
$$
pd_{\infty}(a)-d_{\infty}(a)p\in M_2(SB\otimes {\cal K}) \rforal a\in A.
$$
We also assume that
$$
\|pd_{\infty}(a)-d_{\infty}(a)p\|<\ep/6\rforal a\in {\cal F}.
$$
Note that
 there is $U_1\in M_2(M({\widetilde{SB}}\otimes {\cal K})$
such that $U_1^*1_{M({\widetilde{SB}}\otimes {\cal K})}U_1=p.$
We also have
\beq\label{Abs-2}
&&(U_1s)\phi_{\infty}(x)(U_1s)^*-U_1pd_{\infty}(x)pU_1^*\in SB\otimes {\cal K}\rforal x\in A\andeqn\\
&&\|(U_1s)\phi_{\infty}(x)(U_1s)^*-U_1pd_{\infty}(x)pU_1^*\|<\ep/12\rforal a \in {\cal F}.
\eneq
We may also write, for all $x\in A,$
\beq\label{ABs-3}
&&d_{\infty}(x)-pd_{\infty}(x)p\oplus (1-p)d_{\infty}(x)(1-p)\in M_2(SB\otimes {\cal K})\andeqn\\.
&&\|d_{\infty}(a)-pd_{\infty}(a)p\oplus (1-p)d_{\infty}(a)(1-p)\|<\ep/6\rforal a\in {\cal F}.
\eneq
Put $\psi=(1-p)d_{\infty}(1-p).$ Then  there is a unitary $U_2\in M_2(M(\C\otimes {\cal K})$
such that
\beq\label{Abs-4}
&&\hspace{-0.4in}U_2^*d_{\infty}(x)U_2-s\phi_{\infty}(x)s^*\oplus \psi(x) \\
&&=
U_2^*(d_{\infty}(x)-(pd_{\infty}(x)p+\psi(x)))U_2+\\
&&(pd_{\infty}(x)p\oplus \psi(x)-s\phi_{\infty}(x)s^*\oplus \psi(x)\in M_2(SB\otimes {\cal K})\rforal x\in A \\
&&\andeqn\|U_2^*d_{\infty}(a)U_2-s\phi_{\infty}(a)s^*\oplus \psi(a)\|<\ep/3\rforal a\in {\cal F}.
\eneq
Put $U_3=s\oplus 1_{M({\widetilde{SB}\otimes {\cal K})}}.$ Then
$$
U_3^*U_3=s^*s\oplus 1_{M(\widetilde{SB}\otimes {\cal K})}=1_{M_2(M(\widetilde{SB}\otimes {\cal K}))}
$$
and
\beq\label{Abs-5}
&&U_3^*U_2^*d_{\infty}(x)U_2U_3-\phi_{\infty}(x)\oplus \psi(x)\in M_2(SB\otimes {\cal K})\rforal x\in A\andeqn\\
&&\|U_3^*U_2^*d_{\infty}(a)U_2U_3-\phi_{\infty}(a)\oplus \psi(a)\|<\ep/3\rforal a\in {\cal F}.
\eneq
Equivalently,
\beq\label{ABs-5+}
&&d_{\infty}(x)-U_2U_3(\phi_{\infty}(x)\oplus \psi(x))U_3^*U_2^*\in M_2(SB\otimes {\cal K})\rforal x\in A\andeqn\\
&&\|d_{\infty}(a)-U_2U_3(\phi_{\infty}(a)\oplus \psi(a))U_3^*U_2^*\|<\ep/3\rforal a\in {\cal F}.
\eneq

There is  $U_4\in M_2(M(\C\otimes {\cal K}))$ such that $U_4^*U_4=1_{M(\C\otimes {\cal K})},$
$U_4U_4^*=1_{M_2(M(\C\otimes {\cal K}))}$ and
$$
U_4^*\phi_{\infty}U_4=\phi\oplus \phi_{\infty}.
$$
Let $U_5=U_4\oplus 1_{M(\C\otimes {\cal K})}.$ Then, for all $x\in A,$
\beq\label{Abs-6}
&&U_5^*(U_3^*U_2^*d_{\infty}(x)U_2U_3)U_5-\phi(x)\oplus \phi_{\infty}(x)\oplus \psi(x)\in M_3(SB\otimes {\cal K})\andeqn\\
&&\|U_5^*(U_3^*U_2^*d_{\infty}(a)U_2U_3)U_5-\phi(a)\oplus \phi_{\infty}(a)\oplus \psi(a)\|<\ep/3
\rforal a\in {\cal F}.
\eneq
Let $U_6=1_{M({\widetilde{SB}}\otimes {\cal K})}\oplus (U_2U_3)^*.$
Then, for all $x\in A,$
\beq\label{Abs-7}
&&U_6^*(U_5^*(U_3^*U_2^*d_{\infty}(x)U_2U_3)U_5)U_6-\phi(x)\oplus d_{\infty}(x)\in M_2(SB\otimes {\cal K})\andeqn\\
&& \|U_6^*(U_5^*(U_3^*U_2^*d_{\infty}(a)U_2U_3)U_5)U_6-\phi(a)\oplus d_{\infty}(a)\|<\ep\rforal a\in {\cal F}.
\eneq

\end{proof}

\begin{df}
Let $A$ be a separable \CA\, and let $B$ be a $\sigma$-unital \CA.
Consider two essential extensions $\tau_1, \tau_2: A\to M(SB\otimes {\cal K})/SB\otimes {\cal K}.$
We write $\tau_1\approxeq \tau_2$ if there exists
a unitary $u\in M({\widetilde{SB}}\otimes {\cal K})\subset M(SB\otimes {\cal K})$ such that
$$
\pi(u)^*\tau_1\pi(u)=\tau_2,
$$
where $\pi: M(SB\otimes {\cal K})\to M(SB\otimes {\cal K})/SB\otimes {\cal K}$ is the quotient map.
We will use this convention in the next theorem and its proof.
\end{df}

\begin{thm}\label{Textabs}
Let $A$ be a separable  amenable \CA\, and $B$ be a $\sigma$ -unital \CA.
Suppose that $j: A\to B$ be a  full embedding.
Suppose that $\tau_1, \tau_2: A\to M(SB\otimes {\cal K})/SB\otimes {\cal K}$ are two  essential
extensions. Suppose that $[\tau_1]=[\tau_2]$ in $KK^1(A, SB).$
Then
$$
\tau_1\oplus \pi\circ d_{\infty}\approxeq \tau_2\oplus \pi\circ d_{\infty}.
$$
\end{thm}

\begin{proof}
By \cite{Arv}, there is an essential extension $\tau_1^*: A\to M(SB\otimes {\cal K})/SB\otimes {\cal K}$
such that
$$
\tau_1\oplus \tau_1^*
$$
is trivial.  It follows that $\tau\otimes \tau_1^*$ is trivial.
By \ref{Abs}, there is unitary $v\in M(SB\otimes {\cal K})$ such that
$$
\pi(v)^*(\tau_2\oplus \tau_1^*\oplus \pi\circ d_{\infty})\pi(v)=\pi\circ d_{\infty}.
$$
In particular, $\tau_2\oplus \tau_1^*\oplus \pi\circ d_{\infty}$ is trivial.
It follows \ref{Abs}, in fact,  that
\beq\label{Texabs-1}
\tau_2\oplus \tau_1^*\oplus \pi\circ d_{\infty}\approxeq \pi\circ d_{\infty}.
\eneq
We also have
\beq\label{Texabs-2}
\tau_1\oplus \tau_1^*\oplus \pi\circ d_{\infty}\approxeq \pi\circ d_{\infty}.
\eneq
Therefore
\beq\label{Texabs-3}
\tau_1\oplus (\tau_2\oplus \tau_1^*\oplus \pi\circ d_{\infty}) &\approxeq& \tau_1\oplus \pi\circ d_{\infty}\\
&\approxeq& \tau_1\oplus (\tau_1\oplus \tau_1^*\oplus \pi\circ d_{\infty})\\
&\approxeq& \tau_1\oplus \pi\circ d_{\infty}.
\eneq
However,
it is easy to check that
\beq\label{Texabs-4}
\tau_2\oplus (\tau_1\oplus \tau_1^*\oplus \pi\circ d_{\infty}) \approxeq \tau_1\oplus (\tau_2\oplus \tau_1^*\oplus \pi\circ d_{\infty}).
\eneq
It follows from \eqref{Texabs-3}  and \eqref{Texabs-4}
that
\beq\nonumber
\tau_2\oplus \pi\circ d_{\infty} &\approxeq & \tau_2\oplus (\tau_1\oplus \tau_1^*\oplus \pi\circ d_{\infty})\\
&\approxeq &\tau_1\oplus (\tau_2\oplus \tau_1^*\oplus \pi\circ d_{\infty})\\
&\approxeq & \tau_1\oplus \pi\circ d_{\infty}.
\eneq
\end{proof}

\section{Stable uniqueness theorems for non-unital \CA s,
cutting the tails}

{\it In this section and the next, we will use $Ext(A, SB)$ to identify $KK(A,B)$ when
$A$ is a separable amenable \CA.}

\begin{df}\label{Ddinfty-1}
Let $A\subset B$ be two \CA s. Recall
that $d_{\infty}(a)=\diag(a,a,...,a,...)$ is a diagonal element in $M(B\otimes {\cal K})$ for
all $a\in B.$ Let $\{e_{i,j}\}$ be the matrix units for ${\cal K}$ and let
$e_m=\sum_{i=1}^m e_{i,i}.$ For convenience, in what follows, we use
$d_{\infty-m}(a)$ for
$(1-e_m)d_{\infty}(a).$ In other words,
$d_{\infty-m}(a)=\diag(\overbrace{0,0,...,0}^m, a,a,...).$
\end{df}

In the following lemma, we note that $e_n\not\in A\otimes {\cal K}$ and
$M(A\otimes {\cal K})\not=M({\tilde A}\otimes {\cal K})$ in general.

\begin{lem}\label{562}{\rm (Lemma 5.6.2 of \cite{Lnbk})}
Let $C$ be a $\sigma$-unital \CA\, and let $A$ be a separable \SCA\, of $C.$
Let $\{c_n\}$ be an approximate identity for $C.$
Suppose that $U\in M({\tilde C}\otimes {\cal K})$ is a unitary such that
\beq\label{562-1}
U^*\diag(b(a),d_{\infty}(a))U-\diag(c(a), d_{\infty}(a))\in C\otimes {\cal K}
\eneq
for some $b(a)$ and $c(a)$ (depending on $a$) and for all $a\in A.$ Then, for any $\ep>0$ and
for any finite subset ${\cal F}\subset A,$

(1) there is an  integer $N>1$   such that, for all $a\in {\cal F},$
$$
\|(1-e_N)U^*d_{\infty-1}(a)U(1-e_N)-(1-e_N)d_{\infty-1}(a)(1-e_N)\|<\ep/32,
$$
where $e_N=\sum_{i=1}^n e_{i,i},$

(2) Let $p=Ue_NU^*.$ If $N_1\ge N$ such that
$\|e_{N_1}p-p\|<\ep/8,$ then
$$
\|\diag(0, d_{N_1-1}(a))-[pd_{\infty-1}(a)p+(e_{N_1}-p)d_{\infty-1}(a)(e_{N_1}-p)\|<\sqrt{\ep}+\ep
$$
for all $a\in {\cal F}.$
\end{lem}

\begin{proof}
To simplify notation, we may assume that ${\cal F}$ is in the unit ball of $A_{s.a.}.$
Let ${\cal G}={\cal F}\cup \{ab: a,\, b\in {\cal F}\}.$ Let
$M=\max\{\|b(a)\|\|c(a)\|: a\in {\cal G}\}.$

Note that $e_{i,j}\in {\tilde C}\otimes {\cal K}.$
Therefore $U^*e_{1,1}\in {\tilde C}\otimes {\cal K}.$ Since $\{e_n\}$ forms
an approximate identity for ${\tilde C}\otimes {\cal K},$  there is an integer $N>1$ such that
\beq\label{562-1+}
&&\|(1-e_N)U^*e_{1,1}\|<\ep/64(M+1)\andeqn\\\label{562-2}
&&(1-e_N)U^*\diag(b(a), d_{\infty})U(1-e_n)\approx_{\ep/64} (1-e_N)\diag(c(a), d_{\infty}(a))(1-e_N)
\eneq
for all $b(a), c(a)\in C$ and $a\in {\cal G}.$
Note that
$$
(1-e_N)\diag(c(a), d_{\infty}(a))(1-e_N)=(1-e_N)d_{\infty-1}(a)(1-e_N)=d_{\infty-N}(a)
$$
for all $a\in A.$
Therefore, by \eqref{562-1+} and \eqref{562-2},
\beq\label{562-2+}
\|(1-e_N)U^*d_{\infty-1}(a)U(1-e_N)-d_{\infty-N}(a)\|<{M\ep\over{64(M+1)}}+\ep/64<\ep/32
\eneq
for all $a\in {\cal G}.$ This proves (1).

Suppose now that $N_1>N$ and $\|e_{N_1}p-p\|<\ep/8.$
From what has been established, if $a\in {\cal F},$
\beq\label{562-3}
&&\|(1-e_N)U^*d_{\infty-1}(a^2)U(1-e_N)-d_{\infty-N}(a^2)\|<\ep/32\andeqn\\
&&\|(1-e_N)U^*d_{\infty-1}(a)(1-p)d_{\infty-1}(a)U(1-e_N)-d_{\infty-N}(a^2)\|<\ep/16.
\eneq
Therefore, for $a\in {\cal F},$
\beq\label{562-4}
&&\hspace{-0.4in}(1-e_N)U^*d_{\infty-1}(a)(1-p)d_{\infty-1}(a)U(1-e_N)\\
&\approx_{3\ep/32}& (1-e_N)U^*d_{\infty-1}(a^2)U(1-e_N).
\eneq
Consequently (since $(1-p)=U(1-p)U^*$), for all $a\in {\cal F},$
\beq\label{562-5}
\|(1-p)d_{\infty-1}(a)(1-p)d_{\infty-1}(a)(1-p)-(1-p)d_{\infty-1}(a^2)(1-p)\|<3\ep/32.
\eneq
On the other hand, for all $a\in {\cal F},$
\beq\label{562-6}
(1-p)d_{\infty-1}(a^2)(1-p) &=& (1-p)d_{\infty-1}(a)(1-p)d_{\infty-1}(a)(1-p)\\
            \hspace{0.2in}&+& (1-p)d_{\infty-1}(a)pd_{\infty-1}(a)(1-p).
\eneq
Therefore, for all $a\in {\cal F},$
\beq\label{562-7}
\|(1-p)d_{\infty-1}(a)pd_{\infty-1}(a)(1-p)\|<3\ep/32
\eneq
Since we assume that $a=a^*$ in ${\cal F},$ we have
\beq\label{562-8}
\|pd_{\infty-1}(a)(1-p)\|<\sqrt{3}\sqrt{\ep}/\sqrt{32}<\sqrt{\ep}/2\sqrt{2}\rforal a\in {\cal F}.
\eneq
It follows that
$$
\|d_{\infty-1}(a)-(pd_{\infty-1}(a)p+(1-p)d_{\infty}(a)(1-p)\|<\sqrt{\ep/2}\rforal a\in {\cal F}.
$$
Finally, for all $a\in {\cal F},$
$$
\|\diag(0, d_{N_1-1}(a))-[pd_{\infty-1}(a)p+ (e_{N_1}-p)d_{\infty-1}(a)(e_{N_1}-p)\|<\sqrt{\ep}+\ep.
$$
\end{proof}

The following is a non-unital version of  Theorem 4.3 \cite{Lnsuniq} (see also \cite{ED} and \cite{Lnauct}).
The idea of the proof  presented below is based on
 that of Theorem 5.6.4. of \cite{Lnbk} using what have been proved here.

\begin{thm}\label{T564}
Let $B$ be a $\sigma$-unital \CA\, and let $A$ be a separable \SCA\, of $B.$
Suppose that $j: A\to B$ is full and amenable.
Let $\af,\, \bt: A\to B$ be two \hm s. If $[\af]=[\bt]$ in $KK(A, B),$ then, for any $\ep>0,$ and any
finite subset ${\cal F}\subset A,$ there exists an integer $L>1$ and a unitary
$U\in M_{L+1}({\tilde B})$ such that
$$
\|U^*\diag(\af(a),d_L(a))U-\diag(\bt(a), d_L(a))\|<\ep\rforal a\in {\cal F}.
$$
\end{thm}

\begin{proof}
If either $A$ or $B$ are unital, then this can be easily reduced to the cases both are unital
and $\af$ and $\bt$ are unital.

We now consider the case that neither $A$ nor $B$ are unital.

Fix $1>\ep>0$ and a finite subset ${\cal F}\subset A.$ \Wlog, we may assume that ${\cal F}$ is in the unit ball
of $A_{s.a.}.$ Let ${\cal F}_1={\cal F} \cup\{1_{\tilde A}\}$ and ${\cal G}={\cal F}_1\cup \{ab: a, b\in {\cal F}\}.$
Let $\tau_1$ and $\tau_2$ be two unital essential extensions of $A$ by $SB$
 via
$\af$ and $\bt$ as mapping tori:
\beq\nonumber
&&M_\af=\{(f,a)\in C([0,1],B)\oplus A: f(0)=a\andeqn f(1)=\af(a)\}\andeqn\\
&&M_\bt=\{(f,a)\in C([0,1], B)\oplus A: f(0)=a\andeqn f(1)=\bt(a)\}.
\eneq
Since $[\af]=[\bt]$ in $KK(A,B),$  by \ref{Abs}, there is a unitary $W\in M(S{\tilde{B}}\otimes {\cal K})$
(which is identified with $C^b((0,1) M({\tilde B}\otimes {\cal K})_{\sigma})$) such that
\beq\label{564-2}
\pi(W)^*\diag(\tau_1, \tau_{\infty})\pi(W)=\diag(\tau_2, \tau_{\infty}),
\eneq
where $\pi: M(SB\otimes {\cal K})\to M(SB\otimes {\cal K})/SB\otimes {\cal K}$
is the quotient map and where
$\tau_{\infty}=\pi(d_{\infty}).$

For any $a\in A,$ let $f_a\in M_\af$ and $g_a\in M_\bt$ with
$f_a(0)=a,$ $f_a(1)=\af(a),$ $g_a(0)=a$ and $g_a(1)=\bt(a).$
Set
$$
c(a, f_a, g_a)(t)=W^*(t)\diag(f_a(t), d_{\infty}(a))W-\diag(g_a(t),d_{\infty}(a))
$$
for $t\in (0,1).$  Then, by \eqref{564-2},
$$
c(a, f_a, g_a)\in S{\tilde B}\otimes {\cal K}.
$$
In particular,
$$
\|c(a, f_a, g_a)(t)\|\to 0\,\,\,\, {\rm if}\,\,\, t\to 0\,\,\,{\rm or}\,\,\, t\to 1.
$$
Choose $0<t_1<1$ such that, for all $a\in {\cal G}$ and $t_1\le t<1,$
\beq\label{564-3}
\|c(a, f_a, g_a)(t)\|<\ep/64,\,\,\, \|f_a(t)-\af(a)\|<\ep/64\andeqn \|g_a(t)-\bt(a)\|<\ep/64
\eneq
Choose $0<t_0<t_1<1$ such that, for all $a\in {\cal G}$ and $0<t<t_0,$
\beq\label{564-4}
\|c(a, f_a, g_a)(t)\|<\ep/64,\,\,\, \|f_a(t)-a\|<\ep/64\andeqn \|g_a(t)-a\|<\ep/64.
\eneq
Thus,
\beq\nonumber
&&\hspace{-0.6in}\|W^*(t_1)\diag(\af(a), d_{\infty}(a))W(t_1)-\diag(\bt(a), d_{\infty}(a))\|\\\label{564-5}
&\le&\|f_a(t_1)-\af(a)\|+\|g_a-\bt(a)\|+\|c(a,f_a, g_a)(t_1)\|<3\ep/64
\eneq
for all $a\in {\cal G}.$ Similarly
\beq\label{564-6}
\|W^*(t_0)d_{\infty}(a)W(t_0)-d_{\infty}(a)\|<3\ep/64\rforal a\in {\cal G}.
\eneq
Set
$V(t)=\diag(W(t), W(t_0)^*)$ for $t\in [t_0, t_1].$ We have
that
\beq\label{564-7}
V(t_1)^*\diag(\af(a),d_{\infty}(a), d_{\infty}(a))V(t_1)\approx_{3\ep/32} \diag(\bt(a), d_{\infty}(a), d_{\infty}(a))
\eneq
for all $a\in {\cal G}.$
There is a norm-continuous path of
unitaries $\{V(t): t\in [0, t_0]\}\subset M_2(S{\tilde B}\otimes {\cal K})$ such that
$V(0)=1$ and $V(t_0)=V(t_0)$ (notation already defined) and, for all $t\in [0, t_0],$
\beq\label{564-8}
V(t)^*\diag(d_{\infty}(x), d_{\infty}(x))V(t)-(d_{\infty}(x), d_{\infty}(x))\in {\tilde B}\otimes {\cal K}
\eneq
for all $x\in A$ and
\beq\label{564-9}
V(t)^*\diag(d_{\infty}(a), d_{\infty}(a))V(t)\approx_{3\ep/32} \diag(d_{\infty}(a), d_{\infty}(a))
\eneq
for all $a\in {\cal G}.$

We now identify $d_{\infty}(a)$ with $\diag(d_{\infty}(a), d_{\infty}(a))$ and
$M_2(M({\tilde B}\otimes {\cal K}))$ with $M({\tilde B}\otimes {\cal K}).$ Set
$C=C([0,t_1], {\tilde B}).$ We also identify ${\tilde B}$ with constant
functions in $C.$
We also have that (denote also by
$V, f_a$ and $g_a$ the restriction of $V,$ $f_a$ and $g_a$ on $[0,t_1],$ respectively)
\beq\label{564-10}
V^*\diag(f_a, d_{\infty}(a))V-\diag(g_a, d_{\infty}(a))\in C\otimes {\cal K}
\eneq
for all $a\in {\cal G}.$

For any $(\ep/2^{12})^2>\dt>0,$ it follows from \eqref{564-10} and  \ref{562} that there is an integer $N>1$ such that
\beq\label{564-11}
&&\|e_NV(t)^*\diag(f_a, d_{\infty}(a))V(t)-V(t)^*\diag(f_a, d_{\infty}(a))V(t)e_N\|<\dt\andeqn\\
&&\|(1-e_N)V(t)^*(d_{\infty-1}(a))V(t)(1-e_N)-d_{\infty-1}(a)\|<\dt
\eneq
for all $a\in {\cal G}$ and for all $t\in [0,t_1].$
It follows from \eqref{564-5} and \eqref{564-6} that, for all $a\in {\cal G},$
$$
e_NV(t_1)^*\diag(f_a, d_{\infty}(a))V(t_1)e_N\approx_{6\ep/64} \diag(\bt(a), d_{N-1}(a))
$$
for all $a\in {\cal G}.$

Set $p(t)=V(t)e_NV(t)^*$ for $t\in [0, t_1].$ Since $W(t)|_{(0,1)}$ is in
$C^b((0,1), M({\tilde B}\otimes {\cal K})_\sigma)$ and $V(t)$ is norm-continuous on $[0,t_0],$
$p(t)$ is norm-continuous on $[0,t_1].$ It is important to note that
$p(0)=e_N.$

There is an integer $N_1>N$ such that $\|e_{N_1}p(t)-p(t)\|<\dt/16$ for all $t\in [0, t_1].$
It follows from \ref{562} (with $C=C([0,1], {\tilde B})$) that, for all $t\in [0, t_1],$
\beq\label{564-12}
\hspace{-0.6in}\diag(0, d_{N_1-1}(a))\approx_{\ep/2^{11}}
(p(t)d_{\infty-1}(a)p(t)+(e_{N_1}-p(t))d_{\infty-1}(a)(e_{N_1}-p(t))
\eneq
for all $a\in  {\cal F}_1={\cal F}\cup \{1_{\tilde A}\}$.  By \eqref{564-12},
$$
p(t)\diag(0, e_{N_1-1})\approx_{\ep/2^{10}} \diag(0, e_{N_1-1})p(t)\rforal t\in [0,t_1].
$$

There is a partition:
$$
0=t_m<t_{m-1}<\cdots t_2<t_1
$$
such that
$$
\|p(t_{n+1})-p(t_n)\|<\ep/2^{10},\,\,\, n=1,2,...,m-1.
$$
There are $q_i'\le e_{N_1}$ such that
$$
\|q_i'-p(t_i)\|<\ep/256\andeqn q_i'\diag(0, e_{N_1-1})=\diag(0, e_{N_1-1})q_i',\,\,\,i=1,2,...,m.
$$
Choose $q_m'=p(0)=e_N.$
Set $q_i=q_i'\diag(0, e_{N_1-1}).$ Then
$$
q_m=e_N\diag(0, e_{N_1-1})=\diag(0, e_{N_1-1}).
$$
We have that
\beq\label{564-13}
\|q_i-q_{i+1}\|<\ep/64,\,\,\, i=1,2,...,m-1.
\eneq

Set
\beq
\eta_1(a)&=&p(t_1)d_{\infty-1}(a)p(t_1),\\
\eta_2(a)&=&q_i\diag(0, d_{N_1-1}(a))q_i,\\
\gamma_1(a)&=&(e_{N_1}-p(t_1))\diag(0, d_{N_1-1}(a))(e_{N_1}-p(t_1))\andeqn\\
\gamma_i(a)&=&(e_{N_1}-q_i)\diag(0, d_{N_1-1}(a))(e_{N_1}-q_i),
\eneq
$i=2,3,...,m.$ Immediately we have
$$
\gamma_1(a)\approx_{\ep/128} (e_{N_1}-q_1)\diag(0, d_{N_1-1}(a))(e_{N_1}-q_1)\rforal a\in {\cal F}_1.
$$
Note also that
\beq\nonumber
q_1d_{\infty-1}(a)q_1&\approx_{\ep/64}& \eta_1(a)\rforal a\in {\cal F}_1,\\
\gamma_m(x)&=&(e_{N_1}-e_N)d_{N_1}(x),\\
q_i'\diag(0, d_{N_1-1}(x))q_i'&=&\eta_i(x),\,\,\,i=2,3,...,m-1\andeqn\\\label{564-12-1}
\eta_m(x)&=&\diag(0, d_{N-1}(x))
\eneq
for all $a\in {\tilde A}.$
Moreover, by \eqref{564-12},
\beq\label{564-14}
\diag(0, d_{N-1}(a))\approx_{3\ep/64} \eta_i(a)+\gamma_i(a)
\eneq
for all $a\in {\cal F}_1$ and $1\le i\le m.$ Furthermore, by \eqref{564-13},
\beq\label{564-15}
\|\eta_i(a)-\eta_{i+1}(a)\|<\ep/32\andeqn \|\gamma_i(a)-\gamma_{i+1}(a)\|<\ep/32
\eneq
for all $a\in {\cal F}_1$ and $i=1,2,...,m-1.$  Set
\beq\nonumber
d'_{m(N_1-1)}(a) &=& (e_{N_1}-e_N)d_{N_1}(a)\oplus d_{m-1}(\diag(0, d_{N_1-1}(a))\oplus d_{N-1}(a)),\\\nonumber
d'_{m(N_1-1)+(N-1)}(a)&=&d_m(\diag(0,d_{N_1-1}(a)))\oplus d_{N-1}(a)
\eneq
for all $a\in {\tilde A}.$
Also
\beq\nonumber
\Delta_1(a) &=& \diag(\gamma_1(a), \eta_2(a),\gamma_2(a),...,\eta_m(a), \gamma_m(a), d_{N-1}(a))\\
\Delta_2(a) &=& \diag(\gamma_m(a), \eta_m(a), \gamma_{m-1}(a),...,\eta_2(a),\gamma_1(a), d_{N-1}(a))\\
\Delta_3(a) &=& \diag(\gamma_m(a),\eta_{m-1}(a), \gamma_{m-1}(a),...,\eta_1(a),\gamma_1(a), d_{N-1}(a))
\eneq
for all $a\in {\tilde A}.$  Then, by \eqref{564-14}, \eqref{564-15} and \eqref{564-12-1},
\beq\label{564-16}
&&\|\eta_1(a)\oplus \Delta_1(a)-d'_{m(N_1-1)+(N-1)}(a)\|<3\ep/64\\
&&\|\Delta_2(a)-\Delta_3(a)\|<\ep/32\\
&&\|\Delta_3(a)-d'_{m(N_1-1)}(a)\|<3\ep/64
\eneq
for all $a\in {\cal F}_1.$

Note that
\beq\nonumber
\gamma_1(1_{\tilde A})&=&(e_{N_1}-p(t_1))\diag(0, d_{N_1-1}(1_{\tilde A}))\andeqn\\
\gamma_m(1_{\tilde{A}})&=&e_{N_1}-e_N.
\eneq
Define $P=\Delta_1(1_{\tilde A})$ and $P'=\Delta_2(1_{\tilde A}).$
Then there is a unitary $z\in M_{(m+1)N_1+(N-1)}({\tilde{B}})$ such that
$$
z^*\Delta_1(a)z=\Delta_2(a)\rforal a\in {\tilde A}.
$$
Let $Z=p(t_1)V(t_1)e_N+PzP'.$ Then
\beq\nonumber
ZZ^*=p(t_1)V(t_1)e_NV(t_1)^*p(t_1)+PzP'z^*P=p(t_1)+P\andeqn\\\nonumber
Z^*Z=e_NV(t_1)^*p(t_1)V(t_1)e_N+P'z^*PZP'=e_N+P'=e_N\oplus d_{m(N_1-1)}(1_{\tilde A}).
\eneq
For any $a\in {\cal F}_1,$  by \eqref{564-16},
\beq\nonumber
&&\hspace{-1in}Z^*\diag(\af(a), d'_{m(N_1-1)+(N-1)}(a))Z\\
&\approx_{3\ep/64}& Z^*\diag(\af(a), \eta_1(a)\oplus \Delta_1(a))Z\\
&\approx_{\ep/2^{11}}& e_NV(t_1)^*\diag(\af(a), d_{\infty}(a))V(t_1)e_N\oplus z^*\Delta_1(a)z\\
&\approx_{6\ep/64}& \diag(\bt(a), d_{N-1}(a))\oplus \Delta_2(a)\\
&\approx_{3\ep/64}&\diag(\bt(a),d_{N-1}(a))\oplus \Delta_3(a)\\
&\approx_{3\ep/64}& \diag(\bt(a),d_{N-1}(a))\oplus d'_{m(N_1-1)}(a)
\eneq
for all $a\in {\cal F}_1.$ It follows that there exists a unitary $U\in M_{L+1}({\tilde B})$
($L\ge (m+1)N_1$) such that
$$
\|U^*\diag(\af(a), d_L(a))U-\diag(\bt(a), d_L(a))\|<\ep\rforal a\in {\cal F}.
$$
\end{proof}

\begin{thm}\label{T565}
Let $B$ be a $\sigma$-unital  \CA\, which almost has stable rank one and let $A$ be a separable \SCA\, of $B.$
Suppose that $j: A\to B$ is full and amenable.
Let $\af,\, \bt: A\to B$ be two \hm s. If $[\af]=[\bt]$ in $KK(A, B),$ then, for any $\ep>0,$ and any
finite subset ${\cal F}\subset A,$ there exists an integer $L>1$ and a unitary
$U\in {\widetilde{M_{L+1}(B)}}$ such that
$$
\|U^*\diag(\af(a),d_L(a))U-\diag(\bt(a), d_L(a))\|<\ep\rforal a\in {\cal F}.
$$
\end{thm}

\begin{proof}
For any $\ep>0$ and any finite subset ${\cal F}\subset A,$ by \ref{T564}, there exists an integer
$L>1$ and
 a unitary $V\in M_{L+1}({\tilde B})$ such that
 such that
$$
\|V^*\diag(\af(a),d_L(a))V-\diag(\bt(a), d_L(a))\|<\ep/4\rforal a\in {\cal F}.
$$
\Wlog, we may assume that ${\cal F}$ is in the unit ball of $A.$
There exists $e\in M_{L+1}(B)$ with $0\le e\le 1$
such that
\beq\label{565-1}
\|e\diag(\af(a),d_L(a))e-\diag(\af(a),d_L(a))\|<\ep/16
\rforal a\in {\cal F}.
\eneq
Thus one can choose $1/2>\eta>0$ such that
\beq\label{565-2}
\|f_\eta(e)\diag(\af(a), d_L(a))f_\eta(e)-\diag(\af(a), d_L(a))\|<\ep/8 \rforal a\in {\cal F}.
\eneq
Consider the element $Ve\in M_{L+1}(B).$ Since $M_{L+1}(B)$ almost has stable rank one,
$$M_{L+1}(B)\subset \overline{GL(\widetilde{M_{L+1}(B)}}.$$
By
Theorem 5 of \cite{Pedjot87}, there is a unitary $U\in {\widetilde{M_{n+1}(B)}}$ such
that
$$
U^*f_\dt(e)=V^*f_\dt(e).
$$
Therefore
\beq\nonumber
U^*\diag(\af(a),d_L(a))U&\approx_{\ep/8}&U^*f_\dt(e)\diag(\af(a),d_L(a))f_\dt(e)U\\
&=& V^*f_\dt(e)\diag(\af(a), d_L(a))f_\dt(e)V\\
&\approx_{\ep/4}& \diag(\bt(a), d_L(a))
\eneq
for all $a\in {\cal F}.$

\end{proof}

\section{Stable uniqueness theorems for non-unital \CA s, general cases}

\begin{df}\label{fullunif}
Let $A$ be a \SCA\, of a $\sigma$-unital \CA\, $B.$
An element $a\in A_+\setminus \{0\}$ is said to be {\it uniformly full} in $B,$ if
 there are positive number $M(a)>0$ and an integer $N(a)\ge 1$ such
that, for any $b\in B_+$ with $\|b\|\le 1$ and any $\ep>0,$  there are $x_i(a), x_2(a),...x_{n(a)}(a)\in B$ such
that $\|x_{i}(a)\|\le M(a),$ $n(a)\le N(a)$ and
$$
\|\sum_{i=1}^{n(a)}(x_{i}(a))^*ax_{i}(a)-b\|<\ep.
$$

In this case, we also say $a$ is $(N(A), M(a))$ full.

We say $a$ is {\it strongly uniformly full} in $B,$ if the above holds
for $\ep=0.$

We say $A$ is {\it locally uniformly full},
if every element $ a\in A_+\setminus \{0\}$ is uniformly full;
and we say $A$ is {\it strongly locally uniformly full} if every $a\in A_+\setminus \{0\}$ is
strongly uniformly full.


If $B$ is unital and $A$ is full in $B,$ then $A$ is always  strongly locally uniformly full.
In fact, for each $a\in A\setminus \{0\},$ there are
$x_1, x_2,...,x_m\in B$ such that
$$
\sum_{i=1}^m x_i^*ax_i=1_B.
$$
Choose $M(a)=\max\{\|x_i\|: 1\le i\le m\}$ and $N(a)=m.$
\end{df}

\begin{lem}\label{Lsep}
Let $B$ be a non-unital \CA\, and let $A$ be a separable amenable \CA\, which is a \SCA\,  of  $B.$
Suppose that $h_1, h_2: A\to B$ are two \hm s such that
$[h_1]=[h_2]$ in $KK(A,B).$
Then there exists a separable \SCA\, $C\subset B$ such that
$A, h_1(A), h_2(A)\subset C$ and $[h_1]=[h_2]$ in $KK(A, C).$
If $A$ is full
(locally uniformly full),
then we may assume that $A\subset C$ is
also  full
(locally uniformly).
\end{lem}

\begin{proof}
Put $B_1={\tilde B}.$
Let $\tau_1, \tau_2: A\to M(SB_1)$ be two induced
essential extensions given by
mapping tori:
\beq\label{Lsep-n1}
M_{h_i}=\{(f,a)\in C([0,1], B_1)\oplus A: f(0)=a\andeqn f(1)=h_i(a)\},\,\,\, i=1,2.
\eneq
There is a monomorphism  $\phi_0: A\to M(SB_1\otimes {\cal K})$
and a unitary $w\in M(SB_1\otimes {\cal K})$ such that
$$
w^*(\diag(h_1(a), \phi_0(a))w-\diag(h_2(a), \phi_0(a))\in SB_1\otimes {\cal K}\rforal a\in A.
$$
Let $C_{000}$ be the \SCA\, of $B$ generated by $A, h_1(A)$ and $h_2(A).$
In particular $C_{000}$ is separable.

Let $\{e_{i,j}\}$ be the matrix unit for ${\cal K}$ and choose
a dense sequence $\{t_n\}\subset (0,1).$
Choose an approximate identity $\{E_n\}$ of $SB_1\otimes {\cal K}$
such that $E_n\subset M_{k(n)}(SB_1),$ $n=1,2,....$
Let $\{a_n'\}$ be a dense sequence of $A$ which also contains
a subsequence in $A_{s.a}$  which is dense in $A_{s.a.}$ and a subsequence in $A_+$
which is dense in $A_+.$
Let $\{a_n\}$ be a sequence of elements  which contains $\{a_n'\}$ and if $a_n'\in A_+$
then $f_{1/2}(a_n')\in \{a_n\}.$

Put $D_0$ the \CA\, generated by the subset
$$
\{w^*(\diag(h_1(a_n), \phi_0(a_n))w-\diag(h_2(a_n), \phi_0(a_n)):  n\in \N\}.
$$
Then $D_{0}$ is a separable \SCA\, of $SB_1\otimes {\cal K}$
such that
\beq\label{Lsep-1}
w^*(\diag(h_1(a), \phi_0(a))w-\diag(h_2(a), \phi_0(a))\in D_{0}\rforal a\in A.
\eneq

Put $D_{00}$ the \SCA\, of $SB_1\otimes {\cal K}$ generated by
$$
\{E_n, wE_n, E_nw, E_n\phi_0(a)E_n: a\in A, n\in \N\}.
$$
Let $D_{000}$ be the separable \SCA\, of $SB_1\otimes {\cal K}$ generated by $D_{00}$ and $D_0.$
Denote by $p_{t_n}: SB_1\otimes {\cal K}\to B_1\otimes {\cal K}$
the point-evaluation at $t_n,$ $n=1,2,....$
Let $C_{00}$ be the \SCA\, of $B_1\otimes {\cal K}$ generated by
$$
\{p_{t_n}(D_{000})+C_{000}: n=1,2,...\}.
$$
Denote by $C_{0,n}$ the \SCA\, generated by
$$
\{e_{n,n}C_{00}e_{n,n} +\lambda e_{n,n}: \lambda\in C\},
$$
$n=1,2,....$  Let $C_{0,n}'=e_{1,n}C_{0,n}e_{n,1}.$
So $C_{0,n}'$ is a \SCA\, of $B_1.$
Let $C'$ be the \CA\, generated by $\cup_{n=1}^{\infty} C_{0,n}'.$  Note that $C'$ is unital.
Since $C'\subset B_1={\tilde B},$ we may write $C'={\tilde C''},$ where $C''\subset B$
is a separable \SCA.
Let $\{c_n\}$ be an approximate identity
of $C''$ such that $c_{n+1}c_n=c_nc_{n+1}=c_n,$ $n=1,2,...$
For each $a_k,$ there are
$M(a_k,n)>0,$  $N(a_k,n)\ge 1$  and
$x_{1,n}, x_{2,n}(a_k),...,x_{m(n),a_k}\in B$ such that
$\|x_{i,n}(a_k)\|\le M(a_k,n),$ $m(n,a_k)\le N(a_k,n)$ and
$$
\|\sum_{i=1}^{m(n)}x_{i,n}^*a_kx_{i,n}-c_{n+1}^{1/4}\|<1/2^{n+2},
$$
$n=1,2,....$
Let $C_0'$ be the \SCA\, generated by $C''$ and $\{x_{i,n}(a_k): i, k,n\}.$
Let $C=\overline{\cup_{n=1}^{\infty}c_n(C_0')c_n}.$

It is clear that $C\supset C''.$
We claim that
$A$ is
full
in $C.$  Clearly $\{c_n\}$ is an approximate identity for $C.$

For any $c\in C_+$ with $\|c\|\le 1$ and any $\ep>0,$ there exists $n_1\ge 1$ such that
$$
\|c^{1/2}c_nc^{1/2}-c^{1/2}\|<\ep/4\rforal n\ge n_1.
$$
For each $k,$ choose $n_2\ge 1$ such that
$$
\|c_na_kc_n-a_k\|<\ep/16(((M(a_k, n_1)+1)\cdot m(n_1))^2\rforal n\ge n_2
$$
and $1/2^{n_2}<\ep/4.$
Let $y_{i,k}=c_{n_2}x_{i,n_2}(a_k)c_{n_2}^{1/4}c^{1/2},$ $i=1,2,...,m(n,a_k).$
Then $y_{i,k}\in C$ and $\|y_{i,k}\|\le M(a_k)$ for all $i.$
We have
$$
\|\sum_{i=1}^{m(n,a_k)} y_{i,k}^* a_k y_{i,k}-c^{1/2}c_{n_2}c^{1/2}\|<\ep/2.
$$
It follows that
$$
\|\sum_{i=1}^{m(n,a_k)}y_{i,k}^*a_ky_{i,k}-c\|<\ep.
$$
It follows that each $a_k$ is full in $C.$

Let $I$ be an ideal of $C$ and
let
$$
J=\{a\in A: a\in I\}.
$$
Let  $\pi_J: A\to A/J$ be the quotient map.
Fix $a\in J_+$ with $\|a\|=1.$ For any $1/4>\ep>0,$ by the assumption on $\{a_n'\},$ there is $b\in \{a_n'\}\cap A_+$
such that $\|a-b\|<\ep.$
This implies $\|\pi_J(b)\|<\ep,$ whence $\pi_J(f_{1/2}(b))=f_{1/2}(\pi_J(b))=0.$ In other words,
$f_{1/2}(b)\in J.$ Since $0<\ep<1/4,$ $\|b\|\ge 1-1/4.$ Therefore $f_{1/2}(b)\not=0.$
However, $f_{1/2}(b)\in \{a_k\}.$  So $f_{1/2}(b)$ is full.
This can happen only when $J=\{0\}.$
This holds for any ideal $I.$ Thus $A$ is full in $C.$


Consider $C_{1s}$ the \SCA\, of $B_1\otimes {\cal K}$ generated by ${\tilde C}$ and $\{e_{i,j}\}.$
Then $C_{1s}={\tilde C}\otimes {\cal K}.$
Let $SC_{1s}=C_0((0,1), C_{1s}).$
Suppose that
\beq\label{Lsep-n2}
b\in \{E_n, wE_n, E_nw, E_n\phi_0(a)E_n, E_n\phi_0(a), \phi_0(a)E_n: a\in A, n\in \N\}.
\eneq
Keep in mind that $E_n\in SB_1\otimes {\cal K},$ in particular, $E_n(0)=E_n(1)=0$ $n\in \N.$
Then, for each $t_n\in (0,1),$
$$
\pi_{t_n}(b)\in \pi_{t_n}(SC_{1,s}\otimes {\cal K}).
$$
It follows that
$b\in SC_{1,s}.$  To see this, let $\ep>0,$ there are
$$
0=t_{n_0}<t_{n_1}<t_{n_2}<\cdots t_{n_k}<t=t_{n_{k+1}}
$$
such that $t_{n_i}\in \{t_n\}$, $i=1,2,...,k,$
$$
\|b(t)-b(t_{n_i})\|<\ep/4\rforal t\in(t_{n_i}, t_{n_{i+1}}),\,\,i=0,1,...,k.
$$
Let
$c(t)=(t/t_{n_1}) b(t_{n_1})$ if $t\in (0,t_{n_1}),$
$c(t)=((t_{n_{i+1}}-t)/(t_{n_{i+1}}-t_{n_i}))b(t_{n_i})+((t-t_{n_i})/(t_{n_{i+1}}-t_{n_i}))b(t_{n_{i+1}})$
for $t\in [t_{n_i}, t_{n_{i+1}}),$ $i=1,2,...,k,$ and
$c(t)=((1-t)/(1-t_{n_k}))b(t_{n_k})$ for $t\in (t_{n_k},1).$
Then $c\in SC_{1s}\otimes {\cal K}.$
On the other hand
$$
\|b(t)-c(t)\|<\ep\rforal t\in (0,1).
$$
This proves that $b\in SC_{1s}.$
In particular, $E_n\in SC_{1s}$ and
$\{E_n\}$ forms an approximate identity for
$SC_{1s}.$
Since $wE_n, E_nw \in SC_{1s}\otimes {\cal K},$
this implies that $w\in M(SC_{1s}\otimes {\cal K}).$  Similarly, since
$\phi_0(a)E_n, E_n\phi_0(a)E_n\in SC_{1s}\otimes {\cal K})$ for all $a\in A$ and $n\ge1 ,$
one concludes that we may view that  $\phi_0$ is a monomorphism from
$A$ to $M(SC_{1s}).$

A similar argument shows that $D_0\subset SC_{1s}.$

We have
$$
w^*(\diag(h_1(a), \phi_0(a))w-\diag(h_2(a), \phi_0(a))\in SC_{1s}
$$
for all $a\in A.$
This implies that $[h_1]=[h_2]$ in $KK(A, C_1)=KK(A,{\tilde C}).$
It follows that $[h_1]=[h_2]$ in $KK(A,C).$ Note that if $A$ is locally uniformly full in $B,$
as constructed  above, $A$ is also locally uniformly full in $C.$

\end{proof}

\begin{rem}\label{Rlsep}
The assumption that $A$ is amenable could be weakened to the assumption
that $h_1$ and $h_2$ are amenable.
\end{rem}

%
%
%
%

\begin{lem}\label{Lsep2}
Let $B$ be a non-unital \CA\, and let $A$ be a separable amenable \CA\, which is a \SCA\,  of  $B.$
Suppose that $h_1, h_2: A\to B$ are two \hm s such that
$[h_1]=[h_2]$ in $Hom_{\Lambda}(\underline{K}(A), \underline{K}(B)).$
Then there exists a separable \SCA\, $C\subset B$ such that
$A, h_1(A), h_2(A)\subset C$ and $[h_1]=[h_2]$ in $Hom_{\Lambda}(\underline{K}(A), \underline{K}(C)).$
If $A$ is full
(locally uniformly full),
then we may assume that $A\subset C$ is
also  full
(locally uniformly).
\end{lem}

\begin{thm}\label{T39}
Let $A$ be a separable amenable \CA\, and let $B$ be a $\sigma$-unital \CA.
Suppose that $h_1, h_2: A\to B$ are two \hm s such that
$$
[h_1]=[h_2]\,\,\, {\rm in}\,\,\, KL(A, B).
$$
Suppose that there is an embedding $j: A\to B$ which is locally uniformly full.
Then, for any $\ep>0$ and finite subset ${\cal F}\subset A,$  there is an integer $n\ge 1$ and
a unitary $u\in M_{n+1}({\tilde B})$  such that
$$
\|u^*\diag(h_1(a), d_n(a))u-\diag(h_2(a), d_n(a))\|<\ep\rforal a\in {\cal F}.
$$
\end{thm}

\begin{proof}

Let $C=\prod_{k=1}^{\infty} B,$  $C_0=\bigoplus_{k=1}^{\infty} B$ and let
$\pi: C\to C/C_0$ be the quotient map.
Let $H_i=\{h_i\}: A\to C$ be defined by $H_i(a)=\{h_i(a)\}$ for all $a\in A,$ $i=1,2.$
Define $H_0: A\to C$ by $H_0(a)=\{j(a)\}$ for all $a\in A.$
It follows from 3.5 of \cite{Lnauct} that
\beq\label{T39-2}
[\pi\circ H_1]=[\pi\circ H_2]\,\,\,{\rm in}\,\,\, KK(A, C/C_0).
\eneq
Since $j: A\to B$ is locally uniformly full, for any $a\in A\setminus\{0\},$ let
$M(a)$ and $N(a)$ be defined as in \ref{fullunif} associated with element
$a.$
Let $\{b_n\}\subset (\prod_{n=1}^{\infty} B)_+$ with
$\|\{b_n\}\|\le 1.$
There are
there are $x_{1,n},x_{2,n},...,x_{m(n),n}\in B$ such that
$\|x_{i,n}\|\le M(a)$ and $m(n,a)\le N(a)$ and
$$
\|\sum_{i=1}^{m(n,a)}x_i^*j(a)x_i-b_n\|<\ep.
$$
By adding zeros, if necessarily, \wilog, we may assume that $m(n,a)=N(a)$ for all $n.$
Put $z_i=\{x_{i,n}\},$ $i=1,2,...,N(a).$ Then
$\|z_i\|=\sup\{\|x_{i,n}\|: n\in \N\}\le M(a),$ $i=1,2,...,N(a).$
Therefore $z_i\in \prod_{n=1}^{\infty}B.$
We have
that
$$
\|\sum_{i=1}^{N(a)} z_i^*H_0(a) z_i-\{b_n\}\|<\ep.
$$
This implies that $H_0$ is locally uniformly full in $\prod_{n=1}^{\infty}B.$
Viewing $\pi\circ H_0: A\to C/C_0$ as an embedding. Then it is locally uniformly full.
Combining this with \eqref{T39-2}, applying \ref{Lsep}, we obtain
a separable \SCA\, $D\subset C/C_0$ such that
$\pi\circ H_0(A), \pi\circ H_1(A), \pi\circ H_2(A)\subset D,$
$\pi\circ H_0$ is locally uniformly full in $D$ and
$$
[\pi\circ H_1]=[\pi\circ H_2]\,\,\, {\rm in}\,\,\, KK(A, D).
$$
Applying \ref{T564},
there exists an integer $n\ge 1$ and unitary $U\in M_{n+1}({\tilde D})$ such that
$$
\|U^*\diag(\pi\circ H_1(a), d_n(\pi\circ H_0(a)))U-\diag(\pi\circ H_2(a), d_n(\pi\circ H_0(a)))\|<\ep/4
$$
for all $a\in {\cal F}.$
Note that $U\in M_{n+1}({\widetilde{C/C_0}}).$
There is a unitary $V=\{v_k\}\in {\tilde C}$ such that
$\pi(V)=U.$
Therefore, for all sufficiently large $n,$
$$
\|v_k^*\diag(h_1(a), d_n(a))v_k-\diag(h_2(a), d_n(a))\|<\ep
\rforal a\in {\cal F}.
$$
Choose $u=v_k$ for some sufficiently large $k.$
\end{proof}

\begin{thm}\label{T39sr1}
Let $A$ be a separable amenable \CA\, and let $B$ be a $\sigma$-unital \CA\,  which almost
has stable rank one.
Suppose that $h_1, h_2: A\to B$ are two \hm s such that
$$
[h_1]=[h_2]\,\,\, {\rm in}\,\,\, KL(A, B).
$$
Suppose that there is an embedding $j: A\to B$ which is locally uniformly full.
Then, for any $\ep>0$ and any finite subset ${\cal F}\subset A,$  there is an integer $n\ge 1$ and
a unitary $u\in  {\widetilde{M_{n+1}(B)}}$  such that
$$
\|u^*\diag(h_1(a), d_n(a))u-\diag(h_2(a), d_n(a))\|<\ep\rforal a\in {\cal F}.
$$
\end{thm}

\begin{proof}
This follows from \ref{T39} together with the proof of \ref{T565}.
\end{proof}

\begin{df}{\rm \hspace{-0.05in}(Definition 2.1 of \cite{GL1})}\label{Blbm}
Fix a map $r_0: \N\to \Z_+,$  a map $r_1: \N\to \Z_+,$ a map $T: \N\times \N\to \N,$ integers
$s\ge 1$ and $R\ge 1.$
  We say a   \CA\,
$A\in {\bf C}_{(r_0, r_1,T, s, R)}$ if

(a) for any integer $n\ge 1$ and any {pair} of projections $p,\, q\in M_n({\tilde A})$ with
$[p]=[q]$ in $K_0(A),\,\,\,$  $p\oplus 1_{M_{r_0(n)}(A)}$ and
$q\oplus 1_{M_{r_0(n)}({\tilde A})}$ are Murr{a}y-von Neumann equivalent,
moreover, if $p\in M_n({\tilde A})$ and $q\in M_m({\tilde A})$ and
$[p]-[q]\ge 0,$ there exists $p'\in M_{n+r_0(n)}({\tilde A})$
such that $p'\le p\oplus 1_{M_{r_0(n)}}$ and $p'$ is equivalent to $q\oplus 1_{M_{r_0(n)}};$

(b) if $k\ge 1,$ and $x\in K_0(A)$ such that  $-n[1_{\tilde A}]\le kx \le n[1_{\tilde A}]$ for some integer $n\ge 1,$ then
$${-}T(n,k)[1_{\tilde A}]\le x \le T(n,k)[1_{\tilde A}];$$

 (c) the canonical map $U(M_s({\tilde A}))/U_0(M_s({\tilde A}))\to K_1(A)$ is surjective;

(d) {i}f $u\in {U(M_n({\tilde A}))}$ and $[u]=0$ in $K_1({\tilde A}),$ then $u\oplus 1_{M_{r_1(n)}}\in U_0(M_{n+r_1(n)}({\tilde A}));$

(f) {${\rm cer}(M_m({\tilde A}))\le R$} 
for all $m\ge 1.$

Note that if $K_0(A)=\{0\},$ then $A$ satisfies condition (a) and (b) for any $r_0$ and $T.$
If $A$ has stable rank one, then $r_0$ and $r_1$ can be chosen to be zero.
\end{df}

\begin{df}\label{localfull}
Let $A$ be a separable \CA, let $B$ be non-unital \CA\, and let $L: A\to B$ be a positive linear map.
 Let $F: A_+\setminus \{0\}\to \N\times \R.$
Suppose that ${\cal H}\subset A_+\setminus \{0\}$ is  a subset.
We say $L$ is $F$-${\cal H}$-full, if, for any $b\in B_+$ with $\|b\|\le 1,$ any $\ep>0,$
there are $x_1, x_2,...,x_m\in B$ such that
$m\le N(a)$ and $\|x_i\|\le M(a),$ where
$(N(a), M(a))=F(a),$  and
\beq\label{localfull-1}
\|\sum_{i=1}^mx_i^*L(a)x_i-b\|\le \ep.
\eneq

We say $L$ is exactly $F$-${\cal H}$-full, if in \eqref{localfull-1} holds for $\ep=0.$
\end{df}

\begin{thm}\label{Lauct2} {\rm \hspace{-0.06in}({c.f.} 5.3 of \cite{Lnsuniq}, Theorem 3.1 of \cite{GL1}
\cite{ED}, 5.9 of \cite{Lnauct} and
Theorem 7.1 of \cite{Ln-hmtp})}
Let $A$ be a non-unital separable amenable  \CA\, which satisfies the UCT,  let
$r_0, r_1: \N\to \Z_+,$ $T: \N\times \N\to \N$ be three maps, $s, R\ge 1$ be integers
and let $F: A_+\setminus \{0\}\to \N\times \R_+\setminus \{0\}$  and ${\bf L}: U(M_{\infty}({\tilde A}))\to \R_+$ be  two additional  maps.
  For any $\ep>0$ and any finite subset ${\cal F}\subset A,$ there exists
$\dt>0,$ a finite subset ${\cal G}\subset A,$
a finite subset ${\cal P}\subset \underline{K}(A),$ a finite subset
${\cal U}\subset U(M_{\infty}({\tilde A})),$ a finite subset ${\cal H}\subset A_+\setminus \{0\}$ and an integer $K\ge 1$ satisfying the following:
For any two ${\cal G}$-$\dt$-multiplicative \morp s $\phi, \psi: A\to B,$  where
$B\in {\bf C}_{r_0, r_1, T, s, R},$ and any
${\cal G}$-$\dt$-multiplicative \morp\, $\sigma: A\to M_l(B)$ (for any integer $l\ge 1$) which is also $F$-${\cal H}$-full  such that
\beq\label{Lauct-1}
\hspace{0.3in} &&{\rm cel}(\lceil \phi(u)\rceil  \lceil \psi(u^*)\rceil)\le {\bf L}(u)\tforal u\in {\cal U}\\
&&\tand\,
[\phi]|_{\cal P}=[\psi]|_{\cal P},
\eneq
there exists a unitary $U\in M_{1+Kl}({\tilde B})$ such that
\beq\label{Lauct-2}
\|{\rm Ad}\, U\circ (\phi\oplus S)(a)-(\psi\oplus S)(a)\|<\ep\tforal a\in {\cal F},
\eneq
where
\vspace{-0.14in} $$
S(f)={\rm diag}(\overbrace{\sigma(a), \sigma(a),...,\sigma(a)}^K)\tforal a\in A.
$$
If furthermore, $B$ almost has stable rank one, one can choose $U\in {\widetilde{ M_{1+Kl}(B)}}.$
\end{thm}

\begin{proof}
We may also use $\phi$ and $\psi$ for $\phi\otimes {\rm id}_{M_m}$ and
$\psi\otimes {\rm id}_{M_m},$ respectively.
Fix $A,$ $r_0, r_1, T, s, R,$ $F$ and ${\bf L}$ as described above.
Suppose that the theorem is false.  Then there exists  $\ep_0>0$ and a finite subset ${\cal F}\subset A$
such that there are a sequence of positive numbers $\{\dt_n\}$ with $\dt_n\searrow 0,$ an increasing sequence
$\{{\cal G}_n\} \subset A$ of finite subsets such that $\cup_n {\cal G}_n$ is dense in $A,$
an increasing sequence $\{{\cal P}_n\}\subset \underline{K}(A)$ of finite subsets  such that
$\cup_{{n}}{\cal P}_n=\underline{K}(A),$  an increasing sequence of finite subsets
$\{{\cal U}_n\}\subset U(M_{\infty}({\tilde A}))$ such that $\cup_{{n}}{\cal U}_n\cap U(M_m({\tilde A}))$ is dense
in $U(M_m({\tilde A}))$ for each integer $m\ge1,$  an increasing sequence of finite subsets
$\{{\cal H}_n\}\subset  A_+^{\bf 1}\setminus \{0\}$ such that if $a\in {\cal H}_n$
and $f_{1/2}(a)\not=0,$ then
$f_{1/2}(a)\in {\cal H}_{n+1}$ and $\cup_{{n}}{\cal H}_n$ is dense
in $A_+^{\bf 1},$  a sequences of integers $\{k(n)\}$ with
$\lim_{n\to\infty} k(n)=\infty$,  a sequence of unital \CA s $B_n\in {\bf C}_{r_0, r_1, T, s, R},$
two sequences of ${\cal G}_n$-$\dt_n$-multiplicative \morp s $\phi_{n}, \psi_{n}: A\to B_n$
such that
\beq\label{stableun2-1}
[\phi_{n}]|_{{\cal P}_n}=[\psi_{n}]|_{{\cal P}_n}\andeqn
{\rm cel}(\lceil \phi_{n}(u)\rceil \lceil \psi_{n}(u^*)\rceil)\le {\bf L}(u)
\eneq
for all $u\in {\cal U}_n$
and a sequence of unital ${\cal G}_n$-$\dt_n$-multiplicative \morp s $\sigma_n: A\to
M_{l(n)}(B_n)$ which is also $F$-${\cal H}_n$-full satisfying
\beq\label{stableun2-2}
&&\hspace{0.3in} \inf\{\sup\|v_n^*{\diag}( \phi_{n}(a), S_n(a))v_n-{\diag}(\psi_n(a), S_n(a))\|: a\in {\cal F}\|\}\ge \ep_0,
\eneq
where the infimum is taken among all unitaries $v_n\in M_{k(n)l(n)+1}(B_n)$
and
\vspace{-0.1in} $$
S_n(a)={\diag}(\overbrace{\sigma_n(a),\sigma_n(a),...,\sigma_n(a)}^{k(n)})\rforal a\in A.
$$

Put $B_n'=M_{l(n)}(B_n).$
Let $C_0=\bigoplus_{n=1}^{\infty}B_n',$ $C=\prod_{n=1}^{\infty}B_n',$ $Q(C)=C/C_0$ and
$\pi: C\to Q(C)$ is the quotient map.  Define $\Phi, \Psi, S: A\to C$ by
$\Phi(a)=\{\phi_n(a)\},$ $\Psi(a)=\{\psi_n(a)\}$  and  $S(a)=\{\sigma_n(a)\}$  for
all $a\in A.$  Note that $\pi\circ \Phi,$
$\pi\circ \Psi$ and $\pi\circ {S}$ are \hm s.  Denote also 
$\Phi^{(m)}, \Psi^{(m)}, S^{(m)}: A\to \prod_{{n\ge m}} B_n'$ by
$\Phi^{(m)}(a)=\{\phi_n(a)\}_{n\ge m},$ $\Psi^{(m)}(a)=\{\psi_n(a)\}_{n\ge m}$
and $S^{(m)}(a)=\{\sigma_n(a)\}_{n\ge m}.$

For each $u\in {\cal U}_m,$ we may assume that
$u\in M_{L(m)}({\tilde A})$ {for some integer $L(m)\ge 1$.}  When $n\ge m,$  by the assumption,
there exists a continuous path of unitaries  $\{u_n(t): t\in [0,1]\}\subset M_{L(m)}({\tilde B_n'})$ such that
\vspace{-0.1in} \beq \nonumber
u_n(0)=\lceil \phi_n(u)\rceil, u_n(1)=\lceil \psi_n(u)\rceil\andeqn
{\rm cel}(\{u_n(t)\})\le {\bf L}(u).\nonumber
\eneq
It follows from Lemma 1.1 of \cite{GL1} that, for all $n\ge m,$ there exists a continuous path
$\{U(t): t\in [0,1]\}\subset U_0(\prod_{n{\ge}m} {\tilde B_n'})$ such that
$U(0)=\{\langle \phi_n(u)\rangle\}_{n\ge m}$ and $U(1)=\{\langle \psi_n(u)\rangle\}_{n\ge m}.$
This in particular implies that
\beq\label{stableun2-4}
&&\hspace{0.2in} \langle\Phi^{(m)}(u)\rangle \langle \Psi^{(m)}(u^*) \rangle \in U_0(M_{L(m)}(\prod_{n{\ge}m} {\tilde B_n'}))
\andeqn (\pi\circ \Phi)_{*1}=(\pi\circ \Psi)_{*1}.
\eneq
By (\ref{Lauct-1}),  for all $n\ge m,$
\vspace{-0.14in} \beq\label{stableun2-5}
[\phi_n]|_{{\cal P}_m}=[\psi_n]|_{{\cal P}_m}.
\eneq
By the assumption and by \cite{GL1},  $K_0(C)=\prod_bK_0(B_n').$
It follows that
\beq\label{stableun2-6}
[\Phi^{(m)}]|_{K_0(A)\cap {\cal P}_m}=[\Psi^{(m)}]|_{K_0(A)\cap {\cal P}_m},\,\,m=1,2,....
\eneq
In particular,
\vspace{-0.14in} \beq\label{stableun2-7}
(\pi\circ \Phi)_{*0}=(\pi\circ \Psi)_{*0}.
\eneq
Now let $x_0\in {\cal P}_m\cap K_0(A,\Z/k\Z)$ for some $k\ge 2.$   Let ${\tilde x_0}\in K_1(A)$ be the image
of $x_0$ under the map $K_0(A, \Z/k\Z)\to K_1(A).$
We may assume that ${\tilde x_0}\in {\cal P}_{m_0}$ for some
$m_0\ge m.$  By  (\ref{stableun2-6}), $[\Phi^{(m_0)}]({\tilde x})=[\Psi^{(m_0)}]({\tilde x}).$
Put $y_0=[\Phi^{(m_0)}](x)-[\Psi^{(m_0)}](x).$ Then
$y_0\in K_0((\prod_{n{\ge}m_0}B_n'), \Z/k\Z)$ must be in the image of $K_0(\prod_{n{\ge}m_0}B_n')$
which may be identified with
$K_0(\prod_{n{\ge}m_0}B_n')/kK_0(\prod_{n{\ge}m_0}B_n').$
(see \cite{GL1}).
However, by  (\ref{stableun2-5}),
$$y_0\in {\rm ker}\,\psi_0^{(k)},$$
where $\psi_0^{(k)}: K_0(\prod_{n{\ge}m_0}B_n', \Z/k\Z)\to \prod_{n{\ge}m_0}K_0(B_n', \Z/k\Z)$
is  as in 4.1.4 of \cite{Lncbms}.
By \cite{GL1},
$y_0=0.$ In other words,
$$[\Phi](x)=[\Psi](x)$$
which implies that
\vspace{-0.1in} \beq\label{stableun2-10}
[\pi\circ \Phi]|_{K_0(A,\Z/k\Z)}=[\pi\circ \Psi]|_{K_0(A, \Z/k\Z)},\,\,k=2,3,....
\eneq

Now let $x_1\in K_1(A,\Z/k\Z).$ We may assume that $x\in {\cal P}_m$ for some $m\ge 1.$
Let ${\tilde x_1}\in K_0(A)$ be the image of $x$ under the map $K_1(A,\Z/k\Z)\to K_0(A).$
There is $m_1\ge m$ such that ${\tilde x}\in {\cal P}_{m_1}.$
By (\ref{stableun2-6}), $[\Phi^{(m_1)}]({\tilde x})=[\Psi^{(m_1)}]({\tilde x}).$
Put $y_1=[\Phi^{(m_1)}](x_1)-[\Psi^{(m_1)}](x_1).$ Then
$y_1\in K_1(\prod_{n=m_1}B_n')/kK_1(\prod_{n=m_1}B_n')$ (see (\cite{GL1})).
However, by (\ref{stableun2-4}), $y_1\in {\rm ker}\psi_1^{(k)}$ (see 4.1.4 of \cite{Lncbms})
It follows from
\cite{GL1} that $y_1=0.$ In other words,
$$[\Phi](x_1)=[\Psi](x_1).$$
Thus
\vspace{-0.14in} \beq\label{stableun2-12}
[\pi\circ \Phi]|_{K_1(A,\Z/k\Z)}=[\pi\circ \Psi]|_{K_1(A, \Z/k\Z)}.
\eneq
Therefore, combining (\ref{stableun2-4}), (\ref{stableun2-7}), (\ref{stableun2-10}) and (\ref{stableun2-12}),
$$[\pi\circ \Phi]=[\pi\circ \Psi]\,\,\,{\rm in} \,\,\,  Hom_{\Lambda}(\underline{K}(A), \underline{K}(Q(C)).
$$
For each $a\in {\cal H}_m\subset A_+^{\bf 1}\setminus \{0\},$
any $\{b_n\}\in C_+^{\bf 1},$ and any $\eta>0,$  since $\sigma_n$ is $F$-${\cal H}_n$-full,
for all $n\ge m,$
there are $x_{i,n}(a)\in B_n'$ with $\|x_{i,n}\|\le M(a),$ $i=1,2,...,N(a),$
where $F(a)=M(a)\times N(a),$ such that
$$\|\sum_{i=1}^{N(a)} x_{i,n}(a)^*\sigma_n(a) x_{i,n}(a)- b_n\|<\eta.$$
Define $x(i,a)=\{x_{i,n}(a)\}.$ Then $x(i,a)\in C.$
It follows
that
$$\|\sum_{i=1}^{N(a)}x(i,a)^*S^{(m)}(a)x(i,a)-\{b_n\}\|<\ep.
$$
This shows that $\pi\circ S(a)$ is a full element in $Q(C)$ for all $a\in \cup_{n=1}^{\infty} {\cal H}_n.$
Let $I$ be an ideal of $C$ and
let
$$
J=\{a\in A: \pi\circ S(a)\in I\}.
$$
The same argument used in the proof of \ref{Lsep} shows that $J=\{0\}.$ It follows
that $\pi\circ S$ is full.

It follows from \ref{Lsep2} that there exists a separable \SCA\,
$D\subset Q(C)$ such that
$\pi\circ S(A), \pi\circ \Phi(A), \pi\circ \Psi(A)\subset D$ and
$$
[\pi\circ \Phi]=[\pi\circ \Psi]\,\,\, {\rm in}\,\,\, KL(A, D).
$$


It follows from the end of the proof \ref{T39} that there exists an integer $K\ge 1$ and
unitary $V\in M_{K+1}({\widetilde{Q(C)}})$  such that
$$\|V^*{\diag}(\pi\circ \Phi(a), \Sigma(a))V-{\diag}(\pi\circ \Psi(a), \Sigma(a))\|<\ep_0/4$$
for all $a\in {\cal F},$ where
\vspace{-0.12in} $$\Sigma(a)={\diag}(\overbrace{\pi\circ S(a), \pi\circ S(a),...,\pi\circ S(a)}^K)\rforal a\in A.$$
Therefore there exists a sequence of  unitaries $\{v_n\}\subset M_{K+1}({\tilde C})$ and an integer
$N_1$ such that $k(n)\ge K$ for all $n\ge N_1$ and
$$\|v_n^*{\diag}(\phi_n(a), S_{n,K}(a))v_n-{\diag}(\psi_n(a), S_{n,K}(a))\|<\ep_0/2$$
for all $a\in {\cal F},$ where
\vspace{-0.13in} $$
S_{n,K}(a)={\diag}(\overbrace{\sigma_n(a),\sigma_n(a),...,\sigma_n(a)}^K)\rforal a\in A.
$$
This contradicts  (\ref{stableun2-1}).

If furthermore $B$ almost has stable rank one, then one can deploy the same proof as that of
\ref{T39sr1}.
\end{proof}

\begin{rem}\label{RRLuniq}
Let $K_0(A)\cap {\cal P}={\cal P}_1$ and write
${\cal P}_1=\{z_1,z_2,...,z_m\}.$ Then, by choosing sufficiently large ${\cal P},$
we can always choose ${\cal U}=\{w_1, w_2,...,w_m\}$ so that
$[w_i]=z_i,$ $i=1,2,...,m.$  In other words, that we do not need to consider
unitaries in $U_0(M_{\infty}({\tilde A}).$ In particular, if $K_1(A)=\{0\},$
then we can omit the condition \eqref{Lauct-1}.
Moreover, if $B$ is restricted in the class of \CA s of real rank zero, then
one can choose ${\bf L} \equiv 2\pi +1$ and \eqref{Lauct-1} always holds if
${\cal P}$ is sufficiently large. In other words, in this case, condition \eqref{Lauct-1} can
also be dropped.

\end{rem}

\begin{cor}\label{CLuniq}
Let $A$ be a non-unital separable amenable  \CA\, which satisfies the UCT such that $K_i(A)=\{0\}$
($i=0,1$),
and let $T: A_+\setminus \{0\}\to \N\times \R_+\setminus \{0\}$
be a  map.
  For any $\ep>0$ and any finite subset ${\cal F}\subset A,$ there exists
$\dt>0,$ a finite subset ${\cal G}\subset A,$
a finite subset ${\cal H}\subset A_+\setminus \{0\}$ and an integer $K\ge 1$ satisfying the following:
For any two ${\cal G}$-$\dt$-multiplicative \morp s $\phi, \psi: A\to B,$  where
$B$ is any  $\sigma$-unital \CA,
and any
${\cal G}$-$\dt$-multiplicative \morp\, $\sigma: A\to M_l(B)$ (for any integer $l\ge 1$) which is also
$T$-${\cal H}$-full, 
there exists a unitary $U\in M_{1+Kl}({\tilde B})$ such that
\beq\label{CLauct-2}
\|{\rm Ad}\, U\circ (\phi\oplus S)(a)-(\psi\oplus S)(a)\|<\ep\tforal a\in {\cal F},
\eneq
where
\vspace{-0.1in} $$
S(f)={\rm diag}(\overbrace{\sigma(a), \sigma(a),...,\sigma(a)}^K)\tforal a\in A.
$$
If furthermore, $B$ almost has stable rank one, one can choose $U\in {\widetilde{ M_{1+Kl}(B)}}.$
\end{cor}

\begin{proof}
The only reason that the restriction has to be placed on $B$ is for the computation
of $K$-theory of maps $\phi$ and $\psi.$ Since now $K_i(A)=\{0\},$ this problem disappears.
\end{proof}

\begin{cor}\label{CCLuniq}
Let $A$ be a non-unital separable amenable  \CA\, which satisfies the UCT, 
and let $T: A_+\setminus \{0\}\to \N\times \R_+\setminus\{0\}$
be a  map.
  For any $\ep>0$ and any finite subset ${\cal F}\subset A,$ there exists
$\dt>0,$ a finite subset ${\cal G}\subset A,$
a finite subset ${\cal H}\subset A_+\setminus \{0\}$ and an integer $K\ge 1$ satisfying the following:
For any two ${\cal G}$-$\dt$-multiplicative \morp s $\phi, \psi: A\to B,$  where
$B$ is any  $\sigma$-unital \CA\, with $K_i(B)=\{0\}$ ($i=0,1$)
and any
${\cal G}$-$\dt$-multiplicative \morp\, $\sigma: A\to M_l(B)$ (for any integer $l\ge 1$) which is also
$T$-${\cal H}$-full, 
there exists a unitary $U\in M_{1+Kl}({\tilde B})$ such that
\beq\label{CCLauct-2}
\|{\rm Ad}\, U\circ (\phi\oplus S)(a)-(\psi\oplus S)(a)\|<\ep\tforal a\in {\cal F},
\eneq
where
\vspace{-0.14in} $$
S(f)={\rm diag}(\overbrace{\sigma(a), \sigma(a),...,\sigma(a)}^K)\tforal a\in A.
$$
If furthermore, $B$ almost has stable rank one, one can choose $U\in {\widetilde{ M_{1+Kl}(B)}}.$
\end{cor}

\begin{rem}\label{Rstuniq}
In Theorem \ref{Lauct2}, \ref{CLuniq} as well as \ref{CCLuniq},
if $B$  is $\sigma$-unital and has almost stable rank one, and there is a  $\sigma$-unital hereditary \SCA\, $B_0\subset B$ such that
$\phi(A), \psi(A)\subset B_0,$ then the unitary $U$ can be chosen in
${\tilde B_1},$
where $B_1$ is the hereditary \SCA\, of $M_{1+Kl}(B)$ generated by
$B_0\oplus M_{Kl}(B).$

This can be seen as follows.

First one  may assume that ${\cal F}\subset A^{\bf 1}.$ Let $\ep>0.$
Then  note that $B_1$ is $\sigma$-unital.

Fix $\ep>\eta>0.$
Let $e\in B_1$ be a strictly positive  element such that
$$
\|f_{\dt}(e)x-x\|<\eta/16\andeqn \|x-xf_\dt(e)\|<\eta/16
$$
for all $x\in \{(\phi\oplus S)(a), (\psi\oplus S)(a): a\in {\cal F}\},$
where $1/2>\dt>0.$
Suppose that there is a unitary $U$ in the unitization of $M_{1+Kl}(B)$ such that
$$
\|U^*(\phi\oplus S)(a)U(a)-(\psi\oplus S)(a)\|<\eta/16\rforal a\in {\cal F}.
$$
Put $z=(f_{\dt}(e)Uf_{\dt/2}(e))\in B_1.$
Then
\beq\nonumber
&&zz^*(\phi\oplus S)(a)\approx_{\eta/16}zf_{\dt/2}U^*(\phi\oplus S)(a)\approx_{\eta/16}
zf_{\dt/2}(e)(\psi\oplus S)(a)U^*\\\nonumber
&& \approx_{\eta/16} z(\psi\oplus S)(a)U^*\approx_{\eta/16} f_{\dt}(e)U(\psi\oplus S)(a)U^*\\\nonumber
&&\approx_{\eta/16} f_{\dt}(e) (\phi(a)\oplus S)(a)\approx_{\eta/16} (\phi\oplus S)(a)
\eneq
for all $a\in {\cal F}.$  Exactly the same estimates shows that
$$
(\phi\oplus S)(a)zz^*\approx_{6\eta/16} (\phi\oplus S)(a)\rforal a\in {\cal F}.
$$
With a sufficiently small $\eta,$ one has
$$
|z^*|(\phi\oplus S)(a)\approx_{\ep/64} (\phi\oplus S)(a)\approx_{\ep/64} (\phi\oplus S)(a)|z^*|
$$
for all $a\in {\cal F}.$
One also has
$$
\|z^*(\phi\oplus S)(a)z-(\psi\oplus S)(a)\|<\eta/8\rforal a\in {\cal F}.
$$

Choose $1/2>\dt_1>0$ such that
\vspace{-0.1in} $$
f_{\dt_1}(|z^*|)|z^*|\approx_{\ep/64} |z^*|.
$$
Write $z^*=w|z^|$ as a poler decomposition in $B_1^{**}.$
Since $B_1$ has almost stable rank one, there exists a unitary $v\in {\tilde B_1}$
such that
$$
vf_{\dt_1}(|z^*|)=wf_{\dt_1}(|z^*|).
$$
Note that
\vspace{-0.1in} $$
vf_{\dt_1}(|z^*|)\approx_{\eta/64} z^*.
$$
It follows that
$$
\|v(\phi\oplus S)(a)v^*-(\psi\oplus S)(a)\|<\ep\rforal a\in {\cal F}.
$$
\end{rem}

\begin{rem}\label{Ruct}
Suppose that $A$ has the property that $KK(A,B)=\{0\}$ for every \CA s $B.$
Then the assumption of UCT can be dropped. Note that
UCT condition is used only to ensure that $[\pi\circ \Phi]=[\pi\circ \Psi]$ in $KL(A,D).$
Moreover, in this case, the condition related to $K$-theory can also dropped.
In particular, \ref{CLuniq}  holds if we replace $A$ by  a non-unital separable amenable  \CA\, which satisfies
$KK(A,B)=\{0\}$ for all \CA s $B.$
\end{rem}

\section{Quasi-compact \CA s}

\begin{df}\label{Dcompact}
A $\sigma$-unital  \CA\, is said to be {\it quasi-compact}, if there is $e\in (A\otimes {\cal K})_+$
with $0\le e_1\le 1$
and a partial isometry $w\in (A\otimes {\cal K})^{**}$ such
that
\beq\nonumber
w^*a,\, w^*aw\in A\otimes {\cal K}, ww^*a=aww^*=a \andeqn ew^*aw=w^*awe=w^*aw\rforal a\in A.
\eneq

\end{df}

\begin{prop}
Let $C$ be $\sigma$-unital  \CA\, and let $c\in C\setminus \{0\}$ with $0\le c \le 1$ be a full element in $C.$
Suppose that there is $e_1\in C$  with $0\le e_1\le 1$ such that $e_1c=ce_1=c.$
Then $\overline{cCc}$ is quasi-compact.
\end{prop}

\begin{proof}
$B=\overline{cCc}.$
Let $E\subset M_2(C)$ be a \SCA\,
of the form
$$
E=\{(a_{ij})_{2\times 2}: a_{11}\in B, a_{12}\in \overline{BC}, a_{21}\in \overline{CB}, a_{22}\in C\}.
$$
We view $B\otimes {\cal K}$ and $C\otimes {\cal K}$ as full corners of $E\otimes {\cal K}.$
Moreover, let $p_1$ be the range projection of $B\otimes {\cal K}$ and let $p_2$
be the range projection of $C\otimes {\cal K},$ then
$p_1, p_2\in M(E\otimes {\cal K}).$
By  2.8 of \cite{Br1}, there  is a partial  isometry $W\in M(E\otimes {\cal K})$
such that $W^*(B\otimes {\cal K})W=C\otimes {\cal K},$
$WW^*=p_1$ and $WW^*=p_2.$ Moreover,
$$
p_1Wb=p_1Wbp_1\rforal b\in B\otimes {\cal K}.
$$
Since $W\in M(E\otimes {\cal K}),$ $Wb\in E\otimes {\cal K}.$
Therefore $p_1Wbp_1\in B\otimes {\cal K}.$
Put  $w=(Wp_1)^*.$
Then, for any $b\in B,$
$$
(w)^*b=p_1Wb\in B\otimes {\cal K}, w^*bw=p_1WbW^*p_1\in B\otimes {\cal K}.
$$
Moreover, let $e=p_1We_1W^*p_1\in B\otimes {\cal K},$
$$
ew^*bw=p_1We_1W^*p_1WbW^*p_1=p_1We_1p_2bW^*p_1=p_1We_1bW^*p_1=p_1WbW^*p_1=w^*bw
$$
for all $b\in B.$
We also have $w^*bwe=ewbw^*$ and $ww^*b=p_1W^*Wp_1b=p_1b=b$ and
$bww^*=b$ for all $b\in B.$

 Thus
 $B$ is quasi-compact.
\end{proof}

\begin{cor}\label{CherePA=A}
A $\sigma$-unital full hereditary \SCA\, of a $\sigma$-unital quasi-compact \CA\, is quasi-compact.
\end{cor}

\begin{lem}\label{subcommatrix}
Let $A$ be a $\sigma$-unital and non-unital  \CA\, which is quasi-compact.
Then, there exists an integer $N\ge 1,$  a partial isometry $w\in M_N(A)^{**}$ and $e\in M_N(A)$
with $0\le e\le 1$
such that
$$
w^*aw\in M_N(A), \,\,\, ww^*a=aww^*=a\andeqn w^*awe=ew^*aw=w^*aw\rforal a\in A.
$$
\end{lem}

\begin{proof}
Let $b\in  A$ be a strictly positive element with $0\le b\le 1.$
We may assume that   $A$ is a full hereditary \SCA\, of a
$\sigma$-unital \CA\, $C$  such that $b\in C$ and
there is  $e_1\in C$ with $0\le e_1\le 1$ such that
$be_1=b=be_1.$ Moreover, $bCb=A.$
Since  $b$ is full in $C,$  by  applying \ref{Lrorm},
there exists  $x_1, x_2, ...,x_m\in C$ such that
$$
\sum_{i=1}^m x_i^*bx_i=f_{1/4}(e_1).
$$
Note that $f_{1/4}(e_1)b=b=f_{1/4}(e_1).$
Put $X^*=(x_1^*b^{1/2}, x_2^*b^{1/2},...,x_m^*b^{1/2})$ as an 1-row element in $M_m(C).$
Then
$$
X^*X=f_{1/4}(e_1)\andeqn XX^*\in M_m(b^{1/2}Cb^{1/2})=M_m(A).
$$
Let $X=v|X^*X|^{1/2}$ be the polar decomposition of $X$ in $M_m(C)^{**}.$
Then
$$
vav^*=v|X^*X|^{1/2}a|X^*X|^{1/2}v^*=XaX^*\in M_m(A)\rforal a\in A.
$$
Note $Xb^{1/n}\in M_m(A).$  Let $p$ be the open projection corresponding to $b.$
Then $Xp\in M_m(A)^{**}.$ Note also
that $Xp=v|X^*X|^{1/2}p=vp.$ Set $w=(Xp)^*.$ Then $w^*=vp$ and $ww^*=pX^*Xp=pf_{1/4}(e_1)p=p.$
So $w$ is a partial isometry.  Set $e=XX^*.$
$$
w^*aw=XaX^*\in M_m(A)\rforal a\in A.
$$
Moreover,
\beq\nonumber
w^*awe=XaX^*XX^*=Xaf_{1/4}(e_1)X^*=XaX^*=w^*aw\andeqn\\
ew^*aw=XX^*XaX^*=Xf_{1/4}(e_1)aX^*=XaX^*=w^*aw.
\eneq

\end{proof}



\begin{lem}\label{compactrace}
Let $A$ be a $\sigma$-unital  quasi-compact  \CA\, with $T(A)\not=\{0\}.$
Then the weak *-closure $\overline{T(A)^w}$ of $T(A)$ in ${\tilde T}(A)$ is compact and
$0\not\in \overline{T(A)}^w.$
  Moreover,
if $a\in A$ with $0\le a\le 1$ is strictly positive, then
there is $d>0$ such
that
$$
d_\tau(a)\ge d\rforal \overline{T(A)^w}.
$$

\end{lem}

\begin{proof}
By \ref{subcommatrix}, \wilog, we  may assume that $A$ is a full hereditary \SCA\, of $B\otimes {\cal K}$ for some
$\sigma$-unital
\CA\, $B$ such that  there is $e_1\in (B\otimes {\cal K})_+$ such that $0\le e_1\le 1$ and
$e_1x=xe_1=x$ for all $x\in A.$  Put $C=B\otimes {\cal K}.$
Let ${\tilde T}(C)$ be the convex set of all traces defined on the Pedersen ideal $P(C).$
Note that $A\subset P(C).$
Since $A$ is full in $C,$ we may also assume that each $\tau\in T(A)$ has been extended  to
an element in ${\tilde T}(C).$ Let $a\in A_+$ be a strictly positive element of $A.$
There are $x_1, x_2,....,x_m\in C$ such that
\beq\label{tcompact-2}
\sum_{i=1}^mx_i^*ax_i=f_{1/4}(e_1).
\eneq
Note that $f_{1/4}(e_1)x=xf_{1/4}(e_1)=x$ for all $x\in A.$
But $f_{1/4}(e_1)\in P(C).$
Consider
$$
S=\{\tau\in {\tilde T}(C): \tau(f_{1/4}(e_1))\ge 1\}.
$$
Then $S$ is compact.
Note that
$$
T(A)=\{\tau\in {\tilde T}(C): d_\tau(a)=1\}.
$$
It follows that, for any $\tau\in T(A),$
\beq\label{tcompact-4}
\tau(f_{1/4}(e_1))\ge d_\tau(a)=1.
\eneq
It follows that $T(A)\subset S.$
Therefore the weak*-closure of $T(A)$ is compact and
$0\not\in \overline{T(A)}^w.$

By \eqref{tcompact-2},
\beq\label{tcompact-3}
[f_{1/4}(e_1)]\lesssim  m[a]\,\,\, {\rm in}\,\,\, W(C).
\eneq
Put $d=1/m.$
It follows that, for any $t\in S,$
$$
d_\tau(a)\ge t(f_{1/4}(e_1))/m\ge 1/m=d.
$$

\end{proof}

\begin{cor}\label{Csemicon}
Let $A$ be a $\sigma$-unital \CA\, which is quasi-compact  with  $QT(A)=T(A)\not=\{0\}$ and
let $a\in A$ be a strictly positive element with $0\le a\le 1.$
Suppose
that
$$d=\inf \{d_\tau(a): \tau\in \overline{T(A)}^w\}>0.$$
Then, for any $d/3<d_0<d,$  there exists an integer $n\ge 1$ such that, for all $m\ge n,$
$$
\tau(f_{1/m}(a))\ge d_0\tand \tau(a^{1/m})\ge d_0.
$$
\end{cor}



This is a standard compactness argument.

\begin{thm}\label{Tpedersen}
Let $A$ be a $\sigma$-unital  \CA.  Then $A$ is quasi-compact if and only if
$P(A)=A.$

\end{thm}

\begin{proof}
First assume that $A$ is quasi-compact.
Let $a\in A_+$ be a strictly positive element. Then there exists
$e\in (A\otimes {\cal K})_+$ and a partial isometry $v\in (A\otimes {\cal K})^{**}$ such that
\beq\nonumber
w^*xw\in A\otimes {\cal K},\,\,\, ww^*x=xww^*=x \andeqn w^*xwe=ew^*xw=w^*xw\rforal x\in A.
\eneq
Let $z=w^*a^{1/2}.$ Then $zz^*\in P(A\otimes {\cal K})_+.$
It follows that $a=z^*z\in P(A\otimes {\cal K})_+.$ Therefore
the hereditary \SCA\, generated by $a$ is in $P(A\otimes {\cal K}).$  Consequently
$A\subset P(A\otimes {\cal K}).$

To see  that $A=P(A), $   one applies Theorem 2.1 of \cite{aTz}.




Conversely, suppose that
there are $b_1, b_2,...,b_m\in A_+$ such that
$$
a^{1/2}\le \sum_{i=1}^mg_i(b_i)\andeqn 0\le b_1, b_2,...,b_m\le 1,
$$
where $g_i\in C_0((0,N)$ for some  $N\ge 1$ and the support of $g_i$ is in $ [\sigma, N]$
for some $1/2>\sigma>0.$ Note each $g_i=\sum_{j=1}^K g_{i,j}$ for some $\|g_{i,j}\|+1\ge K\ge 1,$
where  $0\le g_{i,j}\le 1$ and its support still in $[\sigma, N].$ \Wlog, we may assume
that $0\le g_i\le 1.$

Let $c_i=(g_i(b_i))^{1/2},$ $i=1,2,...,m.$

Define
$$
Z=(c_1, c_2,...,c_m)
$$
which we view as a $m\times m$ matrix with zero rows other than the first row.
Define
$$E=\diag(f_{\sigma/2}(b_1), f_{\sigma/2}(b_2),...,f_{\sigma/2}(b_m))\in M_m(A).$$
Note that
\vspace{-0.12in} $$
ZZ^*=\sum_{i=1}^mc_i^2\ge a\andeqn Z^*Z=(d_{i,j})_{m\times m},
$$
where
\vspace{-0.1in} $$
d_{i,j}=c_ic_j,\,\,\,i,j=1,2,...,m.
$$
It follows that
$$
E(Z^*Z)=E(c_ic_j)_{m\times m}=(c_ic_j)_{m\times m}=(Z^*Z)E.
$$
Write $Z^*=V|Z^*|.$ Then
$$
Vx\in M_m(A), VV^*|Z|=|Z|VV^*=|Z|\andeqn
(VxV^*)E=E(VxV^*)=(VxV^*)
$$
for all $x\in \overline{(ZZ^*)M_m(A)(ZZ^*)}.$
Note that $A\subset \overline{(ZZ^*)M_m(A)(ZZ^*)}.$ Therefore
$A$ is quasi-compact.

\end{proof}

The number $n$ below will be used later.

\begin{lem}\label{Lbkqc}
Let $A$ be a $\sigma$-unital   \CA\,  with $QT(A)=T(A)\not=\emptyset$ and
$0\not\in \overline{T(A)}^w.$
Suppose that $A$ has the strong strict comparison for positive elements and has almost stable rank one.
Then $A$ is quasi-compact.
Moreover, let $a\in A$ with $0\le a\le 1$ be a strictly positive element,  let
$$
d=\inf\{d_\tau(a): \tau\in \overline{T(A)}^w\},
$$
and  let $n$ be an integer such that $nd>1,$
there exists elements $e_1, e_2\in M_n(A)$  with $0\le e_1, e_2\le 1,$
$e_1e_2=e_2e_1=e_1$ and $w\in M_n(A)^{**}$
such that
\beq\label{Lbkqc-1}
w^*c, cw\in M_n(A), ww^*c=cww^*=c\rforal c\in A,\\
w^*cwe_1=e_1w^*cw=c\rforal c\in A.
\eneq
Furthermore,
 there exists a full element $b_0\in P(A)$ with $0\le b_0\le 1$ and $e_0\in P(A)_+$
such that $b_0e=eb_0=b_0.$
\end{lem}

\begin{proof}
Let $a\in A_+$ with $0\le a\le 1$ be a strictly positive element.
Since $0\not\in \overline{T(A)}^w$ and $\overline{T(A)}^w$  is compact,
$$
\inf\{d_\tau(a): \tau\in \overline{T(A)}^w\}=d>0.
$$
Suppose that $nd>1.$

By \ref{Csemicon}
there exists  $\ep >0$  such that
$$
\inf\{\tau(f_\ep(a)): \tau\in \overline{T(A)}^w\}=d_0>2d/3
$$
such that $nd_0>1.$
Therefore, by the strong strict comparison,
\vspace{-0.1in} $$
a\lesssim \diag(\overbrace{f_\ep(a), f_\ep(a),...,f_\ep(a)}^n)\,\,\, {\rm in}\,\, M_n(A).
$$
Put $b=\diag(\overbrace{f_\ep(a), f_\ep(a),...,f_\ep(a)}^n).$
Since $A$ is assumed to almost have stable rank one, by \ref{Lalmstr1}, there exists $x\in M_n(A)$
such that
$$
x^*x=a^{1/2}\andeqn xx^*\in \overline{bM_n(A)b}.
$$
There is a partial isometry $w\in M_n(A)^{**}$ such that
$w^*A, Aw\in M_n(A),$ $ww^*c=cww^*=c$ for all $c\in A$ and
$w^*Aw\in  \overline{bM_n(A)b}.$ Put $e=\diag(\overbrace{f_{\ep/2}(a), f_{\ep/2}(a),...,f_{\ep/2}(a)}^n).$
Then $0\le e\le 1$ and
$$
ew^*cw=w^*cwe=w^*cw\rforal c\in A.
$$
Thus $A$ is quasi-compact.
The second part of the statement also holds.

For the last part of the lemma, choose $b_0=f_{\ep}(a)$ and
$e_0=f_{\ep/2}(a).$

\end{proof}

\begin{rem}\label{Dboundedscale}
Let $A$ be a $\sigma$-unital  exact simple \CA. Let $e\in P(A)_+\setminus\{0\}.$
Let $T_e(A)=\{\tau\in {\tilde T}(A): \tau(e)=1\}.$
It is a compact convex set. Let $a\in A$ be the strictly positive element.
A is said to have bounded scale if $d_\tau(a)$ is a bounded function
on $T_e(A)$ (see \cite{Btrace}).  In the absence of strict
comparison, one defines that $A$ has bounded scale if there exists an integer $n\ge 1$ such
that $n\la e\ra \ge \la b\ra$  for any $b\in A_+.$ However, as shown in \cite{Btrace}, this is equivalent
to say $A$ is algebraic simple which in turn is equivalent to say that $A$ is quasi-compact.

\end{rem}

\begin{prop}\label{Pheretc}
Let $A$ be a $\sigma$-unital  \CA\, with $0\not\in \overline{T(A)}^w$ and
every trace in ${\tilde T}(A)$ is finite on $A.$
Let $B\subset A$ be a $\sigma$-unital full hereditary \SCA.
Then $0\not\in \overline{T(B)}^w.$
\end{prop}

\begin{proof}
Let $b\in B$ with $\|b\|=1$ is a strictly positive element of $B$ and
$B=\overline{bBb}.$ Let $e\in A_+$ with $\|e\|=1$ such that
$A=\overline{eAe}.$

Since $b$ is full,  then $\tau(b)>0$ for all $\tau\in \overline{T(A)}^w.$
Then
$$
1>r_0=\inf\{\tau(b): \tau\in \overline{T(A)}^w\}>0.
$$

For any $t\in T(B),$ there is an extension $\tau\in {\tilde T}(A)$
which is finite. Since $\tau$ is positive linear functional, it is continuous.
Therefore
$\tau_0=\tau/\|\tau\|\in T(A).$   Therefore
$t=\|\tau\|\cdot \tau_0|_{B}.$
It follows that
$$
t(b)\ge \|\tau\|\cdot r_0\ge r_0.
$$
This shows that $0\not\in \overline{T(B)}^w.$

\end{proof}

\begin{df}\label{Dkerrho}
Let $A$ be a $\sigma$-unital \CA\, with ${\tilde T}(A)\not=\{0\}.$
Suppose that there  is $e\in P(A)_+$ which is full.

Let $A_e=\overline{eAe}.$ Then $A_e$ is quasi-compact.
Then $0\not\in \overline{T(A_e)}^w.$
Assume that $A$ is not unital. One extends each $\tau\in \overline{T(A_e)}^w$
to a tracial state on ${\tilde A}_e.$
There is an order preserving \hm\, $\rho_{{\tilde A}_e}: K_0({\tilde A}_e)\to \Aff(\overline{T(A_e)}^w).$
By \cite{Br1}, one may identify $K_0(A)$ with $K_0(A_e).$
The composition of maps from $K_0(A)$ to $K_0(A_e),$ then from
$K_0(A_e)$ to $K_0({\tilde A}_e)$ and then to $\Aff(\overline{T(A_e)}^w)$ is an order preserving
\hm\, which will be denoted  by $\rho_A.$
Denote by ${\rm ker}\rho_{A}$ the subgroup
of  $K_0(A)$  consisting of those $x\in K_0(A)$ such that
${\rm ker}\rho_{A}(x)=0.$ Elements in ${\rm ker}\rho_A$  are called infinitesimal elements.
\end{df}

\section{Continuous scale and fullness}

\begin{df}\label{Dconscale}
The previous section discussed \CA s with bounded scales. Let us recall the definition of continuous scale
(\cite{Lncs1} and \cite{Lncs2})

Let $A$ be a $\sigma$-unital and non-unital \CA.
Fix an approximate identity $\{e_n\}$ for $A$
with the property that
$$
e_{n+1}e_n=e_ne_{n+1}=e_n\rforal n\ge 1.
$$
Then $A$ has continuous scale, if for any $b\in A_+\setminus \{0\},$
there exists $N\ge 1$ such that
$$
e_m-e_n\lesssim b
$$
for all $m>n\ge N.$
This definition does not depend on the choice of $\{e_n\}.$

\end{df}

\begin{rem}\label{Rcontext}
It should be noted that for any non-elementary separable simple \CA\, $A$
there is a non-zero hereditary \SCA\, $B\subset A$ such that $B$ has continuous scale (2.3 of \cite{Lncs2}).
\end{rem}

\begin{thm}[cf. \cite{Lncs2}]\label{Pconscale}
Let $A$ be a non-elementary $\sigma$-unital simple \CA\, with continuous scale such that $T(A)\not=\emptyset.$
Then

{\rm (1)} $T(A)$ is compact;

{\rm (2)} $d_\tau(a)$ is continuous on ${\tilde T}(A)$ for any strictly positive element $a$ of $A;$


{\rm (3)} $d_\tau(a)$ is continuous on $\overline{T(A)}^w$  for any strictly positive element $a$ of $A.$

Conversely, if $A$  is exact, has strict comparison for
positive elements and  is quasi-compact, then {\rm (1)}, {\rm (2)} and {\rm (3)} are  equivalent  and also
equivalent to the following:

{\rm (4)} $A$ has continuous scale;

{\rm (5)} $d_\tau(a)$ is continuous on $\overline{T(A)}^w$  for some  strictly positive element $a$ of $A.$

{\rm (6)} $d_\tau(a)$ is continuous  on ${\tilde T}(A)$ for some strictly positive element $a$ of $A.$

\end{thm}

\begin{proof}
Most part of the theorem  are well-known.  It is that (1) holds is  perhaps not well-known.

Suppose that $A$ has continuous scale. Then $A$ is algebraically simple  (3.3 of \cite{Lncs1}).
In particular, $A=P(A).$ It follows from \ref{Csemicon} that
$K=\overline{T(A)}^w$ is compact.
Fix an element $b\in P(A)_+\setminus \{0\}.$
Put $B=\overline{f_{1/2}(b)Af_{1/2}(b)}.$ Note that $A$ is not elementary.
$B_+$ contains infinitely many mutually orthogonal elements $\{x_n\}$
with $0\le x_n\le 1,$ $n=1,2,....$

Thus, for $\ep>0,$ there exists $x_{k(\ep)}$ such that $d_\tau(x_{k(\ep)})<\ep$ for all
$\tau\in K.$
Note that
$$
d_\tau(a)=\lim_{n\to \infty}\tau(f_{1/2^n}(a))\rforal \tau\in K.
$$
Denote $e_n=f_{1/2^n}(a),$ $n=1,2,....$ Since $A$ has continuous scale,  for any $\ep>0,$
there exists $N\ge 1$ such that
$$
e_m-e_n\lesssim x_{k(\ep)}\rforal m>n\ge N.
$$
In particular,
\vspace{-0.12in} \beq\label{Pconscale-2}
\tau(e_m)-\tau(e_n)<\ep\rforal \tau\in K.
\eneq
It follows that $d_\tau(a)$ is continuous on $K.$
Since $d_\tau(a)=1$ on $T(A)$ and $T(A)$ is dense in $K,$
$d_\tau(a)=1$ for all $\tau\in K.$ This implies that
$T(A)=K.$ This proves (1)  and (3).
It is clear that (2) are equivalent to (3).

Conversely,  if $A$ is as stated and if $d_\tau(a)$ is continuous
on $K$ for some strictly positive element,
 then,
\eqref{Pconscale-2} holds for every $\ep.$
Since $A$ has strict comparison for positive elements,
it follows, for any $b\in A_+,$  there exists $N\ge 1$ such that
$$
e_m-e_n\lesssim b\rforal m>n\ge N.
$$
This implies that $A$ has continuous scale.

In other words, in this case, if $A$ does not have continuous scale,
$d_\tau(a)$ is not continuous on $K.$ In particular, $d_\tau(a)$ is not identically
1. This implies $K\not=T(A).$
The above shows that, under the assumption that $A$ as stated in the second part of the theorem,
(1), (4) and (5) are equivalent.  Since (5) and (6) are equivalent, these are also equivalent to (6).
Since the notion of continuous scale is independent of choice of $a,$ these are also equivalent
to (2) and (3).

\end{proof}

\begin{cor}\label{Subcontsc}
Let $A$ be a $\sigma$-unital exact simple \CA\, with
strict comparison and with continuous scale.   Suppose that $T(A)\not=\emptyset.$
Suppose $a\in A$ such that $d_\tau(a)$ is continuous on $T(A).$
Then $\overline{aAa}$ has continuous scale.
\end{cor}

\begin{proof}
Put $B=\overline{aAa}.$
If $d_\tau(a)$ is continuous on $T(A),$
then $d_\tau(a)$ is a continuous function on ${\tilde T}(A).$
Therefore $\{\tau: d_\tau(a)=1\}$ is compact.
\end{proof}

Now we turn to local uniform fullness.

\begin{prop}\label{localunit}
Let $A$ be a  $\sigma$-unital full \SCA\, of a non-unital \CA\, of $B.$
Suppose  that $B$ is quasi-compact.
Then $A$ is strongly locally uniformly full in $B.$
\end{prop}

\begin{proof}
Let $a\in A$ be a strictly positive element.
Then $\overline{aBa}$ is full hereditary \SCA\, $B.$
It suffices to show that $\overline{aBa}$ is locally uniformly full in $B.$
Put $B_1=\overline{aBa}.$

Let $e\in M_n(B)$  with $0\le e\le 1$ and $w\in M_n(B)^{**}$ such that
$$
w^*a,\,w^*aw\in M_n(B), \,\,\, ww^*a=aww^*=a\andeqn w^*awe=ew^*aw=w^*aw\rforal a\in B.
$$
Note also that $aw\in M_n(B)$ for all $a\in B.$

 By Lemma \ref{Lfullep}, for any $1/4>\ep>0$ and any $a\in (B_1)_+\setminus \{0\},$
there are $x_1,x_2,...,x_m\in M_n(B)$ such that
$$
f_{\ep}(e)=\sum_{i=1}^m x_i^*ax_i.
$$
Let $p$ be the range projection of $B.$  Then $px_i\in M_n(B)$ for all
$x_i\in M_n(B).$ We may assume that $px_i=x_i,$ $i=1,2,...,m.$

Note that $f_{\ep}(e)w^*bw=w^*bwf_{\ep}(e)=w^*bw.$   It follows that, for any $x\in B,$
$$
f_\ep(e) w^*xw=w^*xwf_\ep(e) =w^*xw.
$$
Let $M(a)=\max\{\|x_i\|: 1\le i\le m\}$ and $N(a)=m.$
Then, for any $x\in B_+$ with $\|x\|\le 1,$ $w^*x^{1/2}wf_\ep(e)w^*x^{1/2}w=w^*xw.$
Therefore
$$
x^{1/2}wf_\ep(e)w^*x^{1/2}=w(w^*xw)w^*=x.
$$
Put $z_i=x_iw^*x^{1/2},$ $i=1,2,...,m.$
Since $w^*z^{1/2}\in M_n(B)$  and $px_i=x_i,$ $z_i\in B.$
 Then $\|z_i\|\le M(a)$ and
\beq\nonumber
\sum_{i=1}^mz_i^*az_i&=&x^{1/2}w(\sum_{i=1}^mx_i^*ax_i)w^*x^{1/2}\\
&=&x^{1/2}wf_\ep(e)w^*x^{1/2}=x.
\eneq

\end{proof}

\begin{thm}\label{Tqcfull}
Let $A$ be a non-unital separable simple \CA\, which is quasi-compact and with
$T(A)\not=\emptyset.$
Fix an element $e\in A_+\setminus \{0\}$ with $\|e\|=1$ and
$$0<d<\min\{\inf\{\tau(e): \tau\in T(A)\}, \inf\{\tau(f_{1/2}(e)): \tau\in T(A)\}\}.$$
Then there exists a map
$T: A_+\setminus \{0\}\to \N\times \R_+\setminus \{0\}$ satisfying the following:
For any finite subset ${\cal H}_1\subset A_+^{\bf 1}\setminus \{0\},$
there is a finite subset ${\cal G}\subset A$ and $\dt>0$ satisfying the following:
For any exact \CA\, $B$ with $T(B)\not=\emptyset$ and
$0\not\in \overline{T(B)}^w$  which has the strong strict comparison  and has almost stable rank one, and
for any
${\cal G}$-$\dt$-multiplicative \cpc\, $\phi: A\to B$
such that,
$$
\tau(f_{1/2}(\phi(e)))>d/2\rforal \tau\in T(B),
$$
then $\phi$ is exactly $T$-${\cal H}_1$-full.
Moreover, for any $c\in {\cal H}_1,$
$$
\tau(f_{1/2}(c))\ge {d\over{8\min\{M(c)^2\cdot N(c): c\in {\cal H}_1\}}}\rforal \tau\in T(B).
$$

\end{thm}

\begin{proof}
Since $A$ is a $\sigma$-unital simple \CA\, which is quasi-compact,
then there is $T_1: A_+\setminus \{0\}\to \N\times R_+\setminus\{0\}$ such that
the identity map ${\rm id}_A$ is $T_1$-$A_+\setminus\{0\}$-full.

Write $T_1=(N_1,M_1),$ where $N_1: A_+\setminus\{0\}\to \N$ and
$M_1:  A_+\setminus\{0\}\to \R_+\setminus \{0\}.$

Let $n\ge 2$ be  an integer such that $nd/2>1.$
Define $N=2nN_1$ and $M=2M_1$ and $T=(N, M).$

Let ${\cal H}_1\subset A_+\setminus \{0\}$ be a fixed finite subset.

Suppose that $x_{i,h},...,x_{N_(h), h}\in A$ with
$\|x_{i,h}\|\le M_1(h)$ such that
\beq\label{Tqcfull-2}
\sum_{i=1}^{N_1(h)} x_{i,h}^* h^2 x_{i,h}=f_{1/64}(e)\rforal h\in {\cal H}_1.
\eneq
We choose a large ${\cal G}$ and small $\dt>0$ such that, for any ${\cal G}$-$\dt$-multiplicative
\cpc\, $\phi$ from $A$  with $\|\phi\|\not=0$ has
the property
that
\beq\label{Tqcfull-3}
&&\|\phi(f_{1/64}(e))-f_{1/64}(\phi(e))\|<1/64\andeqn\\
&&\|\sum_{i=1}^{N(h)} \phi(x_{i,h})^* \phi(h)^2\phi(x_{i,h})-\phi(f_{1/64}(e))\|<1/64\rforal h\in {\cal H}_1.
\eneq
Now suppose that $\phi: A\to B$ (for any $B$ fits the description in the theorem)
which is ${\cal G}$-$\dt$-multiplicative \cpc\,
such that
\beq\label{Tqcfull-4}
\tau(f_{1/2}(\phi(e)))\ge d/2\rforal \tau\in \overline{T(B)}^w.
\eneq

Applying \ref{Lrorm}, one finds $y_{i,h}\in A$ with
$\|y_{i,h}\|\le 2\|x_{i,h}\|,$ $i=1,2,...,N_1(h)$ such that
\beq\label{Tqcfull-5}
\sum_{i=1}^{N_1(h)}y_{i,h}^*\phi(h)^2y_{i,h}=f_{1/16}(\phi(e))\rforal h\in {\cal H}_1.
\eneq

By the assumption of $B,$ applying \ref{Lbkqc},  choose $e_1, e_2\in M_n(B)_+$
and $w\in M_n(B)^{**}$ as described there.
Put
\vspace{-0.14in} \beq\label{Tqcfull-6}
{\bar E}_0&=&\diag(\overbrace{f_{1/8}(\phi(e)), f_{1/8}(\phi(e)), ..., f_{1/8}(\phi(e))}^{2n})\andeqn\\
{\bar E}_1&=&\diag(\overbrace{f_{1/16}(\phi(e)), f_{1/16}(\phi(e)), ..., f_{1/16}(\phi(e))}^{2n})
\in M_{2n}(B)_+.
\eneq
Then, by the strong strict comparison,
$$
e_2\lesssim {\bar E}_0\in M_{2n}(B).
$$
Since $B$ has almost stable rank one, there exists a unitary $u\in M_{2n}(B)$ such
that
$$
u^*f_{1/6}(e_2)u \in \overline{{\bar E}_0(M_{2n}(B)){\bar E}_0}.
$$
Then
\vspace{-0.11in} \beq\label{Tqcfull-8}
u^*f_{1/16}(e_2)u{\bar E}_1={\bar E}_1u^*f_{1/16}(e_2)u=u^*f_{1/16}(e_2)u.
\eneq
We then write
$$
\sum_{i=1}^{2nN_1(h)} y_{i,h}' \phi(h) y_{i,h}'={\bar E}_1,
$$
where $\|y_{i,h}'\|=\|y_{j,h}\|$ for some $j\in \{1,2,...,N_1(h)\},$ $i=1,2,...,2nN_1(h),$ for all $h\in {\cal H}_1.$
Then
$$
\sum_{i=1}^{2nN_1(h)} (f_{1/16}(e_2)^{1/2}u y_{i,h}' )\phi(h) (y_{i,h}'u^*f_{1/16}(e_2)^{1/2})= f_{1/16}(e_2).
$$
Therefore, for any $b\in B_+$ with $\|b\|\le 1,$
$$
\sum_{i=1}^{2nN_1(h)} (w^*b^{1/2}w)(f_{1/16}(e_2)^{1/2}u y_{i,h}' )\phi(h)^{1/2}\phi(h)\phi(h)^{1/2} (y_{i,h}'u^*f_{1/16}(e_2)^{1/2})(w^*b^{1/2}w)=w^*bw.
$$
Then
$$
\sum_{i=1}^{2nN_1(h)}(b^{1/2}w)(f_{1/16}(e_2)u y_{i,h}'\phi(h)^{1/2} )\phi(h) (\phi(h)^{1/2}(y_{i,h}'u^*f_{1/16}(e_2))w^*b^{1/2}=b.
$$
Note that $b^{1/4}w\in M_n(B)$ and $f_{1/16}(e_2)\in M_n(B).$
Therefore
$$
(b^{1/4}w)f_{1/16}(e_2)\in M_n(B).
$$
It follows that
\vspace{-0.1in} \beq\label{Tqcfull-16}
&&(b^{1/2}w)(f_{1/16}(e_2)^{1/2}u y_{i,h}'\phi(h)^{1/2} \in B
\andeqn\\
&&\|(b^{1/2}w)(f_{1/16}(e_2)^{1/2}u y_{i,h}'\phi(h)^{1/2}\|\le 2M(h)\rforal h\in {\cal H}_1.
\eneq
This implies that $\phi$ is exactly $T$-${\cal H}_1$-full.
\end{proof}

\begin{rem}\label{Rqcfull}
In the light of \ref{Pcom4C}, Theorem \ref{Tqcfull} works for \CA s $B\in {\cal C}'.$
\end{rem}

\section{Non-unital and non-commutative  one dimensional complices }

\begin{df}\label{Dbuild1}
{\rm
Let $F_1$ and $F_2$ be two finite dimensional \CA s.
Suppose that there are two  \hm s
$\phi_0, \phi_1: F_1\to F_2.$
Denote the mapping torus $M_{\phi_1, \phi_2}$ by
$$
A=A(F_1, F_2,\phi_0, \phi_1)
=\{(f,g)\in  C([0,1], F_2)\oplus F_1: f(0)=\phi_0(g)\andeqn f(1)=\phi_1(g)\}.
$$

For $t\in (0,1),$ define $\pi_t: A\to F_2$ by $\pi_t((f,g))=f(t)$ for all $(f,g)\in A.$
For  $t=0,$ define $\pi_0: A\to \phi_0(F_1)\subset F_2$ by $\pi_0((f, g))=\phi_0(g)$ for all $(f,g)\in A.$
For $t=1,$ define $\pi_1: A\to \phi_1(F_1)\subset F_2$ by $\pi_1((f,g))=\phi_1(g))$ for all $(f,g)\in A.$
In what follows, we will call $\pi_t$  a point-evaluation of $A$ at $t.$
There is a canonical map $\pi_e: A \to F_1$ defined by $\pi_e(f,g)=g$ for all
pair $(f, g)\in A.$  It is a surjective map.

The class of all \CA s described above is denoted by ${\cal C}.$

If $A\in {\cal C}$, then $A$ is the pull-back of
\begin{equation}\label{pull-back}
\xymatrix{
A \ar@{-->}[rr] \ar@{-->}[d]^-{\pi_e}  && C([0,1], F_2) \ar[d]^-{(\pi_0, \pi_1)} \\
F_1 \ar[rr]^-{(\phi_0, \phi_1)} & & F_2 \oplus F_2.
}
\end{equation}
Every such pull-back is an algebra in ${\cal C}.$
Infinite dimensional  C*-algebras in ${\cal C}$ are also called {\it one-dimensional non-commutative finite CW complexes} (NCCW)
(see \cite{ELP1} and \cite{ELP2}) and Elliott-Thomsen building blocks.

Denote by ${\cal C}_{0}$ the class of all \CA s $A$  in ${\cal C}$ which
satisfies the following conditions:

(1) $K_1(A)=\{0\},$

(2) $K_0(A)_+=\{0\},$

(3) $0\not\in \overline{T(A)}^w.$


\CA s  in  ${\cal C}_0$ are stably projectionless.

First such examples  can be found in \cite{Raz}.
Let $F_1=M_k$ for some $k\ge 1$ and $F_2=M_{(m+1)k}$
for some $m\ge 1.$
Define $\psi_0, \psi_1: F_1\to F_2$ by
$$
\psi_0(a)=\diag(\overbrace{a,a,...,a}^m, 0)\andeqn
\psi_1(a)=\diag(\overbrace{a,a,..., a,a}^{m+1})
$$
for all $a\in F_2.$
Let
\beq\label{ddraz}
A=A(F_1, F_2, \psi_0, \psi_1)=R(k, m, m+1).
\eneq
Then,  as  shown in \cite{Raz},  $K_0(A)=\{0\}=K_1(A)$ and it is easy to check that $0\not\in \overline{T(A)}^w.$
Denote by ${\cal R}$ the class of \CA s which finite direct sums of \CA s in the form in \eqref{ddraz}.



Denote by ${\cal C}_0^0$ the subclass of \CA s in ${\cal C}_0$ which
also satisfies the condition

(2)' $K_0(A)=\{0\}.$

{{ Let $F_1=\C\oplus \C, F_2=M_{2n}(\C)$. For $(a,b)\in \C\oplus \C=F_1$, define
$$\psi_0(a, b)=\diag(\underbrace{a,a...a}_{n-1},\underbrace{b,b...b}_{n-1}, 0, 0)
~~~\mbox{and}~~~~\psi_1(a, b)=\diag(\underbrace{a,a...a}_n,\underbrace{b,b...b}_n).$$
Then $A(F_1, F_2, \psi_0,\psi_1)=A$ has the property that
$K_0(A)=\{(k,-k )\in \Z\oplus \Z)\}$ which is isomorphic to $\Z$ but $K_0(A)_+=\{0\}$.
Also $K_1(A)=\{0\}.$
Thus $A\in {\cal C}_0$ but $A\notin {\cal C}_0^0.$}}

Let  ${\cal C}'$ denote the class of all full hereditary \SCA s of \CA s in ${\cal C},$
let ${\cal C}_0'$ denote the class of all full hereditary \SCA s of \CA s in ${\cal C}_0$
let ${\cal C}_0^{0'}$ denote the class of all full hereditary \SCA s of \CA s in ${\cal C}_0^0.$

}

\end{df}

%
%
%
%
%



\begin{rem}\label{Runitz}
Let $A=A(F_1, F_2, \psi_0, \psi_1)\in {\cal C}_0.$
Then ${\tilde A}\in {\cal C}.$ Moreover
${\tilde A}=A(F_1', F_2, \psi_0',\psi_1')$ with both
$\psi_0'$ and $\psi_1'$ being unital.

Let $F_1'=F_1\oplus \C$ and let
$p=\psi_0(1_{F_1})\in F_2$ and $q=\psi_1(1_{F_1})\in F_2$.  Define $\psi_0',~{\psi_1'}: F_1'\to F_2$ by
$$
\psi_0'((a,\lambda))=\psi_0(a)\oplus (\lambda\cdot (1_{F_2}-p) {)}~~\mbox{and}~~\psi_1'((a,\lambda))=\psi_{ 1}(a)\oplus (\lambda\cdot (1_{F_2}-q))
$$
 for all $a\in F_1\andeqn \lambda\in \C.$

Since $A\in {\cal C}_0,$ then either $\psi_0$ or $\psi_1$ are not unital. Hence at least one of $\psi_0'$ and $\psi_1'$ is nonzero on the second direct summand $\C$ in $F_1'=F_1\oplus \C$.

\end{rem}


\begin{prop}\label{Pcom4C}
{\rm (1)} Let $A\in {\cal C}'$.
 Then, for any $a_1, a_2\in A_+,$  $a_1\lesssim a_2$ if and only if
 $d_{tr\circ \pi}(\pi(a_1))\le d_{tr\circ \pi}(\pi(a_2))$ for all irreducible representations $\pi$
 of $A.$

{\rm (2)} Let $A\in {\cal C}',$ and
let $c\in A_+\setminus \{0\}.$ Then $c$ is full if and only if, for any
$\tau\in T(A),$ $\tau(c)>0.$

\end{prop}

\begin{proof}
 For {\rm (1)}, we first  consider the case that $A\in {\cal C}.$ By considering ${\tilde A},$ one sees that
this case follows from 3.18 of  \cite{GLN}.

Since \CA s $A\in {\cal C}'$ are hereditary \SCA\, of $A,$
it is easy to see that $A$ also has the above mentioned comparison property.

For {\rm (2)}, we first assume that $A\in {\cal C}.$ It is clear that if $\tau(c)=0,$ for some $\tau,$ then $c$ has zero value
somewhere in $\mathrm{Sp}(A)=\sqcup_j(0,1)_j\cup\mathrm{Sp}(F_1).$ Therefore $c$ is in an proper ideal of $A.$

Now assume that $\tau(c)>0$ for all $\tau\in T(A).$
Write $c=(a,b),$ where $a\in C([0,1], F_2)$ and $b\in F_1.$ Then $b>0$ and $a(t)>0$
for all $t\in [0,1].$  It follows that
$$
d_{tr\circ \pi}(\pi(c))>0
$$
for all $c\in A.$ Since we assume that $0\not\in \overline{T(A)}^w,$
this implies that
$$
\inf\{d_{tr\circ \pi}(\pi(c)): \pi \}>0.
$$
There is an integer $n\ge 1$ such that
$$
d_{tr\circ \pi}(\pi({\bar c}))>2
$$
for all irreducible representations $\pi$ of $A$ (or $M_n(A)$),
where
$$
{\bar c}=\diag(\overbrace{c,c,...,c}^n).
$$
By (1), this implies $a\lesssim {\bar c},$ where $a$ is a strictly positive element.
This implies that $c$ is a full element in $A.$

In general, let $A$ be a full hereditary \SCA\, of $C\in {\cal C}.$
Let $c\in A_+.$ Then  $c\in A_+$ is full if and only if it is full in $C.$
Thus the general case follows from the case that $A\in {\cal C}.$

\end{proof}

\begin{prop}\label{str=1}

{\rm (1)} Every \CA\, in ${\cal C}'$  has stable rank one;

{\rm (2)}  If $A\in {\cal C}$ and $A$ is unital,
 the exponential rank of $A$   is at most $2+\ep,$
 If $A\in {\cal C}$ and $A$ is not unital, then ${\tilde A}$ has exponential rank at most $2+\ep;$

{\rm (3)} Every \CA\, in ${\cal C}'$ is semiprojective.

{\rm (4)} Let $A\in {\cal C}$ and let $k\ge 1$ be an integer.  Suppose that every irreducible representation 
of $C$  has rank at least $k.$ 
 Then, for any $f\in {\rm
LAff}_b(T(C))_+$ with $0\le f \le 1,$ there exists a positive
element $a\in M_2(C)$ such that
$$
\max_{\tau\in T(C)}|d_{\tau}(a)-f(\tau)|\le 2/k.
$$


\end{prop}

\begin{proof}
(1) follows from 3.3 of \cite{GLN}. (2) follows from 3.16 of \cite{GLN} (see also 5.19 of \cite{Lncbms}).

It follows from \cite{ELP1} that every \CA\, in ${\cal C}$ is semiprojective.
Now let $B\subset A$ be a full hereditary \SCA\, of $A.$
We may write $B=\overline{bAb}.$
Then, it is known and standard that, for any $\ep>0,$ there is $b_0\le b$ and
$\|b-b_0\|<\ep$ such that $\overline{b_0Bb_0}=\overline{b_0Ab_0}\in {\cal C}.$
(4) follows exactly the same proof of  10.4 of \cite{GLN}.

\end{proof}




\section{Maps from 1-dimensional non-commutative complices}

\begin{lem}[Lemma 2.1 of \cite{BT}]\label{Lbt1}
Let $A$ be a simple exact \CA\, with strict comparison of positive elements which is also quasi-compact.  Assume  that
$\iota: W(A)_+\to {\rm LAff}_{b+}(\overline{T(A)}^w)$ is  surjective.
Let $a\in A_+\setminus \{0\}$ which is not Cuntz equivalent to a projection.
Then, for any $\ep>0,$  there exists $\dt>0$ and a continuous affine
function  $f: \overline{T(A)}^w\to \R^+$ such that
$$
d_\tau((a-\ep)_+)<f(\tau)<d_\tau((a-\dt)_+)\rforal \tau\in \overline{T(A)}^w.
$$
\end{lem}

\begin{proof}
Lemma 2.1 of \cite{BT} assumes  that $A$ has stable rank one.
It is used in the first sentence of the proof, namely, one can assume that zero is a limit point
of $sp(a).$  If $0$ is an isolated point of $sp(a),$ then there is $0<\dt<1/2$ such that
$f_\dt(a)$ is a projection. It is easy to see that $\la a \ra =\la f_\dt(a) \ra $ in this case.
The rest of the proof works exactly the same as that of Lemma 2.1 of \cite{BT}
which does require $A$ has an identity.

\end{proof}

\begin{lem}[Lemma 2.2 of \cite{BT}]\label{Lbt2}
Let $A$ be as in \ref{Lbt1} and let $a\in A\otimes {\cal K}.$  It follows that there exists
a sequence $\{a_n\}$ of elements  in $(A\otimes {\cal K})_+$  which satisfies the following:

{\rm (1)} $\la a\ra =\sup_n \la a_n \ra ;$

{\rm (2)} $a_n\in M_{n(k)}(A)$  for some $n(k)\in \N;$

{\rm (3)} the function $\tau\to d_\tau(a_n)$ is continuous  on $\overline{T(A)}^w$ for each $n\in \N$ and

{\rm (4)}  $d_\tau(a_n)<d_\tau(a_{n+1})$ for all $\tau\in \overline{T(A)}^w$ and $n\in \N.$

\end{lem}


The following will not be used until later sections.

\begin{lem}\label{Lbt1ACT}
Let $A$ be a non-unital exact  simple \CA\, with strict comparison for positive element
which almost has stable rank one.
Suppose that $A$ is quasi-compact and $\imath: W(A)_+\to {\rm LAff}_{b+}(\overline{T(A)}^w)$
is surjective.

Let $0\le a\le 1$ be a non-zero element in $A$ which is not Cuntz equivalent to
a projection.
Then, for any $\ep>0$   there exists $\dt>0$ and  an element  $e\in A$
with
\beq\label{Lbt1act-1}
0\le f_\ep(a)\le e\le f_{\dt/2}(a)
\eneq
such that the function $\tau\mapsto d_\tau(e)$ is continuous on $\overline{T(A)}^w.$
\end{lem}

\begin{proof}
By \ref{Lbt1}, there is a continuous affine function function
$g_1, g_2\in \Aff(\overline{T(A)}^w$ such that
\beq\label{Lbt1act-2}
d_\tau(f_{\ep/8}(a))<g_1(\tau)<d_\tau(f_{\dt_1}(a))<g_2(\tau)<d_\tau(f_{\dt_2}(a)) \rforal \tau\in \overline{T(A)}^w,
\eneq
where $0<\dt_2<\dt_1<1.$
Since $\iota$ is surjective, there is $c\in M_m(A)$ for some integer $m\ge 1$
such that $0\le c\le 1$ and
$d_\tau(c)=g_2(\tau)$ for all $\tau\in \overline{T(A)}^w.$
It follows from \ref{Lalmstr1}  and \eqref{Lbt1act-2} that there exists $x\in M_m(A)$ such that
\vspace{-0.08in} $$
x^*x=c\andeqn xx^*\in \overline{f_{\dt_2}(a)Af_{\dt_2}(a)}.
$$
Put $c_0=xx^*.$ Then $0\le c_0\le 1.$
Note that
\beq\label{Lbt1act-2+}
d_\tau(c_0)=d_\tau(c)\rforal \tau\in \overline{T(A)}^w.
\eneq
Since $g_1$  and $f_2$ are continuous, there is $m\ge 2$ such that
\beq\label{Lbt1act-2++}
d_\tau(g_1)<\tau(f_{1/m}(c_0))\rforal \tau\in \overline{T(A)}^w.
\eneq
By \eqref{Lbt1act-2} again, applying
\ref{Lalmstr1}, there is a unitary $u$ in the unitization of $\overline{f_{\dt_2}(a)Af_{\dt_2}(a)}$
 such that
\vspace{-0.12in} \beq\label{Lbt1act-3}
u^*f_{\ep/8}(f_{\ep/8}(a))u\in \overline{f_{1/m}(c_0)Af_{1/m}(c_0)}.
\eneq
Define $c_1=uc_0u^*.$
Then
\beq\label{Lbt1act-4}
f_{\ep/8}(f_{\ep/8}(a))\in \overline{f_{1/m}(c_1)Af_{1/m}(c_1)}\subset  \overline{c_1Ac_1}\overline{f_{\dt/2}(a)Af_{\dt/2}(a)}.
\eneq
There is a $g\in C_0((0,1])$ with $0\le g\le 1$ such that $g(t)\not=0$ for all $t\in (0,1],$
$g(t)f_{1/m}=f_{1/m}.$
Put $e=g(c_1).$
Then
\vspace{-0.1in} \beq\label{Lbt1act-5}
d_\tau(e)=d_\tau(c_1)=g_2(\tau) \rforal \tau\in \overline{T(A)}^w\andeqn\\
f_\ep(a)\le f_{\ep/8}(f_{\ep/8}(a)\le e\le f_{\dt_2/2}(a).
\eneq
Choose $\dt=\dt_2/2.$
\end{proof}

\begin{lem}\label{LLbt}
Let $A$ be as in \ref{Lbt1}. Suppose that $a, b\in A_+$
(with $0\le a \le 1$ and $0\le b\le 1$)  such that neither are Cuntz equivalent
to a projection.  Suppose that $a\ll  b.$
Then there exist $\dt>0$ and $c \in A_+$ with $0\le c\le 1$ such that
\beq\label{LLbt-0}
\la a \ra \le \la f_{\dt}(c)\ra, \,\,\, f_{\dt/2}(c)\le f_{\dt/4}(b)\andeqn
\inf \{\tau(f_{\dt}(c))-d_\tau(a): \tau\in \overline{T(A)}^w\}>0.
\eneq
\end{lem}

\begin{proof}
By \ref{Lbt2}, choose $b_n\in A_+$ such that $\{b_n\}$ satisfies {\rm (1)}, {\rm (3)} and {\rm (4)}
in \ref{Lbt2}.  Since $a\ll b,$ there is $n_0\ge 1$ such that
$\la a\ra \le b_n$ for all $n\ge n_0.$
Therefore we have
\beq\label{LLbt-1}
d_\tau(a)\le d_{\tau}(b_{n_0})<d_\tau(b_{n_0+1})<d_\tau(b_{n_0+2})<d_\tau(b_{n_0+3})\le d_\tau(b).
\eneq
Note that
$$
\tau(f_{1/n}(b))\nearrow d_\tau(b)\andeqn \tau(f_{1/n}(b_{n_0+1}))\nearrow d_\tau(b_{n_0+1}) .
$$
It follows, for example,  from 5.4 of \cite{Lnloc} that there exists $n_1\ge 1$ such that,
for all $n\ge n_1,$
$$
\tau(f_{1/n}(b))>d_\tau(b_{n_0+2})\andeqn
\tau(f_{1/n}(b_{n_0+1}))>d_{\tau}(b_{n_0}))\rforal \tau\in \overline{T(A)}^w.
$$
By the strict comparison of positive element, we conclude that
$$
\la f_{1/2n}(b)\ra \ge \la b_{n_0+2} \ra \andeqn  \la f_{1/2n}(b_{n_0+1})\ra \ge \la b_{n_0} \ra .
$$
Put $c=b_{0+1}.$
Since $A$ has stable rank one, one may assume  that
$f_{1/2n}(c)\le f_{1/4n}(b).$
Thus we may choose $0<\dt<1/2n_1.$

Since $\overline{T(A)}^w$ is compact and both
functions in the above inequality are continuous, combining with \eqref{LLbt-1}, we obtain
\vspace{-0.12in} $$
\inf\{\tau(f_\dt(b))-d_\tau(a): \tau\in \overline{T(A)}^w\}>0.
$$

\end{proof}

\begin{thm}[see Theorem 3.3.1 of \cite{Rl}, Theorem 5.2.7 of \cite{Lncbms} and Theorem 8.4 of \cite{GLN}] \label{Lrluniq}
Let $C$ be
a full hereditary \SCA s of
1-dimensional non-commutative complices with $K_1(C)=\{0\}.$
Assume that
$0\not\in \overline{T(C)}^w$
and  let
$\Delta: C^{q,{\bf 1}}\setminus \{0\}\to (0,1)$ be an order preserving map.
Then, for any $\ep>0$ and any finite subset
${\cal F}\subset C,$ there exists a finite subset ${\cal G}\subset C,$ a finite subset ${\cal P}\subset K_0(C),$
a finite subset ${\cal H}_1\subset C_+^{\bf 1}\setminus \{0\},$ a finite subset ${\cal H}_2\subset C_{s.a.},$
$\dt>0,$ $\gamma>0$  satisfying the following:
for any two ${\cal G}$-$\dt$-multiplicative \morp s $\phi_1, \phi_2: C\to A$ for some $A$
which is $\sigma$-unital, simple exact, with strict comparison for positive elements, almost has
stable rank one, is quasi-compact, and the map $W(A)_+\to {\rm LAff}_{b+}(\overline{T(A)}^w)$ is surjective,   such that
\beq\label{Lrluniq-1}
&&[\phi_1]|_{\cal P}=[\phi_2]|_{\cal P},\\
&&\tau(\phi_i)(a)\ge \Delta(\hat{a})\rforal a\in {\cal H}_1\andeqn \rforal \tau\in \overline{T(A)}^w\andeqn\\
&&|\tau(\phi_1(b))-\tau(\phi_2(b))|<\gamma\rforal b\in {\cal H}_2\andeqn \rforal \tau\in \overline{T(A)}^w,
\eneq
there exists a unitary $u\in {\tilde A}$ such that
$$
\|u^*\phi_2(f)u-\phi_1(f)\|<\ep\rforal f\in {\cal F}.
$$

\end{thm}

\begin{proof}
By the terminology in \cite{Rl},  as proved in
\cite{Rl}, $Cu^{\sim}$ classifies \hm s from $C.$

Upon examining the proofs of \cite{Rl}, one sees that Lemma 3.3.1
and Theorem 1.0.1 of \cite{Rl} hold for $B$ almost having stable rank one
instead of having stable rank one as in \ref{Lalmstr1}.

Fix $\ep$ and ${\cal F}\subset C.$
Let $G\subset Cu^{\sim}(C)$ be required by Theorem 3.3.1 of \cite{Rl}  for
$\ep/2$ (in place of $\ep$) and ${\cal F}.$

Note that $C$ has stable rank one, as computed in \cite{Rl}, we may assume
that $G$ consists a finite subset ${\cal P}\subset K_0(C)$
and a finite subset $\{x_1, x_2,...,x_m\}$ in Cuntz group
which can be represented by positive elements
$0\le a_i\le 1$ in $C\otimes {\cal K}$ which are  not Cuntz equivalent to a projection,
$i=1,2,...,m$ (see 3.1.2 of \cite{Rl}).  Put ${\cal P}=\{z_1, z_2,..., z_{m_0}\}.$
Note if $a_i\le z_l$ (for some $1\le l\le m_0$), then $z_l\ge 0$ and
is therefore represented by a projection.

Suppose that $\la a_i\ra \ll \la a_j \ra.$
Then, by \ref{LLbt},  there is a number  $1/4 > \eta(i,j)>0$   and  an element
$0\le c_{i,j}\le 1$ in $A$ such that
\beq\label{Lrluniq-3}
\la a_i \ra \le  \la  f_{\eta_{i,j}}(c_{i,j})\ra\andeqn  f_{\eta_{i,j}/2}(c_{i,j})\le f_{\eta_{i,j}/4}(a_j).
\eneq
Since $a_i$ ($i=1,2,...,m$) is not Cuntz equivalent to a projection
(and assume that $x_i\not=x_j$ if $i\not=j$),
we may choose $\eta(i,j)$ so that
$$
f_{\eta_{i,j}/4}(a_j)-f_{\eta_{i,j}/2}(c_{i,j})\not=0.
$$
Choose  a finite subset ${\cal H}_1\subset C_+$
which contains  nonzero positive elements $b_{i,j}$ such that
$$
b_{i,j}\lesssim f_{\eta_{i,j}/4}(a_j)-f_{\eta_{i,j}/2}(c_{i,j}).
$$
for all possible pair of $i,j$ so that $\la a_i\ra \ll  \la a_j \ra.$
Moreover, if $a_i\ll  z_l,$

Let
\vspace{-0.12in} \beq\label{Lrluniq-4}
\dt_0=\inf \{ \Delta(\hat{g}): g\in {\cal H}_1\}.
\eneq

Let ${\cal H}_2'$ be a finite subset  (possibly in $C\otimes {\cal K})_+$) which contains
$f_{\eta_{i,j}}(c_{i,j}), f_{\eta_{i,j}/2}(c_{i,j}),
f_{\eta_{i,j}/4}(a_i) $ for all possible $i,j$ as described above.

Let ${\cal H}_2\subset C_{s.a.}$ be a finite subset
which contains ${\cal H}_1$ and let
$\dt_1>0$ which have   the following
property:
\beq\label{Lrluniq-4+}
|\tau(h_1(g))-\tau(h_2(g))|<\dt_0/16\tforal g\in {\cal H}_1\cup {\cal H}_2'
\eneq
and for all $\tau\in \overline{T(B)}^w,$
provided that $h_1, h_2: C\to B$ are two \hm s and $B$ is any \CA\,
with $T(C)\not=\emptyset$ and $0\not\in \overline{T(B)}^w$ such that
\beq\label{Lrluniq-4++}
|\tau\circ h_1(f)-\tau\circ h_2(f)|<\dt_1\rforal f\in {\cal H}_2\andeqn \tau\in \overline{T(B)}^w.
\eneq

Put $\gamma=\min\{\dt_0/16, \dt_1/4\}.$
Since $C$ is weakly semiprojective (\cite{ELP1}),
by choosing a large ${\cal G}$ and small $\dt,$ we may assume
that  there are \hm s $\psi_i: C\to A$ such that
\beq\label{Lrluniq-5}
(\psi_i)_{*0}|_{\cal P}=[\phi_i]|_{\cal P}\andeqn \|\psi_i(g)-\phi_i(g)\|<\min\{\ep/16, \gamma\},\,\,\,i=1,2,
\eneq
for all $g\in {\cal F}\cup{\cal H}_1\cup {\cal H}_2,$
provided that $\phi_1$ and $\phi_2$ are ${\cal G}$-$\dt$-multiplicative \cpc s from $C$ to any $A$
which satisfies the assumption of this theorem.

Now assume that $\phi_1, \phi_2: C\to A$ have  the described properties for the above
defined ${\cal G},$ $\dt,$  ${\cal P},$ ${\cal H}_1, $ ${\cal H}_2,$ $\gamma.$

Let $\psi_i: C\to A$ be provided as in \eqref{Lrluniq-5},\,\,\, $i=1,2.$
Then
\beq\label{Lrluniq-6}
&&(\psi_1)_{*0}|_{\cal P}=(\psi_2)_{*0}|_{\cal P},\\
&&\tau\circ \psi_i(g)\ge \dt_0/2\rforal g\in {\cal H}_1\andeqn\\
&&|\tau\circ \psi_1(b)-\tau\circ \psi_2(b)|<\dt_0/2\rforal b\in {\cal H}_2
\eneq
for all $\tau\in \overline{T(A)}^w.$
In particular,  if $a_i\ll a_j,$  by the choice of ${\cal H}_1,$ ${\cal H}_2$ and $\gamma$ above,
\beq\label{Lrluniq-7}
d_\tau(\psi_1(a_i)) &\le& \tau(\psi_1(f_{\eta_{i,j}}(c_{i,j})))<\dt_0/2+\tau(\psi_2(f_{\eta_{i,j}}(c_{i,j})))\\
&\le  &  (\tau(\psi_2(f_{\eta_{i,j}/4}(a_j))-\tau(\psi_2(f_{\eta_{i,j}}(c_{i,j}))))+\tau(\psi_2(f_{\eta_{i,j}}(c_{i,j})))\\
&\le & d_{\tau}(a_j)
\eneq
for all $\tau \in \overline{T(A)}^w.$
Therefore, if $a_i\ll a_j,$
\beq\label{Lruniq-8}
\la \psi_1(a_i)\ra \le \la \psi_2(a_j) \ra.
\eneq
Note also, if $\la a_i\ra \ll z_l,$ then $\la \psi_1(a_i)\ra \ll Cu^{\sim}(\psi_1)(z_l)=Cu^{\sim}(\psi_2)(z_l).$
Combing these with \eqref{Lrluniq-6}, we conclude that, using the terminology in \cite{Rl}
\beq\label{Lruniq-9}
Cu^{\sim}(\psi_1(g))\le Cu^{\tilde {}}(\psi_2(g')\andeqn Cu^{\sim}(\psi_2(g))\le Cu^{\sim}(\psi_1(g'))
\eneq
for all $g,\, g'\in G$ and $g\ll g'.$  By 3.3.1 of \cite{Rl}, there exists a unitary
$u\in A$ such that
$$
\|u^* \psi_2(f)u-\psi_1(f)\|<\ep/2\rforal f\in {\cal F}.
$$
From this and \eqref{Lrluniq-5}, we obtain that
$$
\|u^*\phi_2(f)u-\phi_2(f)\|<\ep\rforal f\in {\cal F}.
$$
\end{proof}

\begin{cor}\label{Cruniq}
The exactly the same statement holds for $C$ being replaced by
full hereditary \SCA s of inductive limits of \CA s in ${\cal C}.$
\end{cor}

\begin{cor}\label{CDdiag}
Let $C\in {\cal C}_0'$ and $A$ be a separable simple exact \CA\,  with strict comparison for positive elements, 
with $K_0(A)=\{0\},$   with stable rank one and with continuous scale.
Suppose also that $W(A)={\rm LAff}_{b+}(T(A)).$ 
Suppose that $\phi: C\to A$ is a \hm. Then, for any $\ep>0,$  any finite subset 
${\cal F}\subset C,$  and any integer $n\ge 1,$ there is another \hm\, $\phi_0: C\to B\subset  M_n(B)\subset A$ 
such that
\beq
\|\phi(x)-\diag(\overbrace{\phi_0(x), \phi_0(x),...,\phi_0(x)}^n)\|<\ep\rforal x\in {\cal F}.
\eneq
\end{cor}

\begin{proof}
Fix a strictly positive element $e\in A_+$ with $\|e\|=1.$ 
There are  mutually orthogonal elements $e_1, e_2,...,e_n\in A_+$ 
such that $\la e_i\ra =\la e_1\ra$ in $Cu(A)$ and 
$\la \sum_{i=1}^n e_i \ra =\la e\ra.$ 
Let $B_1=\overline{e_1Ae_1}\subset A.$ There exist a \SCA\, $D\subset A$ 
such  that $M_n(B_1)\cong D.$ We write $D=M_n(B_1).$ 
Let $\lambda: Cu(C)\to Cu(A)$ defined by
$\lambda(\la a\ra)=1/n (Cu(\phi))(\la a\ra)$ for all $\la a\ra \in Cu(C).$
Note that $K_0(A)=\{0\}.$ Therefore $\lambda$ determines a ${\bf Cu}$ \hm.
(see \cite{Rl}).
It follows from \cite{Rl} that there exists a \hm\, $\phi_0': C\to B_1$ such that
$Cu^{\sim}(\phi_0')=\lambda.$ 
Define $\psi: C\to M_n(B_1)$ by 
$\psi(a)=\diag(\overbrace{\phi_0'(a), \phi_0'(a),...,\phi_0'(a)}^n)$ for all $a\in C.$
Then $Cu^{\sim}(\psi)=Cu^{\sim}(\phi).$ 
It follows from \cite{Rl} that $\phi$ and $\psi$ are approximately unitarily equivalent.
Lemma then follows.
\end{proof}


\section{Tracially one dimensional complices}

\begin{df}\label{Dtr1div}
Let $A$ be a non-unital simple \CA\,  with a strictly positive element $a\in A$
with $\|a\|=1.$  Suppose that
for any $\ep>0,$  any
finite subset ${\cal F}\subset A$ and any $b\in A_+\setminus \{0\},$  there are ${\cal F}$-$\ep$-multiplicative \cpc s $\phi: A\to A$ and  $\psi: A\to D$  for some
\SCA\, $D\subset A$ such that
\beq\label{Dtr1div-1}
&&\|x-\diag(\phi(x),\psi(x))\|<\ep\rforal x\in {\cal F}\cup \{a\},\\\label{Dtrdiv-2}
&& D\in {\cal C}_0^{0'} ({\rm or}\,\, {\cal C}_0'),\\\label{Dtrdiv-3}
&&\phi(a)\lesssim b,\\\label{Dtrdiv-4}
&& \|\psi(x)\|\ge (1-\ep)\|x\|\rforal x\in {\cal F},\\\label{Dtrdiv-4+}
&&f_{1/4}(\psi(a))\,\,\,
{\rm is\,\,\, full\,\,\, in}\,\, D \andeqn
\eneq
$\psi(a)$ is strictly positive in $D.$
  Then we say $A$   is in TA$C_0^{0'}$  (TA${\cal C}_0'$)

Let $A_0=\overline{\phi(a)A\phi(a)}.$ Then  \eqref{Dtrdiv-2} also means
$A_0\perp D.$
Moreover, \eqref{Dtrdiv-3} could be replaced by
$c\lesssim b$ for some strictly positive element $c$ of $A_0.$

\end{df}

\begin{prop}\label{DpreTADone}
Let ${\cal S}={\cal C}_0',$ or ${\cal C}_0^{0'}.$

Suppose that $A$ is a non-unital simple \CA\, which is in TA${\cal S}.$
Then the following holds:

Let  $a\in A$
with $\|a\|=1$ be a strictly positive element.
For any $\ep>0,$  any
finite subset ${\cal F}\subset A,$
any finite subset ${\cal F}_0\subset A_+\setminus \{0\},$ and any $b\in A_+\setminus \{0\},$  there are ${\cal F}$-$\ep$-multiplicative \cpc s $\phi: A\to A$ and  $\psi: A\to D$  for some
\SCA\, $D\subset A$ such that
\beq\label{Ptad1-1}
&&\|x-\diag(\phi(x),\psi(x))\|<\ep\rforal x\in {\cal F}\cup \{a\},\\\label{Ptad1-2}
&& D\in {\cal C}_0^{0'} ({\rm or}\,\, {\cal C}_0'),\\\label{Ptad1-3}
&&\phi(a)\lesssim b,\\\label{Ptad1-4}
&& \|\psi(x)\|\ge (1-\ep)\|x\|\rforal x\in {\cal F},\\\label{Ptad1-4+}
&&f_{1/4}(\psi(x))\,\,\,
{\rm is\,\,\, full\,\,\, in}\,\, D \rforal x\in {\cal F}_0\andeqn
\eneq
$\psi(a)$ is strictly positive in $D.$

\end{prop}

\begin{proof}
We may assume that ${\cal F}\cup {\cal F}_0\subset A^{\bf 1}.$
We may also assume, \wilog, that $\|c\|=1$ for all $c\in {\cal F}_0.$
\Wlog, we may assume that $f_\eta(a)x=xf_\eta(a)$ for all
$x\in {\cal F}\cup {\cal F}_0$ for some $1/4>\eta>0.$
Let ${\cal F}_{0,1}=\{f_{1/4}(c): c\in {\cal F}_0\}\cup {\cal F}_0.$

There are, for each $c\in {\cal F}_{0,1},$  $x_i(c)\in A,i=1,2,...,n(c)$
such that
\beq\label{Ptad1-10}
\sum_{i=1}^{n(c)}x_i(c)^*cx_i(c)=f_{\eta/8}(a).
\eneq

We will choose $0<\dt<\ep$ and a finite subset ${\cal G}\subset A$ satisfy the following:
If $\psi': A\to B'$ is any ${\cal G}$-$\dt$-multiplicative \cpc,
then
\beq\label{Ptad1-11}
&&\hspace{-0.5in}\|\sum_{i=1}^{n(c)}\psi'(x_i(x))^*\psi'(c)\psi'(x_i(c))-f_{\eta/8}(\psi'(c))\|<\min\{\ep/128, \eta/128\}\rforal c\in {\cal F}_{0,1}\\\label{Ptad1-11+}
&&\andeqn\|\psi'(f_{\sigma}(a))-f_{\sigma}(\psi'(a))\|<{\min\{\ep/128,\eta/128\}\over{\max\{2n(c)\cdot \|x_i(c)\|:c\in {\cal F}_{0,1}\}}}
\eneq
for $\sigma=\{\eta/8, 1/4\}.$
Let
$$
{\cal F}_2={\cal F}\cup {\cal F}_{0,1}.
$$
By the assumption, there are ${\cal G}$-$\dt$-multiplicative \cpc s $\phi: A\to A$ and  $\psi: A\to D$  for some
\SCA\, $D\subset A$ such that
\beq\label{Ptad1-12}
&&\|x-\diag(\phi(x),\psi(x))\|<\min\{\ep/64, \eta/64\}\rforal x\in {\cal F}\cup {\cal F}_{0,1},\\\label{Ptad1-13}
&& D\in {\cal S},
\phi(a)\lesssim b,\\
&&\|\psi(x)\|=(1-\ep/16)\|x\| \rforal x\in {\cal F}.
\eneq
 and $f_{1/2}(\psi(a))$ is full in $D.$


By the choice of ${\cal G}$ and $\dt,$ by \eqref{Ptad1-11} and \eqref{Ptad1-11+},  we obtain that
\beq\nonumber
\|\sum_{i=1}^{n(c)}\psi(x_i(c))^*f_{1/4}(\psi(c))\psi(x_i(c))-f_{\eta/2}(\psi(a))\|<\min\{\ep/16, \eta/16\},
\eneq
for all $c\in {\cal F}_0.$
It follows from \ref{Lrorm} that there are $y_i(c)\in D$ such that
$$
\sum_{i=1}^{n(c)}y_i(c)^*f_{1/4}(\psi(c))y_i(c)=f_{\eta}(\psi(a))
$$
for all $c\in {\cal F}_0.$
Since $f_{1/4}(\psi(a))$ is full in $D,$ so is $f_{\eta}(\psi(a)).$ Therefore $f_{1/4}(\psi(c))$
is full in $D$ for all $c\in {\cal F}_0.$



\end{proof}


\begin{cor}\label{ChighrankD}
In the definition of \ref{Dtr1div},  for any integer $k\ge 1,$ one may assume that every irreducible representation 
of $D$ has rank at least $k.$ 
\end{cor}

\begin{proof}
Fix an integer $k\ge 1.$ 
This corollary  can be easily seen by taking ${\cal F}_0$ containing $k$ mutually orthogonal non-zero 
positive elements  $e_1, e_2,...,e_k$  with $\|e_i\|=1$ in \ref{DpreTADone}.
  More precisely,  since $k$ is given,
by taking sufficiently small $\ep,$ we may assume that $D$ contains 
$k$ mutually orthogonal non-zero elements which are full. This forces $\pi(D)$ admits 
$k$ mutually orthogonal non-zero elements in  each irreducible representation 
$\pi.$
\end{proof}

The following follows immediately from the definition.

\begin{prop}\label{PMnTAD}
Let $A$ be a $\sigma$-unital simple \CA\, which is in TA${\cal C}_0^{0'}$ (TA${\cal C}_0'$).
Then, for any integer $k\ge 1,$ $M_k(A)$ is also in TA${\cal C}_0^{0'}$( or TA${\cal C}_0'$).
\end{prop}

\begin{df}\label{DNtr1div}
Let $A$ be a non-unital simple \CA\,  with a strictly positive element $a\in A$
with $\|a\|=1.$  Suppose that there exists
$1> \mathfrak{f}_a>0,$ for any $\ep>0,$  any
finite subset ${\cal F}\subset A$ and any $b\in A_+\setminus \{0\},$  there are ${\cal F}$-$\ep$-multiplicative \cpc s $\phi: A\to A$ and  $\psi: A\to D$  for some
\SCA\, $D\subset A$ such that
\beq\label{DNtr1div-1}
&&\|x-\diag(\phi(x),\psi(x))\|<\ep\rforal x\in {\cal F}\cup \{a\},\\\label{DNtrdiv-2}
&& D\in {\cal C}_0^{0'} ({\rm or}\,\, {\cal C}_0'),\\\label{DNtrdiv-3}
&&\phi(a)\lesssim b,\\\label{DNtrdiv-4}
&&t(f_{1/4}(\psi(a)))\ge \mathfrak{f}_a\rforal t\in T(D).
\eneq
  Then
then we say $A\in {\cal D}_{0}$ (or ${\cal D}$).

\end{df}

\begin{prop}\label{PD0=tad}
Let $A$ be a $\sigma$-unital simple \CA\, in ${\cal D}$ (${\cal D}_{0}$).
Then $A$ is TA${\cal S}$ (${\cal S}={\cal C}_0',$ ${\cal C}_0^{0'}$).
\end{prop}

\begin{proof}
Fix a strictly positive element $a\in A$ with $\|a\|=1.$
We will show in  the definition, we can further assume that $\|\psi(x)\|\ge (1-\ep)\|x\|$ for all $x\in {\cal F}$ and
$\psi(a)$ is strictly positive in $D.$

Let $\ep>0,$ ${\cal F}\subset A$ be a finite subset and $b_0\in A_+\setminus\{0\}$ be given.
\Wlog, we may assume that there is $1/16>\eta>0$ such
that
$$
f_{\eta}(a)x=xf_{\eta}(a)=x\rforal x\in {\cal F}.
$$
By the assumption, there exists  a sequence
of $D_n\in {\cal S}$
and two sequences of \cpc s $\phi_n: A\to A_n,$ and  $\psi_n: \to D_n$
such that
\beq\label{=tad-2}
\lim_{n\to\infty}\|\phi_n(xy)-\phi_n(x)\phi_n(y)\|=0\andeqn\\\label{=tad-3}
\lim_{n\to\infty}\|\psi_n(xy)-\psi_n(x)\psi_n(y)\|=0\rforal x, \, y\in A,\\\label{=tad-4}
\lim_{n\to\infty}\|x-\diag(\phi_n(x),\psi_n(x))\|=0\rforal x\in A\\
\label{=tad-5}
\phi_n(a)\lesssim b_0,\\\label{=tad-6}
\tau(f_{1/4}(\psi_n(a)))\ge \mathfrak{f}_a \rforal \tau\in T(D_n)
\eneq
Put $D_n'=\overline{f_{\eta/2}(\psi_n(a))D_nf_{\eta/2}(\psi_n(a))},$ $n=1,2,....$
By \eqref{=tad-6} and \ref{Pcom4C}, $f_{1/4}(\psi_n(a))$ is full in $D_n.$
Therefore $f_{\eta/2}(\psi_n(a))$ is also full in $D_n.$  This implies that $D_n'\in {\cal S}.$
Define $\psi_{n,0}: A\to D_n'$ by
$$
\psi_{n,0}(x)=(f_{\eta/2}(\psi_n(a)))^{1/2}\psi_n(x)(f_{\eta/2}(\psi_n(a)))^{1/2}\rforal x\in A.
$$
It follows that $\psi_{n,0}(a)$ is full in $D_n'.$
Note that
$$
f_{1/4}(\psi_{n,0}(a))=f_{1/4}(\psi(a)).
$$
Therefore
\vspace{-0.1in} $$
\tau(f_{1/4}(\psi_{n,0}(a)))\ge \mathfrak{f}_a\rforal \tau\in T(D_n').
$$
By choosing large $n,$ using \eqref{=tad-3} and \eqref{=tad-4},
we see that in the definition of \ref{DNtr1div}, we can add the condition
that $\psi(a)$ is full in $D.$

To get  inequality $\|\psi(x)\|\ge (1-\ep)\|x\|$ for all $x\in {\cal F},$  we repeat the above argument. As shown above, we can add
that $\psi_n(a)$ is full in each $D_n$ along with \eqref{=tad-2} to \eqref{=tad-6}.
The condition \eqref{=tad-6} also implies that
\beq\label{=tad-8}
\lim_{n\to\infty}\|\psi_n\|\ge \mathfrak{f}_a.
\eneq
Then, by \eqref{=tad-3}, since $A$ is simple,
\beq\label{=tad-9}
\lim_{n\to\infty}\|\psi_n(x)\|=\|x\|.
\eneq
This implies that, with sufficiently large $n,$ we can always assume
that $\|\psi(x)\ge (1-\ep)\|x\|$ for all $x\in {\cal F}.$
\end{proof}


\begin{thm}\label{UnifomfullTAD}
Let $A$ be a $\sigma$-unital simple \CA\, in ${\cal D}_{0}$ (or  in ${\cal D}$). Then the following holds.
Fix  a strictly positive element $a\in A$
with $\|a\|=1$ and let $1>\mathfrak{f}_a>0$ be a positive number associated
with $a$ in the definition \ref{DNtr1div}. There is a map $T: A_+\setminus \{0\}\to \N\times \R$
($a\mapsto (N(a), M(a))\rforal a\in A_+\setminus \{0\}$) satisfying the following:
For any  finite subset ${\cal F}_0\subset A_+\setminus \{0\}.$
 for any $\ep>0,$  any
finite subset ${\cal F}\subset A$ and any $b\in A_+\setminus \{0\},$   there are ${\cal F}$-$\ep$-multiplicative \cpc s $\phi: A\to A$ and  $\psi: A\to D$  for some
\SCA\, $D\subset A$ such that
\beq\label{Dtad-1}
&&\|x-\diag(\phi(x),\psi(x))\|<\ep\rforal x\in {\cal F}\cup \{a\},\\\label{2Dtrdiv-2}
&& D\in {\cal C}_0^{0'}\,({\rm or}\,\,\,{\cal C}_0'),\\\label{Dtad-3}
&&\phi(a)\lesssim b,\\\label{2Dtrdiv-4}
&& \|\psi(x)\|\ge (1-\ep)\|x\|\rforal x\in {\cal F}\andeqn\\
\eneq
$\psi(a)$ is strictly positive in $D.$
Moreover,  $\psi$ is $T$-${\cal F}_0\cup\{f_{1/4}(a)\}$-full in $\overline{DAD}.$

Furthermore, we may assume that
$$
t\circ f_{1/4}(\psi(a))\ge \mathfrak{f}_a\andeqn t\circ f_{1/4}(\psi(c))\ge \mathfrak{f}_a/4\inf\{M(c)^2\cdot N(c): c\in {\cal F}_0\cup\{f_{1/4}(a)\}\}
$$
for all $c\in {\cal F}_0$ and for all $t\in T(D),$
where $\mathfrak{f}_a$ is given by the definition of ${\cal D}$ (or ${\cal D}_{0}$) associated with $a.$

\end{thm}

\begin{proof}
Since $A$ is simple, for any $b\in A_+\setminus \{0\},$ there exists $N_0(b)\in \N,$
$M_0(b)>0$  and $x_1(b), x_2(b),...,x_{N_0(b)}(b)\in A$ such that $\|x_i(b)\|\le M_0(b)$ and
\beq\label{UnifomfullTAD-n1}
\sum_{i=1}^{N_0(b)} x_i(b)^* bx_i(b)=f_{1/32}(a).
\eneq
Let $\mathfrak{f}_a>0$ be given  as in the definition of \ref{Dtr1div}.

Let $n_0\ge 1$ be an integer such that $n_0\mathfrak{f}_a\ge 4.$

Set $N(b)=n_0N_0(b)$ and $M(b)=2M_0(b)$ for all $b\in A_+\setminus \{0\}.$
Let $T: A_+\setminus \{0\}\to \N\times \R_+\setminus \{0\}$ be defined
by $T(b)=(N(b), M(b))$ for all $b\in A_+\setminus \{0\}.$

Choose $\dt_0>0$ and finite subset ${\cal G}_0\subset A$ such that
\beq\label{UnifomfullTAD-n2}
\|\sum_{i=1}^{N_0}\psi( x_i(b))^* \psi(b)\psi(x_i(b))-f_{1/32}(\psi(a))\|<1/2^{10}\rforal b\in {\cal F}_0,
\eneq
provided $\psi$ is a ${\cal G}_0$-$\dt_0$-multiplicative \cpc\, from $A$ into a \CA .

Let $\ep>0$ and a finite subset ${\cal F}\subset A$ be given.
Let $\dt=\min\{\ep/4, \dt_0/2\}$ and ${\cal G}={\cal F}\cup {\cal G}_0\cup \{a, f_{1/4}(a)\}.$
Let $n\ge 1$ be an integer and let $b_0\in A_+\setminus \{0\}.$

By the assumption one has the following:
 there are ${\cal G}$-$\dt$-multiplicative \cpc s $\phi: A\to A$ and  $\psi: A\to D$  for some
\SCA\, $D\subset A$ such that
\beq\label{UnifomfullTAD-n3}
&&\|x-\diag(\phi(x), \psi(x))\|<\ep\rforal x\in {\cal G},\\\label{UnifomfullTAD-n4}
&& D\in {\cal C}_0^{0'}\,({\rm or}\,\,\,{\cal C}_0'),\\\label{UnifomfullTAD-n5}
&&\phi(a)\lesssim b_0,\\\label{UnifomfullTAD-n6}
&& \|\psi(x)\|\ge (1-\ep)\|x\|\rforal x\in {\cal F}\andeqn
\eneq
$\psi(a)$ is strictly positive in $D,$ and
\beq\label{UnifomfullTAD-n7}
\tau(f_{1/4}(\psi(a)))\ge \mathfrak{f}_a\rforal \tau\in T(D).
\eneq
At this point, we can apply  the argument of \ref{Tqcfull} and its remark \ref{Rqcfull} to conclude
that $\psi$ is $T$-${\cal F}_0\cup \{f_{1/4}(a_0)\}$-full. Lemma then follows.

\end{proof}

\begin{thm}\label{T1D0}
Let ${\cal S}={\cal C}_0', $ or ${\cal C}_0^{0'}.$
Let  $A$ be a non-unital separable  simple \CA\, in  TA${\cal S}$  which is quasi-compact.
Then  $A\in {\cal D}$ (or ${\cal D}_{0}$).
\end{thm}

\begin{proof}
Fix a strictly positive element $a\in A_+.$  We may assume that
there is $e_1\in M_N(A)$ with $0\le e_1\le 1$ and
a partial isometry $w\in  M_N(A)^{**}$
such that
\beq\label{T1D0-1}
w^*xwe_1=e_1w^*xw=w^*xw,\,\,\, x(ww^*)=(ww^*)x=x, w^*xw, xw\in M_N(A)
\eneq
for all $x\in A.$
We may also assume, \wilog, that there are $z_1', z_2',...,z_m'\in M_N(A)$
such that
$$
\sum_{i=1}^m(z_i')^*f_{1/4}(w^*aw)^{1/2}w^*awf_{1/4}(w^*aw)^{1/2}(z_i')=f_{1/8}(e_1).
$$
Therefore
\beq\label{T1D0-n1}
\sum_{i=1}^m f_{1/8}(e_1)(z_i')^*f_{1/4}(w^*aw)^{1/2}w^*awf_{1/4}(w^*aw)^{1/2}(z_i')f_{1/8}(e_1)=(f_{1/8}(e_1))^3.
\eneq
We may also assume that, there are $y_i\in \overline{f_{1/8}(e_1)M_N(A)f_{1/8}(e_1)},$ $i=1,2,...,m,$ such that
\beq\label{T1D0-n2}
\sum_{i=1}^m y_i^*f_{1/4}(w^*aw)y_i=f_{1/8}(e_1)^3.
\eneq
Set $e=f_{1/8}(e_1)^3$ and $z_i=z_i'f_{1/8}(e_1).$  Note that
$$
w^*xwe=ew^*xw=w^*xw\andeqn z_i\in \overline{eM_N(A)e}.
$$
Let $M=\max\{\|y_i\|+\|z_i\|+1: 1\le i\le m\}.$

Let $\mathfrak{f}_a=1/64Mm^2.$  Fix an integer $k\ge 1.$
Fix a finite subset ${\cal F}$ which contains $a$ and $f_{1/4}(a).$
Fix an element $a_0\in A\setminus \{0\}.$
Since $M_N(A)$ is TA${\cal S}$ (for ${\cal S}={\cal C}_0^{0'}, \, {\cal C}_0'$), by \ref{DpreTADone}, there exists a sequence
of $D_n\in {\cal S}$
and two sequences of \cpc s $\phi_n,\, \psi_n: M_N(A)\to M_N(A)$
such that
\beq\label{T1D0-10}
\lim_{n\to\infty}\|\phi_n(xy)-\phi_n(x)\phi_n(y)\|=0\andeqn\\\label{T1D0-11}
\lim_{n\to\infty}\|\psi_n(xy)-\psi_n(x)\psi_n(y)\|=0\rforal x, \, y\in M_N(A),\\\label{T1D0-12}
\lim_{n\to\infty}\|x-\diag(\phi_n(x),\psi_n(x))\|=0\rforal x\in M_N(A),\\
\label{T1D0-13}
\psi_n(x)\in D_n\rforal x\in M_N(A),\\\label{T1D0-14}
\lim_{n\to\infty}\|\psi_n(x)\|=\|x\|\rforal x\in M_N(A),\\\label{T1D0-15}
\phi_n(e)\lesssim a_0\\\label{T1D0-15+}
f_{1/4}(\psi_n(w^*aw))\,\,{\rm is\,\,\, full\,\,\, in}\,\,\, D_n.
\eneq

Note that
\eqref{T1D0-15} implies that
\beq\label{T1D0-15++}
\phi_n(w^*aw)\lesssim w^*a_0w.
\eneq

Put $A_1=\overline{(w^*aw)M_N(A)(w^*aw)}=w^*Aw$ and $A_2=\overline{eM_N(A)e}.$

There are $g_n\in  \psi_n(w^*aw)D_n\psi_n(w^*aw)$ with $0\le g_n\le 1$ such that
\beq\label{T1D0-16}
\lim_{n\to\infty}\|g_nxg_n-g_n\psi_n(w^*xw)g_n\|&=&0\rforal w^*xw\in M_N(A),\\
\lim_{n\to\infty}\|g_n\psi_n(w^*xw)-\psi_n(w^*xw)\|&=&0\rforal x\in A\andeqn\\
\lim_{n\to\infty}\|\psi_n(w^*xw)g_n-\psi(w^*xw)\| &=& 0 \rforal x\in A.
\eneq
It follows that
\beq\label{T1D0-16+}
&&\lim_{n\to\infty}\|g_nw^*xw-w^*xwg_n\|=0\,\,\, {\rm as\,\,\, well \,\,\, as}\\
&&\lim_{n\to\infty}\|w^*x^{1/2}wg_n^2w^*x^{1/2}w-\psi_n(w^*xw)\|=0\rforal x\in A.\label{T1D0-16++}
\eneq
From these, we also have
$$
\lim_{n\to\infty}\|\psi_n|_{A_1}\|=1\andeqn \lim_{n\to\infty}\|\psi_n|_{A_2}\|=1.
$$
Put $\psi_n'=(1/\|\psi_n|_{A_1}\|)\psi_n|_{A_2},$ $n=1,2,....$
Note
$$
\|\psi_n'\|\le 1\andeqn \lim_{n\to\infty}\|\psi_n'\|=1.
$$

Put
$$
\sigma_n=\|\sum_{i=1}^m\psi_n'(z_i)^*\psi_n'(f_{1/4}(w^*aw))\psi_n'(z_i)-f_{1/8}(\psi_n'(e))\|,\,\,\, n=1,2,....
$$
By \eqref{T1D0-11} and \eqref{T1D0-n2},  we also have
that
\beq\label{T1D0-17}
\lim_{n\to\infty}\sigma_n=0.
\eneq
\Wlog, we may assume that
\beq\label{T1D0-18}
\sigma_n< \min\{1/128 Mm^2\}\rforal n.
\eneq
Note that
$$
f_{1/32}(e)w^*xw=w^*xwf_{1/32}(e)=w^*xw\rforal x\in A.
$$
Put $H_n=\overline{\psi_n(w^*aw)D_n\psi_n(w^*aw)},$ $n=1,2,....$
Thus, for any tracial state of $t_n\in T(H_n),$
$t_n\circ \psi_n'$ is a state of $A_1.$ One can uniquely extended to a state on $A_2.$
 Since
$$
1\ge e\ge x\rforal 0\le x\le 1\andeqn x\in A_1,
$$
$$
t_n\circ \psi_n'(e)=1\rforal n.
$$
Take any weak*-limit $\tau$ of $\{t_n\circ \psi_n'\}.$
Then $\tau(e)=1.$
Moreover, $\tau$ is a trace on $A_2.$   By \eqref{T1D0-17},
\beq\label{T1D0-19}
\tau(f_{1/4}((w^*aw))\ge  1/4Mm^2\andeqn \tau(w^*aw))\ge 1/16Mm^2.
\eneq
We may assume, \wilog,  that, for all $n,$
\beq\label{T1D0-20}
t\circ (\psi_n(f_{1/4}(w^*aw)))\ge 1/5Mm^2\rforal t\in T(D_n).
\eneq
By \eqref{T1D0-11},
\beq\label{T1D0-20-1}
\lim_{n\to\infty}\|\psi_n(f_{1/4}(w^*aw))-f_{1/4}(\psi_n(w^*aw))\|=0.
\eneq
Therefore, \wilog, we may assume that, for all $n,$
\beq\label{T1D0-20-2}
t\circ (f_{1/4}(\psi_n(w^*aw)))\ge 1/6Mm^2\rforal t\in T(D_n).
\eneq
We also have that
\beq\label{T1D0-20+1}
&&\lim_{j\to\infty}\|w^*aw-f_{1/2j}(w^*aw)^{1/2}w^*awf_{1/2j}(w^*aw)^{1/2}\|=0
\andeqn\\\label{T1D0-20+2}
&&\lim_{j\to\infty}\|w^*xw-f_{1/2j}(w^*aw)^{1/2}w^*xwf_{1/2j}(w^*aw)^{1/2}\|=0
\rforal  x\in A.
\eneq

Put $L_n(a)=\diag(\phi_n(w^*aw), ,\psi_n(w^*aw)).$
It follows from \eqref{T1D0-12} and \ref{LRL}  that,
for any $j\ge 2,$  there exists $n(j)\ge j$ and  a partial isometry
$v_{j}\in M_N(A)^{**}$ such that
\beq\label{T1D0-23}
&&\hspace{-0.8in}v_{j}v_{j}^*f_{1/2j}(L_{n(j)}(w^*aw))=f_{1/2j}(L_{n(j)}(w^*aw)v_jv_{j}^*=f_{1/2j}(L_{n(j)}(w^*aw)),\\
&&\hspace{-0.8in}v_{j}^*cv_{j}\in A_1\rforal c\in \overline{f_{1/2j}(L_{n(j)}(w^*aw))M_N(A)f_{1/2j}(L_{n(j)}(w^*aw))}
\andeqn\\\label{T1D0-23+}
&&\hspace{-0.9in}\lim_{j\to\infty}\sup\{\|v_{j}^*cv_{j}-c\|: 0\le c\le 1\andeqn c\in\overline{f_{1/2j}(L_{n(j)}(w^*aw))Af_{1/2j}(L_{n(j)}(w^*aw))}\}=0.
\eneq
Note that $f_{1/2j}(\psi_{n(j)}(w^*aw))\le f_{1/2j}(L_{n(j)}(w^*aw)),$ $j=1,2,....$
It follows that
$$
v_{j}^*cv_{j}\in A_1
$$
for all $c\in \overline{f_{1/2j}(\psi_{n(j)}(w^*aw)Af_{1/2j}(\psi_{n(j)}(w^*aw))}.$
By \eqref{T1D0-20-2},
$f_{1/2j}(\psi_{n(j)}(w^*aw))$ is full in $D_{n(j)}$ for all $j\ge 2.$
Put
$$
E_{n(j)}'=\overline{f_{1/2j}(\psi_n(w^*aw))D_{n(j)}f_{1/2j}(\psi_n(w^*aw))},
\,\,\,j=2,3,....
$$
Then $E_{n(j)}'\in {\cal S},$ $j=2,3,....$
Put
$$
E_{j}=v_{j}^*E_{n(j)}'v_{j},\,\,\,j=j_0+1,j_0+2,....
$$
Then $E_{j}\in {\cal S}$ and $E_{j}\subset A_1,$ $j=j_0+1, j_0+2,....$

Define $\Phi_{j}=v^*_j\phi_{n(j)}v_j^*$
and $\Psi_j: A_1\to E_j,$ by $\Psi_j(x)=v_jf_{1/2j}(\psi(w^*aw)\psi_{n(j)}'(x)f_{1/2j}(\psi_{n(j)}(w^*aw)v_j^*,$ $j=3, 4,....$
Note that $\Psi_j$ maps $A_1$ into $E_{j}$ with
$\|\Psi_j\|=1,$ $j=1,2,....$
We have
\beq\label{T1D0-25}
\lim_{j\to\infty}\|\Phi_{j}(xy)-\Phi_{j}(x)\Phi_j(y)\|=0\rforal x, y\in A_1,\\
\limsup_{j\to\infty}\|\Psi_j(xy)-\Psi_j(x)\Psi_j(y)\|=0\rforal x,y\in A_1.
\eneq
Moreover,  by applying \eqref{T1D0-12}, \eqref{T1D0-23+} and \eqref{T1D0-20+2},
\beq\label{T1D0-27}
\limsup_{j\to\infty}\|x-\diag(\Phi_j(x),\Psi_j(x))\|=0\rforal x\in A_1.
\eneq
We also have
that
$$
\Psi_j(w^*aw)\lesssim a_0.
$$
Moreover, by \eqref{T1D0-20-2},
$$
t\circ f_{1/4}(\Psi_j(w^*aw))\ge \mathfrak{f}_a\rforal t\in T(E_{n(j)}).
$$
The theorem follows from the fact that $a\mapsto w^*aw$ ($a\in A$) is an isomorphism
from $A_1$ onto $A$ (and choosing a sufficiently large $j$).

\end{proof}

\begin{prop}\label{Phered}
Let ${\cal S}$ denote ${\cal C}_0^{0'},$ or ${\cal C}_0'.$
Let $A$ be a non-unital separable simple \CA\, and let $B\subset A$ be a hereditary \SCA.
Then

(1) If $A$ is TA${\cal S},$ so is $B;$

(2) If $A$ is in ${\cal D}$ (or ${\cal D}_{0}$), so is $B.$

\end{prop}

\begin{proof}
Let ${\cal S}$ denote ${\cal C}_0^{0'},$ or ${\cal C}_0'.$
Let $b\in A_+$ with $\|b\|=1$ and $B=\overline{bBb}.$
Let $e\in A_+$ be a strictly positive element with $\|e|=1.$
We assume that $\|b\|=1$ and $\|e\|=1.$
Fix $b_0\in B_+\setminus \{0\}.$

In both cases (1) and (2),  by \ref{DpreTADone},
there exists
a sequence
of $D_n\in {\cal S}$
and two sequences of \cpc s $\phi_n,\, \psi_n: A\to A$
such that
\beq\label{Phered-1}
\lim_{n\to\infty}\|\phi_n(xy)-\phi_n(x)\phi_n(y)\|=0\andeqn\\
\lim_{n\to\infty}\|\psi_n(xy)-\psi_n(x)\psi_n(y)\|=0\rforal x, \, y\in A,\\\label{Phered-2}
\lim_{n\to\infty}\|x-\diag(\phi_n(x),\psi_n(x))\|=0\rforal x\in A,\\
\phi_n(e)\lesssim b_0,\\\label{Phered-2+}
\lim_{n\to\infty}\|\psi_n(x)\|=\|x\|\rforal x\in  A,
\eneq
$f_{1/4}(\psi_n(b))$ is full in $D_n,$  and $\psi_n(e)$ is a strictly positive element of $D_n,$ $n=1,2,....$
Moreover, in the case (2),  by \ref{UnifomfullTAD}, we can also have
\beq\label{Phered-1n}
t\circ f_{1/4}(\psi_n(e))\ge \mathfrak{f}_e,\,\,\, t\circ f_{1/4}(\psi_n(b))\ge r_0
\eneq
for all $t\in T(D_n)$ and $n,$   where  $r_0$ is previously given
(as $\mathfrak{f}_e/4\inf\{M(c)\cdot N^2(c): c=\{b, f_{1/4}(e)\}\}$).


By \eqref{Phered-2+},
$$
\lim_{n\to\infty}\|\psi_n|_{B}\|=1.
$$

We also have that
\beq\label{Phered-20+1}
&&\lim_{j\to\infty}\|b-f_{1/2j}(b)^{1/2}bf_{1/2j}(b)^{1/2}\|=0
\andeqn\\\label{Phered-20+2}
&&\lim_{j\to\infty}\|x-f_{1/2j}(b)^{1/2}xf_{1/2j}(b)^{1/2}\|=0
\rforal  x\in B.
\eneq

Put $L_n(x)=\diag(\phi_n(x),\psi_n(x))$ for all $x\in A.$
As in the proof of \ref{T1D0}, by applying \ref{LRL},
for any $j\ge 2,$  there exists $n(j)\ge j$ and  a partial isometry
$v_{j}\in A^{**}$ such that
\beq\label{Phere-23}
&&\hspace{-0.8in}v_{j}v_{j}^*f_{1/2j}(L_{n(j)}(b))=f_{1/2j}(L_{n(j)}(b))v_jv_{j}^*=f_{1/2j}(L_{n(j)}(b)),\\
&&\hspace{-0.8in}v_{j}^*cv_{j}\in B\rforal c\in \overline{f_{1/2j}(L_{n(j)}(b))Af_{1/2j}(L_{n(j)}(b))}
\andeqn\\\label{Phere-23+}
&&\hspace{-0.8in}\lim_{j\to\infty}\sup\{\|v_{j}^*cv_{j}-c\|: 0\le c\le 1\andeqn c\in
\overline{f_{1/2j}(L_{n(j)}(b))Af_{1/2j}(L_{n(j)}(b))}\}=0.
\eneq
Note that $f_{1/2j}(\psi_{n(j)}(b))\le f_{1/2j}(L_{n(j)}(b)),$ $j=1,2,....$
It follows that
$$
v_{j}^*cv_{j}\in B
$$
for all $c\in \overline{f_{1/2j}(\psi_{n(j)}(b)Af_{1/2j}(\psi_{n(j)}(b))}.$
Since $f_{1/4}(\psi_{n(j)}(f_{1/4}(b)))$ is full in $D_{n(j)},$
$f_{1/2j}(\psi_{n(j)}(f_{1/4}(b)))$ is full in $D_{n(j)}$ for all $j\ge 2.$
Put
$$
E_{n(j)}'=\overline{f_{1/2j}(\psi_n(b))D_{n(j)}f_{1/2j}(\psi_n(b))},
\,\,\,j=2,3,....
$$
Then $E_{n(j)}'\in {\cal S},$ $j=2,3,....$
Put
$$
E_{j}=v_{j}^*E_{n(j)}'v_{j},\,\,\,j=3,4,....
$$
Then $E_{j}\in {\cal S}$ and $E_{j}\subset B,$ $j=3, 4,....$

Define $\Phi_{j}=v^*_j\phi_{n(j)}v_j^*$
and $\Psi_j: B\to E_j,$ by $\Psi_j(x)=v_jf_{1/2j}(\psi_{n(j)}(b))\psi_{n(j)}(x)f_{1/2j}(\psi_{n(j)}(b))v_j^*,$ $j=3, 4,....$
For $j>4,$
\beq\label{Phere-n10}
&&\hspace{-0.4in}f_{1/4}(f_{1/2j}(\psi_{n(j)}(b))\psi_{n(j)}(b)f_{1/2j}(\psi_{n(j)}(b)))=f_{1/4}(\psi_{n(j)}(b))\\\label{Phere-n11}
\hspace{0.4in} &&=
 f_{1/2j}(\psi(b))f_{1/4}(\psi_{n(j)}(b))f_{1/2j}(\psi_{n(j)}(b)).
\eneq
It follows that
$f_{1/4}(\Psi_j(b))$ is full in $E_j,$ $j=4,5,....$
We have
\beq\label{Phere-n13}
\lim_{j\to\infty}\|\Phi_{j}(xy)-\Phi_{j}(x)\Phi_j(y)\|&=&0\rforal x, y\in B\andeqn\\
\limsup_{j\to\infty}\|\Psi_j(xy)-\Psi_j(x)\Psi_j(y)\|&=&0\rforal x,y\in B.
\eneq
Moreover,  by applying \eqref{Phered-2}, \eqref{Phere-23+} and \eqref{Phered-20+2},
\beq\label{Phere-n14}
\limsup_{j\to\infty}\|x-\diag(\Phi_j(x),\Psi_j(x))\|&=&0\rforal x\in B\andeqn\\
\lim_{n\to\infty}\|\Psi_n(x)\|&=&\|x\|\rforal x\in  B.
\eneq
We also have
that
$$
\Phi_j(b)\lesssim b_0.
$$
These implies that $B$ is  in TA${\cal S}.$
Moreover,  in case (2), by \eqref{Phered-1n} and \eqref{Phere-n10},
$$
t\circ f_{1/4}(\Psi_j(b))\ge r_0/2\rforal t\in T(E_{n(j)}).
$$

The proposition follows when  one chooses  a sufficiently large $j.$



\end{proof}

\begin{prop}\label{PtadMk}
Let $A\in {\cal D}$ (or ${\cal D}_{0}$) which is   quasi-compact.  Then, for every integer $k\ge 1,$ $M_k(A)\in {\cal D}$ (or
${\cal D}_0,$).
\end{prop}

\section{
Traces and the comparison  in \CA s in ${\cal D}$}

\begin{prop}\label{PD0qc}
Let $A$ be a separable simple \CA\, in ${\cal D}.$
Then $QT(A)=T(A).$ Moreover, $0\not\in \overline{T(A)}^w.$
\end{prop}

\begin{proof}
Let $a_0\in A$ be a strictly positive element of $A$ with $\|a_0\|=1.$
Let $\mathfrak{f}_{a_0}>0$ be in the definition of \ref{DNtr1div}.
Fix any $b_0\in A_+\setminus \{0\}.$
Choose a sequence of positive elements
$\{b_n\}$ which has the following:
$b_{n+1}\lesssim b_{n,1},$
where $b_{n,1}, b_{n,2},....,b_{n,n}$ are mutually orthogonal positive elements
in $\overline{b_nAb_n}$ such that $b_nb_{n,i}=b_{n,i}b_n=b_{n,i},$
$i=1,2,...,n,$ and $\la b_{n,i}\ra =\la b_{n,1}\ra,$ $i=1,2,...,n.$

One obtains two sequences of \SCA s $A_{0,n},$ $D_n$ of $A,$ where
$D_n\in {\cal C}_0',$  two sequences
of \cpc s $\phi_{0,n}: A\to A_{0,n}$  with
and
$\phi_{1,n}: A\to D_n$ satisfy the following:
\beq\label{TD0qc-1}
\lim_{n\to\infty}\|\phi_{i,n}(ab)-\phi_{i,n}(a)\phi_{i,n}(b)\|=0\rforal a,\, b\in A,\\
\lim_{n\to\infty}\|a-\diag(\phi_{0,n}(a), \phi_{1,n}(a))\|=0\rforal a\in A,\\
c_n\lesssim b_n,\\\label{TD0qc-4}
\lim_{n\to\infty}\|\phi_{1,n}(x)\|=\|x\|\rforal x\in A,\\
\tau(f_{1/4}(\psi_{1,n}(a_0)))\ge \mathfrak{f}_{a_0}\rforal \tau\in T(D_n)
\eneq
 and $\phi_{1,n}(a_0)$ is a strictly positive element in $D_n,$
 where $c_n$ is a strictly positive element of $A_{0,n}.$

It follows that
\beq\label{TD0qc-10}
\lim_{n\to\infty}\sup\{|\tau(a)-\tau\circ \phi_{1,n}(a)|: \tau\in QT(A)\}=0\rforal a\in A.
\eneq
Then $\lim_{n\to\infty}\|\tau|_{D_n}\|=1$ for all $\tau\in QT(A).$

Let $t={\tau\over{\|\tau|_{D_n}\|}}|_{D_n}.$ Then $t\in T(D_n).$
Given any $\tau\in QT(A).$
Define $t_n(a)=\tau\circ \phi_{1,n}(a)$ for all $a\in A.$
Then $t_n$ is a state of $A.$ Let $t_0$ be a weak*-limit of $\{t_n\}.$
Then $t_0\not=0,$ since by
\eqref{TD0qc-4} and \eqref{TD0qc-1},
\vspace{-0.1in} \beq\label{TD0qc-11}
t_0(f_{1/4}(a_0))\ge \mathfrak{f}_{a_0}.
\eneq
It follows from \eqref{TD0qc-1} that $t_0$ is a trace. It follows from \eqref{TD0qc-10}
that every $\tau\in QT(A)$ is a tracial state.

Let $\tau\in T(A).$ Then \eqref{TD0qc-11} shows that  $0\not\in \overline{T(A)}^w.$

\end{proof}

\begin{rem}\label{Rfa0}
Let $A\in {\cal D}$ and let $a\in A_+$ be a strictly positive element with $\|a\|=1.$
Define
$$
r_0=\inf\{\tau(f_{1/4}(a)): \tau\in T(A)\}>0.
$$
The above proof also shows that we can choose
$
\mathfrak{f}_{a}=r_0/2.
$
In fact in the case that $A$ is quasi-compact, one can choose
$\mathfrak{f}_{a}$  arbitrarily close to
$$
\inf\{\tau(a): \tau\in \overline{A}^w\}.
$$
and
in the case that $A$ has continuous scale, we can always choose
such a strictly positive element that the above $r_0$ could be arbitrarily close
to 1.
\end{rem}

\begin{prop}\label{Pprojless}
Every $\sigma$-unital simple \CA\, in ${\cal D}$ is stably projectionless.
\end{prop}

\begin{proof}
Let $A\in {\cal D}.$
Since $A\in {\cal D}$ if and only if $M_n(A)\in {\cal D}$ for each $n,$ we only need to show that
$D$ has no non-zero projections.
Let $p\in A$ be a projection and let
$$
r=\inf\{\tau(p): \tau\in \overline{T(A)}^w\}>0.
$$
Choose $r/4>\ep>0.$
Then
\vspace{-0.14in}\beq\label{Pprojless-1}
\|p-(x_1\oplus x_2)\|<\ep/2,
\eneq
where $x_1\in (A_0)_+$ and $x_2\in  D_+,$ where $A_0=\overline{bAb}$ for some
$b\in A_+$ with
$d_\tau(b)<r/4$ for all $\tau\in \overline{T(A)}^w,$
$x_2\in D_+$ and $D\in {\cal D}.$
A  standard argument shows that there are projections
$p_1\in A_0$ and $p_2\in D$ such that
\beq\label{Pprojless-2}
\|p-(p_1\oplus p_2)\|<\ep.
\eneq
Since $D$ is projectionless, $p_2=0.$
This implies that
$\tau(p)<\tau(p_1)+\ep<r/2$ for all $\tau\in T(A).$
Impossible.

\end{proof}


\begin{thm}\label{Comparison}
Let $A\in {\cal D}$ be a
a separable simple \CA\,
Suppose that $a, b\in A$ with $0\le a\le b\le 1.$
If $d_\tau(a)< d_{\tau}(b)$ for all $\tau\in \overline{T(A)^w},$
then
$$
a\lesssim b.
$$
\end{thm}

\begin{proof}
Fix a strictly positive element $a_0\in A$ with $0\le a_0\le 1.$
Let $a,\,b\in A_+$ be two non-zero elements such that
\beq\label{scomp-1}
d_\tau(a)<d_\tau(b)\tforal \tau\in \overline{T(A)}^w.
\eneq
For convenience, we assume that $\|a\|{=} \|b\|= 1.$
Let $1/2>\ep>0.$   Put $c=f_{\ep/16}(a).$
If $c$ is Cuntz equivalent to $a,$ then
zero is an isolated point in ${\rm sp}(a).$ So, $a$ is Cuntz equivalent to a projection.
Since $A$ is projectionless (see \ref{Pprojless}),
 zero is not an isolated point in ${\rm sp}(a).$
 There is a nonzero element $c'\in \overline{aAa}_+$ such that
$c'c=cc'=0.$  Therefore
$$r_0:=\inf\{d_\tau(b)-d_{\tau}(c): \tau\in \overline{T(A)}^w\}>0.$$

Put $c_1=f_{\ep/64}(a).$  Then a standard compactness argument (see Lemma 5.4 of \cite{Lnloc}) shows that
{t}here is $1>\dt_1>0$ such that
$$\tau(f_{\dt_1}(b))>\tau(c)\ge d_{\tau}(c_1)\tforal \tau\in \overline{T(A)}^w.$$
Put $b_1=f_{\dt_1}(b).$
Then
\beq\label{scomp-3+2}
&& \,\,\,\,\,\,\,r=\inf\{\tau(b_1)-d_{\tau}(c_1): \tau\in T(A)\}
\ge \inf\{\tau(b_1)-\tau(c): \tau\in T(A)\}>0.
\eneq
Note that $\|b\|=1.$
By choosing smaller $\dt_1,$ we may assume that there exist non-zero elements
 $e, e'\in f_{2\dt_1}(b)Af_{2\dt_1}(b)$ with
$0\le e\le e'\le 1$ and $e'e=ee'=e$ such that
$$\tau(e')<r/8\tforal \tau\in \overline{T(A)}^w.$$
Let $r_1=\inf\{\tau(e): \tau\in T(A)\}.$  Note that, since $A$ is simple and $\overline{T(A)}^w$ is compact, $r_1>0.$
Let $b_2=(1-e')b_1(1-e').$
Thus, there is $0<\dt_2<\dt_1/2<1/2$ such that
$$7r/8<\inf\{\tau(f_{\dt_2}(b_2))-\tau(c_1):\tau\in T(A)\}<r-r_1.$$

Since $f_{\dt_2}(b_2)f_{3/4}(b_2)=f_{3/4}(b_2)$ and since $\overline{f_{3/4}(b_2)Af_{3/4}(b_2)}$ is non-zero,
there is a non-zero $e_1\in A$  with $0\le e_1\le 1$ such that $e_1f_{\dt_2}(b_2)=e_1$ with
$\tau(e_1)<r/18$ for all $\tau\in T(A).$


There are $x_1, x_2,...,x_m\in A$ such that
\beq\label{scomp-5}
\sum_{i=1}^m x_i^*e_1x_i=f_{1/4}(a_0).
\eneq
Let $\mathfrak{f}_{a_0}>0$ be given as part of the definition for
$A\in {\cal D}.$

Let $\sigma=\mathfrak{f}_{a_0}\cdot \min\{\ep^2/2^{17}(m+1), \dt_1/8, r_1/2^7(m+1)\}.$

By \eqref{scomp-3+2}  and \cite{CP}, there are $z_1, z_2,...,z_K\in A$  and $b'\in A_+$ such that
\beq\label{scomp-6-1}
&&\|f_{\ep^2/2^{12}}(c)-\sum_{j=1}^Kz_j^*z_j\|<\sigma/4\andeqn\\\nonumber
&&\|f_{\dt_2}(b_2)-(b'+e_1+\sum_{j=1}^Kz_jz_j^*)\|<\sigma/4.
\eneq
 Since $A\in {\cal D},$
for any $\eta>0,$ there exist \SCA s\,   $A_0, D\subset A,$  $D\in {\cal C}_0'$
with $A_0\perp D,$
such that
\beq\label{scomp-6}
&&\|f_{\ep^2/2^{14}}(c)- (f_{\ep^2/2^{14}}(c_{0,2})+f_{\ep^2/2^{14}}(c_2))\|<\sigma\\\label{scomp-6+1}
&&\|f_{\dt_1/2}(b_2)-(f_{\dt_1/2}( b_{0,3})+f_{\dt_1/2}(b_3))\|<\sigma,\\\label{scomp-6+2}
&&\|f_{\dt_2/4}(b_2)-(f_{\dt_2/4}(b_{0,3})+f_{\dt_2/4}(b_3))\|<\sigma,\\
&&\|e_1-(e_{0,1}+e_{1,1})\|<\sigma\\
&&c \lesssim e,\\
&& \|a_0-(a_{0,0}\oplus a_{1,0})\|<\sigma,\\
&& \tau(f_{1/4}(a_{1,0}))\ge \mathfrak{f}_{a_0}\rforal \tau\in \overline{T(D)}^w,
\eneq
where $c\in A_0$ is a strictly positive element of $A_0,$
and,  (using \eqref{scomp-5} and (\ref{scomp-6-1})),  such that
\beq\label{scomp-9}
&&\|\sum_{i=1}^m (x_i')^*e_{1,1}x_i'-f_{1/4}(a_{1,0})\|<\sigma,\\\label{scomp-9-}
\vspace{-0.8in} &&\|f_{\ep^2/2^{14}}(c_2)-\sum_{j=1}^K (z_j')^*z'_j\|<\sigma \andeqn\\\label{scomp-9+}
&&\|f_{\dt_2}(b_3)-(\sum_{j=1}^K z_j'(z_j')^*+e_{1,1}+b'')\|<\sigma,
\eneq
where $0\le a_{0,0},  b_{0,2}, b_{0,3}, c_{0,0}, e_{0,1} \le 1$ are in $A_0,$
$0\le a_{1,0}, b_{2}, b_{3}, c_2, e_{1,1}\le 1$ are in $D,$
$z_j', z_j' , x_i',  b''\in D,$ ($i=1,2,...,m$ and $j=1,2,...,K$)

Note that, by \eqref{scomp-9}
\beq\label{scomp-n9}
t(e_{1,1})\ge \mathfrak{f}_{a_0}/2(m+1)\tforal  t\in  T(D).
\eneq

Therefore, by (\ref{scomp-9-}), \eqref{scomp-9} and \eqref{scomp-9+},
\vspace{-0.14in} {$$\begin{aligned}d_t(f_{\ep^2/2^{13}}(c_2))
  &\le  t(f_{\ep^2/2^{14}}(c_2))
  \le \sigma+\sum_{j=1}^Kt((z_j')^*z_j') \\
&= \sigma+\sum_{j=1}^Kt(z_j'(z_j')^*)
\le  t(e_{1,1})+\sum_{j=1}^Kt(z_j'(z_j')^*)\\
&\le  t(f_{\dt_1}(b_3))
\le d_t(f_{\dt_1/2}(b_3))
\end{aligned}$$}
for all $t\in T(D).$
It follows from \ref{Pcom4C}
that
$$f_{\ep^2/2^{13}}(c_2)\lesssim f_{\dt_1/2}(b_3).$$
By (\ref{scomp-6+1}) and Lemma 2.2 of  \cite{RorUHF2}
\vspace{-0.1in} $$f_{\dt_1/2}(b_3)\le f_{\dt_1/4}(b_2)\le b_2.$$
It the follows
\vspace{-0.12in} \beq\nonumber
\hspace{0.2in}f_{\ep/2}(c) &\lesssim & f_{\ep^2/2^{11}}(c_{0,2}+c_2)\lesssim f_{\ep^2/2^{11}}(c_2)\oplus c\\\nonumber
&\lesssim & b_2 +e \lesssim b_1\lesssim b.
\eneq
We also have
\beq\nonumber
f_{\ep}(a)\lesssim f_{\ep/2}(f_{\ep/16}(a))=f_{\ep/2}(c)\lesssim b.
\eneq
Since this holds for all $1>\ep>0,$ by 2.4 of  \cite{RorUHF2}, we conclude that
\vspace{-0.1in} $$
a\lesssim b.
$$

\end{proof}



\section{Tracially approximate divisibility}


%

\begin{df}\label{Dappdiv}
Let $A$ be a non-unital and $\sigma$-unital simple \CA.
We say that $A$ has (non-unital)  tracially approximately divisible property
if the following holds:

For any $\ep>0,$ any finite subset ${\cal F}\subset A,$ any
$b\in A_+\setminus \{0\},$ and any
integer $n\ge 1,$ there are $\sigma$-unital \SCA s
$A_0, A_1$ of $A$ such that
$$
{\rm dist}(x, B_d)<\ep \rforal x\in {\cal F},
$$
where $B_d\subset B\subset A,$
\vspace{-0.1in} \beq\label{Dappdiv-1}
B=A_0\bigoplus  M_n(A_1),\\
B_d=\{(x_0, \overbrace{x_1,x_1,...,x_1}^n): x_0\in A_0, x_1\in A_1\}
\eneq
and $a_0\lesssim  b,$ where $a_0$ is a strictly positive element of $A_0.$
\end{df}

\begin{lem}\label{Lmatrixpullb}
Let $D$ be a non-unital separable simple \CA\, which  can be written as 
$D=\lim_{n\to\infty}(D_n, \phi_n),$ where each $D_n\in {\cal C}_0^{0'}.$
Let $K\ge 1$ be an integer, 
let $\ep>0$ and let ${\cal F} \subset D_n$  for some $n\ge 1,$ there exists 
an integer $m\ge n$ and \SCA\, $D_m'=M_K(D_m'')\subset D_m$ and a finite subset ${\cal F}_1\subset D_m''$ satisfy the following:
\beq\label{Lmatripb-1}
(\phi_{n, m}(f), d({\cal F}_1))<\ep\rforal f\in {\cal F},
\eneq
where $d(x)=\diag(\overbrace{x,x,...,x}^K).$
If each $D_n\in {\cal C}_0',$ then
there exists an integer $m\ge n$ and $D_m'=M_K(D_m'')\subset D_m$ and a finite subset ${\cal F}_1\subset D_m''$ 
satisfy the following:
\beq\label{Lnz1702-1}
\|\phi_{n,m}(f)-(r(f)+d({\cal F}_1)\|<\ep\rforal f\in {\cal F},
\eneq
where $r(f)\in \overline{eD_me}$ for all $f\in {\cal F},$ $e\in (D_m)_+,$  $e\lesssim e_d,$ 
where $e_d$ is a strictly positive element of $D_m.$

\end{lem}

\begin{proof}
One note that $K_0(D)=K_1(D)=\{0\}.$ 
One also note that $D\otimes Q$ is an inductive limit of \CA s in ${\cal C}_0^{0'}.$ 
It follows from \cite{Raz} (see also \cite{Tsang} and \cite{Rl}) that $D\cong D\otimes Q.$
Therefore, there exists  a \SCA\, $C$ 
such that
\vspace{-0.12in}\beq\label{Lmatripb-2}
{\rm dist}(a, \af(a))<\ep/4\rforal a\in \phi_{n, \infty}({\cal F}),
\eneq
where 
\vspace{-0.12in} $$
\af(a)=\diag(\overbrace{c(a),c(a),...,c(a)}^K)\subset C\otimes M_K
$$
for all $a\in \phi_{n, \infty}({\cal F})$ and for some $c(a)\in C.$

To simplify notation, \wilog, we may assume that there is $c_0\in C_+$ with 
$\|c_0\|=1$ such that 
$c_0c(a)=c(a)c_0=c(a)$ for all $a\in \phi_{n, \infty}({\cal F}).$ 
Consider \SCA\, $C_0\otimes M_K,$ where $C_0$ is the \SCA\, of $C$ generated by 
$c_0.$ Then $C_0\otimes M_K$ 
is a quotient of $C_0((0,1])\otimes M_K$ 
(or $C([0,1])\otimes M_K$).  Let $B=C_0((0,1])\otimes M_K$  (or $C([0,1])\otimes M_K$)
as just mentioned and let $q: B\to C_0\otimes M_K$ be  the standard quotient map. 
Since $B$ is semprojective, there is a \hm\, 
$H: B\to D_{m_1}$ for some $m_1\ge n$ such that
$\phi_{m_1,\infty}\circ H=q.$  Let $c_{00}$ be the pre-image of $c_0$ under $q.$
The lemma follows if we choose a sufficiently large $m\ge m_1$ and 
$D_{m}''=\overline{\phi_{m_1, m}(c_{00})D_{m}\phi_{m_1, m}(c_{00})}$
as well as $D_m'=D_m''\otimes M_K.$

In the case that $D_n\in {\cal C}_0',$ by \cite{aTz}, 
$D\otimes {\cal Z}\cong D.$ 
In ${\cal Z}$ (see the proof of Lemma 2.1 of \cite{Rlz}, and also
Lemma  4.2 of \cite{Rrzstable}),  there are $e_1, e_2,...,e_K, b\in {\cal Z}_+$ 
such that $\sum_{j=1}^K e_j+d=1_{\cal Z},$ $e_1, e_2,...,e_K$ are mutually orthogonal,
there exist $w_1, w_2,...,w_K\in {\cal Z}$ such that $e_j=w_jw_j^*$ and 
$e_{j+1}=w_j^*w_j$ and $d\lesssim e_1.$  Moreover, from the proof of Lemma 4.2 
of \cite{Rrzstable}, there is a unitary $v\in {\cal Z}$ such that
$v^*dv\le e_1.$ 
\Wlog,  by identifying $D$ with $D\otimes {\cal Z},$ we may assume 
that  $\phi_{n, \infty}(x)=y\otimes 1$ for some $y\in D$ and for every element $x\in {\cal F}.$
Let $d'=c_0\otimes d, $ $v'=c_0^{1/2}\otimes v,$ $e_j'=c_0\otimes e_j, $ $w_j'=c_0^{1/2}\otimes w_j,$
$j=1,2,...,K.$ Note that $d'+\sum_{j=1}^K e_j'=c_0.$
With sufficiently large $m,$ with a standard perturbation, 
we may assume $,d', v', e_j', w_j'\in \phi_{m, \infty}(D_m),$ $j=1,2,...,m,$ and 
$\phi_{n,\infty}(x)$ commutes with $d', e_j'$ and $w_j'$ for all $x\in {\cal F}.$
With possibly even larger $m,$  \wilog, there are $d'', v'' e_j'', w_j''\in  D_m$
such that 
$d''+\sum_{j=1}^K=c_0',$ $d''=(v'')(v'')^*,$ $(v'')(v'')^*\le e_1'',$ 
$e_1'', e_2'',..., e_K''$ are mutually orthogonal, $(w_j'')(w_j'')^*=e_j'$ and
$e_{j+1}''=(w_j'')^*(w_j''),$ where $c_0'\in (D_m)_+$ such 
that $c_0'\phi_{n,m}(x)=\phi_{n,m}(x)c_0'=\phi_{n,m}(x)$ for all $x\in {\cal F}$ 
and 
\beq
\|[\phi_{n,m}(x), y]\|<\ep/16K^2\rforal x\in {\cal F}
\eneq
and $y\in \{d''^{1/2}, d'',  v'', e_j'', w_j'', j=1,2,...,K\}.$
Define  $D_m''=\overline{e_j''D_me_j''},$ $r(f)=(d'')^{1/2}\phi_{n,m}d''$ and 
${\cal F}_1=\{e_1''\phi_{n,m}(x)e_1'': x\in {\cal F}\}$ and identify
$d(e_1''\phi_{n,m}(x)e_1'')$ with 
$$
\sum_{j=1}^Ke_j''\phi_{n,m}(x)e_j''\in M_K(D_m'')\,\,\,\,\rforal x\in {\cal F}.
$$
Lemma follows.
\end{proof}

\begin{thm}\label{TDapprdiv}
Let $A$  be a separable simple \CA\, in ${\cal D}_0.$
Then $A$ has tracially approximately divisible property in the sense of  \ref{Dappdiv}.
\end{thm}

\begin{proof}
Let $a_0\in A$ be a strictly positive element of $A$ with $\|a_0\|=1.$
Let $1>\mathfrak{f}_{a_0}>0$ be in the definition of \ref{Dtr1div}.
Fix an integer $k_0\ge 1$
such that $\mathfrak{f}_{a_0}>2^{-k_0}.$

Upon replacing $a_0$ by $g(a_0)$ for some $g\in C_0((0,1])$ with
$0\le g\le 1,$ we may assume that
\beq\label{TDIV-n1}
\tau(a_0)>\mathfrak{f}_{a_0}\rforal \tau\in T(A)
\eneq
(see \ref{Rfa0}).
Fix any $b_0\in A_+\setminus \{0\}.$
Choose a sequence of positive elements
$\{b_n\}$ which has the following property:
$b_{n+1}\lesssim b_{n,1},$
where $b_{n,1}, b_{n,2},....,b_{n,2^{n+k_0+5}}$ are mutually orthogonal positive elements
in $\overline{b_nAb_n}$ such that $b_nb_{n,i}=b_{n,i}b_n=b_{n,i},$
$i=0,1,2,...,n,$ and $\la b_{n,i}\ra =\la b_{n,1}\ra,$ $i=1,2,...,2^{n+k_0+3}.$

It should be noted that
\beq\label{TDadiv-0}
\sum_{k=m}^{\infty} \sup\{\tau(b_n): \tau\in \overline{T(A)}^w\}<\mathfrak{f}_{a_0}/2^{m+5}\rforal m\ge 1.
\eneq

One obtains two sequences of \SCA s $A_{0,n},$ $D_n$ of $A,$ two sequences
of \cpc s $\phi_n^{(0)}: A\to A_{0,n}$  with $\|\phi_n^{(1)}\|=1$ ($i=0,1$) and
$\phi_n^{(1)}: A\to D_n\in {\cal C}_0^{0'}$ satisfy the following:
\beq\label{TDappdiv-1}
&&\lim_{n\to\infty}\|\phi_n^{(i)}(ab)-\phi_n^{(i)}(a)\phi_n^{(i)}(b)\|=0\rforal a,\, b\in A, \,\,i=0,1,\\\label{TDappdiv-1+}
&&\lim_{n\to\infty}\|a-\diag(\phi_n^{(0)}(a), \phi_n^{(1)}(a))\|=0\rforal a\in A,\\\label{TDappdiv-1+2}
&&c_n\lesssim b_n,\\\label{TDappdiv-2}
&&\tau(f_{1/4}(\psi_n^{(1)}(a_0)))\ge \mathfrak{f}_{a_0}\rforal \tau\in T(D_n)
\eneq
 and $\phi_n^{(1)}(a_0)$ is a strictly positive element in $D_n,$
 where $c_n$ is a strictly positive element of $A_{0,n}.$
As in the proof of  \ref{PD0qc},
\beq\label{TDappdiv-6}
\lim_{n\to\infty}\sup\{|\tau(a)-\tau\circ \phi_n^{(1)}(a)|: \tau\in T(A)\}=0\rforal a\in A.
\eneq

Let $a_1=\phi_1^{(1)}(a_0),$ $a_2=\phi_2^{(1)}(a_1),..., a_n=\phi_n^{(1)}(a_{n-1}),$
$n=0,1,....$

\Wlog, by passing to a subsequence, if necessary,  by \eqref{TDappdiv-1+},
we may assume that, for all $m > n,$
\beq\label{TDapdiv-7-}
&&\|a_n-\diag(\phi_m^{(0)}(a_n), \phi_m^{(1)}(a_n))\|<{\mathfrak{f}_{a_0}\over{2^{(n+4)^2}}}\\\label{TDapdiv-7}
&&\|f_{1/4}(a_n)-\diag(f_{1/4}(\phi_m^{(0)}(a_n)), f_{1/4}(\phi_m^{(1)}(a_n)))\|<{\mathfrak{f}_{a_0}\over{2^{(n+4)^2}}},\,\,\, n=1,2,....
\eneq

Claim (1):  For any $n\ge 1,$
\beq\label{TDapdiv-7+1}
\liminf_n \{\tau(f_{1/4}(\phi_m^{(1)}(a_n))): \tau\in T(D_m)\andeqn m>n\}\ge {\mathfrak{f}_{a_0}\over{8}}.
\eneq

Claim (2):  If we first take a subsequence $\{N(k)\}$ and define
$a_1:=\phi^1_{(N(1))}(a_0),$ $a_2=\phi_{N(2)}^{(1)}(a_1),..., a_n=\phi_{N(n)}^{(1)}(a_{n-1}),$
$n=0,1,...,$  then the Claim (1) still holds, when
$m$ is replaced by $N(m).$

Let us first explain that Claim (2) follows from Claim (1) since we first take a subsequence
in the above construction and then apply Claim (1).

We now prove the Claim (1).

Assume  Claim (1) is false.

Then
there exists $\eta>0$ such that $ {\mathfrak{f}_{a_0}\over{8}}-\eta>0$ and
\beq\label{TDapdiv-7+2}
\lim_n\inf \{\tau(\phi_m^{(1)}(a_n)): \tau\in T(D_m)\andeqn m>n\}\le {\mathfrak{f}_{a_0}\over{8}}-\eta.
\eneq
There is $n_0\ge 1$ such that, for all $m>n\ge n_0$ and $k\ge 1,$
\beq\label{TDappdiv-7+3}
\tau(f_{1/4}(\phi_{m+k}^{(1)}(a_m)))\le \tau(f_{1/4}(\phi_{m+k}^{(1)}(a_n)))+\eta/2\rforal \tau\in T(D_{m+k}).
\eneq

Hence there exists a subsequence  $\{n_k\}$
 which has the following property:

If $k'\ge k,$
\beq\label{TDapdiv-8}
t_{n_{k'}}(f_{1/4}(\phi_{n_{k'}}^{(1)}(a_{n_k})))\le \mathfrak{f}_{a_0}/8-\eta/2.
\eneq

Consider the states $\tau_k$ defined by
$t_k(a)=t_k(\phi_{n_k}^{(1)}(a))$
for all $a\in A,$ $k=1,2,....$

Let $\tau$ be a weak*-limit of $\{t_k\}.$ It follows \eqref{TDappdiv-2} that
$\tau$ is not zero.  On the other hand, by \eqref{TDapdiv-8},
\beq\label{TDapdiv-9}
\tau(f_{1/4}(a_{n_k}))<\mathfrak{f}_{a_0}/8\rforal k.
\eneq

It follows from \eqref{TDadiv-0} and \eqref{TDapdiv-7} that, if $m>n\ge 1,$
\beq\label{TDapdiv-10}
t(f_{1/4}(\phi_m^{(1)}(a_n)))\ge \tau(f_{1/4}(a_n))-\mathfrak{f}_{a_0}/2^{m+5}-\mathfrak{f}_{a_0}/2^{(n+4)^2}
\eneq
for all $t\in T(A).$

Therefore, also using \eqref{TDIV-n1},  for all $k,$
\beq\label{TDapdiv-11}
t(f_{1/4}(a_{n_k}))\ge t(f_{1/4}(a_0))-(\sum_{j=1}^{n+k}(\mathfrak{f}_{a_0}/2^{j+5}-\mathfrak{f}_{a_0}/2^{(j+1)^2}))>\mathfrak{f}_{a_0}/4.
\eneq
This contradicts  with \eqref{TDapdiv-9} which proves the Claim (1).

Now define
$\psi_n': D_n\to D_{n+1}$ by
$\psi_n'(d)=\phi_{n+1}^{(1)}(d)$ for all $d\in D_n,$ $n=1,2,....$

Since $A$ is simple, there is, for each $n,$ a map
$T_n=(N_n, M_n): A_+\setminus \{0\}\to \N\times \R_+\setminus \{0\}$ such that, for any $a\in A_+\setminus\{0\},$
there are $x(a)_{1,n}, x(a)_{2,n},...,x(a)_{m_n(a),n}\in A$ with
$m_n(a)\le N_n(a)$ and $\|x(a)_j\|\le M_n(a)$ such that
\beq\label{TDapdiv-12}
\sum_{j=1}^{m_n(a)}x(a)_{j,n}^*ax(a)_{j,n}=f_{1/16}(\phi_{n}^{(1)}(a_1)), n=1,2,....
\eneq
Here we view $D_n\subset A.$

Now fix a finite subset ${\cal F}\subset A_+$ and $1/2>\ep>0.$

Let $\{{\cal F}_{k,n}'\}$ be an increasing sequence of finite subsets of $(D_k)_+$ such that
the union of these subsets is dense in $(D_k)_+.$
We assume that $\phi_k^{(1)}({\cal F})\subset {\cal F}_{k,1}'.$

Set
\vspace{-0.1in}\beq\label{Ddvi-n12}
{\cal F}_{1,n}''=\{(a-\|a\|/2)_+:a\in {\cal F}_{1,n}'\}\andeqn {\cal F}_{1,n}'''={\cal F}_{1,n}'\cup{\cal F}_{1,n}''
\eneq
Let
${\cal F}_{1,n}$ be a finite subset which also contains
$$
{\cal F}_{1,n}'''
\cup\{a_i,
f_{1/16}(a_i),  f_{1/4}(a_i), \,i=0,1\}.
$$
Since $D_1$ is semiprojective,
there exists a \hm\, $\psi_1: D_1\to D_{n_2}$ such that
\beq\label{TDapdiv-13}
&&\|\psi_1(a)-\psi_{n_2}'(a)\|<\min\{\mathfrak{f}_{a_0}/16, \ep/8\}\andeqn\\
&&\hspace{-0.3in}\|\sum_{i=1}^{m_n(a)}\phi_{n_2}^{(1)}(x(a)_{i,1})^*\psi_1(a)\phi_{n_2}^{(1)}(x(a)_{i,1})-f_{1/16}(\phi_{n_2}^{(1)}(a_1))\|<\min\{\mathfrak{f}_{a_0}/16, \ep/8\}
\eneq
for all  $a\in {\cal F}_{1,1}'''.$

Therefore, by applying \ref{Lrorm},  there are $y(a)_{i,n_2}\in D_{n_2}$ with $\|y(a)_{i,n_2}\|\le \|x(a)_{i,1}\|+\mathfrak{f}_{a_0}/16$ such that
\beq\label{TDapdiv-14}
\sum_{i=1}^{m_n(a)}y(a)_{i, n_2}^*\psi_1(a)y(a)_{i,n_2}=f_{1/8}(\phi_{n_2}^{(1)}(a_1))\rforal a\in {\cal  F}_{1,1}'''.
\eneq
To simplify notation, by passing to a subsequence, if necessary,   \wilog,
we may assume that $n_2=2.$


Set
\vspace{-0.1in}\beq\label{Ddvi-n12+}
{\cal F}_{2,n}''=\{(a-\|a\|/2)_+:a\in {\cal F}_{2,n}'\}\andeqn {\cal F}_{2,n}'''={\cal F}_{2,n}'\cup {\cal F}_{2,n}'''
\eneq

Let
${\cal F}_{2,n}$ be a finite subset which also contains
$$
{\cal F}_{2,n}''
\cup\{a_i,
f_{1/16}(a_i),  f_{1/4}(a_i),\, i=0,1,\, f_{1/16}(\phi_2^{(1)}(a_1))\}\cup \psi_1({\cal F}_{1,1}).
$$
Since $D_2$ is semiprojective,
there exists a \hm\, $\psi_2: D_2\to D_{n_3}$ such that
\beq\label{TDapdiv-15}
\|\psi_2(a)-\psi_{n_3}'(a)\|<\min\{\mathfrak{f}_{a_0}/2^{3+2}, \ep/2^{2+2}\}\rforal a\in {\cal F}_{2,2}\andeqn\\
\|\sum_{i=1}^{m_n(a)}\phi_{n_3}^{(1)}(x(a)_{i,2})^*\psi_2(a)\phi_{n_3}^{(1)}(x(a)_{i,2})-f_{1/16}(\phi_{n_3}^{(1)}(a_2))\|<\min\{\mathfrak{f}_{a_0}/2^{3+2}, \ep/2^{3+1}\}
\eneq
for all  $a\in {\cal F}_{2,2}'''.$

Therefore, by applying \ref{Lrorm},  there are $y(a)_{i,n_3}\in D_{n_3}$ with $\|y(a)_{i,n_3}\|\le \|x(a)_{i,2}\|+\mathfrak{f}_{a_0}/2^{3+2}$ such that
\beq\label{TDapdiv-16}
\sum_{i=1}^{m_n(a)}y(a)_{i,n_3}^*\psi_n(a)y(a)_{i,n_3}=f_{1/8}(\phi_{n_3}^{(1)}(a_2))\rforal a\in {\cal  F}_{2,2}'''
\eneq
To simplify notation, by passing to a subsequence, if necessary,   \wilog,
we may assume that $n_3=3.$

Continue this process,
one then obtains a sequence
of \hm\, $\psi_n: D_n\to D_{n+1}$ such that
\beq\label{TDapdiv-17}
\|\psi_n(a)-\psi_{n}'(a)\|<\min\{\mathfrak{f}_{a_0}/2^{3+n}, \ep/2^{2+n}\}\rforal a\in {\cal F}_{n,n}\andeqn\\
\|\sum_{i=1}^{m_n(a)}\phi_{n}^{(1)}(x(a)_{i,n})^*\psi_n(a)\phi_{n}^{(1)}(x(a)_{i,n})-f_{1/16}(a_{n})\|<\min\{\mathfrak{f}_{a_0}/2^{3+n}, \ep/2^{3+n}\}
\eneq
for all  $a\in {\cal F}_{n,n}'''.$
Moreover, there  are $y(a)_{i,n}\in D_{n}$ with $\|y(a)_{i,n}\|\le M_n(a)+\mathfrak{f}_{a_0}/2^{3+n}$ such that
\beq\label{TDapdiv-18}
\sum_{i=1}^{m_n(a)}y(a)_{i,n}^*\psi_n(a)y(a)_{i,n}=f_{1/8}(a_{n})\rforal a\in {\cal  F}_{n,n}'''.
\eneq

Let $D=\lim_{n\to\infty} (D_n, \psi_n).$ (Again, one should note that we have taken a subsequence to
simplify notation.)

We now verify $D$ is simple.  Fix a non-zero positive element $d_0\in D_+$ with $\|d_0\|=1.$
Since each $D_n$ is stably projectionless, so is $D.$
Fix $1/64>\ep_1>0.$ There is $d\in D_+$ such that
$d=\psi_{m,\infty}(d')$ for $d'\in (D_m)_+$ with $\|d'\|=1$
and
\beq\label{TDapdiv-19}
\|d-d_0\|<\ep_1/32.
\eneq
It follows from \ref{Lrorm} that there is $z\in D$ such
that
\beq\label{TDapdiv-20}
(d-\ep_1/16)_+=z^*d_0z.
\eneq

By the construction, there is $d''\in {\cal F}_{m',m'}'$ for some
$m'\ge m+16$ such that
\beq\label{TDapdiv-21}
\|\psi_{m,m'}((d'-\ep_1/16)_+)-d''\|<\ep_1/64
\eneq
There is $y\in D_{m'}$ such
that
\beq\label{TDapdiv-22}
(d''-\ep_1/8)_+=y^*\psi_{m,m'}((d'-\ep_1/4)_+)y.
\eneq
Note that $\ep_1/2\le \|d''\|/8.$
By the construction, there $x_1, x_2,...,x_L\in D_{m'+1}$
such that
\beq\label{TDapdiv-23}
\sum_{i=1}^L x_i^*\psi_{m', m'+1}((d''-\ep_1/2)_+)x_i=f_{1/8}(\phi_{m'+1}(a_{m'}))=f_{1/8}(a_{m'+1}).
\eneq

Claim (3): $a_{00}:=\psi_{m'+1, \infty}(f_{1/4}(a_{m'+1})$ is full.

In fact, for any $m'>m'+1,$  it follows from \eqref{TDapdiv-17} and Claim (2)  that
\beq\label{TDapdiv-24}
\tau(\psi_{m'+1, m''}(f_{1/4}(a_{m'+1})))&\ge& \tau(\phi_{m''}^{(1)}(f_{1/4}(a_{m'+1})-\mathfrak{f}_{a_0}/2^{m'+3}\\
&=&\tau(f_{1/4}(\phi_{m''}^{(1)}(a_{m'+1}))-\mathfrak{f}_{a_0}/2^{m'+3}> \mathfrak{f}_{a_0}/16
\eneq
for all $\tau\in T(D_{m''}).$

By \ref{Pcom4C}, we conclude that
$\psi_{m'+1, m''}(f_{1/4}(a_{m'+1}))$ is full in $D_{m''}.$
Therefore $a_{00}$ is full in $\psi_{m'', \infty}(D_{m''})$ for all
$m''>m'+1.$
Hence the closed ideal generated by $a_{00}$ contains all \\
$\cup_{m''>m'+1}\psi_{m'', \infty}(D_{m''}).$ This implies that $a_{00}$ is full in $D.$
This proves Claim (3).

It follows from \eqref{TDapdiv-23} that $\psi_{m', \infty}((d''-\ep_1)_+)$ is full
in $D.$  By \eqref{TDapdiv-22}, $\psi_{m, \infty}((d'-\ep_1/4)_+)$ is full in $D.$
This, in turn, by \eqref{TDapdiv-20}, $b_0$ is full in $D.$
Since $b_0$ is arbitrarily chosen in $D_+\setminus \{0\}$ with $\|b_0\|=1,$
this implies that $D$ is indeed simple.

On the other hand $D$ is non-unital and  an inductive limit of \CA s in ${\cal C}_0^{0'}.$
This lemma then follows a direct application of \ref{Lmatrixpullb}.

In fact,  for any fixed integer $K\ge 1,$ by \ref{Lmatrixpullb}, 
there exists $m\ge 1$ such that
\beq\label{TDapdiv-26}
 {\rm dist}(\psi_m(f), d({\cal F}_1))<\ep/8\rforal f\in {\cal F},
 \eneq
 where $d(x)=\diag(\overbrace{x,x,...,x})^K)$ for all $x\in {\cal F}_1$
 and where ${\cal F}_1\subset D_m'',$  $D_m'=M_K(D_m'')\subset D_m.$
 By choosing possibly even larger $m,$ by  \eqref{TDappdiv-1+},
 we may also assume that
 \beq\label{TDadpdiv-27}
 \|a-\diag(\phi_m^{(0)}(a), \phi_m^{(1)}(a))\|<\ep/4\rforal a\in {\cal F}.
 \eneq
 It follows from \eqref{TDapdiv-17} that
 \beq\label{TDapdiv-28}
 {\rm dist}(\phi_m^{(1)}(f), d({\cal F}_1))<\ep/4\rforal f\in {\cal F}.
 \eneq
Let  
$$
B_{1,d}=\{(\overbrace{x,x,...,x}^K): x\in D_m''\}\subset M_K(D_m'')\andeqn B_d=A_{0,m}\oplus B_{1,d}.
$$
Then, 
$$
{\rm dist}(a, B_d)<\ep\rforal a\in {\cal F}.
$$
Note also, by \eqref{TDappdiv-1+2},
$$
c_m\lesssim b_m\lesssim b_0
$$
(recall that $c_m$ is a strictly positive element for $A_{0,m}$). The lemma follows.
\end{proof}

\begin{cor}\label{LappZ}
Let $A$ be a simple  \CA\, in ${\cal D}.$ Then $A$ has the following property:
For any $\ep>0,$ any finite subset ${\cal F}\subset A,$ any $a_0\in A_+\setminus \{0\}$
and any integer $n\ge 1,$ 
there are  $e_0, e_{00}, e_{01}\in A_+,$ \cpc s $\phi_0: A\to E_0,$ $\phi_1: A\to E_1$
and $\phi_2: A\to D_n(E_3),$  where $E_0, E_1, E_2$ are \SCA s of $A,$ 
$E_0=\overline{e_0Ae_0},$ $e_{00}\in E_1,$ $e_{01}\in E_2,$ 
$E_0\perp E_1,$ $M_n(E_2)\subset E_1,$ $E_2\subset \overline{e_{01}Ae_{01}}$  such that 
\beq\label{DappZ-1}
&&\|x-\diag(\phi_0(x), \phi_1(x))\|<\ep/2\andeqn\\
&&\|\phi_1(x)-(r(x)+\diag(\overbrace{\phi_2(x), \phi_2(x), ...,\phi_2(x)}^n))\|<\ep/2,\\
&&r(x)\in \overline{e_{00}Ae_{00}}\rforal x\in {\cal F},
\eneq
and $e_0+e_{00}\lesssim a_0$ and $e_{00}\lesssim e_{01}.$
\end{cor}
\begin{proof}
In the proof of \ref{TDapprdiv}, we replace ${\cal C}_0^{0'}$ by ${\cal C}_0$ and keep entire proof 
to the line ends ``... $D$ is indeed simple" shortly before \eqref{TDapdiv-26}.
Instead of applying the first part of \ref{Lmatrixpullb} in the last few lines 
of the proof, we apply the second part of \ref{Lmatrixpullb}.
\end{proof}

\begin{thm}\label{TCCdvi}
Let $A$ be a non-unital  and $\sigma$-unital simple \CA\, in ${\cal D}_{0}$
with a strictly positive element $a\in A_+$ with $\|a\|=1.$
Then the  following is true.

  There exists
$1> \mathfrak{f}_a>0,$ for any $\ep>0,$  any
finite subset ${\cal F}\subset A$ and any $b\in A_+\setminus \{0\}$ and any integer
$n\ge 1,$  there are ${\cal F}$-$\ep$-multiplicative \cpc s $\phi: A\to A$ and  $\psi: A\to D$  for some
\SCA\, $D\subset A$  with $D\in {\cal C}_0^{0'}$ such that $\|\psi\|=1,$ 
\beq\label{CCdiv-1}
&&\|x-\diag(\phi(x), \overbrace{\psi(x), \psi(x),...,\psi(x)}^n)\|<\ep\rforal x\in {\cal F}\cup \{a\},
\\\label{CCdiv-3}
&&\phi(a)\lesssim b,
\\\label{CCdiv-4+}
&&t(f_{1/4}(\psi(a)))\ge \mathfrak{f}_a\rforal t\in T(D)\andeqn
\eneq
$\psi(a)$ is strictly positive in $D.$

Let $A_0=\overline{\phi(a)A\phi(a)}.$ Then
$A_0\perp D.$ Moreover, \eqref{CCdiv-3} could be replaced by
$c\lesssim b$ for some strictly positive element $c$ of $A_0.$
\end{thm}

Note that there is a \SCA\, $M_n(D)$ in $A,$ where $D\oplus D\oplus \cdots \oplus D$ sitting
in the diagonal.

\begin{proof}
By \ref{TDapprdiv}, $A$ has the property of tracially approximate divisibility.
Fix a strictly positive element $a\in A_+$ with $\|a\|=1.$
It follows from \ref{PD0qc} that $0\not\in \overline{T(A)}^w.$
Let
\beq\label{CCdvi-10}
r_0=\inf\{\tau(f_{1/2}(a)): \tau\in \overline{T(A)}^w\}>0.
\eneq

Let $\mathfrak{f}_{a_0}=r_0/4.$ Choose an integer $k_0\ge 1$ such that
$r_0/16>1/k_0.$

Let $1>\ep>0$ and ${\cal F}\subset A$ be a finite subset.
Choose $\ep_1=\min\{\ep/16, r_0/128\}.$ Let ${\cal F}_1\supset {\cal F}\cup \{a, f_{1/4}(a)\}$
be a finite subset of $A.$
Let
$b\in A_+\setminus \{0\},$ and any
integer $n\ge 1$  be given.

Choose $b_0',b_1',...,b_{n+2k_0}'\in \overline{bAb}$ such
that $b_0'$  and $b_1'$ are mutually orthogonal and mutually equivalent in the sense of Cuntz
and there are non-zero elements $b_0, b_1,...,b_{n+2k_0}\in A_+$ such
that $b_ib_0'=b_i,$ $i=0,1,...,n+2k_0$

By \ref{TDapprdiv}, $A$ has the property of tracially approximate divisibility.
Therefore
there are $\sigma$-unital \SCA s
$A_0, A_1$ of $A$ such that
\vspace{-0.1in}$$
{\rm dist}(x, B_d)<\ep_1/2 \rforal x\in {\cal F}_1,
$$
where $B_d\subset B\subset A,$
\vspace{-0.12in}\beq\label{CCdvi-11}
&&B=A_0\bigoplus  M_n(A_1),\\
&&B_d=\{(x_0, \overbrace{x_1,x_1,...,x_1}^n): x_0\in A_0, x_1\in A_1\}
\eneq
and $a_0\lesssim  b_0,$ where $a_0$ is a strictly positive element of $A_0.$
Moreover, there are $y_0\in A_0$ and $y_1\in A_1$ such that
\vspace{-0.12in}\beq\label{CCdvi-12}
&&\|a-\diag(y_0, \overbrace{y_1,y_1,...,y_1}^n)\|<\ep/2\andeqn\\
&&\|f_{1/4}(a)-\diag(f_{1/4}(y_0), \overbrace{f_{1/4}(y_1),f_{1/4}(y_1),...,f_{1/4}(y_1)}^n)\|<\ep_1/2.
\eneq

Note that
\vspace{-0.1in}\beq\label{CCdvi-12+}
\tau(\diag(0, \overbrace{f_{1/4}(y_1),f_{1/4}(y_1),...,f_{1/4}(y_1)}^n))\ge r_0-1/(n+2k_0)-\ep_1/2>r_0/3
\eneq
for all $\tau\in T(A).$

Let $A_0'=\overline{a_0Aa_0}$ and $A_1'=\overline{y_{1}Ay_{1}}.$
Note that $0\not\in \overline{T(A_1')}^w$ as \ref{Pheretc}.
Moreover, if $\tau\in T(A),$ then $\|\tau|_{A_1}\|\ge r_0/3.$
We also have
\beq\label{CCdvi-12+1}
\tau(f_{1/4}(y_1))\ge r_0/3\rforal \tau\in T(A_1').
\eneq

Note, by \ref{Rfa0}, in the definition of \ref{DNtr1div}
the constant $\mathfrak{f}_{y_1}$ can be chosen
$r_0/6.$

Let ${\cal G}\subset A_1$ be a finite subset such that the following holds
\beq\label{CCdvi-13}
{\rm dist}(f, \diag(x_0, \overbrace{x_1,x_1,...,x_1}^n): x_0\in A_0, x_1\in {\cal G}\})<\ep_1/2\rforal f\in {\cal F}_1
\eneq
and $y_1\in {\cal G}.$

Note that $A_1'$ is a hereditary \SCA\, of $A.$
Therefore, there exist  two \SCA s  $B_0$ and $D$ of $A_1',$ where
$D\in {\cal C}_0^{0'}$ 
and two
${\cal G}$-$\ep_1$-multiplicative \cpc s
$\phi_0: A_1'\to B_0$ and $\psi_0: A_1'\to D$ such that
\beq\label{CCdvi-13+}
&&\|x-\diag(\phi_0(x), \psi_0(x))\|<\ep_1/2\rforal x\in {\cal G},\\
&&\phi_0(c_0)\lesssim b_1,\\
&&\|\psi_0\|=1\andeqn\\
&&\tau\circ f_{1/4}(\psi_0(y_1))\ge r_0/6\rforal \tau\in T(D)
\eneq
and $\psi_0(y_1)$ is a strictly positive element in $D,$ where $c_0$ is a strictly positive element of
$A_1'.$

Let $A_{00}=A_0\oplus \overbrace{A_0', A_0',..., A_0'}^n$
and let
\vspace{-0.12in}$$
c=\diag(a_0, \overbrace{c_0,c_0,...,c_0}^n).
$$
Choose  a function $g\in C_0((0,1]),$ define
$\phi_{00}: A\to A_{00}$ by
$$
\phi_{00}(x)=\diag(g(a_0)xg(a_0), \overbrace{\phi_0(x),\phi_0(x),...,\phi_0(x)}^n)\rforal x\in A.
$$
Then, with a choice of $g,$ we have
\vspace{-0.12in}\beq\label{CCdvi-20}
\|x-\diag(\phi_{00}(x), \overbrace{\psi_0(x),\psi_0(x),...,\psi_0(x)}^n)\|<\ep\rforal x\in {\cal F}.
\eneq
Moreover,
$$
c\lesssim b_0\oplus b_1\oplus \cdots \oplus b_n\lesssim b.
$$

\end{proof}

The following follows from  the combination of \ref{UnifomfullTAD}  and \ref{TCCdvi}.

\begin{cor}\label{Cuniformful}
Let $A$ be a $\sigma$-unital simple \CA\, in ${\cal D}_{0}$ 
Then the following holds.
Fix  a strictly positive element $a\in A$
with $\|a\|=1$ and let $1>\mathfrak{f}_0>0$ be as in \ref{DNtr1div} (see also \ref{Rfa0}). There is a map $T: A_+\setminus \{0\}\to \N\times \R$
($a\mapsto (N(a), M(a))\rforal a\in A_+\setminus \{0\}$) satisfying the following:
For any  finite subset ${\cal F}_0\subset A_+\setminus \{0\}.$
 for any $\ep>0,$  any
finite subset ${\cal F}\subset A$ and any $b\in A_+\setminus \{0\}$ and any integer
$n\ge 1,$  there are ${\cal F}$-$\ep$-multiplicative \cpc s $\phi: A\to A$ and  $\psi: A\to D$  for some
\SCA\, $D\subset A$ such that
\beq\label{CDtad-1}
&&\|x-\diag(\phi(x), \overbrace{\psi(x), \psi(x),...,\psi(x)}^n)\|<\ep\rforal x\in {\cal F}\cup \{a\},\\\label{CDtrdiv-2}
&& D\in {\cal C}_0^{0'}
\\\label{CDtad-3}
&&a_{0}\lesssim b,\\\label{CDtrdiv-4}
&& \|\psi\|=1\andeqn
\eneq
$\psi(a)$ is strictly positive in $D,$  where $a_{0}\in \overline{\phi(a)A\phi(a)}$ is
a strictly positive element.
Moreover,  $\psi$ is $T$-${\cal F}_0\cup\{f_{1/4}(a)\}$-full in $\overline{DAD}.$

Furthermore, we may assume that
\beq
t\circ f_{1/4}(\psi(a))\ge \mathfrak{f}_a\andeqn\\
 t\circ f_{1/4}(c)\ge \mathfrak{f}_a/4\inf\{M(c)\cdot N^2(c): c\in {\cal F}_0\cup\{f_{1/4}(a)\}\}
\eneq
for all $c\in {\cal F}_0$ and for all $t\in T(D).$

\end{cor}

\begin{rem}\label{Rm17div}
It is clear from the proof that, both \ref{TCCdvi} and \ref{Cuniformful} hold 
if $A\in {\cal D}$ with $n=1.$ 
\end{rem}

\section{Stable rank}

The proof of the following is very similar to that of  Lemma 2.1 of \cite{Rlz}.

\begin{lem}\label{NTstr1pre}
Let $A$ be a non-unital and $\sigma$-unital simple \CA\,
with $0\not\in \overline{T{{(}}A)}^w $ which
has strict comparison for positive elements.
Suppose that every hereditary \SCA\, $B$ of $A$
satisfies the conclusion of \ref{LappZ}.
Then, for any hereditary \SCA\, $B$ of $A,$
$$
B\subset {\overline{GL({\tilde B})}}
$$
\end{lem}

\begin{proof}
Since every hereditary \SCA\, $B$ of $A$ has the same said properties,
it suffices to show that $A\subset {\overline{GL({\tilde {{A}}})}}.$

Fix an element $x\in A$
and $\ep>0.$
Let $e\in A$ with $0\le e\le 1$ be a strictly positive element.
Upon replacing $x$ by
$f_\eta(e)xf_\eta(e)$ for some small $1/8>\eta>0,$  we may assume that
$x\in \overline{f_\eta(e)Af_{\eta}(e)}.$ Put
$B_1=\overline{f_\eta(e)Af_{\eta}(e)}.$

{{By the assumption, we know that}} $e$ is not a projection.
we obtain a positive element ${{b_0}}\in B_1^{\perp}\setminus \{0\}.$

Note that
$$
B_{{1}}^{\perp}=\{a\in A: ab=ba=0\rforal b\in B_1\}
$$
is a non-zero  hereditary \SCA\, of $A.$
Since we assume that $A$ is infinite dimensional, $\overline{b_0Ab_0}$
contains  non-zero positive elements $b_{0,1},b_{0,1}', b_{0,2}, b_{0,2}\in B_1^{\perp}$
such that
$$
b_{0,1}'\lesssim b_{0,2}'\andeqn b_{0,1}b'_{0,1}=b_{0,1}, b_{0,2}b_{0,2}'=b_{0,2}\andeqn
b_{0,1}'b_{0,2}'=0.
$$

Choose an integer $n\ge 2,$ since $A$ has strict comparison for positive element as
given in \ref{Comparison},
such that, if there are $n$-mutually orthogonal and mutually equivalent positive elements
$a_1,a_2,...,a_n\in A_+,$
then
$$
a_1+a_2\lesssim b_{0,1},\,\,\,i=1,2,...,n.
$$

There is $B_1'\subset B_1$
such that
$$
B_1'=B_{1,1}+ \overbrace{D\oplus D\oplus \cdots \oplus D}^n,
$$
where $B_{1,1}$ is a hereditary \SCA\, with a strictly positive element $b_{11}\lesssim b_{0,1}$
and
there are $x_0\in B_{1,1}$ and $x_1\in D\setminus \{0\}$ such that
\beq\label{NTstr1-pre-10}
\|x-(x_0+{\bar x}_1)\|<\ep/16\andeqn\\
{\bar x_1}= \diag(\overbrace{x_1,x_1,...,x_1}^n)\|<\ep/16.
\eneq

Let $d_0\in D$ be a strictly positive element.  By the choice of
$n,$  $d_0\lesssim b_{0,1}.$

Choose $0<\eta_1<1/4$ such that
\beq\label{NTstr1-pre-11}
\|f_{\eta_1}(d_0)x_1f_{\eta_1}(d_0)-x_1\|<\ep/16.
\eneq
Put $x_1'=f_{\eta_1}(d_0)x_1f_{\eta_1}(d_0).$
Note
\beq\label{NTstr1-pre-11+}
f_{\eta_1/8}(d_0)\lesssim  b{{'}}_{0,2}.
\eneq
There are $w_i\in A$ such that
\vspace{-0.12in}\beq\label{NTstr1-pre-12}
w_iw_i^*&=&\diag(\overbrace{0,0,...,0}^{i-1},f_{\eta_1/4}(d_0),0,...,0),i=1,2,...,n,\\
 w_i^*w_i&=&\diag((\overbrace{0,0,...,0}^{i},f_{\eta_1/4}(d_0),0,...,0),\,\,\, i=1,2,...,n-1,\andeqn\\
 w_n^*w_n&\in& \overline{b{{'}}_{0,2}Ab{{'}}_{0,2}}.\,\,\,
\eneq

There is $v\in A$
such that
\beq\label{NTstr1-pre-13}
v^*v=x_0+\diag(x_1',0,...,0)\andeqn vv^*\in \overline{(b_{0,1}'+b_{0,2}')A(b_{0,1}'+b_{0,2}')}.
\eneq

Put
\vspace{-0.12in}\beq
&&x_i''=\diag(\overbrace{0,0,...,0}^{i-1},x_1',0,...,0),\,\,\,i=1,2,...,n,\\
&&y_i''=\diag(\overbrace{0,0,...,0}^{i-1},f_{\eta_1/4}(d_0),0,...,0),\,\,\,i=1,2,...,n,\\
&&z_1=v^*,\,\,\, z_2=v, \\
&&z_3=\sum_{i=1}^{n-1}w_i^*x_{i}''\andeqn
z_4=\sum_{i=1}^{n-1}y_i''w_i.
\eneq
Note  that
\vspace{-0.12in}\beq
z_3z_2=0, z_1z_4=0.
\eneq
Therefore
\beq
(z_1+z_3)(z_2+z_4)&=&z_1z_2+z_3z_4\\
&=& v^*v+\diag(0, x_1',x_1',...,x_1')\\
&=&x_0+ \diag(\overbrace{x_1',x_1',...,x_1'}^n).
\eneq
On the other hand,
\vspace{-0.12in}\beq\label{NTstr1-pre-14}
z_1^2=v^*v^*=0,\,\,\, z_1z_3=0.
\eneq
We also compute
that
\vspace{-0.12in}\beq\label{NTstr1-pre-15}
z_3^2=\sum_{i,j}w_i^*x_{i}''w_j^*x_j''=\sum_{i=2}^{n-1}w_i^*x_{i}''w_{i-1}^*x_{i-1}''.
\eneq
Inductively, we compute that
\vspace{-0.12in}\beq\label{NTstr1-pre-16}
z_3^n=0.
\eneq
Thus, by \eqref{NTstr1-pre-14},
\beq\label{NTstr1-pre-17}
(z_1+z_3)^k=\sum_{i=1}^k z_3^iz_1^{k-i}\rforal k.
\eneq
Therefore, by \eqref{NTstr1-pre-14},
and \eqref{NTstr1-pre-16}, for $k=n+1,$ 
\vspace{-0.1in}\beq\label{NTstr1-pre-18}
(z_1+z_3)^{n+1}=0.
\eneq
We also have that  $z_2z_4=0$ and $z_2^2=0.$
A similar computation shows that $z_4^n=0.$
Therefore as above, $(z_2+z_4)^{n+1}=0.$
We estimate
that
$$
\|x-(z_1+z_3)(z_2+z_4)\|<\ep/4.
$$
Suppose that $\|z_i\|\le M$ for $i=1,..,4.$
Consider
$$
z_5=z_1+z_3+\ep/16(M+1)\andeqn z_6=z_2+z_4+\ep/16(M+1).
$$
Since $(z_1+z_3)$ and $(z_2+z_4)$ are nilpotents,
both $z_5$ and $z_6$ are invertible in ${\tilde A}.$
We also estimate
that, by \eqref{NTstr1-pre-10},
$$
\|x-z_5z_6\|<\ep.
$$

\end{proof}

\begin{cor}\label{CD0str1}
Let $A$ be a non-unital simple separable \CA\, which is in ${\cal D}.$
Then $A$ almost has stable rank one.
\end{cor}

\begin{cor}
Every separable simple \CA\, in ${\cal D}$ is quasi-compact.
\end{cor}

\begin{proof}
This follows from  \ref{PD0qc}, \ref{Comparison}, \ref{CD0str1} and \ref{Lbkqc}.
\end{proof}

\begin{lem}\label{Lorthperp}
Let $A$ be a separable  simple  \CA\, and let $a, a_1\in A_+$ be two non-zero elements
such that $aa_1=a_1a=a.$
Then
$$
B_0=\{b\in {\tilde A}: ba=ab=0\}
$$
is a  hereditary \SCA\, of $A.$

Let $b\in B_0$ be a strictly positive element with $0\le b\le 1$ and
$p=\lim_{n\to\infty} b^{1/n}$ be the open projection in $\pi_U(A)'',$
where $\pi_U$ is the universal representation of $A.$
Then, either $B_0$ has a unit,  in which case,  $B_0=B_0\cap A+\C\cdot 1_{B_0},$
or  $B_0\cap A+ \C \cdot p=B_0+\C\cdot p\cong {\tilde B_0}.$
\end{lem}

\begin{proof}
It is clear that $B_0$ is a hereditary \SCA\, of ${\tilde A}.$
Let $B_{00}=A\cap B_0.$  Note that $1-a_1\in B_0.$
Let $\pi: {\tilde A}\to \C$ be the quotient map.
Then $\pi(1-a_1)=1.$  Thus we have the following short exact sequence:
$$
0\to  B_{00}\to B_0\to \C\to 0.
$$
If $B_0$ has a unit, then the above short exact sequence shows that
$B_0=B_{00}+\C\cdot 1_{B_0}.$

Otherwise,
let $b\in B_0$ a strictly positive element of $B_0$ with $0\le b\le 1.$
Let $p=\lim_{n\to\infty}b^{1/n}$ in $\pi_U(A)''.$
Let $B_1=B_0+\C\cdot p.$
Note that $\lim_{n\to\infty}\pi(b^{1/n})=1$ for all $n.$ Therefore $B_1/B_{00}\cong \C.$

\end{proof}

\begin{thm}\label{TTstr1}
Let $A$ be a non-unital separable simple projectionless \CA\, which is in ${\cal D}.$
 Then $A$ has stable rank one.
\end{thm}

\begin{proof}
Let $x\in {\tilde A}.$ We will show that $x\in\overline{GL({\tilde A})}.$
By applying 3.2  and 3.5 of \cite{Rr11}, \wilog, we may assume that
there exists a non-zero positive element $e_0'\in {\tilde A}$ such that
$xe_0'=e_0'x=0.$ We may further assume that there exists a non-zero element
that $e_0\in {\tilde A}$ such that $e_0e_0'=e_0'e_0=e_0.$

Let $A_0=\overline{e_0{\tilde A}e_0}$ and let $\pi: {\tilde A}\to \C$ be
the quotient map. If $\pi(A_0)\not=\{0\},$
then $\pi(x)=0.$ It other words, $x\in A.$
It follows from \ref{CD0str1} that $x\in \overline{GL({\tilde A})}.$
So for the rest of the proof, we assume that $\pi(A_0)=\{0\}.$
In other words, $e_0\in A_+$ and $A_0$ is a hereditary \SCA\, of $A.$
We may assume that there is  a non-zero element
$e_{00}, e_{00}'\in A_0$ with $0\le e_{00}\le e_{00}'\le 1$
such that
$$
e_{00}e_{00}'=e_{00}.
$$

By multiplying a scalar of multiple of identity, \wilog, we may assume that
$x=1+a',$ where $a'\in A.$

Let
$$
B=\{y\in {\tilde A}: ye_0=e_0y=0\}.
$$
Then, by \ref{Lorthperp}, $B$  is a hereditary \SCA\, of ${\tilde A}.$
Moreover, ${\tilde B}\cong B\cap A+\C\cdot p,$
where $p=\lim_{n\to\infty}(b_0)^{1/n}$ in $\pi_U(A)''$ for some strictly positive element $b_0\in B.$
Therefore one may rewrite $x=p+a$ for some $a\in B\cap A.$
Put $B_0=B\cap A.$

Now since $B_0\in {\cal D}, $ there exist   a hereditary \SCA\, $B_{01}\subset B_0$ and a
\SCA\,
$D\subset B_0$ such that
\beq\label{TTstr1-5}
\|a-(x_0+x_1)\|<\ep/4,
\eneq
where $x_0\in B_{0,1}$ and $x_1\in D,$ $B_{0,1}D=DB_{0,1}=\{0\},$
\beq\label{TTstr1-6}
b_{0,1}\lesssim e_{00}
\eneq
where $b_{0,1}$ is a strictly positive element of $B_{0,1}$ and where $D\in {\cal C}_0'.$

We may assume, \wilog,  that there  are $e_{0,1}, e_{0,1}', e_{0,1}'' e_{0,1}'''\in B_{0,1}$ 
with $0\le e_{0,1}\le e_{0,1}'\le e_{0,1}''\le e_{0,1}'''\le 1$
such that
\beq\label{TTstr1-7}
e_{0,1}x_0=x_0e_{0,1}=x_0, \,\,\,e_{0,1}'''e_{0,1}''=e_{0,1}'',\,\,\,e_{0,1}''e_{0,1}'=e_{0,1}'\andeqn e_{0,1}'e_{0,1}=e_{0,1}.
\eneq
Let $A_2=\overline{(e_{0,0}'+e_{0,1}'')A(e_{0,0}'+e_{0,1}'')}.$
Since $b_{0,1}\lesssim e_{00}$ and since $A$ almost has stable rank one,
there is  unitary $u_1'\in {\tilde A_2}$ such that
\beq\label{TTstr1-8}
(u_1')^*e_{0,1}'(u_1')\in \overline{e_{0,0}'Ae_{0,0}'}\subset A_0.
\eneq

Let $A_3$ be the hereditary \SCA\, of $A$ generated by $A_0$ and $A_2.$
Let $q$ be the open projection  in $A^{**}$ corresponding to $A_3,$
and let $q_0$ be the open projection in $A^{**}$ corresponding to $\overline{e_{0,1}Ae_{0,1}}.$
Note that
\beq\label{TTstr1-9}
q_0\le e_{0,1}'\le e_{0,1}''\le q.
\eneq
Note also that
\beq\label{TTstr1-10}
&&\|x-(p+x_0+x_1)\|=\|a-(x_0+x_1)\|<\ep/4\andeqn\\
&&x-(p+x_0+x_1)=a-(x_0+x_1)\in B_0\cap A.
\eneq
In particular, $p+x_0+x_1\in {\tilde A}.$
Put $z=p+x_0+x_1.$ Then we also view that $z\in {\tilde B}_0.$
Put
\beq\label{TTstr1-11}
z_0=zq_0=(p+x_0)q_0=q_0(p+x_0)\andeqn z_1=z(p-q_0)=(p-q_0)z.
\eneq
Keep in mind
that
\vspace{-0.1in}\beq\label{TTstr1-11+}
z_0+z_1=z.
\eneq
Now write $u_1'=\lambda 1_{{\tilde A}_2}+y$ for some $y\in A_2$ and for some
scalar $\lambda\in \C$ with $|\lambda|=1.$
Let $u_1=\lambda q+y.$
Therefore, by multiplying ${\bar \lambda},$
we may choose so that $u_1$ has the form $q+y.$
Define
$$
u=1+y=u_1+(1-q).
$$
{{Note that $q_0\leq 1_{{\tilde A}_2}$, we have $q_0u_1=q_0u_1'$.}}
Then
\beq\label{TTstr1-12}
&&\hspace{-0.2in}z_0u=(p+x_0)q_0u_1=(q_0+x_0)e_{0,1}'u_1{{=(q_0+x_0)e_{0,1}'u_1'}}=e_{0,1}'(q_0+x_0)e_{0,1}'u{{'}}_1.\\
&&\hspace{-0.2in}z_0u=(p+x_0)q_0u_1=(q_0+x_0)e_{0,1}'u_1=e_{0,1}'(q_0+x_0)e_{0,1}'u_1.
\eneq
Since $e_{0,1}'e_{0,0}'=0,$  by \eqref{TTstr1-8},
\beq\label{TTstr1-13}
(z_0u)(z_0u)=((q_0+x_0)e_{0,1}'u_1')e_{0,1}'z_0u)=0.
\eneq
In other words, $z_0u$ is a nilpotent in $A^{**}.$

On the other hand,
$$
z_1=(p-q_0)z=(p-q_0)+x_1.
$$
Put $D_1=D+\C\cdot (p-q_0).$
Then $D_1\cong {\tilde D}.$  Since $D$ has stable rank one,
there is  an invertible element $z_1'\in D_1$
such that
\beq\label{TTstr1-14}
\|z_1-z_1'\|<\ep/4.
\eneq
Write $z_1'=\lambda_1\cdot (p-q_0)+d$ for some scalar $\lambda_1\in \C$
with $|\lambda_1|=1$ and $d\in D.$
By looking the quotient $D_1/D,$ we may also write
\beq\label{TTst1-15}
z_1'=(p-q_0)+d+\eta(p-q_0)\andeqn |\eta|<\ep/4.
\eneq
\Wlog, we may insist that $\eta\not=0$ (since elements near $z_1'$ are invertible).
We can also write that
\beq\label{TTstr1-15}
z_1'=z_1+(z_1'-z_1)=z_1+d-x_1+\eta(p-q_0).
\eneq

Now view $z_0u\in (1-(p-q_0))A^{**}(1-(p-q_0)).$
Since $z_0u$ is a nilpotent, $z_0u+\eta(1-(p-q_0))$ is invertible in $(1-(p-q_0))A^{**}(1-(p-q_0)).$
Let $\zeta_1$ be the inverse of $z_0u+\eta(1-(p-q_0))$ in $(1-(p-q_0))A^{**}(1-(p-q_0))$ and
$\zeta_2$ be the inverse of $z_1'$ in $D_1=D+\C\cdot (p-q_0).$
Then
\beq\label{TTstr1-15+}
(z_0u+\eta(1-(p-q_0))\oplus z_1')(\zeta_1\oplus \zeta_2)=(1-(p-q_0))+(p-q_0)=1.
\eneq
It follows that
\beq\label{TTstr1-16}
z_2=z_0u+\eta(1-(p-q_0))+z_1'\in GL(A^{**}).
\eneq
However, by \eqref{TTst1-15},
\beq\label{TTstr1-17}
z_2 &=&z_0u+\eta(1-(p-q_0)) +z_1'\\
&=& z_0u_1+\eta(1-(p-q_0))+z_1+(d-x_1)+\eta(p-q_0)\\
&=&z_0u_1+z_1 +(d-x_1)+\eta\cdot 1\\
&=& (z_0+z_1)u+(d-x_1)+\eta\cdot 1\\
&=& zu+(d-x_1)+\eta \cdot 1\in {\tilde A}.
\eneq
It follows that $z_2\in GL({\tilde A}).$
We have
\beq\label{TTstr1-18}
\|zu-z_2\| &=&\|(z_0u_1+z_1)-z_2\|\\
&\le & \|z_1-(\eta(1-(p-q_0))+z_1')\|\\
&\le &\|z_1-z_1'\|+\eta<\ep/4+\ep/4=\ep/2.
\eneq
Therefore (using also \eqref{TTstr1-10}),
\beq\label{TTstr1-19}
\|xu-z_2\|<\ep,\,\,\, {\rm or}\,\,\,\|x-z_2u^*\|<\ep.
\eneq
Since $z_2$ is invertible so is $z_2u^*.$ Since $u\in {\tilde A},$ $z_2u^*$ is in $GL({\tilde A}).$

\end{proof}

At this point, we would like to introduce the following:

\begin{df}\label{DgTR}
Let $A$ be a $\sigma$-unital simple \CA.
We say $A$ is a projectionless simple \CA s with generalized tracial rank at most one,
and write $gTR(A)\le 1,$ if $A$ is TA${\cal C}_0'.$

\end{df}

\begin{prop}\label{PDgTR}
A non-unital  separable  stably projectionless  simple \CA\, $A$ has  generalized tracial rank at most one,
i.e., $gTR(A)\le 1,$
if  and  only if,  for some $a\in P(A)_+\setminus \{0\},$ $\overline{aAa}\in {\cal D}.$
\end{prop}

\begin{proof}
Suppose that $A$ is TA${\cal C}_0'.$ Then, by \ref{Phered}, for any $a\in P(A)_+\setminus \{0\},$
$B:=\overline{aAa}$ is TA${\cal C}_0'.$  Now $B$ is quasi-compact, by  \ref{T1D0},
$B\in {\cal D}.$

It follows from \ref{Phered}  and Brown's stable isomorphism theorem (\cite{Br1})
that it suffices to show that $gTR(B\otimes {\cal K})\le 1.$

Conversely, let $a\in P(A)_+\setminus \{0\}$ and $B=\overline{aAa}.$
Fix  $\ep>0,$ a finite subset ${\cal F}$ and $e\in B\otimes {\cal K}_+\setminus\{0\}$
with $\|e\|=1.$  It is easy to find $a_0\in (B\otimes M_N)_+$ with $\|a_0\|=1$ and
$a_0\lesssim e$ for some $N\ge 1.$

\Wlog, we may assume that ${\cal F}\subset M_{N_1}$ for some $N_1\ge N.$
Therefore it suffices to show that $B\otimes M_{N_1}$ is TA${\cal C}_0'.$
But this follows from the statement of \ref{PMnTAD}.

\end{proof}

\begin{prop}\label{PWAff}
Let $A\in {\cal D}$ with continuous scale.  Then the map $W(A)\to V(A)\sqcup {\rm LAff}_{b+}(T(A))$ 
is surjective.
\end{prop}

\begin{proof}
This follows from the same lines of the  proof of 5.3 of \cite{BPT} as shown in 10.5 of \cite{GLN}
using  (4) of \ref{str=1}. (One can also use the proof 6.2.1 of \cite{Rl} by 
applying  \ref{LappZ} as (D) in that proof.)
\end{proof}

\begin{prop}\label{D=D0K0}
Let $A\in {\cal D}$ with $K_0(A)=\{0\}.$  Then, $A$ has the properties described in  \ref{TCCdvi} as well 
as in \ref{Cuniformful} but replacing ${\cal C}_0^{0'}$ by ${\cal C}_0'.$
\end{prop}

\begin{proof}
We may assume, \wilog, that $A$ has continuous scale. 
It follows from \ref{PWAff} that the map $W(A)\to {\rm LAff}_{b+}(T(A))$ is surjective.
Then, by \ref{CDdiag}, $A$ has tracially approximate divisible property.   
The proof of \ref{TCCdvi} applies to $A$ with ${\cal C}_0^{0'}$ replaced by ${\cal C}_0'.$ 
One then also obtains the conclusion of \ref{Cuniformful} with ${\cal C}_0^{0'}$ replaced by ${\cal C}_0'.$

\end{proof}

We would like to state the following
\begin{prop}\label{Pgtr1}
Let $A$ be a separable simple \CA\, which is stably projectionless and $gTR(A)\le 1.$
Then the following hold.

(1) $A$ has stable rank one;

(2) Every quasi-trace of $A$ is a trace;

(3) $A$ has strict comparison for positive elements;

(4) If $A=P(A),$  then $A\in {\cal D};$

(5) If $B\subset A$ is a  hereditary \SCA, then $B$ also has $gTR(B)\le 1;$

(6)  $M_n(A)$ is stably projectionless and has $gTR(M_n(A))\le 1$ for every integer $n\ge 1.$
\end{prop}



\section{ The range of invariant}

In this first part of this research,  the isomorphism theorem will be established only for  the case that
the non-unital simple
\CA s with $K_0(A)=K_1(A)=\{0\}.$   We would like to revisit some of known results.

\begin{df}\label{DM0}
 Let us denote by ${\cal M}_0$ for the class of non-unital simple \CA s which are inductive limits
of \CA s in ${\cal C}_0^0.$  We insist that the connecting maps are injective and
maps strictly positive elements to strictly positive elements.
\end{df}

\begin{df}\label{DW}
Recall $W$ is an inductive limit of \CA s of the form described in \eqref{ddraz} (see \cite{Tsang}
and \cite{Jb}) which has $K_0(W)=K_1(W)=\{0\}$ and
with a unique tracial state. In particular, $W\in {\cal M}_0.$ 
Moreover, $W$ is the closure of a union of an increasing sequence of \CA s in ${\cal R}$ (in 
fact in  \ref{ddraz}). 
\end{df}

The following is known (see \cite{Raz}, \cite{Tsang} and \cite{Jb}).

\begin{thm}\label{R2}
For any metrizable Choquet simplex $\Delta,$
there exists a non-unital simple \CA\, $A$  which has a continuous scale  and
$A=\lim_{n\to\infty}(B_n, \imath_n),$
 where each $B_n$ is a finite direct sum
of $W$  and  $\imath_n$ maps strictly positive elements to strictly positive elements such that
$$
(K_0(A),  K_1(A), T(A))=(\{0\}, \{0\}, \Delta).
$$
Moreover,  $A$ is also isomorphic to a \CA\,  in ${\cal M}_0$ with continuous scale.
\end{thm}

\begin{proof}
By 3.10 of \cite{Btrace}, there exists a
a unital  simple AF-algebra  $D$ with $T(D)=\Delta.$
Define $A=D\otimes W.$

\end{proof}

\begin{cor}\label{R1}
For any metrizable Choquet simplex $\Delta,$
there exists a non-unital simple \CA\, $A$ in ${\cal M}_0$  which has continuous scale such that
$$
(K_0(A), K_1(A),  T(A))=(\{0\}, \{0\}, \Delta).
$$
Moreover, $A\in {\cal D}_{0}.$  In fact
$A\otimes Q\cong A.$
\end{cor}

\begin{cor}[\cite{Tsang}]\label{Ctsang}
Let ${\tilde T}$ be a topological cone with a base $T$ which is a metrizable Choquet simplex
and let $\gamma: {\tilde T}\to (0, \infty]$ be a lower semi-continuous function.
Then there exist a non-unital simple \CA\, $A$ in ${\cal M}_0$  such that
$$
({\tilde T}(A), \gamma)=({\tilde T}(A), \Sigma_A).
$$
Moreover, $A$ is an inductive limit of finite direct sum of $W.$ 
\end{cor}

\begin{proof}
Let $B$ be a unital simple AF-algebra with $T(B)=T.$
There is a positive element
$a\in B\otimes {\cal K}$ such that $d_\tau(a)=\gamma(\tau)$ for all $\tau\in T=T(B).$
Let $B_1=\overline{a(B\otimes {\cal K})a}.$
Put $A=B_1\otimes W.$  Since $B_1$ is a simple AF-algebra and since
$M_k(W)\cong W$ for all integer $k\ge 1,$ one sees
that $A$ is an inductive limit of finite direct sum $W.$
One then check that
$$
{\tilde T}(A)={\tilde T}\andeqn \Sigma_A=\gamma.
$$

\end{proof}

\begin{prop}[\cite{Rl}]\label{PM0her}
Let $A\in {\cal M}_0.$ Then every hereditary \SCA\, $B\subset A$ is in ${\cal M}_0.$
\end{prop}



\section{ Non-unital version of  some results of Winter}

In the following statement, the words ``every hereditary \SCA\, $B$"  may be replaced
by each $B$ has "non-unital approximate divisible property".

\begin{lem}\label{LWL1}
Let $A$ be a  separable simple \CA\, which almost has stable rank one which  has continuous scale and
has strong strict comparison for positive element.
Suppose that  there is $1>\eta>0$ such that every   hereditary \SCA\, $B$ has the following property:

Let  $a\in B$ be a strictly positive element
with $\|a\|=1.$  Suppose that
for any $\ep>0,$  any
finite subset ${\cal F}\subset A$ and any $b\in A_+\setminus \{0\}$
there are ${\cal F}$-$\ep$-multiplicative \cpc s $\phi: B\to B$ and  $\psi: B\to D$  for some
\SCA\, $D\subset B$ such that
\beq\label{LWLtr1div-1}
&&\|x-\diag(\phi(x), \psi(x))\|<\ep\rforal x\in {\cal F}\cup \{a\},\\\label{LWLn-2}
&&D\in {\cal C}_0' (\in {\cal C}_0^{0'}),\\\label{LWLn-3}
&&\tau(\psi(a))\ge \eta\tforal \tau\in T(B),\\\label{LWLn-4}
&&f_{1/4}(\psi(a))\,\,\,
is\,\,\, full\,\,\, in
\,\, D \tand
\eneq
$\psi(a)$ is strictly positive in $D.$

  Then $A\in {\cal D}$ ( or ${\cal D}_0$).

\end{lem}

\begin{proof}
Let $1>\eta_1>\eta.$
Let $a_0\in A$ be a strictly positive  element with $\|a_0\|=1$
such that
$\tau(a_0)>\max\{3/4, \eta_1\}.$  Let $b_0\in A_+\setminus \{0\}$ with $\|b_0\|=1.$

Fix $\ep>0$ and finite subset ${\cal F}\subset A^{\bf 1}.$
Choose some $g\in C_0((0,1])$ and let $a_1=g(a_0)$ such that
$a_1\ge a_0$ and
\vspace{-0.1in}\beq\label{LWL1-1-}
\|a_1xa_1-x\|<\ep/64\rforal x\in {\cal F}.
\eneq

Let ${\cal F}_1$ be a finite subset containing ${\cal F}\cup\{a_i, f_{1/4}(a_i), i=0,1\}$
and let $\dt_1>0$ with $\dt_1<\ep/16.$

By the assumption,   there are ${\cal F}_1$-$\dt_1$-multiplicative \cpc s $\phi_1: A\to A$ and  $\psi_1: A\to D_1$  for some
\SCA\, $D_1\subset B$ such that
\beq\label{LWL-1}
&&\|x-\diag(\phi_1(x), \psi_1(x))\|<\dt_1/32\rforal x\in {\cal F}_1,\\\label{LWL-2}
&& D_1\in {\cal C}_0^{0'} (\in {\cal C}_0'),\\\label{LWL-3}
&&\tau(\psi_1(a_0))\ge \eta\tforal \tau\in T(B),\\\label{LWL-4}
&&f_{1/4}(\psi_1(a_0))\,\,\,
is\,\,\, full\,\,\, in
\,\, D_1 \andeqn
\eneq
$\psi(a_0)$ is strictly positive in $D_1.$

By \eqref{LWL-3},
$$
\|\psi_1\|\ge \eta.
$$
It follows from the last part of the proof of \ref{PD0=tad} that  we may also assume
that
$$
\|\psi_1(x)\|\ge (1-\dt_1/4)\|x\|\rforal x\in {\cal F}.
$$

Put ${\cal F}_2'=\{\phi_1(a): a\in {\cal F}_1\}.$
By choosing a sufficiently small $\dt_1,$
we may assume
that
\beq\label{LWL-10}
\|\phi_1(a_1)\phi_1(x)\phi_1(a_1)-\phi_1(x)\|<\ep/64\rforal x\in {\cal F}\cup\{a_0\}.
\eneq
Therefore, for some $\sigma>0,$
\beq\label{LWL-11}
\|f_{\sigma}(\phi_1(a_1))\phi_1(x)f_{\sigma}(\phi_1(a_1))-\phi_1(x)\|<\ep/32\rforal x\in {\cal F}\cup\{a_0\}.
\eneq
By \ref{Lbt1ACT}, there exists $0\le e\le 1$ such that
\beq\label{LWL-12}
f_{\sigma}(\phi_1(a_1))\le e\le f_{\sigma'}(\phi_1(a_1))
\eneq
 and $d_\tau(e)$ is continuous on $\overline{T(A)}^w,$
 where $0<\sigma' <\sigma/2.$
Define
$\phi_1': A\to A$ by
\beq\label{LWL-13}
\phi_1'(a)=e^{1/2}\phi_1(a)e^{1/2}\rforal a\in A.
\eneq
We also have, by \eqref{LWL-12},
\beq\label{LWL-14}
e^{1/2}((\phi_1(a_1)-\sigma'/2)_+)e^{1/2}\le e^{1/2}\phi_1(a_1)e^{1/2}\le e.
\eneq
But
\beq\label{LWL-14+}
e&=&e^{1/2}f_{\sigma'/2}(\phi_1(a_1))e^{1/2}\le   e^{1/2}( (2/\sigma')(\phi_1(a_1)-\sigma'/2)_+)e^{1/2}\\
&=&(2/\sigma')(e^{1/2}((\phi_1(a_1)-\sigma'/2)_+)e^{1/2}).
\eneq
Combining these two inequalities, we conclude that
$d_\tau(\phi_1'(a_1))=d_\tau(e)$ for all $\tau\in T(A).$
In particular,
$\phi_1'(a_1)$  is continuous on $T(A).$
Note that, by the choice of ${\cal F}_1$ and $\eta,$
\beq\label{LWL-15}
\|\phi_1'(a_1)-\phi_1(a)\|<\ep/16\rforal a\in {\cal F}\cup\{a_0\}.
\eneq

Therefore,  \wilog,  we may assume that, in addition to  \eqref{LWL-1} to \eqref{LWL-4},
$\overline{\phi_1(a_1)A\phi_1(a_1)}$ has continuous scale and
let $B_1=\overline{\phi_1(a_1)A\phi_1(a_1)}.$
So we can apply these conditions to $B_1$ and $\phi_1({\cal F})$ (using $\ep/32^2$).
Therefore the process continues.
Let $k\ge 1$ such that
$$
(1-\eta)^k<\inf\{d_\tau(b_0): \tau\in T(A)\}.
$$

Then we can stop at the stage $k.$
That way, we obtain hereditary \SCA s $B_1,B_2,...,B_k,$
and \SCA s $D_1, D_2,...,D_k$ such that
$B_{i+1}\subset B_i,$ $B_i\perp D_i,$ $D_{i+1}\subset B_i,$
$D_i\in {\cal C}_0'$ (or ${\cal C}_0^{0'}$), ${\cal F}_i$-$\dt_i$-multiplicative
\cpc s $\phi_{i+1}: B_i\to B_{i+1}$ and $\phi_{i+1}: B_i\to D_{i+1}$
such that
\beq\label{LWL-20}
\|x-\diag(\phi_{i+1}(x), \psi_{i+1}(x))\|&<&\ep/32^{i+1}\rforal x\in \phi_i({\cal F}),\\
\tau(\psi_{i+1}(a_0)) &\ge& \eta\rforal \tau\in T(B_i),
\eneq
$f_{1/4}(\psi_{i+1}(a_0))$ is full in $D_{i+1}$ and $B_{i+1}$ has continuous scale,
$i=1,2,...,k-1.$
Now $D=\bigoplus_{i=1}^kD_i$ and let
$\Psi: A\to D$ be defined by
$$
\Psi(a)=\diag(\psi_1(a), \psi_2(a),...,\psi_k(a))\rforal a\in A.
$$
We have, by \eqref{LWL-20},
\beq
\|x-\diag(\Phi(x), \Psi(x))\|<\ep\rforal x\in {\cal F},\\
\|\Psi(x)\|\ge (1-\ep)\|x\|\rforal x\in {\cal F},
\eneq
$f_{1/4}(\Psi(a_0))$ is full in $D$ and $\Psi(a_0)$ is a strictly positive element of $D,$
where
$\Phi: A\to B_k$ is defined by
$\Phi=\phi_k\circ \phi_{k-1}\circ \cdots \phi_1.$

We compute that, if $b_k\in B_k$ is a strictly positive element for $B_k$ with $\|b_k\|=1,$
then
$$
d_\tau(b_k)\le (1-\eta)^k\tforal \tau\in T(A).
$$
This implies that
$$
b_k\lesssim b_0,
$$
since $A$ is assumed to have strong strict comparison for positive elements.

Note that $A$ has continuous scale,  in particular, it is quasi-compact. It follows from the above
and \ref{T1D0} that $A\in {\cal D}$ (or in ${\cal D}_0$).
\end{proof}

The following is a  non-unital variation of a result of W. Winter.
\begin{prop}\label{WinterP1}
Let $A$ be a non-unital separable simple \CA\, with strong strict comparison
for positive elements.
Let $F$ be a finite dimensional $C^*$-algebra and let
$\phi: F\to A$ and
$\phi_i: F\to A$ ($i\in \N$) be order zero \cpc s such that,
for each $x\in F_+$ and $f\in C_0((0,1])_+,$
\beq
&&\lim_{i\to\infty}\sup\{|\tau(f(\phi(x)))-\tau(f(\phi_i(x))|: \tau\in T(A)\}=0\tand\\
&&\limsup_{i\to\infty}\|f(\phi_i(x))\|\le \|f(\phi(x))\|.
\eneq
Then there are $s_i\in (M_4\otimes A)^{\bf 1},$ $i\in \N,$ such that,
for all $y\in F_+,$
\beq\label{WP1-3}
&&\lim_{i\to\infty}\|s_i(1_4\otimes \phi(y))-(e_{11}\otimes \phi_i(y))s_i\|=0\tand\\\label{WP1-4}
&&\lim_{i\to\infty}\|(e_{11}\otimes \phi_i(y))s_is_i^*-(e_{11}\otimes \phi_i(y))\|=0.
\eneq

\end{prop}

\begin{proof}
Since $\phi$ and $\phi_i$ are order zero maps, it is easy to see
that we may reduce the general case to the case
that $F=M_n$ for some integer $n\ge 1.$

We first consider the case that $F=\C.$
Write $h=\phi(1_\C)$ and $h_i=\phi_i(1_\C),$ $i\in \N.$

Then exactly the same proof as in the proof of 2.1 of \cite{Wccross}, without changing a single symbol, provides
the elements $s_i$ such that \eqref{WP1-3} and \eqref{WP1-4} hold.
It should be noted that, for any $x\in A,$ $1_4\otimes x\in M_4(A)$ and
$e_{11}\otimes x\in M_4(A).$

The general case can also be reduced to this case in the non-unital case.
 In fact, since $\phi$ and $\phi_i$ are order zero maps, it is easy to see
that we may reduce the general case to the case
that $F=M_n$ for some integer $n\ge 1.$

By 1.2 of \cite{Wann05},  there are \hm  s  $\psi, \psi_i: C_0((0,1], F)\to A$
such that $\phi(x)=\psi(\iota\otimes x)$ and
$\phi_i(x)=\psi_i(\iota\otimes x)$ for all $x\in F,$ where
$\iota(t)=t$ for all $t\in [0,1].$

Let us use $\{\eta_{kj}\}$ for the matrix unit for $M_n.$
Consider $\phi^{(1)}=\phi|_{\eta_{11}M_n\eta_{11}}$ and
$\phi_i^{(1)}=\phi_i|_{\eta_{11}M_n\eta_{11}},$ $i\in \N.$

Let $s_i$ be the sequence satisfies \eqref{WP1-3} and \eqref{WP1-4}
for $F=\eta_{11}M_n\eta_{11}.$
We may assume that
\beq\label{WP1-6}
&&\lim_{i\to\infty}\|s_i(1_4\otimes \phi(y))-(e_{11}\otimes \phi_i(y))s_i\|<1/n^22^{i+1}\tand\\\label{WP1-7}
&&\lim_{i\to\infty}\|(e_{11}\otimes \phi_i(y))s_is_i^*-(e_{11}\otimes \phi_i(y))\|<1/n^22^{i+1}
\eneq
for all $y\in \eta_{11}M_n\eta_{11}$ with $\|y\|\le 1.$
For each $i,$ there is $g_i\in C_0((0,1])$ with $0\le g_i \le 1$ such that
\beq\label{WP1-6+4}
\|\iota g_i^4-\iota \|<1/n^22^{i+2}.
\eneq
Note that, for any $g\in C_0((0,1]),$
\beq\label{WP1-6+2}
1_4\otimes \psi(g\otimes \eta_{11})=1_4\otimes \psi(g(\iota\otimes \eta_{11}))=g(1_4\otimes \phi(\eta_{11}))\andeqn\\
\label{WP1-6+3}
e_{11}\otimes \psi_i(g\otimes \eta_{11})=e_{11}\otimes g(\psi_i(\iota\otimes \eta_{11}))=g(e_{11}\otimes \phi_i(\eta_{11}).
\eneq

For each $i,$ $k(i)\ge i$ such that
\beq\label{WP1-6+1}
\|s_{k(i)}(1_4\otimes \psi(g_i^2\otimes \eta_{11}))s_{k(i)}^*-(e_{11}\otimes \psi_{k(i)}(g_i^2\otimes \eta_{11})\|<1/n 2^{i+1}
\eneq
To simplify the notation, \wilog, we may assume that $k(i)=i.$
Define
$$
S_i=\sum_{j=1}^n\psi_i(g_i\otimes \eta_{1j})^*s_i(1_4\otimes \psi(g_i\otimes \eta_{1j})),\, i\in \N.
$$
Then $S_i\in M_4(A),$ $i\in \N.$ Moreover,
\beq\label{WP1-8}
S_iS_i^*=\sum_{j=1}^n\psi_i(g_i\otimes \eta_{1j})^*s_i(1_4\otimes \psi(g_i\otimes \eta_{1j})))(1_4\otimes \psi(g_i\otimes \eta_{1j}))^*s_i^*\psi(g_i\otimes \eta_{1j})\\
\le \sum_{j=1}^n\psi_i(g_i\otimes \eta_{1j})^*s_is_i^*\psi_i(g_i\otimes \eta_{1j})\\
\le \sum_{j=1}^n\psi_i(g_i\otimes \eta_{1j})^*\psi_i(g_i\otimes \eta_{1j})\\
\le \sum_{j=1}^n\psi_i(g_i^2\otimes \eta_{jj})\le \psi_i(g_i^2\otimes 1_{M_n})
\eneq
Therefore $S_i\in M_4(A)^{\bf 1},$ $i\in\N.$
Let $x=(a_{kj})_{n\times n}\in M_n$ such that $|a_{kj}|\le 1.$
We compute that
\beq\label{WP1-9}
\hspace{-0.4in}S_i(1_4\otimes \phi((a_{kj})))&=&\sum_{j=1}^n\psi_i(g_i\otimes \eta_{1j})^*s_i(1_4\otimes \psi(g_i\otimes \eta_{1j}))(1_4\otimes \psi((\iota\otimes (a_{kj})))\\\nonumber
&=& \sum_{j=1}^n \psi_i(g_i\otimes \eta_{1j})^*s_i(\sum_{k=1}^n(1_4\otimes \psi(\iota\otimes a_{jk}\otimes \eta_{11})\psi(g_i\otimes \eta_{1k}))\\\nonumber
&\approx_{1/2^{i+1}}&  \sum_{j=1}^n \psi_i(g_i\otimes \eta_{1j})^*\sum_{k=1}^n(e_{11}\otimes \psi_i(\iota\otimes a_{jk}\cdot\eta_{11})s_i(1_4\otimes \psi(g_i\otimes \eta_{1k}))\\\nonumber
&=& \sum_{j=1}^n \sum_{k=1}^n(e_{11}\otimes \psi_i(\iota\otimes a_{jk}\otimes \eta_{jk})\psi_i(g_i\otimes \eta_{k1})s_i(1_4\otimes \psi(g_i\otimes \eta_{1k})))\\\nonumber
&=&\sum_{k=1}^n(e_{11}\otimes \psi_i(\iota\otimes (\sum_{j=1}^n( a_{jk}\cdot \eta_{jk})))\psi_i(g_i\otimes \eta_{k1})s_i(1_4\otimes \psi(g_i\otimes \eta_{1k})))\\\nonumber
&=&\sum_{k=1}^n(e_{11}\otimes \psi_i(\iota\otimes (a_{st})_{n\times n})\psi_i(g_i\otimes \eta_{k1})s_i(1_4\otimes \psi(g_i\otimes \eta_{1k})))\\\nonumber
&=&e_{11}\otimes \psi_i(\iota\otimes (a_{st})_{n\times n}))(\sum_{k=1}^n\psi_i(g_i\otimes \eta_{k1})s_i(1_4\otimes \psi(g_i\otimes \eta_{1k})))\\\nonumber
&=&e_{11}\otimes \phi((a_{kj})_{n\times n})S_i.
\eneq

By \eqref{WP1-6+1} and \eqref{WP1-6+4},
\beq\label{WP1-11}
e_{11}\otimes \phi_i((a_{ij}))S_iS_i^*&=&\\\nonumber
&&\hspace{-1.2in}e_{11}\otimes \phi_i((a_{ij}))\sum_{j=1}^n\psi_i(g_i\otimes \eta_{1j})^*s_i(1_4\otimes \psi(g_i\otimes \eta_{1j})))(1_4\otimes \psi(g_i\otimes \eta_{1j}))^*s_i^*(\psi_i(g_i\otimes \eta_{1j})\\\nonumber
&=&e_{11}\otimes \phi_i((a_{ij}))\sum_{j=1}^n\psi_i(g_i\otimes \eta_{1j})^*s_i(1_4\otimes \psi(g_i^2\otimes \eta_{11})))s_i^*(\psi_i(g_i\otimes \eta_{1j})\\\nonumber
&\approx_{1/2^{i+1}}&e_{11}\otimes \phi_i((a_{ij}))\sum_{j=1}^n\psi_i(g_i\otimes \eta_{1j})^*\psi_i(g_i^2\otimes \eta_{11})))(\psi_i(g_i\otimes \eta_{1j}))\\\nonumber
&=&e_{11}\otimes \phi_i((a_{ij}))(\sum_{j=1}^n\psi_i(g_i^4\otimes \eta_{jj}))\\\nonumber
&\approx_{1/2^{i+1}}& e_{11}\otimes \phi_i((a_{ij})).
\eneq
\end{proof}

The following is a non-unital variation of Proposition 3.2 of \cite{Wnz}.

\begin{prop}[\cite{Wnz}]\label{WP2}
Let $A$ be a non-unital separable simple \CA\, with ${\rm dim}_{nuc} A\le m<\infty$
 and with strictly positive element $a\in A$ with $0\le a\le 1.$
Suppose that
$$
(F_n, F_n^{(1)}, F_n^{(2)},...,F_n^{(m)}, \psi_n, \psi_n)_{ n\in \N}
$$
is a system, where $F_n, F_n^{(i)}$ are finite dimensional \CA s with
$F_n=\bigoplus_{i=1}^m F_n^{(i)},$
$$
\psi_n: A\to F_n,\,\,\, \phi_n\circ \psi_n: A\to A
$$
are \cpc s,
$$ \phi_n: F_n\to A$$
are completely positive maps,
$\phi_n^{(i)}=\phi_n|_{F_n^{(i)}}$ are order zero \cpc s,  such that
\beq\label{WP2-1}
\lim_{n\to\infty}\|\phi_n\circ \psi_n(b)-b\|=0\tforal a\in A.
\eneq
Then, for any $\ep>0$ and $b\in \overline{f_{\ep}(a)Af_{\ep}(a)},$
\beq\label{WP2-2}
&&\lim_{n\to\infty}\|{[}\phi_n\circ \psi_n(b),\, \phi_n^{(i)}\circ \psi_n^{(i)}(g(a)){]}\|=0\tand\\\label{WP2-2+}
&&\lim_{n\to\infty}\|b\phi_n^{(i)}\circ \psi_n^{(i)}(g(a))-\phi_n^{(i)}\circ \psi_n^{(i)}(b)\|=0,
\eneq
where $\psi_n^{(i)}: A\to F_n^{(i)}$ is the composition of $\psi_n$ with
the projection map on the $i$-th summand $F_n^{(i)},$ $i=1,2,...,m$ and
$n\in \N,$ and where
$g\in C_0((0,1])$ is a function such that $0\le g(t)\le 1$ and
$g(t)=1$ for all $t\ge \ep/2.$
\end{prop}

\begin{proof}
Fix $1>\ep>0$ and $g$ as described.

Let $B_n=\overline{\psi_n(g(a))F_n\psi_n(g(a))},$ $n=1,2,....$
Note that $B_n$ are unital finite dimensional \CA s.
For each $n,$ define
$$
{\hat \psi}_n(\cdot)=\psi_n(g(a))^{-1/2}\psi_n(\cdot)\psi_n(g(a))^{-1/2},
$$
where the inverse is taken in $B_n.$  Also, for each $n,$ define
$$
{\hat \phi}_n(c)=\phi_n(\psi_n(g(a))^{1/2} c\psi_n(g(a))^{1/2})\rforal c\in A.
$$
Then
$$
{\hat \phi}_n\circ {\hat \psi}_n|_{C}=\phi_n\circ \psi_n|_{C},
$$
where $C=\overline{f_{\ep}(a)Af_{\ep}(a)}.$

From Lemma 3.6 of \cite{KW},  for $i\in \{1,2,...,m\},$ $x\in B_n$ and $b\in A,$
\beq\label{WP2-4}
\|{\hat\phi}_n(x{\hat \psi}_n(b))-{\hat \phi}(x){\hat \phi}\circ {\hat \psi}(b)\|<3\|x\|\max \{\|{\hat \phi}_n{\hat \psi}_n(b)-b\|,
\|{\hat \phi}_n{\hat \psi}_n(b^2)-b^2\|\}.
\eneq
Write $B_n=\bigoplus_{i=1}^m \psi_n^{(i)}(g(a))F_n^{(i)}\psi_n^{(i)}(g(a))$ and choose
$x=q_n^{(i)}$ which is the identity of $\psi_n^{(i)}(g(a))F_n^{(i)}\psi_n^{(i)}(g(a)).$
Then, by  \eqref{WP2-4}, one obtains
$$
\lim_{n\to\infty}\|\phi_n^{(i)}\circ \psi_n^{(i)}(b)-\phi_n^{(i)}\circ \psi_n^{(i)}(g(a))(\phi_n\circ \psi_n(b))\|=0
$$
for all $b\in C.$ This gives \eqref{WP2-2+}. From this, \eqref{WP2-2} also follows.

\end{proof}


The following is a   non-unital version  of 2.2 of \cite{Wccross}.
The proof is based on that of 2.2 of \cite{Wccross} with some subtle
modification to fit the non-unital case with continuous scale.

\begin{thm}\label{TWv}
Let $A$ be a non-unital   separable simple \CA\,
with $K_0(A)={\rm ker}\rho_A,$
with ${\rm dim}_{nuc} A=m<\infty$
which has continuous scale.

Suppose that for every non-zero hereditary \SCA\, $B$ of $A$ the following holds.

Fix a strictly positive element $a_0\in B_+$ with $0\le a_0\le 1.$

Let $C\in {\cal M}_0$ be a non-unital simple \CA\, with $K_0(C)=\{0\},$ $K_1(A)=\{0\},$
quasi-compact, and
$T(C)=T(B).$
Suppose that there is a sequence of \cpc s $\sigma_n: B\to C$ and
\hm s $\rho_n: C\to B$ such that
\beq\label{TWv-1}
&&\lim_{n\to\infty}\|\sigma_n(ab)-\sigma_n(a)\sigma_n(b)\|=0\tforal a, b\in A,\\
&&\lim_{n\to\infty}\sup\{\tau(\rho_n\circ \sigma_n(a))-\tau(a)\|:\tau\in T(A)\}=0\tand
\eneq
$\sigma_n(a_0)$ is a strictly positive in $C$ for all $n\in \N.$

Then, for any $\ep>0$ and any finite subset ${\cal F}\subset B$
and any integer $k_0\ge 1,$
 there are ${\cal F}$-$\ep$-multiplicative \cpc s $\phi: B\to M_{4(m+1)}(B)$ and  $\psi: B\to D\in {\cal C}_0^{0'}$
 for some
\SCA\, $D\subset M_{4(m+1)}(B)$ such that
\beq\label{TWvtr1div-1}
&&\|1_{4(m+1)}\otimes x-\diag(\phi(x),
\overbrace{\psi(x), \psi(x),...,\psi(x)}^{k_0})
\|<\ep\rforal x\in {\cal F},
\\\label{TWvtrdiv-2}
&&  D\in {\cal C}_0^{0'}\,\\\label{TWvtrdiv-3}
&&\tau(\psi(a_0))\ge 1/12k_0(m+1)\tforal \tau\in T(M_{4(m+1)}(B)),\\\label{TWvtrdiv-4}
&& \overline{\phi(a_0)B\phi(a_0)} \,\,\, \text{has\,\,continuous\,\,scale,}\\\label{TWvtrdiv-4-}
&& \|\psi\|=1,\\\label{TWvtrdiv-4+}
&&f_{1/4}(\psi(a_0))\,\,\,
is\,\,\, full\,\,\, in
\,\, D \tand
\eneq
$\psi(a_0)$ is strictly positive in $D.$

\end{thm}

\begin{proof}
This is a modification of the proof of 2.2 of \cite{Wccross}.


To simplify notation, \wilog, we may assume
that ${\cal F}\subset f_{\dt}(a_0)Bf_{\dt}(a_0)$ for some
$0<\dt<1/4.$
We may assume, by \ref{Csemicon},
that
\beq\label{Wv-0}
\tau(f_{\dt}(a_0))\ge 1-1/64(m+1)\rforal \tau\in T(B).
\eneq

Let  $F^{(0)}_j, F_j^{(1)},..., F_j^{(m)}$ be
finite dimensional \SCA s,
$F_j=F^{(0)}_j\oplus \cdots F_j^{(m)},$
$\psi_j: A\to F_j$ be a sequence of \cpc s,
$\phi_j: F_j\to B$ be a sequence of completely positive linear maps
such that $\phi_j\circ \psi_j$ are \cpc s,
$\phi_j^{(l)}=\phi_j|_{F_j^{(l)}}$ ($l\in \{1,2,...,m\}$)
are order zero maps, and
$$
\lim_{n\to\infty}\|\phi_n\circ \psi_n(b)-b\|=0\rforal b\in B.
$$
Denote by $\psi_j^{(l)}$ be the composition of $\psi_j$ with
the projection from $F_j$ onto  $F_j^{(l)}.$
By replacing $F_j^{(l)}$ by the hereditary  \SCA\, generated by
$\psi_j^{(l)}(a_0),$ we may $\psi_j^{(l)}(a_0)$ is a strictly positive element which is invertible
in $F_j^{(l)}.$

By the weak stability of order zero maps,
there are ${\tilde \phi}_{j,n}^{(l)}:  F_n^{(l)}\to C$ $(n\in \N$) such that
$$
\lim_{n\to\infty}\|{\tilde \phi}_{j,n}^{(l)}(x)-\sigma_n\circ \phi_j^{(l)}(x)\|=0\rforal x\in F_j^{(l)}.
$$

For each $x\in (F_j^{(l)})_+$ and $f\in C_0((0,1])_+,$  we have
$$
\lim_{n\to\infty}\|f({\tilde \phi}_{j,n}^{(l)}(x))-\sigma_n(f(\phi_j^{(l)}(x)))\|=0,
$$
hence
\beq\label{Wv-10}
\lim_{n\to\infty}\sup_{\tau\in T(A)}\{|\tau(f(\rho_n\circ {\tilde \phi}_{j,n}^{(l)}(x)))-\tau(f(\phi_j^{(l)}(x)))|\}=0\andeqn\\
\lim_{n\to\infty}\|\rho_n(f({\tilde \phi}_{j,n}^{(l)}(x)))-\rho_n\circ \sigma_n(f(\phi_j^{(l)}(x)))\|=0.
\eneq
Since $\rho_n$ are injective \hm s, $\sigma_n$ are eventually nonzero and
approximately multiplicative, we see that
$$
\limsup_{n\to\infty}\|\rho_n\circ \sigma_n(f(\phi_j^{(l)}(x)))\|=\|f(\phi_j^{(l)}(x))\|,
$$
whence
$$
\limsup_{n\to\infty}\|f(\rho_n\circ {\tilde \phi}_{j,n}^{(l)}(x))\|\le \|f(\phi_j^{(l)}(x))\|
$$
for $x\in (F_j^{(l)})_+$ and $f\in C_0((0,1])_+.$

By \ref{WinterP1}, there are
$$
s_{j,n}^{(l)}\in M_4(A)^{\bf 1},\,\, n\in \N,
$$
such that
\beq\label{Wv-12}
\lim_{n\to\infty}\|s_{j,n}^{(l)}(1_4\otimes \phi_j^{(l)}(x))-(e_{11}\otimes \rho_n\circ {\tilde \phi}_{j,n}^{(l)})s_{j,n}^{(l)}\|=0\andeqn\\
\lim_{n\to\infty}\|(e_{11}\otimes \rho_n\circ {\tilde \phi}_{j,n}^{(l)}(x))s_{j,n}^{(l)}(s_{j,n}^{(l)})^*
-e_{11}\otimes \rho_n\circ {\tilde\phi}_{j,n}^{(l)}(x)\|=0
\eneq
for each $x\in F_j^{(l)}.$
Put $B_{\infty}=\prod B/\oplus B.$
We obtain contractions
$$
s_j^{(l)}\in (M_4\otimes B)_{\infty}\cong M_4\otimes B_{\infty}
$$
with
$$
s_j^{(l)}(1_4\otimes \imath\circ \phi_j^{(l)}(x))=
(e_{11}\otimes {\bar \rho}\circ {\bar \sigma}\circ \phi_j^{(l)}(x))s_j^{(l)}
$$
'and
$$
(e_{11}\otimes {\bar \rho}\circ {\bar \sigma}\circ \phi_j^{(l)}(x))s_j^{(l)}(s_j^{(l)})^*
=e_{11}\otimes {\bar \rho}\circ {\bar \sigma}\circ \phi_j^{(l)}(x)),
$$
where
$$
{\bar \sigma}: B\to C_{\infty}=\prod C/\oplus C
$$
is the \hm\, induced by $\sigma_n,$
$$
{\bar \rho}: C_{\infty}=\prod C/\oplus C\to B_{\infty}
$$
is the \hm\, induced by $\rho_n$ and $\imath: B\to B_{\infty}$
is the canonical embedding.
Let ${\bar \imath}: B_{\infty}\to (B_{\infty})_{\infty}$ be the map induced
by $\imath,$ i.e.,
$\{(a_m)_{m\in \N}\}\stackrel{\bar \imath}{ \mapsto}\{(\imath(a_m))_{m\in \N}\}.$
Let
$$
{\bar \gamma}: A_{\infty}\to (A_{\infty})_{\infty}
$$
be induced by ${\bar \rho}{\bar \sigma},$
$$
{\bar \phi}^{(l)}: \prod_jF_j^{(l)}/\oplus_j F_j^{(l)}\to B_{\infty}
$$
and
$$
{\bar \psi}^{(l)}: A\to \prod_jF_j^{(l)}/\oplus_j F_j^{(l)}
$$
be the order zero maps induced by $\phi_j^{(l)}$ and $\psi_j^{(l)},$ respectively.

Define
$$
{\bar s}^{(l)}=\{(s_j^{(l)})_{j\in\N}\}\in (M_4\otimes B_{\infty})_{\infty},
$$
then
\beq\label{Wv-15}
&&{\bar s}^{(l)}(1_4\otimes {\bar \imath}\circ {\bar \phi}^{(l)}\circ {\bar \psi}^{(l)}(b))=
(e_{11}\otimes {\bar \gamma}\circ {\bar \phi}^{(l)}\circ {\bar \psi}^{(l)}(b)){\bar s}^{(l)}\andeqn\\\label{Wv-16}
&&(e_{11}\otimes {\bar \gamma}\circ {\bar \phi}^{(l)}\circ {\bar \psi}^{(l)}(b)){\bar s}^{(l)}({\bar s}^{(l)})^*
=e_{11}\otimes {\bar \gamma}\circ {\bar \phi}^{(l)}\circ {\bar \psi}^{(l)}(b).
\eneq
for all $b\in B.$

Note that, for each $b\in \overline{f_{\dt}(a_0)Bf_{\dt}(a_0)}$ and
each $g\in C_0((0,1])$ with $0\le g\le 1,$
$g(t)=1$ for all $t\ge \dt/2,$ by \ref{WP2},
\beq\label{Wv-16+}
{\bar \phi}^{(l)}\circ {\bar \psi}^{(l)}(g(a))\imath(b)={\bar \phi}^{(l)}\circ {\bar \psi}^{(l)}(b).
\eneq
As a consequence,
$$
({\bar \phi}^{(l)}\circ {\bar \psi}^{(l)}(g(a)))^{1/2}\imath(b)\in  C^*({\bar \phi}^{(l)}\circ {\bar \psi}^{(l)}(b): b\in B)
$$
which implies that, for each $b\in \overline{f_{\dt}(a_0)Bf_{\dt}(a_0)}$
\beq\label{Wv-17}
{\bar s}^{(l)}(1_4\otimes ({\bar \imath}\circ {\bar  \phi}^{(l)}\circ {\bar \psi}^{(l)}(g(a))^{1/2})
(1_4\otimes {\bar \imath}\circ \imath(b))\\
={\bar s}^{(l)}(1_4\otimes ({\bar \imath}(( {\bar \phi}^{(l)}\circ {\bar \psi}^{(l)}(g(a))^{1/2}) \imath(b)))\\
=(e_{11}\otimes ({\bar \gamma}(\imath(b))(({\bar \gamma}\circ
 {\bar \phi}^{(l)}\circ {\bar \psi}^{(l)}(g(a)))^{1/2})){\bar s}^{(l)}\\
 =(e_{11}\otimes {\bar \gamma}\circ\imath(b))(e_{11}\otimes
 {\bar \gamma}\circ ({\bar \phi}^{(l)}\circ {\bar \psi}^{(l)}(g(a)))^{1/2}){\bar s}^{(l)}.
\eneq
We now fix  continuous functions $g_i$ (on $[0,1]$) such that $g_1(t)\not=0$ for all $t\in (0,1]$
and $g_i(t)=1$ for all $t\ge \dt/8(i+1),$ $i=0,1.$ In particular $g_1(a_0)$ is a strictly positive element.
We also require that $g_1g_0=g_0.$
Moreover,
by \eqref{Wv-0},
\beq\label{Wv-17+}
\tau(g_i(a_0))\ge (1-1/64(m+1))\rforal \tau\in T(A),\,\,\, i=0,1.
\eneq
Set, for $i=0,1,$
\beq\label{Wv-18}
{\bar v}_i&=&\sum_{l=0}^m e_{1l}\otimes (e_{11}\otimes ({\bar \gamma}\circ {\bar \phi}^{(l)}\circ {\bar \psi}^{(l)}
(g_i(a)))^{1/2}){\bar s}^{(l)}\\
&=& \sum_{l=0}^m e_{1l}\otimes ({\bar s}^{(l)}(1_4\otimes
 ({\bar \imath}\circ {\bar \phi}^{(l)}\circ {\bar \psi}^{(l)}(g_i(a_0)))^{1/2})
 \in M_{m+1}(M_4\otimes (B_{\infty})_{\infty}).
\eneq
Then (also using \eqref{Wv-16+})
\beq\label{Wv-20}
{\bar v}_i{\bar v}_i^*&=&
\sum_{l=0}^me_{11}\otimes e_{11}\otimes
{\bar \gamma}\circ {\bar \phi}^{(l)}\circ {\bar \psi}^{(l)}(g_i(a_0))\\
&=&e_{11}\otimes e_{11}\otimes {\bar \gamma}\circ {\bar \phi}\circ {\bar \psi}(g_i(a_0))\\
&=&e_{11}\otimes e_{11}\otimes {\bar \gamma}(\imath(g_i(a_0))),\,\,\,i=0,1.
\eneq
Moreover,  for $b\in \overline{f_{\dt}(a_0)Bf_{\dt}(a_0)},$
by \eqref{Wv-16}, \eqref{Wv-15} and \eqref{Wv-18},
\beq\label{Wv-25}
&&\hspace{-0.3in}{\bar v}_i(1_{m+1}\otimes 1_4\otimes {\bar \imath}(\imath(b))\\
&=&(\sum_{l=0}^m e_{1l}\otimes ({\bar s}^{(l)}(1_4\otimes
 ({\bar \imath}\circ {\bar \phi}^{(l)}\circ {\bar \psi}^{(l)}(g_i(a)))^{1/2}))
(1_{m+1}\otimes 1_4\otimes {\bar \imath}(\imath(b)))\\
&=&\sum_{l=0}^m e_{1l}\otimes ({\bar s}^{(l)}(1_4\otimes
 ({\bar \imath}\circ {\bar \phi}^{(l)}\circ {\bar \psi}^{(l)}(b))))\\
 &=&\sum_{l=0}^m e_{1l}\otimes (e_{11}\otimes {\bar \gamma}\circ {\bar \phi}^{(l)}\circ {\bar \psi}^{(l)}(b)){\bar s}^{(l)}\\
 &=&\sum_{l=0}^m e_{1l}\otimes (e_{11}\otimes {\bar \gamma}\circ {\imath}(b))
 (e_{11}\otimes {\bar \gamma}\circ  {\bar \phi}^{(l)}\circ {\bar \psi}^{(l)}(b)){\bar s}^{(l)}\\
&=&\sum_{l=0}^m e_{1l}\otimes (e_{11}\otimes {\bar \gamma}(\imath(b))){\bar v}_i.
\eneq
Hence, for $b\in \overline{f_{\dt}(a_0)Bf_{\dt}(a_0)},$
\beq\label{Wv-30}
{\bar v}_i^*{\bar v}_i(1_{m+1}\otimes 1_4\otimes {\bar \imath}\imath(b))&=&
{\bar v}_i^*(e_{11}\otimes {\bar \gamma}\circ \imath(b)){\bar v}_i\\
&=&(1_{m+1}\otimes 1_4\otimes {\bar \imath}\imath(b)){\bar v}_i^*{\bar v}_i.
\eneq

Now fix a finite subset ${\cal F}_1\subset {\cal F}\cup \{a\}$ and fix $0<\eta<\ep$
such that
\beq\label{Wv-n-1}
\|f_{\dt}(\Phi_0(a_0)\Phi(b)f_{\dt}(\Phi_0(a_0))-\Phi(b)\|<\ep/64\rforal b\in {\cal F},
\eneq
provided that $\Phi$ is a ${\cal F}_1$-$\eta$-multiplicative \cpc.
For $x\in {\cal F}_1,$
 there are $j,n\in \N$ and
$v_i\in M_{4(m+1)}(B)$ such that
\beq\label{Wv-32}
&&v_iv_i^*=e_{11}\otimes e_{11}\otimes \rho_n(\sigma_n(g_i(a_0)))\,\,
\\\label{Wv-32+1}
&&\tau(e_{11}\otimes e_{11}\otimes \rho_n(\sigma_n(f_\dt(a_0))))\ge {1-1/64(n+1)(m+1)\over{4(m+1)}}\\
&&\hspace{2in}\rforal \tau\in T(M_{4(m+1)}(B)),\\\label{Wv-33}
&&\|[v_i^*v_i,\, 1_{m+1}\otimes 1_4\otimes b]\|<\eta/16\rforal b\in {\cal F}_1\\\label{Wv-34}
&&\|v_i^*v_i(1_{m+1}\otimes 1_4\otimes  b)-v_i^*(e_{11}\otimes e_{11} \otimes \rho_n\circ \sigma_n(b))v_i\|<\eta/16
\\\label{Wv-35}
&&\hspace{2in}  \rforal  b\in {\cal F}_1 \andeqn\\\label{Wv-36}
&&\|v_1^*(e_{11}\otimes e_{11} \otimes \rho_n\circ \sigma_n(b))v_1-
v_0^*(e_{11}\otimes e_{11} \otimes \rho_n\circ \sigma_n(b))v_0\|<\eta/16
\eneq
for all $ b\in {\cal F}_1.$
Then
$$\kappa: \, b\to v_0^*(e_{11}\otimes e_{11}\otimes \rho_n(b))v_0$$ is
a \cpc\, on $C$ and is
a monomorphism on $\overline{f_{\dt}(c)Cf_{\dt}(c)},$
where $c=\sigma_n(a_0).$
Since $C\in {\cal M}_0,$ $C':=\overline{f_{\dt}(c)Cf_{\dt}(c)}\in {\cal M}_0.$
One can write
$C'=\lim_{k\to\infty} (D_k, \iota_k),$ where  $D_k\in {\cal C}_0^0$ and where
each $\iota_k$ is injective and maps strictly positive elements to
strictly positive elements.
Since $C'\otimes Q\cong C',$
we may assume that
$C_{2k+1}=C_{2k}\otimes M_{k!}$ and
$\iota_{2k+1}: C_{2k}\to C_{2k+1}$ is defined by
$\iota_{2k+1}(x)=x\otimes 1_{k!},$ $k=1,2,....$
Define $B_k=\kappa(\iota_{k,\infty}(D_k).$  Since $D_k$ is amenable, there is, for each $k,$ a
\cpc \,
$s_k: C'\to \iota_{k, \infty}(D_k)$ such that
\beq\label{Wv-37}
\lim_{k\to\infty}\|s_k(c)-c\|=0\rforal c\in  C.
\eneq
Now fix $k_0\ge 1$ and assume $k>k_0.$
Define,
$L_0: B\to M_{4(m+1)}(B)$ by
$$
L_0(b)=(1-v_1^*v_1)(1_{m+1}\otimes 1_4\otimes b)(1-v_1^*v_1)\rforal b\in A,
$$
define $L': A\to B_{2k}\to B_{2k+1}=B_{2k}\otimes M_{k!}$ by
$$
L'(b)=\iota_{2k+1}(s_{2k}(\kappa(e_{11}\otimes e_{11}\otimes \rho_n\circ \sigma_n(b)))\tforal b\in A.
$$
Note that, for any $b\in A,$
\beq\label{Wv-38}
L_0(b)L'(b)=L'(b)L_0(b)=0.
\eneq
By \ref{Lbt1ACT}, there exists $0\le e\le 1$ such that
\beq\label{Wv-n2}
f_{\dt/4}(L_0(a_0))\le e\le f_{\dt_1}(L_0(a_0))
\eneq
for some $0<\dt_1<\dt_2$ and $d_\tau(e)$ is continuous on $\overline{T(B)}^w.$
Define
$L_0': B\to B$ by
\beq\label{Wv-n3}
L_0'(b)=e^{1/2}L_0(b)e^{1/2}\rforal b\in B.
\eneq
We also have, by \eqref{Wv-n2} (see the lines from \eqref{LWL-14} to \eqref{LWL-14+}),
\beq\label{Wv-n4}
\la  L_0'(a_0)\ra =\la e^{1/2}L_0(a_0)e^{1/2}\ra =\la e\ra .
\eneq
So $d_\tau(L_0'(a_0))=d_\tau(e)$ is continuous on $\overline{T(B)}^w.$
Note that, by the choice of ${\cal F}_1$ and $\eta,$
\beq\label{Wv-39}
\|L_0(b)-L_0'(b)\|<\ep/64\rforal b\in {\cal F}.
\eneq
Moreover,
\beq\label{Wv-39+1}
L_0'(b)L'(b)=L'(b)L_0'(b)=0\rforal b\in B.
\eneq

By choosing a sufficiently large $k,$ we may assume, by \eqref{Wv-37} and \eqref{Wv-35},
\beq\label{Wv-39++}
\|v_1^*v_1(1_{m+1}\otimes 1_4\otimes b)-L'(b)\|<\ep/4\rforal b\in {\cal F}.
\eneq
Since $C'$ is simple, we may also assume
that $L'(f_{1/2}(a_0))$ is full in $B_k.$
We estimate that, by \eqref{Wv-33} and by \eqref{Wv-39},
\beq\label{Wv-41}
\|1_{m+1}\otimes 1_4\otimes b-{\rm diag}(L_0'(b), L'(b))\|<\ep\rforal b\in {\cal F}.
\eneq
By \eqref{Wv-32+1},
\beq\label{Wv-42}
\tau(L'(a_0))\ge 1/12(m+1)\rforal \tau\in T(M_{4(m+1)}(B)).
\eneq

\end{proof}


%


Combing the above with \ref{LWL1} and \ref{Subcontsc}, we obtain the following:
\begin{cor}\label{TWLQ}
Let $A$ be a non-unital   separable simple \CA\,
with $K_0(A)={\rm ker}\rho_A,$
with stable rank one  and with ${\rm dim}_{nuc} A=m<\infty$
which has  continuous scale.  Fix a strictly positive element $a\in A_+$ with $0\le a\le 1.$

Let $C\in {\cal M}_0$ be a non-unital simple \CA\,
which is quasi-compact with $K_0(C)=\{0\},$ $K_1(A)=\{0\}$ and
$T(C)=T(A).$
Suppose that there is a sequence of \cpc s $\sigma_n: A\to C$ and
\hm s $\rho_n: C\to A$ such that
\beq\label{CCWv}
&&\lim_{n\to\infty}\|\sigma_n(ab)-\sigma_n(a)\sigma_n(b)\|=0\tforal a, b\in A,\\\label{CCWv-2}
&&\lim_{n\to\infty}\sup\{|\tau(\rho_n\circ \sigma_n(a))-\tau(a)|:\tau\in T(A)\}=0\tand
\eneq
$\sigma_n(a)$ is a strictly positive in $C$ for all $n\in \N.$

Then $A\otimes U\in {\cal D}_{0}$ for any UHF-algebra $U.$

\end{cor}

\begin{thm}\label{TTTAD}
Let $A$ be a non-unital   separable simple \CA\, 
with $K_0(A)={\rm ker}\rho_A,$
which has almost stable rank one  and with ${\rm dim}_{nuc} A=m<\infty$
which has  continuous scale.

Let $C\in {\cal M}_0$ be a non-unital simple \CA\,  which is quasi-compact with $K_0(C)=\{0\},$ $K_1(A)=\{0\}$ and 
$\gamma: T(C)\to T(A)$ be  an affine homeomorphism. 
Suppose that there is a sequence of \cpc s $\sigma_n: A\to C$ and 
\hm s $\rho_n: C\to A$ such that
\beq\label{TTWv-1}
&&\lim_{n\to\infty}\|\sigma_n(ab)-\sigma_n(a)\sigma_n(b)\|=0\tforal a, b\in A\andeqn\\\label{TTWv-2}
&&\lim_{n\to\infty}\sup\{|\tau (\rho_n\circ \sigma_n(a))-\gamma(\tau)(a)|:\tau\in T(A)\}=0.
\tand
\eneq
$\sigma_n(a)$ is a strictly positive in $C$ for all $n\in \N.$
Suppose that every non-zero hereditary \SCA\, $A$ has the tracially approximate divisible property.

Then $A\in {\cal D}_{0}.$ 
\end{thm}

\begin{proof}
We first notice that  the existence of $\rho_n$ and 
the condition \eqref{CCWv-2} also holds by \cite{Rl}.




Let $B$ be a hereditary  \SCA\, with continuous scale.
Fix a strictly positive element $e\in B$ with $\|e\|=1$ and positive element $e_1\in B$ 
such that $ee_1=e_1=ee_1$ with $d_\tau(e_1)> 1-1/64(m+2)$ for all $\tau\in T(B).$
Let $\ep>0,$ ${\cal F}\subset B$ be a finite subset and let $b\in B_+\setminus \{0\}.$
Choose $b_0\in B_+\setminus \{0\}$ and $16(m+1)\la b_0\ra \le b$ in 
$Cu(A).$ 
There are 
$e_0\in B_+$ and a   $\sigma$-unital  hereditary \SCA\,
$A_0$ of $A$ such that $e_0\perp M_n(A_0)$ $e_0\lesssim b_0$ and 
$$
{\rm dist}(x, B_{1,d})<\ep/64(m+1)\rforal x\in {\cal F}\cup \{e\},
$$
where $B_{1,d}\subset  \overline{e_0Be_0}\oplus M_n(A_0)\subset  B$ and 
\beq\label{TTTAD-1}
B_{1,d}=\{ \diag( x_0, (\overbrace{x_1,x_1,...,x_1}^{4(m+1)}):x_0\in \overline{e_0Be_0},\, x_1\in A_0\}.
\eneq
and strictly positive elements of $B_{1,d}$ are strictly positive elements of $M_n(D).$
Moreover,  $A_0$ has continuous scale.

\Wlog, we may assume that ${\cal F}\subset B.$
Write 
$$
x=\diag(\overbrace{x_1, x_1,....,x_1}^{4(m+1)}).
$$
Let ${\cal F}_1=\{x_1: x\in {\cal F}\}.$ 
Note that we may write $\diag(\overbrace{x_1, x_1,....,x_1}^{4(m+1)})=x_1\otimes 1_{4(m+1)}.$
Then ${\rm dim}_{nuc}A_0=m$ (see \cite{WZ}).
Also $A_0$ is a non-unital separable simple \CA\, 
with $K_0(A_0)={\rm ker}\rho_{A_0}$ which has continuous scale.
We then apply \ref{TWv}   to $A_0.$ 
Thus, for $\eta={1\over{16(m+1)}},$ the conditions in \ref{LWL1} are satisfied. 
We then   apply \ref{LWL1}.

\end{proof}

\section{\CA s ${\cal W}$ and class ${\cal D}_{0}$}

%


\begin{df}\label{DWtrace}
Let $A$ be a non-unital separable \CA. Suppose that
$\tau\in T(A).$
We say that $\tau$ is a $C_0$-trace if there exists a sequence of
\cpc s $\{\phi_n\}$ from $A$ into $D_n\in {\cal C}_0$ such that
\beq\nonumber
&&\lim_{n\to\infty}\|\phi_n(ab)-\phi_n(a)\phi_n(b)\|=0\rforal a,\,b\in A\andeqn\\
&&\tau(a)=\lim_{n\to\infty}t_n(\phi_n(a))\rforal a\in A,
\eneq
where $t_n\in  T(D_n),$ $n=1,2,....$

\vspace{0.1in}

We say that $\tau$ is a $W$-trace
if there exists a sequence of
\cpc s $\{\phi_n\}$ from $A$ into $W$ such that
such that
\beq\nonumber
&&\lim_{n\to\infty}\|\phi_n(ab)-\phi_n(a)\phi_n(b)\|=0\rforal a,\,b\in A\andeqn\\
&&\tau(a)=\lim_{n\to\infty}t_0(\phi_n(a))\rforal a\in A,
\eneq
where $t_0$ is the unique tracial state on $W.$

\end{df}

\begin{prop}\label{wxtrace}
Let $A$ be a non-unital separable simple \CA\, and
let $\tau\in T(A).$
The following are equivalent.

{\rm (1)}  $\tau$ is a $W$-trace.

{\rm (2)} There exists a sequence of
\cpc s $\{\phi_n\}$ from $A$ into $D_n\in {\cal C}_0^{0'}$ such that
such that
\beq\nonumber
&&\lim_{n\to\infty}\|\phi_n(ab)-\phi_n(a)\phi_n(b)\|=0\rforal a,\,b\in A\andeqn\\
&&\tau(a)=\lim_{n\to\infty}t_n(\phi_n(a))\rforal a\in A,
\eneq
where $t_n\in T(D_n).$
\end{prop}

\begin{proof}
Suppose that (1) holds.  Since $W$ is an inductive limit of \CA s in $C_0^0,$
(2) holds immediately.

Suppose (2) holds.
Let $\Gamma_n: T(W)=\{t_0\}\to T(D_n)$ by
$\Gamma_n(t_0)=t_n.$ Then there is affine map
$\gamma_n: {\rm LAff}_b(T(D_n))\to {\rm LAff}_b(T(W))=\R$ defined by
$\gamma_n(f)(t_0)=f(t_n),$ $n=1,2,....$
This induces a \hm\, from $\gamma_n^*: Cu^{\sim}(D_n)\to Cu^{\sim}(W).$
It follows from \cite{Rl} that there is a \hm\,
$\psi_n: D_n\to W$ such that $Cu^{\sim}(\psi_n)=\gamma_n^*.$
In particular,
$$
t_0\circ \psi_n(a)=t_n(a)\rforal a\in (D_n)_+.
$$
Define $\Psi_n: A\to W$ by $\Psi_n=\psi_n\circ \phi_n.$
It follows that
\beq
\lim_{n\to\infty}t_0\circ \Psi_n(a)&=&\lim_{n\to\infty}t_0\circ \psi_n\circ \phi_n(a)\\
&=&
\lim_{n\to\infty}t_n\phi_n(a)=\tau(a)\rforal a\in A.
\eneq
\end{proof}

\begin{thm}\label{Ttracekerrho}
Let $A$ be a non-unital separable simple \CA\, which is quasi-compact.
If every tracial state $\tau\in T(A)$ is a $W$-trace, then
$K_0(A)={\rm ker}_A.$
\end{thm}

\begin{proof}
Fix $\tau\in T(A).$
Suppose that there are two projections $p, q\in M_k({\tilde A})$ such that
$x=[p]-[q]\in K_0(A)$ and $\tau(p)\not=\tau(q).$

Let $d=|\tau(p)-\tau(q)|.$
Note that $\pi(p)$ and $\pi(q)$ have the same rank in $M_k(\C),$
where $\pi: M_k({\tilde A})\to M_k(\C)$ is the quotient map.

Denote still by $\tau$ the extension of $\tau$ on ${\tilde A}$ as well as
on $M_n({\tilde A}).$
If $\tau$ were $W$-trace, there there would be a sequence $\{\phi_n\}$
of \cpc s from $M_k(A)$ into $M_k(W)$  such that
\beq\nonumber
&&\lim_{n\to\infty}\|\phi_n(a)\phi_n(b)-\phi_n(ab)\|=0\rforal a, b\in M_k(A)\andeqn\\
&&\tau(a)=\lim_{n\to\infty}t_0\circ\phi_n(a)\rforal a\in M_k(A),
\eneq
where $t_0$ is the unique tracial state on $W.$
Let ${\tilde \phi}_n: A\to M_k(W)$ be
an extension of \cpc\, such that ${\tilde \phi_n}(1)=1_{\tilde W}.$
It follows that
$$
\lim_{n\to\infty}\|{\tilde \phi}_n(a){\tilde \phi}(b)-{\tilde \phi}(ab)\|=0\rforal a, b\in M_k({\tilde A}).
$$
Let $t_0$ also denote the extension of $t_0$ on $M_k({\tilde W}).$
Then we also have
$$
\tau(a)=\lim_{n\to\infty} t\circ {\tilde \phi}_n(a)\rforal a\in M_k({\tilde A}).
$$
To simplify notation, \wilog, we may assume that
\beq\label{Ttkr-1}
|t\circ{\tilde \phi}_n(p)-t\circ {\tilde \phi}_n(q)|\ge d/2\rforal n.
\eneq
There are projections $p_n,\, q_n\in M_k({\tilde A})$ such that
\beq\label{Ttkr-2}
&&\lim_{n\to\infty}\|{\tilde \phi}_n(p)-p_n\|=0\andeqn\\\label{Ttkr-3}
&&\lim_{n\to\infty}\|{\tilde \phi}_n(q)-q_n\|=0.
\eneq	

Since $\pi(p)$ and $\pi(q)$ has the same rank, there are $v\in M_k({\tilde A})$ such that
$\pi(v^*v)=p$ and $\pi(vv^*)=q.$
Let $\pi_w: M_k({\tilde W})\to M_k$ be the quotient map.  Then
\beq\label{Ttkr-4}
\lim_{n\to\infty}\|\pi_n\circ \phi_n(v^*v)-\pi_n(p_n)\|=0\andeqn
\lim_{n\to\infty}\|\pi_n\circ \phi_n(vv^*)-\pi_n(q_n)\|=0.
\eneq
It follows that $\pi_w(p_n)$ and $\pi_w(q_n)$ are equivalent projections in $M_k$ for all
large $n.$  Since $K_0(W)=0,$ it follows
that $[p_n]-[q_n]=0$ in  $K_0(W)$ which means
that $p_n$ and $q_n$ are equivalent in $M_k({\tilde W})$ since ${\tilde W}$ has stable rank one.
In particular,
$$
t_0(p_n)=t_0(q_n)
$$
for all sufficiently large $n.$  This contradicts with  the fact that \eqref{Ttkr-1}, \eqref{Ttkr-2} and \eqref{Ttkr-3}
hold.

\end{proof}

\begin{prop}\label{Pwtracest}
Let $A$ be a non-unital simple \CA\,  with a $W$-tracial state $\tau\in T(A).$
Let $0\le a\le 1$ be a strictly positive element of $A.$
Then there exists a sequence of \cpc s  $\phi_n: A\to W$ such that
$\phi_n(a)$ is a strictly positive element,
\beq\nonumber
&&\lim_{n\to\infty}\|\phi_n(a)\phi_n(b)-\phi_n(ab)\|=0\rforal a,\, b\in A\andeqn\\
&&\tau(a)=\lim_{n\to\infty} t_w\circ \phi_n(a)\rforal a\in A,
\eneq
where $t_w$ is the unique tracial state of $W.$
\end{prop}

\begin{proof}
We may assume
that
\vspace{-0.1in}$$
\tau(a^{1/n})\ge 1-1/2n,\,\,\, n=1,2,....
$$
Since $\tau$ is a $W$-tracial state, there exists a sequence  of \cpc s
$\psi_n: A\to W$ such that
\beq\nonumber
\lim_{n\to\infty}\|\psi_n(a)\psi_n(b)-\psi_n(ab)\|=0\rforal a,\, b\in A\andeqn\\
\tau(a)=\lim_{n\to\infty} t_w\circ \psi_n(a)\rforal a\in A.
\eneq

Put $b_n=\phi_n(a^{1/n}).$  \Wlog, passing to a subsequence of $\{\psi_n\},$ we may assume
\beq\label{Pwtst-1}
t_w(b_n)\ge 1-1/n,\,\,\, n=1,2,....
\eneq

Consider $B_n=\overline{b_nWb_n}.$  Choose a strictly positive element
$0\le e\le 1$ of $W.$   By \cite{Rl}, there is a \hm\,
$h_n: B_n\to W$ such that $[h_n(b_n)]=[e]$ in $Cu(W).$ In particular,
$h_n(b_n)$ is strictly positive.  Since $B_n$ has a unique trace,
there is $\af_n>0$ such that
\beq\label{Pwtst-2}
\af_nt_w(b)=t_w\circ h_n(b)\rforal b\in B_n,\,\,\, n=1,2,...
\eneq
Since now $h_n(b_n)$ is strictly positive,
\beq\label{Pwtst-3}
\lim_{k\to\infty}t_w\circ h_n(b^{1/k})=1.
\eneq
Since $B_n\subset W,$ \eqref{Pwtst-2} and \eqref{Pwtst-3} imply that
\beq\label{Pstst-4}
\af_n\ge 1.
\eneq

On the other hand, by \eqref{Pwtst-1}, we have that
$$
1-1/n \le {t_w\circ h_n(b_n)\over{\af_n}}\le 1/\af_n,\,\,\, n=1,2,....
$$
Therefore
$$
\af_n\le {1\over{1-1/n}}.
$$
It follows that $\lim_{n\to\infty} \af_n=1.$
Let $\phi_n=h_n\circ \psi_n.$ One verifies that $\{\phi_n\}$ meets the requirement.
\end{proof}




\begin{prop}\label{Pwtrace}
Let $A$ be a separable \CA\,  which is quasi-compact and every tracial state $\tau$ is
quasidiagonal.
Let $Z\in {\cal D}_0$ be a  simple \CA\, which is an inductive limit \CA s in ${\cal C}_0'$
such that $Z$ is quasi-compact, $K_0(Z)={\rm ker}\rho_Z$ with a unique
tracial state.
Then all  tracial states of $A\otimes Z$  and $A\otimes W$  are  $W$-tracial states.
\end{prop}

\begin{proof}
Let $\tau\in T(A).$ Denote by $t$ the unique tracial state of $Z.$
We will show $\tau\otimes t$ is a W-trace on $A\otimes Z.$

For each $n,$ there is \hm\, $h_n: M_n(Z)\to W$ (by \cite{Rl}) such that
$h_n$ maps a strictly positive element of $M_n(Z)$ to an one in $W.$
Let $t_w\in T(W).$ Then $t_w\circ h_n$ is a tracial state of $Z.$
Therefore $t(a)=t_w\circ h_n(a)$ for all $a\in M_n(Z).$
Moreover,  for any $a\in M_n$ and $b\in Z,$
$$
tr_n(a)t(b)=t_w\circ h_n(a\otimes b),
$$
where $tr_n$ is the normalized on $M_n,$ $n=1,2,....$

Since $\tau$ is quasidiagonal, there is a sequence
$\psi_n: A\to M_{k(n)}$ of \cpc s such that
\beq\label{Pwt-1}
&&\lim_{n\to\infty}\|\psi_n(ab)-\psi_n(a)\psi_n(b)\|=0\rforal a, \, b\in A\andeqn\\
&&\tau(a)=\lim_{n\to\infty}tr_{k(n)}\circ\psi_n(a)\rforal a\in A.
\eneq

Define $\phi_n: A\otimes Z\to W$
by $\phi_n(a\otimes b)=h_{k(n)}\circ  (\psi_n(a)\otimes b)$ for all $a\in A$ and $b\in Z.$
Then $\phi_n$ is \cpc\, and, for any $a\in A$ and $b\in Z,$
\beq\label{Pwt-3}
(\tau\otimes t)(a\otimes b)&=&\lim_{n\to\infty} tr_{k(n)}(\psi_n(a))t(b)\\
&=& \lim_{n\to\infty} t_w\circ h_n(\psi_n(a)\otimes b)\\
&=& \lim_{n\to\infty} t_w\circ h_n(\phi_n(a\otimes b).
\eneq
Therefore $\tau\otimes t$ is $W$-trace.

A similar proof shows that every tracial state of $A\otimes W$ is also a $W$-tracial state.
\end{proof}

\begin{thm}\label{TMW}
Let $A$ be a non-unital separable simple \CA\, with finite nuclear dimension  which is quasi-compact such that
$T(A)\not=\emptyset$ and
every tracial state is a $W$-trace.
Suppose also that $K_0(A)={\rm ker}\rho_A.$
Then $A\otimes U\in {\cal D}_{0}$ for any UHF-algebra $U.$
\end{thm}

\begin{proof}
By \ref{Rcontext}, $A$ has a hereditary \SCA\, which has continuous scale.
It follows that, \wilog, we may assume that $A$ has continuous scale.

It follows from \ref{R2} that there is a non-unital simple \CA\, $B=\lim_{n\to\infty} (B_n, \imath_n),$
where each $B_n$ is a finite direct sum of $W$  and
$\imath_{n, \infty}$ maps strictly positive elements to strictly positive elements such that $B_n$ is quasi-compact and
\beq\label{TMW-1}
T(B)\cong T(A).
\eneq
Note that  $T(A)$  is a  metrizable Choquet simplex.
Denote by $\gamma: T(B)\to T(A)$ the affine homeomorphism.
Let $(\imath_{n, \infty})_{\sharp}: T(B)\to T(B_n)$
such that
$$
t\circ \imath_{n, \infty}(b)=(\imath_{n, \infty})_{\sharp}(t)(b)
$$
for all $b\in B_n,$ $n=1,2,....$  It is an affine continuous map.
Note that $T(W)=\{ t_w\},$ where
$t_w$ is the unique tracial state of $W.$

Fix a strictly positive element $a_0\in A.$
Fix $\ep>0$ and a finite subset ${\cal F}\subset A,$
Since $B=\lim_{n\to\infty} (B_n, \imath_n),$ then, it is standard  and easy to see
that there is  an integer $n_1\ge 1$ and $\kappa: T(B_{n_1})\to T(A)$
such that
\beq\label{TMW-2}
\sup_{\tau\in  T(B)}|\kappa\circ (\imath_{n_1, \infty})_{\sharp}(\tau)(f)-\gamma(\tau)(f)|<\ep/2\rforal f\in {\cal F}.
\eneq
Write $B_{n_1}=W_1\oplus W_2\oplus \cdots \oplus W_{m},$
where each $W_i\cong W.$
Denote by $t_{w,1}, t_{w,2},...,t_{w,m}$ the unique tracial states on $W_i,$
and $\theta_i=\kappa(t_{w,i}),$ $i=1,2,...,m.$
By the assumption, there exists, for each $i,$ a sequence  of \cpc\,
$\phi_{n,i}: A\to W_i$ such that
\beq\label{TMW-3}
&&\lim_{n\to\infty}\|\phi_{n,i}(a)\phi_{n,i}(b)-\phi_{n,i}(ab)\|=0\rforal a,b\in A\andeqn\\
&&\theta_i(a)=\lim_{n\to\infty}t_{w,i}\circ \phi_{n,i}(a)\rforal a\in A.
\eneq
Moreover, by \ref{Pwtracest}, we may assume that $\phi_{n,i}(a_0)$ is strictly positive.
Define $\phi_n: A\to B_{n_1}$ by
\vspace{-0.12in}\beq\label{TMW-4}
\phi_n(a)= \phi_{n,1}(a)\oplus \phi_{n,2}(a)\oplus\cdots \phi_{n,m}(a)\rforal a\in A.
\eneq
Then
\vspace{-0.1in}\beq\label{TMW-4+}
\lim_{n\to\infty}\sup_{\tau\in T(B_{n_1})}\{|\tau(\phi_n(a))-\kappa(\tau)(a)|\}=0
\rforal a\in A.
\eneq
Define $\psi_n: A\to B$ by
\vspace{-0.12in}\beq\label{TMW-5}
\psi_n(a)=\imath_{n_1,\infty}\circ \phi_n(a)\rforal a\in A.
\eneq
Note that $\psi_n(a_0)$ is a strictly positive element.
We also have  that
\beq\label{TMW-6}
\lim_{n\to\infty}\|\psi_n(ab)-\psi_n(a)\psi_n(b)\|=0\rforal a,\, b\in A.
\eneq
Moreover,  for any $\tau\in T(B)$ and any $f\in {\cal F},$
\beq\label{TMW-7}
|\gamma(\tau)(f)-\tau\circ \psi_n(f)| &\le & |\gamma(\tau)(f)-\kappa\circ (\imath_{n_1, \infty})_{\sharp}(\tau)(f)|\\
&&+|\kappa\circ (\imath_{n_1, \infty})_{\sharp}(\tau)(f)-\tau\circ \psi_n(f)|\\
&<&\ep/2+|\kappa\circ (\imath_{n_1, \infty})_{\sharp}(\tau)(f)-\tau\circ \imath_{n_1, \infty}\circ \phi_n(f)|\\
&\le & \ep/2+\sup_{t\in T(B_{n_1})}\{|\tau(\phi_n(f))-\kappa(\tau)(f)|\}.
\eneq
By \eqref{TMW-4+}, there exists $N\ge 1$ such that, for all $n\ge N,$
\beq\label{TMW-8}
\sup_{\tau\in T(B)}\{|\gamma(\tau)(f)-\tau\circ \psi_n(f)|\}<\ep\rforal f\in {\cal F}.
\eneq
Thus,  we obtain a sequence of \cpc s $\sigma_n: A\to B$ such that \eqref{TTWv-1}
and \eqref{TTWv-2} hold.
Thus \ref{TTTAD} applies.
\end{proof}

\begin{thm}\label{TWtrace}
Let $A$ be a non-unital separable simple  \CA\, with finite nuclear dimension  which is quasi-compact.
Suppose that $T(A)\not=\emptyset.$
Then $A\otimes Z, A\otimes W\in {\cal D}_{0},$
where $Z$ is as described in \ref{Pwtrace}.
\end{thm}

\begin{proof}
It follows \ref{Pwtrace} that every tracial state of $A\otimes Z$ is a $W$-trace.  It follows from \ref{Ttracekerrho}
that $K_0(A\otimes Z)=\rho_{A\otimes Z}.$ Then \ref{TMW} applies.
\end{proof}

\section{Classification of simple \CA s with zero $K_0$ and $K_1$}

\begin{thm}\label{TTMW}
Let $A$ and $B$ be two separable simple amenable \CA s with  continuous scale.
Suppose that both $A$ and $B$ are ${\cal D}$ such that
$KK(A, D)=KK(B,D)=\{0\}.$
Then  $A\cong B$ if and only if
there is an affine homeomorphism $\gamma: T(A)\to T(B).$
Moreover,  there is an isomorphism $\phi: A\to B$ such that
$\phi_T=\gamma^{-1}.$
\end{thm}

\begin{proof}
By \ref{R2},
there exists a simple \CA\, $C=\lim_{n\to\infty} (C_n, \imath_n),$  where each
$B_n$ is a finite direct sum of $W$ and $\imath_n$ maps strictly positive elements to
strictly positive elements, which has continuous scale and
$$
T(A)=T(C).
$$
It suffices to show that $A\cong B.$
We will use $\Gamma: T(C)\to T(A)$ for  the  affine homeomorphism given above.
We will use an approximate intertwining argument of Elliott.

{\bf Step 1}: Construction of $L_1.$

\noindent
Fix a finite subset ${\cal F}_1\subset A$ and $\ep>0.$
\Wlog, we may assume that ${\cal F}_1\subset A^{\bf 1}.$

 Since $A$ has the continuous scale,  $A$ is quasi-compact.
Choose  a strictly positive element $a_0\in A_+$ with $\|a_0\|=1.$
We may assume, \wilog,
that
\beq\label{TC0k-nn1}
a_0y=ya_0=y,\,\, a_0\ge y^*y\andeqn  a_0\ge yy^* \rforal y\in {\cal F}_1
\eneq
Let $T: A_+\setminus \{0\}\to \N\times \R$
 with $T(a)=(N(a), M(a))$ ($a\in A_+\setminus \{0\}$) be the map given
 by \ref{Cuniformful}.

 Let $\dt_1>0$ (in place of  $\dt$), let
 ${\cal G}_1\subset A$ (in place of ${\cal G}$) be a finite subset,
 let ${\cal H}_{1,0}\subset A_+\setminus \{0\}$ (in place of ${\cal H}$)
 be a finite subset, and let $K_1\ge 1$ (in place of $K$) be an integer
 given by \ref{CLuniq} (together with \ref{Rstuniq} and \ref{Ruct}) for the above $T$ (in place of $F$), $\ep/16$ (in place of $\ep$) and ${\cal F}_1.$

 \Wlog, we may assume that ${\cal F}_1\cup {\cal H}_{1,0}\subset {\cal G}_1\subset A^{\bf 1}.$

 Choose $b_0\in A_+\setminus \{0\}$ with
 $d_\tau(b_0)<1/8(K_1+1).$

 It follows from  and \ref{Cuniformful}  (and \ref{D=D0K0}) that there are ${\cal G}_1$-$\dt_1/64$-multiplicative \cpc s
 $\phi_0: A\to A$ and $\psi_0: A\to D$ for some $D\subset A$ with
 $D\in {\cal C}_0'$  such that
 \beq\label{TCzeroK1-1}
&&\hspace{-0.4in} \|x-\diag(\phi_0(x),\overbrace{\psi_0(x), \psi_0(x),...,\psi_0(x)}^{2K_1+1})\|<\min\{\ep/128, \dt_1/128\}\rforal x\in {\cal G}_1\\
&&  \phi_0(a_0)\lesssim b_0, \|\psi_0(x)\|\ge (1-\ep/128)\|x\|\rforal x\in {\cal G}_1,\andeqn
\eneq
 $\psi_0(a_0)$ is strictly positive.
 Moreover,  $\psi_0$ is $T$- ${\cal H}_{1,0}\cup \{f_{1/4}(a_0)\}$-full
 in $\overline{DAD}.$

 Define $\phi_0': A\to A$ by $\phi_0'(a)=\diag(\phi_0(a), \psi_0(a))$ for all $a\in A.$
 Put $A_{00}=\overline{\phi_0'(a_0)A\phi_0'(a_0)}$ and $A_{0,1}=\overline{DAD},$
 let $D_{1,1}=M_{2K_1}(D)$ and $D_{1,1}'=M_{2K_1+1}(D).$
 Let $j_1: D\to M_{2K_1}(D)$ be defined by
 $$
 j_1(d)=\diag(d,d,...,d)\rforal d\in D.
 $$

 Note that
 \vspace{-0.12in}\beq\label{TC0k1-1+}
 \phi_0(a_0)\oplus \psi_0(a_0)\lesssim \psi_0(a_0)\oplus \psi_0(a_0).
 \eneq
  Let
\vspace{-0.12in} \beq\nonumber
 &&d_{00}'=\diag(\overbrace{\psi_0(a_0),\psi_0(a_0),...,\psi_0(a_0)}^{2K_1+1})\in D_{1,1}'\andeqn\\
&& d_{00}=\diag(0,\overbrace{\psi_0(a_0),\psi_0(a_0),...,\psi_0(a_0)}^{2K_1})\in D_{1,1}.
 \eneq	

 Let $\imath_1: D_{1,1}'\to A$ be the embedding.  Let $Cu^{\sim}(\imath_1): Cu^{\sim}(A_{1,1})\to Cu^{\sim}(A)$ be the induced
 map.
 Let $\Gamma^{\sim}: Cu^{\sim}(A)\to Cu^{\sim}(C)$ be
 the isomorphism given by
 $\Gamma^{\sim}(f)(\tau)=f(\Gamma(\tau))$ for all $f\in L\Aff_b(T(A))$ and $\tau\in T(A).$
 It follows Theorem 1.0.1 of \cite{Rl} that there  is a \hm\, $h_1':
 D_{1,1}'\to C$ such
that
 \beq\label{TC0k1-5}
 Cu^{\sim}(h_1')=\Gamma^{\sim}\circ Cu^{\sim}(\imath_1)\andeqn
\la  h_1'(d_{00}')\ra =\Gamma^{\sim}\circ Cu^{\sim}(\imath_1)(\la d_{00}'\ra).
 \eneq
  Let $h_1=(h_1')|_{D_{1,1}}.$  Denote by $C'=\{c\in C: ch_1(d)=h_1(d)c=0\rforal d\in D_{1,1}\}.$
 Note that
 $$
h_1'( \diag(\psi_0(a), \overbrace{0,0,...,0}^{2K_1}))\subset C'\rforal a\in A.
$$
Define  $h_0': A\to C'$ by
\vspace{-0.12in}$$
h_0'(a) =h_1'(\diag(\psi_0(a)\overbrace{0,0,...,0}^{2K_1}))\rforal a\in A.
$$
Define $L_1: A\to C$ by
\vspace{-0.16in}$$
L_1(a)=\diag(h_0'(a),h_1(\diag(\overbrace{\psi_0(a),\psi_0(a),...,\psi_0(a)}^{K_1}))\rforal a\in A.
$$

{\bf Step 2}: Construct $H_1$ and the first approximate commutative diagram.

It follows Theorem 1.0.1 of \cite{Rl} that there is a
\hm\, $H: C\to  A$ such that
\vspace{-0.1in}\beq\label{TC0k1-6}
Cu^{\sim}(H)=(\Gamma^{\sim})^{-1}.
\eneq
\vspace{-0.09in}Note that
\vspace{-0.12in}\beq\label{TC0k1-9}
\la  H\circ h_0'(a_0)\ra \le \la \phi_0(a_0)\ra.
\eneq
Choose $\dt_1/4>\eta_0>0$ such that
\beq\label{TC0k1-10}
\hspace{-0.2in}\|f_{\eta_0}(H\circ h_0'(a_0)) x-x\|,\,\,\|x-xf_{\eta_0}( H\circ h_0'(a_0))\|<
\min\{\ep/128, \dt_1/128\}
\eneq
for all $x\in  H\circ h_0'({\cal G}_1).$
Since $A$ has almost stable rank one, there is a unitary $u_0\in {\tilde A}$ such that
\beq\label{TC0k1-11}
u_0^*f_{\eta_0}( H\circ h_0'(a_0))u_0\in \overline{\psi_{00}(a_0)A\psi_{00}(a_0)},
\eneq
where
\vspace{-0.13in}$$
\psi_{00}(a_0)=\diag(\psi_0(a_0), \overbrace{0,0,...,0}^{2K_1})).
$$
Define $H': A\to \overline{\psi_{00}(a_0)A\psi_{00}(a_0)}$
by
$$
H'(a)=u_0^*(f_{\eta_0}(H\circ h_0'(a_0)))H\circ h_0'(a)(f_{\eta_0}(H\circ h_0'(a_0)))u_0\rforal a\in A.
$$
Note that $H'$ is ${\cal G}_1$-$\dt_1/32$-multiplicative \cpc.
Moreover, by \eqref{TC0k1-10},
\vspace{-0.12in}\beq\label{TC0k1-11+}
\|{\rm Ad}\, u_0\circ H\circ h_0'(a)-{\rm Ad}\, u_0\circ H'(a)\|<\min\{\ep/128, \dt_1/128\}\rforal a\in {\cal G}_1.
\eneq

Consider two \hm s
${\rm Ad}\, u_0\circ H\circ h_1\circ j_1$ and $\imath_1\circ j_1.$
Then, by \eqref{TC0k1-5} and \eqref{TC0k1-6},
\beq\label{TC0k1-7}
Cu^{\sim}({\rm Ad}\, u_0\circ H\circ h_1\circ j_1)=Cu^{\sim}(\imath_1\circ j_1).
\eneq
Put $A'=\{a\in A: a\perp \overline{\psi_{00}(a)A\psi_{00}(a)}\}.$
Note that we may view that both $\imath_1\circ j_1$ and ${\rm Ad}\, u_0\circ H\circ h_1\circ j_1$
are maps into $A'.$

It follows from Theorem 3.3.1 of \cite{Rl} that
there exists a unitary $u_1\in {\tilde A'}$ such that
\beq\label{TC0k1-8}
\|u_1^* ({\rm Ad}\, u_0\circ H\circ h_1\circ j_1(x))u_1-\imath_1\circ j_1(x)\|<\min\{\ep/16,\dt_1/16\}\rforal x\in {\cal G}_1.
\eneq
One may write $u_1=\lambda+z$ for some $z\in A'.$ Therefore
we may view $u_1$ is a unitary in ${\tilde A}.$ Note that, for any
$b\in \overline{\psi_{00}(a)A\psi_{00}(a)},$  $u_1^*bu_1=b.$
In particular, for any $a\in A,$
\beq\label{TC0k1-15}
{\rm Ad}\, u_1\circ H'(a)=H'(a)\rforal a\in A.
\eneq

By applying \ref{CLuniq} and its remarks (including \ref{Ruct}),  there is $u_2\in {\tilde A}$
such that
\beq\label{TC0k1-16}
\|{\rm Ad}\, u_2\circ \diag(H'(a), \imath_1\circ j_1\circ \psi_0(a))-\diag(\psi_0'(a), \imath_1\circ j_1\circ \psi_0(a))\|<\ep/16
\eneq
 for all $a\in {\cal F}_1.$
 Combining \eqref{TC0k1-8}, \eqref{TC0k1-15} and \eqref{TC0k1-11+},
 we have
 \beq\label{TC0k1-17}
 \|{\rm Ad}\, u_2\circ u_1\circ u_0\circ H\circ L_1(a)-{\rm Ad}\, u_2\circ \diag(H'(a), \imath_1\circ j_1\circ \psi_0(a))u_2\|<\ep/7
 \eneq
 for all $ a\in {\cal F}_1.$
On the other hand,  by \eqref{TCzeroK1-1},
 \beq\label{TC0k1-18}
 \|{\rm id}_A(a)-\diag(\psi_0'(a), j_1\circ \psi_0(a))\|<\ep/16\rforal a\in {\cal F}_1.
 \eneq
 Put $U_1=u_0u_1u_2.$  By \eqref{TC0k1-18} and \eqref{TC0k1-17},
 we conclude that
 \beq\label{TC0k1-19}
 \|{\rm id}_A(a)-{\rm Ad}\, U_1\circ H\circ L_1(a)\|<\ep\rforal a\in {\cal F}_1.
 \eneq

 Put $H_1={\rm Ad}\, U_1\circ H.$
 Then we obtain the following diagram:
 \begin{displaymath}
\xymatrix{
A \ar[r]^{\id} \ar[d]_{L_1} & A\\
C \ar[ur]_{H_1}
}
\end{displaymath}
which is approximately commutative on the subset ${\cal F}_1$ within $\ep.$

 {\bf Step 3}: Construction of $L_2$ and the second approximate commutative diagram.

 There is an easy way to obtain a map from $A$ to $C.$ However, since we want the process continue,
 we will repeat the construction in {\bf Step 1}.

 We first back to $C.$
 Define $\Delta: C^{\bf 1, q}\setminus \{0\}\to (0,1)$ by
 \beq\label{TC0k1-20}
 \Delta(\hat{a})=(1/2) \inf\{\tau(a): \tau\in T(C)\}
 \eneq
 (Recall that  $T(C)$ is compact since $C$ has continuous scale.)

 Fix any $\eta_1>0$ and a finite subset ${\cal S}_1\subset C.$
 We may assume that ${\cal S}_1\subset C^{\bf 1}$ and $L_1({\cal F}_1)\subset {\cal S}_1.$

 Let ${\cal G}_2\subset C$ (in place of ${\cal G}$), let ${\cal H}_{1,1}\subset C_+^{\bf 1}\setminus \{0\}$
 (in place of ${\cal H}_1$) be a finite subset, let ${\cal H}_{1,2}\subset C_{s.a.}$ (in place
 of ${\cal H}_2$) be a finite subset, $\dt_2>0$ (in place of $\dt$), $\gamma_1>0$ (in place of $\gamma$)
 be as required by \ref{Lrluniq} for $C,$ $\eta_1/16$ (in place of $\ep$) and ${\cal S}_1$ (in place of ${\cal F}$)
 as well as $\Delta$ above.

 \Wlog, we may assume that ${\cal S}_1\cup {\cal H}_{1,2}\subset {\cal G}_2\subset C^{\bf 1}.$

 Fix $\ep_2>0$ and a finite subset ${\cal F}_2$ such that
 $H_1({\cal S}_1)\cup {\cal F}_1\subset {\cal F}_2.$

 We may assume that ${\cal F}_2\subset A^{\bf 1}.$

 Let
 $$
 \gamma_0=\min\{\gamma_1, \inf\{\Delta(\hat{a}): a\in {\cal H}_{1,1}\cup {\cal H}_{1,2}\}\}.
 $$

Fix a strictly positive element $a_1\in A_+$ with $\|a_1\|=1.$
We may assume, \wilog,
that
\beq\label{TC0k-n1}
a_1y=ya_1=y,\,\, a_1\ge y^*y\andeqn  a_1\ge yy^* \rforal y\in {\cal F}_2
\eneq
Let $T: A_+\setminus \{0\}\to \N\times \R$
 with $T(a)=(N(a), M(a))$ ($a\in A_+\setminus \{0\}$) be the map given
 by \ref{Cuniformful} as mentioned in the {\bf Step 1}.

 Let $\dt_2'>0$ (in place of  $\dt$), let
 ${\cal G}_2\subset A$ (in place of ${\cal G}$) be a finite subset,
 let ${\cal H}_{2,0}\subset A_+\setminus \{0\}$ (in place of ${\cal H}$)
 be a finite subset, and let $K_2'\ge 1$ (in place of $K$) be an integer
 given by \ref{CLuniq} (and its remarks) for the above $T,$ $\ep_1/16$ (in place of $\ep$) and ${\cal F}_2.$

 \Wlog, we may assume that ${\cal H}_{2,0}\subset {\cal G}_2\subset A^{\bf 1}.$

 We may assume that $\dt_2'<\dt_2.$

 Choose $K_2\ge K_2'$ such that
 $1/K_2<\gamma_0/4.$

 Choose $b_{2,0}\in A_+\setminus \{0\}$ with
 \beq\label{TC0k-n2}
 d_\tau(b_{2,0})<1/8(K_2+1).
 \eneq

 It follows from \ref{Cuniformful} (and \ref{D=D0K0})
 that there are ${\cal G}_2$-$\dt_2'/64$-multiplicative \cpc s
 $\phi_{2,0}: A\to A$ and $\psi_{2,0}: A\to D_2$ for some $D_2\subset A$ with
 $D_2\in {\cal C}_0'$  such that
 \beq\label{TCzeroK1-31}
&&\hspace{-0.4in}\hspace{-0.4in} \|x-\diag(\phi_{2,0}(x),\overbrace{\psi_{2,0}(x), \psi_{2,0}(x),...,\psi_{2,0}(x)}^{2K_2+1})\|<\min\{\ep_2/128, \dt_2'/128\}\rforal x\in {\cal G}_2\\
&& \phi_{2,0}(a_1)\lesssim b_{2,0},\\
&& \|\psi_{2,0}(x)\|\ge (1-\ep_2/128)\|x\|\rforal x\in {\cal G}_2\andeqn
\eneq
 $\psi_{2,0}(a_1)$ is strictly positive.
 Moreover, every element $c\in {\cal H}_{2,0}\cup \{f_{1/4}(a_1)\}$ is
 $(N(c), M(c))$-full in $\overline{D_2AD_2}.$

 Define $\phi_{2,0}': A\to A$ by $\phi_{2,0}'(a)=\diag(\phi_{2,0}(a), \psi_{2,0}(a))$ for all $a\in A.$
 Put $A_{2, 00}=\overline{\phi_{2,0}'(a_1)A\phi_{2,0}'(a_1)}$ and $A_{2, 0,1}=\overline{D_2AD_2},$
 let $D_{2,1}=M_{2K_2}(D_2)$ and $D_{2,1}'=M_{2K_2+1}(D_2).$
 Let $j_2: D_2\to M_{2K_2}(D_2)$ be defined by
 $$
 j_2(d)=\diag(d,d,...,d)\rforal d\in D_2.
 $$

 Note that
\vspace{-0.14in}  \beq\label{TC0k1-31+}
 \phi_{2,0}(a_1)\oplus \psi_{2,0}(a_1)\lesssim \psi_{2,0}(a_1)\oplus \psi_{2,0}(a_1).
 \eneq

 Let
\vspace{-0.14in} \beq\nonumber
 &&d_{2,00}'=\diag(\overbrace{\psi_{2,0}(a_1),\psi_{2,0}(a_1),...,\psi_{2,0}(a_1)}^{2K_2+1})\in D_{2,1}'\andeqn\\
&& d_{2,00}=\diag(0,\overbrace{\psi_{2,0}(a_1),\psi_{2,0}(a_1),...,\psi_{2,0}(a_1)}^{2K_2})\in D_{2,1}.
 \eneq	

 Let $\imath_2: D_{2,1}'\to A.$  Let $Cu^{\sim}(\imath_2): Cu^{\sim}(A_{2,1})\to Cu^{\sim}(A)$ be the induced
 map.
 It follows Theorem 1.0.1 of \cite{Rl} that there  is a \hm\, $h_2':
 D_{2,1}'\to C$ such
 that
 \beq\label{TC0k1-35}
 Cu^{\sim}(h_2')=\Gamma^{\sim}\circ Cu^{\sim}(\imath_2)\andeqn
\la  h_2'(d_{2, 00}')\ra =\Gamma^{\sim}\circ Cu^{\sim}(\imath_2)(\la d_{2, 00}'\ra).
 \eneq

 Let $h_2=(h_2')|_{D_{2,1}}.$  Denote by $C''=\{c\in C: ch_2(d)=h_2(d)c=0\rforal d\in D_{2,1}\}.$
 Note that
 $$
h_2'( \diag(\psi_{2,0}(a), \overbrace{0,0,...,0}^{2K_2}))\subset C''\rforal a\in A.
$$
Define  $h_{2,0}': A\to C''$ by
\vspace{-0.12in}$$
h_{2,0}'(a) =h_2'(\diag(\psi_{2,0}(a)\overbrace{0,0,...,0}^{2K_2}))\rforal a\in A.
$$
Define $L_2: A\to C$ by
\vspace{-0.12in}$$
L_2(a)=\diag(h_{2,0}'(a),h_2(\diag(\overbrace{\psi_{2,0}(a),\psi_{2,0}(a),...,\psi_{2,0}(a)}^{K_2}))\rforal a\in A.
$$
By \eqref{TC0k-n2} and $1/K_2<\gamma_0/4,$ we compute that
\beq\label{TC0k1-36}
\sup\{ |\tau\circ L_2\circ {\rm Ad}\, U_1\circ H(x)-\tau(x)|: \tau\in T(C)\}<\gamma_0\rforal x\in {\cal G}_2.
\eneq
This also implies that
\beq\label{TCk1-37}
\tau(L_2\circ {\rm Ad}\, U_1\circ H(b))\ge \Delta(\hat{b})\rforal b\in {\cal H}_{1,1}.
\eneq
Note also that we have assume that $K_0(C)=K_1(C)=\{0\}.$
Thus \ref{Lrluniq} applies. Therefore we obtain a unitary $V_1\in {\tilde C}$ such that
\beq\label{TCk1-38}
\|{\rm Ad}\, V_1\circ L_2\circ H_1(a)-{\rm id}_{C}(a)\|<\eta_1/2\rforal a\in {\cal S}_1.
\eneq
Put $L_2={\rm Ad}\, V_1\circ L_2.$
Then, we obtain the following diagram:
 \begin{displaymath}
\xymatrix{
A \ar[r]^{\id} \ar[d]_{L_1} & A \ar[d]^{L_2}\\
C \ar[ur]_{H_1}\ar[r]_\id & C
}
\end{displaymath}
with the upper triangle approximately commutes on $\mathcal F_1$ up to $\ep$ and the lower triangle approximately commutes on $\mathcal S_1$ up to $\eta_1$.

  {\bf Step 4}: Show the process continues.

 We will repeat the argument in {\bf Step 2}.
 Note  that
\vspace{-0.1in}\beq\label{TC0k1-49}
\la  H\circ h_{2,0}'(a_1)\ra \le \la \phi_{2,0}(a_1)\ra.
\eneq
Choose $\dt_2/4>\eta_1>0$ such that
\beq\label{TC0k1-50}
\hspace{-0.2in}\|f_{\eta_1}(H\circ h_{2,0}'(a_1)) x-x\|,\,\,\|x-xf_{\eta_1}( H\circ h_{2,0}'(a_1))\|<
\min\{\ep_2/128, \dt_2'/128\}
\eneq
for all $x\in  H\circ h_{2,0}'({\cal G}_2).$
Since $A$ has almost stable rank one, there is a unitary $u_{2,0}\in {\tilde A}$ such that
\beq\label{TC0k1-51}
u_{2,0}^*f_{\eta_1}( H\circ h_{2,0}'(a_1))u_{2,0}\in \overline{\psi_{2,00}(a_1)A\psi_{2,00}(a_1)},
\eneq
where
\vspace{-0.1in}$$
\psi_{2,00}(a_1)=\diag(\psi_{2,0}(a_1), \overbrace{0,0,...,0}^{2K_2})).
$$
Define $H'': A\to \overline{\psi_{2,00}(a_1)A\psi_{2,00}(a_1)}$
by
$$
H''(a)=u_{2,0}^*(f_{\eta_1}(H\circ h_{2,0}'(a_1)))H\circ h_{2,0}'(a)(f_{\eta_1}(H\circ h_{2,0}'(a_1)))u_{2,0}\rforal a\in A.
$$
Note that $H''$ is ${\cal G}_2$-$\dt_2'/32$-multiplicative \cpc.
Moreover, by \eqref{TC0k1-50},
\beq\label{TC0k1-52+}
\|{\rm Ad}\, u_{2,0}\circ H\circ h_{2,0}'(a)-{\rm Ad}\, u_{2,0}\circ H''(a)\|<\min\{\ep_2/128, \dt_2'/128\}\rforal a\in {\cal G}_2.
\eneq

Consider two \hm s
${\rm Ad}\, u_{2,0}\circ H\circ h_{2,1}\circ j_2$ and $\imath_2\circ j_2.$
Then, by \eqref{TC0k1-35} and \eqref{TC0k1-36},
\beq\label{TC0k1-57}
Cu^{\sim}({\rm Ad}\, u_{2,0}\circ H\circ h_2\circ j_2)=Cu^{\sim}(\imath_2\circ j_2).
\eneq

Put $A''=\{a\in A: a\perp \overline{\psi_{2,00}(a)A\psi_{2,00}(a)}\}.$
Note that we may view that both $\imath_2\circ j_2$ and ${\rm Ad}\, u_{2,0}\circ H\circ h_2\circ j_2$
are maps into $A''.$

It follows from Theorem 3.3.1 of \cite{Rl} that
there exists a unitary $u_{2,1}\in {\tilde A''}$ such that
\beq\label{TC0k1-8n}
\hspace{-0.2in}\|u_{2,1}^* ({\rm Ad}\, u_{2,0}\circ H\circ h_2\circ j_2(x))u_{2,1}-\imath_2\circ j_2(x)\|<\min\{\ep_2/16,\dt_2'/16\}\rforal x\in {\cal G}_2.
\eneq
One may write $u_{2,1}=\lambda+z'$ for some $z'\in A'.$ Therefore
we may view $u_{2,1}$ is a unitary in ${\tilde A}.$ Note that, for any
$b\in \overline{\psi_{2,00}(a)A\psi_{2,00}(a)}$  $u_{2,1}^*bu_{2,1}=b.$
In particular, for any $a\in A,$
\beq\label{TC0k1-15n}
{\rm Ad}\, u_{2,1}\circ H''(a)=H''(a)\rforal a\in A.
\eneq

By applying \ref{CLuniq} and its remark,  there is $u_{2,2}\in {\tilde A}$
such that
\beq\label{TC0k1-16n}
\|{\rm Ad}\, u_{2,2}\circ \diag(H''(a), \imath_2\circ j_2\circ \psi_{2,0}(a))u_{2,2}-\diag(\psi_{2,0}'(a), j_2\circ \psi_{2,0}(a))\|<\ep_2/16
\eneq
 for all $a\in {\cal F}_2.$
 Combining \eqref{TC0k1-8n}, \eqref{TC0k1-15n} and \eqref{TC0k1-50},
 we have
 \beq\label{TC0k1-17n}\hspace{-0.3in} \|{\rm Ad}\, u_{2,2}\circ u_{2,1}\circ u_{2,0}\circ H\circ L_2(a)-{\rm Ad}\, u_{2,2}\circ \diag(H''(a), \imath_2\circ j_2\circ \psi_{2,0}(a))u_{2,2}\|<\ep_2/7
 \eneq
 for all $ a\in {\cal F}_2.$

On the other hand,  by \eqref{TCzeroK1-31},
 \beq\label{TC0k1-18n}
 \|{\rm id}_A(a)-\diag(\psi_{2,0}'(a), j_2\circ \psi_{2,0}(a))\|<\ep_2/16\rforal a\in {\cal F}_2.
 \eneq
 Put $U_2=u_{2,0}u_{2,1}u_{2,2}.$  By \eqref{TC0k1-18n} and \eqref{TC0k1-17n},
 we conclude that
 \beq\label{TC0k1-19n}
 \|{\rm id}_A(a)-{\rm Ad}\, U_2\circ H\circ L_2(a)\|<\ep_2\rforal a\in {\cal F}_2.
 \eneq
 Thus we expand the diagram to the following:
  \begin{displaymath}
\xymatrix{
A \ar[r]^{\id} \ar[d]_{L_1} & A \ar[d]^{L_2}  & \hspace{-0.3in}{\stackrel{\id}{\longrightarrow}} A\\
C \ar[ur]_{H_1}\ar[r]_\id & C
\ar[ur]_{H_2}
}
\end{displaymath}
where the last triangle is approximately commutative on ${\cal  F}_2$ within $\ep_2.$

We then continue to construct $L_3$ and the Elliott approximately intertwining argument
shows that $A$ and $C$ are isomorphic.

 \end{proof}

\begin{cor}\label{D00=D}
Let $A\in {\cal D}$
with  $K_0(A)=K_1(A)=\{0\}.$
Then $A\in {\cal D}_{0}.$
\end{cor}

\begin{cor}\label{CWWW}
$W\otimes W\cong W.$
\end{cor}

\begin{proof}
Note that $W\otimes Q\cong Q.$ It follows from \ref{TMW} that $W\otimes W\in {\cal D}.$ Then Theorem \ref{TTMW} applies.
\end{proof}

\begin{cor}\label{CZtW}
Let $A$ be a non-unital separable amenable  simple \CA\, with continuous scale  and
satisfy the UCT. Suppose that $A\in {\cal D}$ with $K_0(A)={\rm ker}\rho_A.$
Suppose that $B\in {\cal D}_{0}$ which satisfies the UCT  with continuous scale  and
with $K_0(B)=K_1(B)=\{0\},$ and suppose
that there is an affine homeomorphism $\gamma: T(A)\to T(B).$
Then there is an embedding $\phi: A\to B$ such that $\phi_T=\gamma^{-1}.$
\end{cor}

\begin{proof}
The proof of this corollary is contained in the  proof of the previous theorem but much simpler.
Note since ${\rm ker}\rho_A=K_0(A),$
then, in the previous proof,  $\Gamma$ indeed gives a \hm\, from
$Cu^{\sim}(A)$ to $Cu^{\sim}(C).$
We simply omit the construction of $H_1$ and keep step 1  and step 3
(in the step one we ignore anything related to step 2).  The one-sided Elliott intertwining
will give us a \hm\, from $A$ to $C.$
\end{proof}


{\bf The proof of Theorem \ref{TTCLM}}:

\begin{proof}
Let $\Gamma:({\tilde T}(A), \Sigma_A)\cong ({\tilde T}(B), \Sigma_B)$  be given by the above statement.
Let $a\in P(A)_+$ with $\|a\|=1$ such that $A_0=\overline{aAa}$ has the continuous scale.
Then $T(A_0)$ is a  metrizable Choquet simplex. We may view that ${\tilde T}(A)$ as a cone with the base
$T(A_0).$
Let $b\in B_+$ such that
$$
d_{\Gamma(\tau)}(b)=d_\tau(a)\rforal \tau\in {\tilde T}(A).
$$
Let $B_0=\overline{bBb}.$ Then
$\Gamma$ gives an affine homeomorphism from $T(A_0)$ onto $T(B_0).$
It follows from \ref{TTMW} that there is an isomorphism
$\phi: A_0\to B_0$ such that
$\phi_T$ gives $\Gamma|_{T(A_0)}.$   By \cite{Br1}, this induces
an isomorphism ${\tilde \phi}: A\otimes {\cal K}\to B\otimes {\cal K}.$
Fix a strictly positive element $a_0\in A$ with $\|a_0\|=1$ such
that
$$
d_\tau(a_0)=\Sigma_A(\tau)\rforal \tau\in {\tilde T}(A).
$$
Let $b_0={\tilde \phi}(a_0).$
Then ${\tilde \phi}$ gives an isomorphism from $A$ to $B_1=\overline{b_0(B\otimes {\cal K})b_0}.$
Let $b_1\in B$ be a strictly positive element with $\|b_1\|=1$ such that
$$
d_\tau(b_1)=\Sigma_B(\tau) \rforal \tau\in {\tilde T}(B).
$$
Then
\vspace{-0.1in}$$
d_\tau(b_1)=d_\tau(b_0)\rforal \tau\in {\tilde T}(B).
$$
Since $B$ is a  separable simple \CA\, with stable rank one,  this implies
there exists $\phi_1: B_1\to B$ such that $(\phi_1)_T={\rm id}_{\tilde{T(B)}}.$
Then $\phi_1\circ {\tilde \phi}_A$ gives a  required isomorphism.

\end{proof}

\begin{rem}\label{Ruct2}
The assumption of that both $A$ and $B$ satisfy the UCT can be replaced by
$KK(A,D)=KK(B, D)=\{0\}$ as stated in \ref{TTMW}.
\end{rem}

\begin{cor}\label{CZtWg}
Let $A$ be a non-unital separable amenable  simple \CA\,   and
satisfy the UCT. Suppose that $A\in {\cal D}$ with $K_0(A)={\rm ker}\rho_A.$
Suppose that $B\in {\cal D}$ which satisfies the UCT   and
with $K_0(B)=K_1(B)=\{0\}.$
Suppose also
that there is a cone  homeomorphism   $\gamma: {\tilde T}(A)\to {\tilde T}(B)$
such that
$\Sigma_B\circ \gamma^{-1}=\Sigma_A.$
Then there is an embedding $\phi: A\to B$ such that $\phi_T=\gamma^{-1}$
and $d_\tau\circ \phi(a)=\Sigma_B(\tau)$ for all $\tau\in {\tilde T}(B),$
where $a\in A$ is a strictly positive element.
\end{cor}

\begin{cor}\label{CM0=D0}
Every separable amenable simple \CA\,  $A$ in ${\cal D}$ with $K_i(A)=\{0\}$ ($i=0,1$)
which satisfies the UCT is isomorphic to a \CA\, in ${\cal M}_0.$
\end{cor}

\begin{rem}\label{RTTCLM}
It should be noted that $A$ is not assumed to have finite nuclear dimension.
However, this brings the question when  $gTR(A)\le 1?$
Theorem \ref{TMW} states that, if $A$ has finite nuclear dimension and if we also assume
that every traces is a $W$ trace then $A$ has $gTR(A)\le 1.$
In particular, $gTR(A\otimes W)\le 1.$
\end{rem}

Thus we have the following corollary:

\begin{cor}\label{CCNinW}
Let $A$ be a separable simple \CA\, with finite nuclear dimension  and with non-zero
traces which satisfies the UCT.
Then $A\otimes W$ is isomorphic to a \CA\, in ${\cal M}_0.$
\end{cor}


\section{Finite nuclear dimension case}

\begin{rem}\label{RbC05}
Let $A$ be a non-unital separable \CA.
Since ${\tilde A}\otimes Q$ is unital, we may view ${\widetilde{A\otimes Q}}$ as
a \SCA\, of ${\tilde A}\otimes Q$ with the unit $1_{{\tilde A}\otimes Q}.$
In the following corollary we use $\imath$ for the embedding
from $A\otimes Q$ to ${\tilde A}\otimes Q$ as well as from
${\widetilde{A\otimes Q}}$ to ${\tilde A}\otimes Q.$
Since $K_1(Q)=\{0\},$  from the six-term exact sequence in $K$-theory, one concludes that
 the \hm\,
$\imath_{*0}: K_0(A\otimes Q)\to K_0({\tilde A}\otimes Q)$  is injective.

We will use this fact and identify $x$ with $\imath_{*0}(x)$ for all $x\in K_0(A\otimes Q)$ in the following corollary.
\end{rem}

\begin{lem}\label{LfactorQ}
Let $A$ be a non-unital separable \CA\, and let $\{\psi_n\}$ be a sequence of approximately multiplicative  \cpc  s
from ${\tilde A}\otimes Q$ to $Q,$ then $\phi_n=\psi_n\circ \imath$ is a  sequence of approximately multiplicative  \cpc  s
from $A$ into $Q,$ where $\imath: A\to {\tilde A}\otimes Q$
is the embedding defined by $a\mapsto a\otimes 1$ for all $a\in A.$

Conversely,  if $\{\phi_n\}$ is  a sequence of approximately multiplicative  \cpc  s
from $A$ to $Q,$
then, there exists a sequence of approximately multiplicative \cpc s
$\{\psi_n\}: {\tilde A}\otimes Q\to Q$
such that
$$
\lim_{n\to\infty}\|\phi_n(a)-\psi_n\circ \imath(a)\|=0\rforal a\in A,
$$
where $\imath: A\to {\tilde A}\otimes Q$
is the embedding defined by $a\mapsto a\otimes 1$ for all $a\in A.$

Moreover, if $\lim\sup \|\phi_n(a)\|\not=0$ for some $a\in A$  and if $\{e_n\}$ is an approximate
identity, then, we can choose $\psi_n$ such that
$$
tr(\psi_n(1))=d_{tr}(\phi_n(e_n)))\tforal n.
$$
\end{lem}

\begin{proof}
We prove only the second part.
Write $Q=\overline{\cup_{n=1}^{\infty} M_{n!}}$ with the embedding
$j_n: B_n:=M_{n!}\to M_{n!}\otimes M_{n+1}=M_{(n+1)!},$ $n=1,2,....$
\Wlog, we may assume
that $\phi_n$ maps $A$ into $B_n,$  $n=1,2,....$
Consider $\phi_n'(a)=\phi_n(e_n^{1/2}ae_n^{1/2}),$ $n=1,2,....$
Choose $p_n$ to be  the range projection of $\phi_n(e_n)$ in $B_n.$
Define $\psi_n': {\tilde A}\otimes Q\to Q\otimes Q$ by $\psi_n(a\otimes 1_Q)=\phi_n'(a)\otimes 1_{Q}$
for all $a\in A,$ $\psi'_n(1\otimes r)=p\otimes r$ for all $r\in Q.$
Then
$$
\lim_{n\to\infty}\|\psi_n'(a\otimes 1)-\phi_n(a)\otimes 1\|=0\rforal a\in A.
$$
Moreover, $tr(\psi_n'(1))=d_{tr}(\phi_n(e_n))$ for all $n.$
There is an isomorphism $h: Q\otimes Q\to Q$ such that
$h\circ \imath$  is approximately unitarily equivalent to ${\rm id}_Q.$
By choosing  some unitaries $u_n\in Q,$ we can choose $\psi={\rm ad}\, u_n\circ h\circ \psi_n',$
$n=1,2,....$

\end{proof}

\vspace{-0.1in}The following is a non-unital version of Lemma 4.2 of \cite{EGLN}.

\begin{lem}\label{Lpartuniq}
Let $A$ be a non-unital simple separable amenable C*-algebra  with
$T(A)\not=\emptyset$
which is quasi-compact and
which satisfies the UCT.
Fix a strictly positive element $a\in A_+$ with $\|a\|=1$
such that
\beq\label{Lpartuniq-n1}
\tau(f_{1/2}(a))\ge d\rforal \tau\in \overline{T(A)}^w.
\eneq

For any $\ep>0$ and any finite subset ${\mathcal F}$ of $A$, there exist
$\dt>0,$   a finite subset ${\mathcal G}$ of $A$,
and a finite subset ${\mathcal P}$ of $K_0(A)$ with the following property.
\noindent
Let $\psi,\phi: A\to Q$ be two ${\mathcal G}$-$\dt$-multiplicative \cpc s such that
\beq\label{puniq-1}
&&[\psi]|_{\mathcal P}=[\phi]|_{\mathcal P}\tand\\\label{puniq-1n}
&&tr(f_{1/2}(\phi_0(a)))\ge d/2 \tand tr(f_{1/2}(\phi_1(a)))\ge d/2,
\eneq
where $tr$ is the unique tracial state of $Q.$
Then there is a unitary $u\in Q$ and an  ${\mathcal F}$-$\ep$-multiplicative \cpc\, $L: A\to {\rm C}([0,1], Q)$
such that
\beq\label{puni-2}
&& \pi_0\circ L=\psi,\,\,\, \pi_1\circ L={\rm Ad}\, u\circ \phi.
\eneq
Moreover, if
\begin{equation}\label{puniq-3}
 |\mathrm{tr}\circ \psi(h)-\mathrm{tr}\circ \phi(h)|<\ep'/2\tforal  h\in {\mathcal H},
\end{equation}
for a finite set ${\mathcal H}\subset A$ and $\ep'>0$,
then $L$ may be chosen such that
\begin{equation}\label{puniq-4}
 |\mathrm{tr}\circ \pi_t\circ L(h)-\mathrm{tr}\circ \pi_0\circ L(h)|<\ep'\rforal  h\in {\mathcal H}\andeqn t\in [0, 1].
\end{equation}
Here, $\pi_t: {\rm C}([0,1], Q)\to Q$ is the point evaluation at $t\in [0,1].$
\end{lem}

\begin{proof}
 Let $T: A_+\setminus \{0\}\to \N\times \R_+\setminus \{0\}$ be given by \ref{Tqcfull}
 (with above $d$ and $a$).
 In the notation in \ref{Blbm}, $Q\in {\bf C}_{0,0, t, 1, 2},$ where $t: \N\times \N\to \N$ is defined
 to be $t(n,k)=n/k$ for all $n, k\ge 1.$  Now ${\bf C}_{0,0, t, 1,2}$ is fixed.
 We are going to apply Theorem \ref{Lauct2} together with the Remark \ref{RRLuniq} (note
 that $Q$ has real rank zero and $K_1(Q)=\{0\}$).

Let ${\mathcal F}\subset A$ be a finite subset and let $\ep>0$ be given.   We may assume that $a\in {\mathcal F}$ and every element of ${\mathcal F}$ has norm at most one.
Write  ${\mathcal F}_1=\{ab: a, b\in {\mathcal F}\}\cup {\cal F}$.

Let  $\delta_1>0$ (in place $\dt$), ${\mathcal G}_1$ (in place
of ${\cal G}$) and  ${\cal H}_1$(in place of ${\cal H}$), ${\mathcal P},$ and $K$ be as assured by
Theorem \ref{Lauct2} for ${\mathcal F}_1$ and $\ep/4$ as well as $T$ (in place of $F$). (As stated earlier we will also use the Remark \ref{RRLuniq} so that we drop ${\bf L}$  and
 condition \eqref{Lauct-1}.)
Since $K_1(Q)=\{0\}$ and
$K_0(Q)=\Q,$ we may choose ${\cal P}\subset K_0(A).$

We may also assume that $\mathcal F_1\cup {\cal H}_1\subset \mathcal G_1$ and
$K\ge 2.$

Now, let ${\cal G}_2\subset A$ (in place of ${\cal G}$) be a finite subset and
let $\dt_2>0$ (in place of $\dt_1$) given by \ref{Tqcfull} for the above ${\cal H}_1$ and
$T.$

Let $\dt=\min\{\ep/4, \dt_1/2, \dt_2/2\}$ and ${\cal G}={\cal G}_1\cup {\cal G}_2.$
\Wlog, we may assume that ${\cal G}\subset A^{\bf 1}.$

Since $Q\cong Q\otimes Q,$ we may assume, \wilog,
that $\phi(a), \psi(a)\in Q\otimes 1$ for all $a\in A.$
Pick mutually equivalent projections
$e_0, e_1,e_2,...,e_{2K}\in Q$ satisfying $\sum_{i=0}^{2K} e_i=1_Q.$
Then, consider the maps $\phi_i, \psi_i: A\to Q\otimes e_iQe_i$, $i=0, 1, ..., 2K$, which are defined by
$$\phi_i(a)=\phi(a)\otimes e_i\quad\mathrm{and}\quad \psi_i(a)=\psi(a)\otimes e_i,\quad a\in A,$$ 
and consider the maps
$$\Phi_{K+1}:=\phi=\phi_0\oplus \phi_1\oplus\cdots \oplus \phi_{2K},\quad \Phi_0:=\psi=\psi_0\oplus \psi_1\oplus \cdots \oplus \psi_{2K}$$ and
\vspace{-0.12in}$$\Phi_i:=\phi_0\oplus\cdots  \oplus \phi_{i-1}\oplus \psi_{i}\oplus \cdots \oplus \psi_{2K},\quad i=1,2,...,2K.$$
Since $e_i$ is unitarily equivalent to $e_{{0}}$ for all $i$, one has
$$[\phi_i]|_{\mathcal P}=[\psi_j]|_{\mathcal P},\quad 0\leq i, j\leq 2K.$$
and in particular, 
\vspace{-0.1in}\beq\label{puniq-10+}
[\phi_{i}]|_{\mathcal P}=[\psi_{i}]|_{\mathcal P},\quad i=0, 1, .., 2K.
\eneq

Note that, for each $i=0, 1, ..., n$, $\Phi_i$ is unitarily equivalent to
$$
\psi_{i}\oplus (\phi_0\oplus \phi_1\oplus \cdots \oplus \phi_{i-1}\oplus  \psi_{i+1}\oplus  \psi_{i+2}\oplus \cdots \oplus \psi_{2K}),
$$
and $\Phi_{i+1}$ is unitarily equivalent to
$$
 \phi_{i}\oplus (\phi_0\oplus \phi_1\oplus \cdots  \oplus \phi_{i-1}\oplus \psi_{i+1}\oplus  \psi_{i+2}\oplus \cdots \oplus \psi_{2K}).
$$

By \eqref{puniq-1n},
applying \ref{Tqcfull}, each $\phi_i$ as well as $\psi_i$ are
$T$-${\cal H}_1$-full in $e_iQe_i,$ $i=0,1,2,...,2K.$

In view of this, and \eqref{puniq-10+}, applying Theorem \ref{Lauct2} (and its remarks),
 we obtain unitaries $u_i\in Q$, $i=0, 1, ..., 2K$,  
such that
\begin{equation}\label{puniq-11}
\|{\tilde \Phi}_{i+1}(a)-{\tilde \Phi}_{i}(a)\|<\ep/4,\quad a\in {\mathcal F}_1,
\end{equation}
where
\vspace{-0.12in}$${\tilde \Phi}_0:=\Phi_0=\psi \quad\textrm{and}\quad {\tilde \Phi}_{i+1}:={\rm Ad}\,  u_{i}\circ \cdots \circ {\rm Ad}\ u_1 \circ   {\rm Ad}\ u_0\circ \Phi_{i+1},\quad i=0, 1,...,2K.$$
Put $t_i=i/(2K+1)$, $i=0, 1, ..., 2K+1$,
and define $L: A\to {\rm C}([0,1], Q)$ by
$$
\pi_t\circ L=(2K+1)(t_{i+1}-t){\tilde \Phi_i}+(2K+1)(t-t_i){\tilde \Phi_{i+1}},\quad t\in [t_i, t_{i+1}],\ i=0,1,...,2K.
$$
By construction,
\begin{equation}\label{eq-verification}
\pi_0\circ L={\tilde \Phi}_0=\psi\quad\mathrm{and}\quad \pi_1\circ L={\tilde \Phi}_{n+1}={\rm Ad}\, u_n\circ\cdots \circ{ \rm Ad}\ u_1 \circ{ \rm Ad}\ u_0 \circ \phi.
\end{equation}
Since $\tilde{\Phi}_i$, $i=0, 1, ..., 2K$, are $\mathcal G$-$\delta$-multiplicative (in particular $\mathcal F$-$\ep/4$-multiplicative), it follows from \eqref{puniq-11} that $L$ is ${\mathcal F}$-$\ep$-multiplicative. By \eqref{eq-verification}, $L$ satisfies \eqref{puni-2} with $u=u_{2K}\cdots u_1 u_0$.

Moreover, if there is a finite set $\mathcal H$ such that \eqref{puniq-3} holds, it is then also straightforward to verify that $L$ satisfies \eqref{puniq-4}, as desired.
\end{proof}

\begin{rem}\label{Ruct3}
If $KK(A, D)=\{0\}$ for all \CA s $D,$ then the assumption
that $A$ satisfies the UCT can be dropped (see \ref{Ruct}).
\end{rem}

\begin{thm}\label{TalWtrace}
Let $A$ be a non-unital separable amenable simple \CA\, with $K_0(A)={\rm Tor}(K_0(A)$ and with
$T(A)\not=\emptyset$ which satisfies the UCT.  Suppose that $A$ is quasi-compact.
Then every trace in $\overline{T(A)}^w$  is  a $W$-trace.
\end{thm}

\begin{proof}
It suffices to show that every tracial state of $A$ is a $W$-trace.
It follows from \cite{TWW} that every trace is quasidiagonal.
For a fixed  $\tau\in T(A),$  there exists a sequence of approximate multiplicative \cpc s
$\{\phi_n\}$ from $A$ into $Q$ such that
$$
\lim_{n\to\infty} tr\circ \phi_n(a)=\tau(a)\rforal a\in A.
$$
By \ref{LfactorQ}, we may assume that $\phi_n=\psi_n\circ \imath,$
where $\imath: A\to A\otimes Q$  is the embedding defined by
$\imath(a)=a\otimes 1_Q$ for all $a\in A$ and $\psi_n: A\otimes Q\to Q$ is a sequence
of approximate multiplicative \cpc s.

Therefore it suffices to show that every tracial state of $A\otimes Q$ is a $W$-trace.
Set $A_1=A\otimes Q.$  Then  $K_0(A_1)=\{0\}.$

Fix $1>\ep>0,$ $1>\ep'>0,$  a finite subset ${\cal F}\subset A_1$ and a finite ${\cal H}\subset A_1.$
Put ${\cal F}_1={\cal F}\cup {\cal H}.$
\Wlog, we may assume that ${\cal F}_1\subset A_1^{\bf 1}.$
Note that $A$ is non-unital. Choose  a strictly positive element $a\in A_+$ with $\|a\|=1.$
We also assume  that
$$
\tau(f_{1/2}(a))\ge d>0\rforal \tau\in \overline{T(A)}^w.
$$




Let $1>\dt>0,$ ${\cal G}\subset A_1$ be a finite subset be required by
\ref{Lpartuniq} for $A_1$ (in place of $A$), $d/2$ (in place of $d$), $\ep/16$ (in place of $\ep$)
and ${\cal F}_1.$ (Note since $K_0(A_1)=\{0\},$ the required set ${\cal P}$  in \ref{Lpartuniq}  does not  appear here.)



Let  ${\cal G}_1={\cal G}\cup {\cal F}_1$ and let
$
\ep_1={\ep\cdot \ep'\cdot \dt\over{2}}.
$

%



Let $\tau\in T(A_1).$  Since $\tau$ is quasi-diagonal,
 there exists a ${\cal G}_1$-$\ep_1$-multiplicative \cpc\,
$\phi: A_1\to Q$
such that
\beq\label{Talwtr-11}
&&|\tau(b)-tr\circ \psi(b)|<\ep'/16\rforal b\in {\cal G}\cup {\cal F},\\
&&tr(f_{1/2}(\psi(a)))>2d/3.
\eneq


Choose an integer $m\ge 3$ such that
$$
1/m<\min\{\ep_1/64, d/8\}.
$$

Let $e_1, e_2,...,e_{m+1}\in Q$ be a set of mutually orthogonal and mutually equivalent  projections
such that
$$
\sum_{i=1}^{m+1}e_i=1_Q\andeqn tr(e_i)={1\over{m+1}},\,\,\,i=1,2,...,m+1.
$$
Let $\psi_i: A_1\to (1\otimes e_i)(Q\otimes Q)(1\otimes e_i)$ be defined by
$\psi_i(b)=\psi(b)\otimes e_i,$ $i=1,2,...,m+1.$
Let
$$\Psi_0=\sum_{i=1}^m\psi_i\andeqn \Psi_1=\sum_{i=1}^{m+1}\psi_i.$$
Identify $Q\otimes Q$ with $Q.$
Note that
\beq\label{Talwtr-12}
tr(f_{1/2}(\Psi_i(a)))\ge d/2,\,\,\,i=0,1.
\eneq
Moreover,
$$
|\tau\circ \Phi_0(b)-\tau\circ \Phi_1(b)|<{1\over{m+1}}<\min\{\ep_1/64, d/8\}\rforal b\in A_1.
$$

Again, keep in mind that $K_0(A_1)=\{0\}.$ Applying \ref{Lpartuniq}, we obtain
 a unitary $u\in Q$ and a  ${\mathcal F}_1$-$\ep/16$-multiplicative \cpc\,  $L: A\to {\rm C}([0,3/4], Q)$
such that
\hspace{-0.1in}\beq\label{Talwtr-13}
 \pi_0\circ L=\Psi_0,\,\,\, \pi_{3/4}\circ L={\rm Ad}\, u\circ \Psi_1.
\eneq
Moreover,
\begin{equation}\label{Talwtr-14}
 |\mathrm{tr}\circ \pi_t\circ L(h)-\mathrm{tr}\circ \pi_0\circ L(h)|<1/m<\ep'/64, \quad h\in {\mathcal F}_1,\ t\in [0, 3/4].
\end{equation}
Here, $\pi_t: {\rm C}([0,3/4], Q)\to Q$ is the point evaluation at $t\in [0,3/4].$
There is a continuous path of unitaries $\{u(t): t\in [3/4,1]\}$ such that
$u(3/4)=u$ and $u(1)=1_Q.$
Define $L_1: A_1\to C([0,1], Q)$
by
 $\pi_t\circ L_1=\pi_t\circ L$ for $t\in [0, 3/4]$ and
$\pi_t\circ L_1={\rm Ad} u_t \circ \Psi_1$ for $t\in (3/4, 1].$
$L_1$ is a ${\cal F}_1$-$\ep/16$-multiplicative \cpc\, from $A_1$ into $C([0,1], Q).$
Note now
\vspace{-0.12in}\beq\label{Talwtr-15}
&&\pi_0\circ L_1=\Psi_0\andeqn \pi_1\circ L_1=\Psi_1\andeqn\\
&&|tr\circ \pi_t\circ L(h)-tr\circ \Psi_1(h)|<\ep'/64\rforal h\in {\cal H}.
\eneq

Fix an integer $k\ge 2.$ Let $\kappa_i: M_k\to M_{k(m+1)}$  ($i=0,1$) be defined by
\beq\label{Talwtr-16}
\kappa_0(c)&=&\diag(\overbrace{c,c,...,c}^m,0)\andeqn\\
\kappa_1(c)&=&\diag(\overbrace{c,c,...,c}^{m+1})
\eneq
\hspace{-0.1in}for all $c\in M_k.$ Define
$$
C_0=\{(f,c): C([0,1], M_{k(m+1)})\oplus M_k:  f(0)=\kappa_0(c)\andeqn f(1)=\kappa_1(c)\}.
$$
and defined
$$
C_1=C_0\otimes Q.
$$
Note that $C_0\in {\cal C}_0^0$  and  $C_1$ is an inductive limit
of  Razak algebras $C_0\otimes M_{n!}.$  Moreover
$K_0(C_1)=K_1(C_1)=\{0\}.$
Put $p_0=\sum_{i=1}^m1_Q\otimes e_i.$
Define  ${\bar \kappa}_0: Q\to p_0(Q\otimes Q)p_0$ to be the unital \hm\,
defined by ${\bar \kappa}_0(a)=a\otimes \sum_{i=1}^me_i$ and
${\bar \kappa}_1(a)=a\otimes 1_Q$ for all $a\in Q.$

Then
one may write
$$
C_1=\{(f, c)\in C([0,1], Q)\oplus Q: f(0)={\bar {\kappa}}_0(c)\andeqn f(1)={\bar \kappa}_1(c)\}.
$$
Note that ${\bar \kappa}_0\circ \psi(b)=\Psi_0(b)$ for all $b\in A_1$ and
${\bar \kappa}_1\circ \psi(b)=\Psi_1(b)$ for all $b\in A_1.$  Thus
one can define  $\Phi': A_1\to C_1$ by
$\Phi'(b)=(L_1(b), \psi(b))$ for all $b\in A_1.$

Then  $\Phi'$ is a ${\cal F}_1$-$\ep/16$-multiplicative \cpc\,
such that
\beq\label{Talwtr-17}
|tr(\pi_t\circ \Phi'(h))-tr\circ \psi(h)|<\ep'/4 \rforal h\in {\cal H}.
\eneq
Let $\mu$ be the Lebesque measure on $[0,1].$
There is a \hm\, $\Gamma: Cu^{\sim}(C_1)\to Cu^{\sim}(W)$ such that
$\Gamma(f)(\tau_0)=(\mu\otimes tr)(f)$ for all $f\in Aff(T(C_1)),$
where $\tau_0$ is the unique tracial state of $W.$
By \cite{Rl}, there exists a \hm\,
$\lambda: C_1\to W$
such that
\beq\label{Talwtr-18}
\tau_0\circ \lambda((f,c))=\int_0^1 tr(f(t))d t\rforal (f, c)\in C_1.
\eneq
Finally, let
$\Phi=\lambda\circ \Phi'.$
Then $\Phi$ is a ${\cal F}_1$-$\ep$-multiplicative \cpc\, from
$A_1$ into $W.$  Moreover, one computes
that
\beq\label{Talwtr-19}
|\tau_0\circ \Phi(h)-\tau(h)|<\ep'\rforal h\in {\cal H}.
\eneq
This proves the theorem.
\end{proof}

\begin{df}\label{DTAmap}
Let $D, C$ be two non-unital separable amenable \CA s and  let $u\in {\tilde C}.$  Suppose that $\phi: D\to C$ be a \hm.
 Let $\phi^{\sim}: {\tilde D}\to {\cal C}$ be the unital extension 
of $\phi.$  In what follows $z$ denotes the standard unitary generator of $C(\T).$

If $u\phi(d)=\phi(d)u$ for all $d\in D,$  for some unitary $u\in {\tilde C}$ with $\pi^{C\sim}(u)=1,$
where $\pi^{C\sim}: {\tilde C}\to \C$ is the quotient map,
then one can defined a \hm\,  $\Phi_{u, \phi}: C(\T)\otimes {\tilde D}\to {\tilde C}$ such 
that $\Phi(z\otimes 1_{\tilde D})=u$ and $\Phi(1\otimes d)=\phi^{\sim}(d)$ for all 
$d\in {\tilde D}.$  For any finite subset ${\cal F}\subset C(\T)\otimes D,$ 
any finite subset ${\cal F}_d\subset {\tilde D},$ and $\ep>0,$ 
there exists a finite subset ${\cal G}\subset D$ and $\dt>0$ such that, whenever 
$\phi: D\to  C$ is a ${\cal G}$-$\dt$ -multiplicative \cpc \, (for any \CA\, $C$)  and 
$\|[u, \phi(g)]\|<\dt$ for all $g\in {\cal G},$ there exists a ${\cal F}$-$\ep$-multiplicative 
\cpc\, $L: C(\T)\otimes {\tilde D}\to {\tilde D}$ such that 
\beq
\|L(z\otimes 1)-u\|<\ep\andeqn \|L(1\otimes d)-\phi^{\sim}(d)\|<\ep\tforal d\in {\cal F}_d.
\eneq
We will denote such $L$ by $\Phi_{u, \phi}.$ 
\end{df}

The following  is a  non-unital version  of 8.4 of \cite{GLN}. The proof is almost identical to that
of 8.4 of \cite{GLN}. We include 
it here for convenience of the reader. 

\begin{lem}\label{UniqN1}
Let $A=C(\T)\otimes {\tilde D},$ where $D\in {\cal C}_0^0.$  
Let ${\cal F}\subset A$ be a finite subset, let
$\ep>0$ be a positive number and let $\Delta: A_+^{q, {\bf 1}}\setminus \{0\}\to (0,1)$  be an order preserving map. There exists  a finite subset ${\cal H}_1\subset A_+^{\bf 1}\setminus \{0\},$
$\gamma_1>0,$ $\gamma_2>0,$ $\dt>0,$ a finite subset
${\cal G}\subset A$ 
a finite subset ${\cal H}_2\subset A$
satisfying the following:
For any unital ${\cal G}$-$\dt$-multiplicative \morp s 
$\Phi_{u,\psi}, \Phi_{v, \psi}: A\to {\tilde C}$ 
for some $C\in {\cal R}$  
where $u, v\in U({\tilde C})$ and 
$\phi, \psi: D\to C$ are two ${\cal G}_d$-$\dt$-multiplicative \cpc s
($\{g\otimes 1: g\in {\cal G}_d\}\subset {\cal G}$)
such that
\begin{equation}\label{Uni1-1}
\tau(\Phi_{u,\phi}(a))\ge \Delta(a),\,\,\, \tau(\Phi_{v,\psi}(a))\ge \Delta(a) \tforal \tau\in T(C) \tand a\in {\cal H}_1,
\end{equation}
\begin{equation}\label{Uni1-3}
|\tau\circ \phi(a)-\tau\circ \psi(a)|<\gamma_1 \tforal a\in {\cal H}_2\tand
\end{equation}
\begin{equation}\label{Uni1-3+1}
{\rm dist}(\bar{u}, \bar{v})<\gamma_2,
\end{equation}
there exists a unitary $W\in {\tilde C}$ such that
\begin{equation}\label{Uni1-4}
\|W(\Phi_{u,\phi}(f))W^*-(\Psi_{v, \psi}(f))\|<\ep,\tforal f\in {\cal F}.
\end{equation}
\end{lem}

\begin{proof}
One computes $K_0(A)=\Z$ and $K_1(A)=\Z,$ while $K_0({\tilde C})=\Z$ and $K_1({\tilde C})=\{0\}.$
Therefore $KK(A, {\tilde C})={\rm Hom}(K_0(A), K_0({\tilde C})={\rm Hom}(\Z, \Z).$


Let ${\cal H}_1'\subset A_+\setminus \{0\}$ (in place of ${\cal H}_1$)
for $\ep/32$  (in place of $\ep$) and ${\cal F}$ required by 
6.7 of \cite{GLN}.
 Let
 $\dt_1>0$ (in place of $\dt$), ${\cal G}_1\subset A$ (in place of ${\cal G}$) be a finite subset and let ${\cal P}_0\subset \underline{K}(A)$ (in place of ${\cal P}$) be a finite subset required by 
 6.7 of \cite{GLN}
for $\ep/32$ (in place of $\ep$), ${\cal F}$ and  $\Delta.$
We may assume that $\dt_1<\ep/32$ and $(2\dt_1, {\cal G}_1)$
is a $KK$-pair (see the end of 
2.12 of \cite{GLN}).

Moreover, we may assume that $\dt_1$ is sufficiently small  that
if $\|uv-vu\|<3\dt_1,$ then the Exel formula
$$
\tau({\rm bott}_1(u,v))={1\over{2\pi\sqrt{-1}}}(\tau(\log(u^*vuv^*))
$$
holds for any pair of unitaries $u$ and $v$ in any unital \CA\, $C$ with tracial rank zero and any $\tau\in T(C)$ (see Theorem 3.6 of \cite{Linajm}). Moreover
if $\|v_1-v_2\|<3\dt_1,$ then
$$
{\rm bott}_1(u,v_1)={\rm bott}_1(u,v_2).
$$

Let  $g_1\in U(A)$  be given  by 
$z\otimes 1_{\tilde D}.$  Let $\bar{g_1}\in U(A)/CU(A)$ be the image.
Let $\dt_u=\min\{1/256, \dt_1/16\},$
${\cal G}_u={\cal F}\cup{\cal G}_1\cup {\cal U}_0$ and
let ${\cal P}_u={\cal P}_0\cup \{[g_1]\}.$

Let $\dt_2>0$ (in place of $\dt$), ${\cal G}_2\subset A$ (in place of ${\cal G}$), ${\cal H}_2'\subset A_+\setminus \{0\}$ (in place of ${\cal H}$), $N_1\ge 1$ (in place of $N$) be the finite subsets and the constants as required by 
7.4 of \cite{GLN}
for 
$\dt_u$ (in place of $\ep$), ${\cal G}_u$ (in place of ${\cal F}$), ${\cal P}_u$ (in place of ${\cal P}$)  and
$\Delta$ and with $\bar{g}_1$ (in place of $g_j$).

Let $\dt_3>0$  and let ${\cal G}_3\subset A\otimes C(\T)$
be a finite subset satisfying the following:
For any ${\cal G}_3$-$\dt_3$-multiplicative \morp\, $L': A\otimes C(\T)\to C'$ (for any unital \CA\, $C'$ with $T(C')\not=\emptyset$),
 \beq\label{Uni-10}
 |\tau([L]({\boldsymbol{\bt}}(\bar{g}_1))|<1/8N_1.
\eneq

Without loss of generality, we may assume
that
$$
{\cal G}_3=\{g\otimes z': g\in {\cal G}_3'\andeqn z'\in \{1,z, z^*\}\},
$$
where
${\cal G}_3'\subset A$ is a finite subset containing $1_A$  (by  choosing a smaller $\dt_3$ and large ${\cal G}_3'$).

 Let $\ep_1'=\min\{d/27N_1, \dt_u/2,  \dt_2/2, \dt_3/2\}$ and let ${\bar \ep}_1>0$ (in place of $\dt$) and ${\cal G}_4\subset A$ (in place of ${\cal G}$) be a finite subset as
required by  6.4 of \cite{GLN}
for $\ep_1'$ (in place of $\ep$) and ${\cal G}_u\cup {\cal G}_3'.$
Put
$
\ep_1=\min\{\ep_1', \ep_1'',{\bar \ep_1}\}.
$
Let ${\cal G}_5={\cal G}_u\cup {\cal G}_3'\cup {\cal G}_4.$

Let $\mathcal H_3'\subseteq A_+^{\bf 1}\setminus \{0\}$ (in place of $\mathcal H_1$), $\dt_4>0$ (in place of $\dt$),  ${\cal G}_6\subset A$  (in place of ${\cal G}$), ${\cal H}_4'\subset A_{s.a.}$ (in place of $\mathcal H_2$), ${\cal P}_1\subset \underline{K}(A)$ (in place of ${\cal P}$) and $\sigma>0$ 
be the finite subsets and constants  as required by Theorem 
5.9  of \cite{GLN}
with respect to $\ep_1/4$ (in place $\ep$) and ${\cal G}_5$ (in place of ${\cal F}$) and $\Delta$.

Choose $N_2\ge N_1$ such that $(k(A)+1)/N_2<1/8N_1.$
 Choose ${\cal H}_5'\subset A_+^{\bf 1} \setminus \{0\}$ and
 $\dt_5>0$ and a finite subset ${\cal G}_7\subset A$ such that, for any $M_m$ and unital ${\cal G}_7$-$\dt_5$-multiplicative \morp\, $L': A\to M_m,$
 if
 ${\rm tr}\circ L'(h)>0\tforal h\in {\cal H}_5',$
 then
 $m\ge 16N_2.$

Put $\dt=\min\{\ep_1/16,  \dt_4/4, \dt_5/4\},$
${\cal G}={\cal G}_5\cup {\cal G}_6\cup {\cal G}_7$, and  ${\cal P}={\cal P}_u\cup {\cal P}_1.$
Put
$$
{\cal H}_1={\cal H}_1'\cup {\cal H}_2'\cup {\cal H}_3'\cup {\cal H}_4'\cup {\cal H}_{{5}}'
$$
and let ${\cal H}_2={\cal H}_4'.$
Let $\gamma_1=\sigma$ and let
$0<\gamma_2<\min\{d/16N_2, \dt_u/9, 1/256\}.$

Now suppose that $C\in {\cal R}$ and $\Phi_{u, \phi}, \Psi_{v,\psi}: A\to {\tilde C}$ are two unital
${\cal G}$-$\dt$-multiplicative \morp s satisfying the condition of the theorem for the given $\Delta,$ ${\cal H}_1,$ $\dt,$ ${\cal G},$ ${\cal P},$ ${\cal H}_2,$ $\gamma_1,$ $\gamma_2$ and ${\cal U}.$ 
Here we also assume that  
\beq
\|\pi^{C\sim}(\Phi_{u, \phi}(g))-\pi^{C\sim}(\Phi_{v, \psi}(g))\|<\ep_1'\rforal g\in {\cal G}_5,
\eneq
$\pi^{C\sim}: {\tilde C}\to \C$ is  the 
quotient map.

In what follows we will use $\Phi$ for $\Phi_{u, \phi}$ and $\Psi$ for $\Phi_{v, \psi}.$ 

By considering only one summand, we may write $C=A(F_1, F_2, h_0, h_1)=A(k_1, m_0, m_0+1),$ 
where $F_1=M_{k_1},$ $F_2=M_{k_1(m_0+1)},$
$h_1(f)=f\otimes 1_{M_{m_0+1}}$ and $h_0(f)=f\otimes p$ for all $f\in F_1,$ where $p\in M_{m_0+1}$
with rank $m_0.$ 
By the choice of ${\cal H}_5',$ one has that
$k_1\ge 16 N_2.$
Let
\vspace{-0.14in}$$
0=t_0<t_1<\cdots <t_n=1
$$
be a partition of $[0,1]$ so that
\vspace{-0.12in}\beq\label{Uni-11}
\|\pi_{t}\circ \Phi(g)-\pi_{t'}\circ \Phi(g)\|<\ep_1/16\andeqn
\|\pi_{t}\circ \Psi(g)-\pi_{t'}\circ \Psi(g)\|<\ep_1/16
\eneq
for all $g\in {\cal G},$ provided $t, t'\in [t_{i-1}, t_i],$ $i=1,2,...,n.$

Note that, by \ref{Runitz}, 
${\tilde C}=A(F_1\otimes \C, F_2, h_0^{\sim}, h_1^{\sim}),$ 
where $h_0^{\sim}((f,\lambda))=h_0(f)\oplus \lambda (1_{F_2}-h_0(1_{F_1}))$
and $h_1^{\sim}((f, \lambda))=h_1(f).$ 

Applying Theorem 
5.9 of \cite{GLN},
one obtains a unitary
$w_i\in F_2,$ if $0<i<n$ 
and $w_n\in h_1(F_1).$ 
such that, for  $0<i\le n,$ 
\vspace{-0.1in}\beq\label{Uni-12}
\|w_i\pi_{t_i}\circ \Phi(g)w_i^*-\pi_{t_i}\circ \Psi(g)\|<\ep_1/16\tforal g\in {\cal G}_5.
\eneq
Let $\pi^{F_1^{\sim'}}: h_0^{\sim}(F_1^{\sim})\to \C$ and $\pi^{D'}: {\tilde D}\to \C$  be the quotient maps. 
Let $\pi': h_0(F_1^{\sim})\to h_0(F_1)$ be the quotient map.  Consider $\pi'\circ \Phi$ and $\pi'\circ \Psi.$ 
Then, by applying  Theorem 
5.9 of \cite{GLN}, one also has a unitary $w_0'\in  h_0(F_1)$ such that
\beq\label{Uni-12-+1}
\|(w_0')\pi'\circ \pi_{t_0}\circ \Phi(g)(w_0')^*-\pi'\circ \pi_{t_0}\circ \Psi(g)\|<\ep_1/16\tforal g\in {\cal G}_5.
\eneq
Put $w_0=w_0'\oplus (F_2-h_0(1_{F_1})).$ 
Then we have 
\beq\label{Uni-12-+2}
\|w_0\pi_{t_0}\circ \Phi(g)(w_0)^*-\pi_{t_0}\circ \Psi(g)\|<\ep_1'\tforal g\in {\cal G}_5.
\eneq

It follows from 
8.2 of \cite{GLN} that  we may assume that there is a unitary  $w_e\in F_1\oplus  \C$ such that 
$h_0^{\sim}(w_e)=w_0$ and $h_1^{\sim}(w_e)=w_n.$ Write $w_0=w_0'\oplus 1_\C.$ 
Then $h_0(w_e')=w_0$ and $h_1(w_0')=w_n.$

By \eqref{Uni1-3+1}, there is a unitary  $\omega_1\in {\tilde C}$
such that $\omega_1\in CU({\tilde C}))$ and
\beq\label{Uni-13}
\|\langle (\Phi(g_1^*)\rangle \langle (\Psi(g_1)\rangle  -\omega_1\|<\gamma_2.
\eneq
({\it note that we now have $w_1$ as well as $\omega_1$ in the proof}.) 
Note that $\omega_1$ can be chosen so that $\omega_1=1+\omega_{01}$ for
some $\omega_{01}\in C.$
Write
\vspace{-0.12in}$$
\omega_1=\prod_{l=1}^{e(1)}\exp(\sqrt{-1}a_1^{(l)})
$$
for some selfadjoint element $a_1^{(l)}\in {\tilde C},$
$l=1,2,...,e(1).$ 
Let $r_l=\pi^{C^\sim}(a_1^{(l)}),$ $l=1,2,...,e(1).$ 
Then $r_l\in \R.$  Then $\sum_{l=1}^{e(1)}r_l\in 2\pi \Z.$ Replacing $a_1^{(l)}$ by
$a_1^{(l)}-r_l,$ $l=1,2,...,e(1),$ we may assume that $\pi^{C^\sim}(a_1^{(l)})=0,$
$l=1,2,...,e(1).$

Then
$$
\sum_{l=1}^{e(j)}{k_1(m_0+1)({\rm tr}(a_1^{(l)}(t)))\over{2\pi}}\in \Z,\quad t\in (0,1),
$$
where ${\rm tr}$ is the normalized trace on $F_2.$ One checks also 
the above also has an integer value at $1$ as well as $0.$ 
In particular,
\beq\label{Uni-15}
\sum_{l=1}^{e(j)}(k_1(m_0+1))({\rm tr})(a_1^{(l)}(t))=\sum_{l=1}^{e(j)}(k_1(m_0+1)){\rm tr}(a_1^{(l)}(t'))
\rforal t, t''\in [0,1].
\eneq

Then
\begin{eqnarray}\label{Ui-16}
&&\|\pi_{t_i}(\langle \Phi)(g_1^*)\rangle) w_i(\pi_{t_i}(\langle \Phi)(g_1)\rangle)w_i^*-\omega_1(t_i)\|
<3\ep_1/8+2\gamma_2<1/32.
\end{eqnarray}
Let $\pi_e': C\to F_1$ be the restriction $(\pi_e')|_{F_1}$ and $\Phi_e'=\pi'_e\circ \Phi.$
Write  $\pi_e(\omega_1)=\omega_{1e}\oplus 1_\C,$  where $\omega_{1e}\in F_1.$
 Then \eqref{Uni-17} also means
\beq\label{Uni-17+1}
\|\langle (\Phi_e')(g_1^*)\rangle w_e'(\langle \Phi_e')(g_1)\rangle)(w_e')^*-\omega_{1e}\|<3\ep_1/8+2\gamma_2<1/32.
\eneq
We also have (with $\Phi_e=\pi_e\circ \Phi$)
\beq\label{Uni-17}
\hspace{-0.4in}\|\langle (\Phi_e)(g_1^*)\rangle w_e(\langle \Phi_e)(g_1)\rangle)w_e^*-\pi_e(\omega_1)\|<3\ep_1'/8+2\gamma_2<1/32.
\eneq

It follows from \eqref{Ui-16} that there exists selfadjoint elements $b_i\in F_2$ such that
\beq\label{Uni-18}
\hspace{-0.2in}\exp(\sqrt{-1}b_{i})=\omega_1(t_i)^*(\pi_{t_i}(\langle \Phi)(g_1^*)\rangle)w_i(\pi_i(\langle \Phi)(g_1)\rangle) w_i^*,
\eneq
and $b_{e}\in F_1$ such that
\beq\label{Uni-19}
\hspace{-0.3in}&&\exp(\sqrt{-1}b_{e})=\pi_e(\omega_1)^*(\pi_e(\langle \Phi)(g_1^*)\rangle)
w_e'(\pi_e(\langle \Phi)(g_1)\rangle) w_e^*,\andeqn\\
\label{Uni-20}
\hspace{-0.3in}&&\|b_{i}\|<2\arcsin (3\ep_1/16+\gamma_2) \,i=0,1,...,n, \andeqn \|b_e\|<2\arcsin (3\ep_1'/16+\gamma_2)
\eneq
Put $b_e'=\pi_e'(b_e).$ 
Then
\beq\label{Uni-20+1}
b_e'=\omega_{1e}^*(\pi_e(\langle \Phi)(g_1^*)\rangle)w_{1e}w_e'(\pi_e(\langle \Phi)(g_1)\rangle) (w_e')^*
\andeqn \|b_e'\| <2\arcsin (3\ep_1/16+\gamma_2).
\eneq
We have that
\beq\label{Uni-20+}
h_0^{\sim} (b_{e})=b_{0}\andeqn h_1^{\sim}(b_{e})=b_{n}.
\eneq
Note that
\beq\label{Uni-21}
(\pi_{t_i}(\langle \Phi(g_j^*)\rangle))w_i(\pi_{t_i}(\langle \Phi)(g_1)\rangle) w_i^*
=\pi_{t_i}(\omega_1)\exp(\sqrt{-1}b_{i}),
\,\,\,i=0,1,...,n,e.
\eneq
Then,
\beq\label{Uni-22}
{k_1(m_0+1)\over{2\pi}}( {\rm tr})(b_{i})\in \Z,
\eneq
$i=0,1,...,n.$
We also have
\beq\label{Uni-23}
{k_1\over{2\pi}}({\rm tr}_{k(1)})(b_{e}')\in \Z,
\eneq
where ${\rm tr}_{k(1)}$ is the normalized trace on $F_1.$
Put
$$
\lambda_i={k_1(m_0+1)\over{2\pi}}( {\rm tr})(b_{i})\in \Z,  \,\,\,  i=0,1,2,...,n,\andeqn
\lambda_{e}={k_1\over{2\pi}}({\rm tr}_{k_1})(b_{e})\in \Z.
$$

 We have, by (\ref{Uni-20}),
\beq\label{Uni-24}
|{\lambda_i\over{k_1(m_0+1)}}|&<&1/4N_1,\,\,\,i=1,2,...,n,\andeqn\\\label{Uni-24+}
|{\lambda_{e}\over{k_1}}|&<&1/4N_1,
\eneq

Define $\af_i^{(0,1)}: K_1(A)\to \Z=K_0(F_2)$ by mapping $[g_1]$ to $\lambda_{i},$  $i=0,1,2,...,n,$ and
define
$\af_e^{(0,1)}: K_1(A)\to \Z=K_0(F_1)$ by
mapping $[g_1]$ to $\lambda_{e}.$ 
We write $K_0(A\otimes C(\T))=K_0(A)\oplus {\boldsymbol{\bt}}(K_1(A)))$
(see  
%
2.10 of \cite{LnHomtp}
for the definition
of ${\boldsymbol{\bt}}$).
Define $\af_i: K_*(A\otimes C(\T))\to K_*(F_2)$ as follows:
On $K_0(A\otimes C(\T)),$ define
\beq\label{Uni-25}
\af_i|_{K_0(A)}=[\pi_i\circ \Phi]|_{K_0(A)},\,\,\,
\af_i|_{{\boldsymbol{\bt}}(K_1(A))}=\af_i\circ {\boldsymbol{\bt}}|_{K_1(A)}=\af_i^{(0,1)}
\eneq
and on $K_1(A\otimes C(\T)),$ define
\beq\label{Uni-25+}
\af_i|_{K_1(A\otimes C(\T))}=0,\,\,\,i=0,1,2,...,n.
\eneq

Also define  $\af_e\in {\rm Hom}(K_*(A\otimes C(\T)), K_*(F_1)),$ by
\beq\label{Uni-26}
\af_e|_{K_0(A)}=[\pi_e'\circ \Phi]|_{K_0(A)},\,\,\,
\af_e|_{{\boldsymbol{\bt}}(K_1(A))}=\af_e\circ {\boldsymbol{\bt}}|_{K_1(A)}=\af_e^{(0,1)}
\eneq
on $K_0(A\otimes C(\T))$ and  $(\af_e)|_{K_1(A\otimes C(\T))}=0.$
Note that
\beq\label{Uni-26+0}
(h_0)_{*}\circ \af_e=\af_0\andeqn (h_1)_{*}\circ\af_e=\af_n.
\eneq
Since $A\otimes C(\mathbb{T})$ satisfies the UCT,  the map $\alpha_e$ can be lifted to an element of $KK(A\otimes C(\mathbb T), F_1)$ which is still denoted by $\alpha_e$. Then define
\beq\label{Uni-26+}
\af_0=\af_e\times [h_0] \andeqn \af_n=\af_e\times [h_1]
\eneq
in $KK(A\otimes C(\mathbb T), F_2)$.

For $i=1, ..., n-1$, also pick a lifting of $\alpha_i$ in $KK(A\otimes C(\mathbb T), F_2)$, and still denote it by $\alpha_i$.
We estimate that
\beq\label{Uni-26+n}
\|(w_{i}^*w_{i+1})\pi_{t_i}\circ \Phi(g)-\pi_{t_i}\circ \Phi(g)(w_{i}^*w_{i+1})\|<\ep_1/4\tforal g\in {\cal G}_5,
\eneq
$i=0,1,...,n-1.$
Let $\Lambda_{i,i+1}: C(\T)\otimes A\to F_2$ be a unital \morp\, given
by the pair $w_{i}^*w_{i+1}$ and $\pi_{t_i}\circ \Phi$  (by 
6.4 of \cite{GLN}, see 2.8 of \cite{LnHomtp}).
Denote $V_{i}=\langle \pi_{t_i}\circ \Phi(g_1) \rangle,$  $i=0,1,2,...,n-1.$

We have
\beq\label{Uni-27}
\|w_{i}V_{i}^*w_{i}^* V_{i}V_{i}^*w_{i+1}V_{i}w_{i+1}^*-1\|<1/16\\
\|w_{i}V_{i}^*w_{i}^*V_{i}V_{i+1}^*
w_{i} V_{i+1}w_{i+1}^*-1\|<1/16
\eneq
and there is a continuous path $Z(t)$
of unitaries  such that $Z(0)=V_{i}$ and $Z(1)=V_{i+1}.$ Since
$$
\|V_{i}-V_{i+1}\|<\dt_1/12,\,\,\,j=1,2,...,k(A),
$$
we may assume that $\|Z(t)-Z(1)\|<\dt_1/6$ for all $t\in [0,1].$
We obtain a continuous path
$$
w_{i}V_{i}^*w_{i}^*V_{i}Z(t)^*w_{i+1} Z(t)w_{i+1}^*
$$
which is in $CU(M_{nm(A)})$
for all $t\in  [0,1]$ and
$$
\|w_{i}V_{i}^*w_{i}^*V_{i}Z(t)^*w_{i+1} Z(t)w_{i+1}^*-1\|<1/8\tforal t\in [0,1].
$$

It follows that
$$
(1/2\pi\sqrt{-1})( {\rm tr})[\log(w_{i}V_{i}^*w_i^*V_{i}Z(t)^*w_{i+1}Z(t)w_{i+1}^*)]
$$
is a constant integer,  where ${\rm tr}$ is the normalized trace on $F_2.$
In particular,
\beq\label{Uni-28}
&&(1/2\pi\sqrt{-1})({\rm tr})(\log(w_{i}V_{i}^*w_{i}^*w_{i+1} V_{i}w_{i+1}^*))\\\label{Uni-28+}
&&=(1/2\pi\sqrt{-1})({\rm tr})(\log(w_{i}V_{i}^*w_{i}^*V_{i}V_{i+1}^*w_{i +1}V_{i}w_{i+1}^*)).
\eneq
One also has
\begin{eqnarray}\label{Uni-29}
&&w_{i}V_{i}^*w_{i}^*V_{i}V_{i+1}^*w_{i+1} V_{i+1}w_{i+1}^*\\
&&=(\omega_1(t_{i})\exp(\sqrt{-1}b_{i}))^*\omega_1(t_{i+1})
\exp(\sqrt{-1}b_{i+1})\\\label{Uni-30}
&&=\exp(-\sqrt{-1}b_{i})\omega_1(t_{i})^*\omega_1(t_{i+1})
\exp(\sqrt{-1}b_{i+1}).
\end{eqnarray}
Note that, by (\ref{Uni-13}) and (\ref{Uni-11}), for $t\in [t_i, t_{i+1}],$
\begin{equation}
\|\omega_1(t_{i})^*\omega_1(t)-1\|<2(m(A)^2)\ep_1/16+2\gamma_2<1/32,
\end{equation}
 $i=0,1,..., n-1.$

By Lemma 3.5 of \cite{Lin-AU11},
\begin{equation}\label{LeasyApp}
{\rm tr}(\log(\omega_1(t_{i})^*\omega_1(t_{i+1})))=0.
\end{equation}

It follows that (by the Exel formula (see \cite{HL}), using (\ref{Uni-28+}), (\ref{Uni-30}) and (\ref{LeasyApp}))
\beq\label{Uni-31}
&&\hspace{-0.6in}{\rm tr}({\rm bott}_1(V_{i}, w_{i}^*w_{i+1}))\\
 \hspace{-0.2in}&=&
({1\over{2\pi \sqrt{-1}}})({\rm tr})(\log(V_{i}^*w_{i}^*w_{i+1}V_{i}w_{i+1}^*w_{i}))\\
 \hspace{-0.2in}&=&({1\over{2\pi \sqrt{-1}}}){\rm tr})(\log(w_{i}V_{i}^*w_{i}^*w_{i+1}V_{i}w_{i+1}^*))\\
&=&({1\over{2\pi \sqrt{-1}}}) {\rm tr})(\log(w_{i}V_{i}^*w_{i}^*V_{i}V_{i+1}^*
w_{i+1}V_{i+1}w_{i+1}^*))\\
&=& ({1\over{2\pi \sqrt{-1}}}){\rm tr}(\log(\exp(-\sqrt{-1}b_{i})\omega_j(t_{i})^*
\omega_j(t_{i+1})\exp(\sqrt{-1}b_{i+1}))\\
&=& ({1\over{2\pi \sqrt{-1}}})[({\rm tr})(-\sqrt{-1}b_{i})+({\rm tr})(\log(\omega_1(t_{i})^*\omega_1(t_{i+1}))\\
&&\hspace{1.4in}+({\rm tr})(\sqrt{-1}b_{i})]\\
&=&{1\over{2\pi}}({\rm tr})(-b_{i}+b_{i+1})
\eneq
for all $t\in T(F_2).$  In other words,
\beq\label{Uni-32}
{\rm bott}_1(V_{i}, W_{i}^*W_{i+1}))=-\lambda_{i}+\lambda_{i+1}
\eneq
 $i=0,1,...,  n-1.$

Applying
7.4 of \cite{GLN}
(using \eqref{Uni-24+}, \eqref{Uni1-1},  among other items), there are unitaries
$z_i\in F_2$, $i=1,2,...,n-1$, and $z_e\in F_1$ such that
\begin{eqnarray}\label{Uni-33}
&&\|[z_i,\,\pi_{t_i}\circ\Phi(g)]\|<\dt_u\tforal g\in {\cal G}_u\\
&& {\rm Bott}(z_i, \pi_{t_i}\circ \Phi)=\alpha_i\andeqn
 {\rm Bott}(z_e, \pi_e'\circ \Phi)=\alpha_e.
\end{eqnarray}
Put $$z_0=h_0(z_e)\otimes (1_{F_2}-h_0(1_{F_1}))\quad \mathrm{and} \quad z_n=h_1(z_e).$$
One verifies (by \eqref{Uni-26+}) that
\begin{eqnarray}\label{Uni-34}
{\rm Bott}(z_0, \pi_{t_0}\circ \Phi)=\alpha_0\andeqn
{\rm Bott}(z_n, \pi_{t_n}\circ \Phi)=\alpha_n.
\end{eqnarray}

Denote by $\Lambda_e: A\otimes C(\T)\to F_1$ the unital \cpc\, given by the pair 
$\pi_e\circ \Phi$ and $z_e$ as in 6.4 of \cite{GLN} (see 2.8 of \cite{LnHomtp}).
Then 
\beq\nonumber
&&[h_0\circ \Lambda_e]|_{{\boldsymbol{\bt}(\underline{K}(A))}}=(\af_e\times [h_0])|_{{\boldsymbol{\bt}(\underline{K}(A))}}\andeqn\\\nonumber
&&{[}h_1\circ \Lambda_e{]}|_{{\boldsymbol{\bt}(\underline{K}(A))}}=(\af_e\times [h_1])|_{{\boldsymbol{\bt}(\underline{K}(A))}}
\eneq
One verifies (by \eqref{Uni-26+}) that
$$\hspace{0.2in}{\rm Bott}(z_0, \pi_{t_0}\circ \Phi)=(\alpha_0)|_{{\boldsymbol{\bt}(\underline{K}(A))}}\andeqn
{\rm Bott}(z_n, \pi_{t_n}\circ \Phi)=(\alpha_n)|_{{\boldsymbol{\bt}(\underline{K}(A))}}.$$

Let $U_{i}=z_{i}(w_{i})^*w_{i+1}(z_{i+1})^*,$
$i=0,1,2,...,n-1.$
Then, by \eqref{Uni-33} and \eqref{Uni-26+n},
\beq\label{NT-8}
&&\|[U_{i}, \pi_{t_i}\circ \Phi(g)]\|<2\ep_1+2\dt_u<3\dt_u {<\dt_1/2}
\rforal g\in {\cal G}_u,
\eneq
$ i=0,1,2,...,n-1.$
Moreover, for $i=0,1, 2, ..., n-1$,
\beq\label{NT-8+1}
\hspace{0.5in}{\rm bott}_1(U_{i}, \pi_{t_i}\circ \Phi)&=&
{\rm bott}_1(z_{i},\, \pi_{t_i}\circ \Phi{)}+{\rm bott}_1{(}w_{i}^*w_{i+1}, \pi_{t_i}\circ \Phi{)}\\\nonumber
&&\quad+{\rm bott}_1{(}z_{i+1}^*,\, \pi_{t_i}\circ\Phi)\\\nonumber
&=& (\lambda_{i,j})+(-\lambda_{i,j}+\lambda_{i+1,j})+(-\lambda_{i+1,j})\\\nonumber
&=&0.
\eneq
Note that $K_1(F_2, \Z/k\Z)=K_1(F_1,\Z/k\Z)=\{0\}$ 
for $k\ge 2,3,....$ Therefore, in  the case that  $K_1(A)$ is torsion free, \eqref{NT-8+1} implies 
that 
\beq\label{Uni-37}
{\rm Bott}(U_{i}, \pi_{t_i}\circ \Phi)=0,\,\,\, i=0,1,...,n-1.
\eneq
Note that, by the assumption \eqref{Uni1-1},
\beq\label{Uni-38}
{\rm tr}\circ \pi_t\circ \Phi(h)\ge \Delta(\hat{h})\tforal h\in {\cal H}_1',
\eneq
where ${\rm tr}$ is the normalized trace on $F_2.$

By applying 
6.7 of \cite{GLN}, using \eqref{Uni-38}, \eqref{NT-8} and \eqref{Uni-37}, there exists
a continuous path of unitaries, $\{\tilde{U}_{i,i+1}(t): t\in [t_i, t_{i+1}]\}\subset F_2$ such that
\begin{equation}\label{Uni-39}
\tilde{U}_{i,i+1}(t_i)={\rm id}_{F_2},\,\,\, \tilde{U}_{i, i+1}(t_{i+1})=(z_iw_i^*w_{i+1}z^*_{i+1}),
\end{equation}
and
\begin{equation}\label{Uni-39+}
\|\tilde{U}_{i, i+1}(t)\pi_{t_i}\circ \Phi(f)\tilde{U}_{i, i+1}(t)^*-\pi_{t_i}\circ \Phi(f)\|<\ep/32
\end{equation}
for all $f\in {\cal F}$ and for all $t\in [t_i, t_{i+1}].$ Define
$W\in {\tilde C}$ by
\beq\label{Uni-40}
W(t)=(w_iz_i^*)\tilde{U}_{i, i+1}(t)\tforal t\in [t_i, t_{i+1}],
\eneq
$i=0,1,...,n-1.$ Note that $W(t_i)=w_iz_i^*,$ $i=0,1,...,n.$
Note also that
\beq 
&&W(0)=w_0z_0^* =h_0(w_ez_e^*)\oplus (1_{F_2}-h_0(1_{F_1})\andeqn\\
&&W(1)=w_nz_n^*=h_1(w_ez_e^*).
\eneq
So $W\in {\tilde C}.$
One then checks that, by  (\ref{Uni-11}), (\ref{Uni-39+}) , (\ref{Uni-33}) and (\ref{Uni-12}),
\begin{eqnarray}
&&\|W(t)((\pi_t\circ \Phi)(f)W(t)^*-(\pi_t\circ \Psi)(f)\|\\
&<&\|W(t)((\pi_t\circ \Phi)(f))W(t)^*-W(t)((\pi_{t_i}\circ \Phi)(f))W^*(t)\|\\
 &&+\|W(t)(\pi_{t_i}\circ \Phi)(f)W(t)^*-W(t_i)(\pi_{t_i}\circ \Phi)(f) W(t_i)^*\|\\
 &&+\|W(t_i)((\pi_{t_i}\circ \Phi)(f) W(t_i)^*-(w_i(\pi_{t_i}\circ \Phi)(f)w_i^*)\|\\
 &&+\|w_i(\pi_{t_i}\circ \Phi)(f)w_i^*-\pi_{t_i}\circ \Psi(f)\|\\
 && +\|\pi_{t_i}\circ \Psi(f)-\pi_t\circ \Phi(f)\|\\
 &<&\ep_1/16+\ep/32+\dt_u+\ep_1/16+\ep_1/16<\ep
\end{eqnarray}
 for all $f\in {\cal F}$ and for $t\in [t_i, t_{i+1}]$.




\end{proof}

\begin{cor}\label{CUniqN1}
Lemma \ref{UniqN1} holds for $D, C\in {\cal M}_0$ with continuous scales.
\end{cor}

\begin{proof}
Note that \CA s in ${\cal M}_0$ are  inductive limits of \CA s in ${\cal R}$ with injective connecting maps.
It is then easy to see that the lemma holds for $D, C\in {\cal M}_0.$ 
\end{proof}

\begin{lem}\label{densesp}
Let $A$  be a  non-unital  C*-algebra and $T(A)\not=\emptyset,$
  let $U$ be an infinite dimensional  UHF-algebra
and $B\subset A$ be a hereditary \SCA\, of $B.$ 
Suppose that there exists  $e\in A_+$ with $\|e\|=1$ and $eb=be=b$ for all 
$b\in B.$ Then there is a unitary $w\in {\tilde A\otimes U}$ 
with the form $w=\exp(i \pi (e\otimes h))$ for some $h\in U_{s.a.}$  with 
$\tau_U(h)=0$ (where $\tau_U$ is the unique tracial state of $U$) such that for
any unitary $u=\lambda +x\in {\tilde A}$   with  
$\lambda \in \T\subset \C$ and $x\in B,$  one has, for any $b\in B$ and $f\in C(\T),$ 
\beq\label{densesp-1}
\tau(bf((u\otimes 1)w))=\tau(b)\tau(f(1_A\otimes \exp(i h)))=\tau(b)\int_{\T}f dm
\eneq
and for all $\tau\in T(A\otimes U),$
where $m$ is the normalized Lebesgue
 measure on $\T.$
Moreover, for any $a\in B$  and $\tau\in T(A\otimes U)$, $\tau ((a\otimes 1)  w^j)=0$ if $j\not=0$.
Furthermore, if $A$ has continuous scale, then, for any $\ep>0,$  and any $N\ge 1,$ one can choose $e$ such that
\beq\label{densesp-2}
|\tau((u\otimes 1)w)^j)|<\ep\rforal  0<|j|\le N.
\eneq
\end{lem}

\begin{proof}
Denote by $\tau_U$ the unique trace of $U$. Then any trace $\tau\in T(A\otimes U)$ is a product trace, i.e.,
$$\tau(a\otimes b)=\tau(a\otimes 1)\otimes\tau_U(b),\quad a\in A, b\in U.$$

Pick a selfadjoint element $h\in U$ such that the spectral measure of the  unitary $w_0=\exp(ih)$ 
is the Lebesque measure (a Haar unitary).  Moreover, ${\rm sp}(h)=[-\pi, \pi]$ and 
$\tau(h)=0.$

Then one has, for each $n\in \Z,$ 
\vspace{-0.12in}$$
\tau_U(w_0^n)=
\left\{
\begin{array}{ll}
1, & \textrm{if $n=0$},\\
0, & \textrm{otherwise}.
\end{array}
\right.
$$
Put $w=\exp(i (e\otimes h))\in {\widetilde{A\otimes U}}.$ 
Thus $w=\sum_{k=0}^{\infty} {i e^k\otimes h^k\over{k!}}.$ 
Hence, for any $\tau\in T(A\otimes U)$, one has, for each $n\in \Z,$
$$\tau(b((u\otimes 1) w)^n)=\tau(b(u^n\otimes 1)(e\otimes 1)(1\otimes w_0^n))=\tau(bu^n\otimes 1)\tau_U(w_0^n)=
\left\{
\begin{array}{ll}
1, & \textrm{if $n=0$},\\
0, & \textrm{otherwise};
\end{array}
\right.
$$
and therefore
$$\tau((b\otimes 1)P(u\otimes 1)w))=\tau(b)\tau(P(1\otimes w))=\tau(b)\int_{\T}P(z)dm$$
for any polynomial $P$. Similarly, $\tau(bP(u\otimes w)^*)=\tau(b)\int_TP({\bar z})dm$
for any polynomial $P.$
Since polynomials of $z$ and $z^{-1}$ are dense in $\mathrm{C}(\T)$, one has
$$\tau((b\otimes 1)f((u\otimes 1)w))=\tau(b)\tau(f(1\otimes w))=\tau(b)\int_{\T}fdm,\quad f\in \mathrm{C}(\T),$$
as desired.

For the second part of this lemma, 
assume that $A$ has continuous scale.  Then,  for any $\dt>0$ and any integer $N_1\ge 1,$ 
we can choose $e_1, e\in A_+$ such that $1\ge e\ge e_1,$ $e_1e=ee_1=e_1$
$\tau(e^k)\ge \tau(e_1^{N_1})>1-\dt$ for all $\tau\in T(A)$ and $k\in \N.$
Fix  $N$ and $\ep>0.$ 
a simple calculation shows the second part of
the lemma follows by choosing sufficiently small $\dt$ and large $N_1.$ 

\end{proof}

\begin{prop}\label{densesp3}
Let $C$ be a  non-unital  amenable simple \CA\, and let $U$ be an infinite dimensional UHF-algebra.
{{For any $ \dt>0, \dt_c>0,$  $1>\sigma_1,\, \sigma_2>0,$ any finite subset ${\cal G}\subset
C\otimes C(\T),$  any finite ${\cal G}_c\subset {\tilde C},$ any finite subset ${\cal H}_1\subset C(\T)_+\setminus \{0\}$ and any finite subset ${\cal H}_2\subset  (C\otimes C(\T))_{s.a.}$ 
and any integer $N\ge 1,$
 there exist $\dt_1>0$ and a finite subset
 ${\cal G}_1\subset C$ satisfying the following:}}
For 
 any unital
 ${\cal G}_1$-$\dt_1$-multiplicative \morp\, $L: C\to  A$ and a unitary 
 $u\in {\tilde A}$ with $\|L(g),\, u]\|<\dt_1$ for all $g\in {\cal G}_1,$ 
 where $A$ is another non-unital \CA\, with $T(A)\not=\emptyset$ and with continuous scale,
 there exists  a positive element $e\in A$ with $\|e\|=1$ and $h\in U$ 
satisfying the following: 
there are unital $\mathcal G$-$\dt$-multiplicative \cpc s $L_1, L_2: {\tilde C}\otimes C(\T)\to {\tilde B}$ such that 
 \beq\label{densesp3-1}
 |\tau(L_1(f))-\tau(L_2(f))|<\sigma_1\tforal f\in {\cal H}_2,\ \tau\in T(B),\ \textrm{and}\ \label{densesp3-2} \\
 \tau(g(w))\ge \sigma_2(\int g dm)\tforal g\in {\cal H}_1,\ \tau\in T(B),
 \eneq
where $B=A\otimes U$ and $m$ is the normalized Lebesgue
 measure
 on $\T,$  and
  \beq\label{densesp3-n11}
 \|L_i(c\otimes 1_{C(\T)})-L^{\sim}(c)\otimes 1_U\|<\dt_c\tforal c\in {\cal G}_c,\,\,\, i=1,2,\\\label{densesp3-n12}
 \|L_1(c\otimes z^j)-L^{\sim}(c)(u\exp(ie\otimes h))^j\|<\dt_c\tforal c\in {\cal G}_c\andeqn\\\label{densesp3-n13}
 \|L_2(c\otimes z^j)-L(c)^{\sim} \exp(ie\otimes h)^j\|<\dt_c\tforal c\in {\cal G}_c
 \eneq
 and for all $0<|j|\le N,$ where $L^{\sim}: {\tilde C}\to {\tilde A}$ is 
 the unital extension of  $L.$  Moreover, $\tau(e\otimes h)=0$ for all $\tau\in A\otimes Q.$
%
\end{prop}

\begin{proof}
\Wlog, we may assume that there are  finite subsets ${\cal G}_c, {\cal H}_{c,1}\subset {\tilde C}$
such that ${\cal G}=\{c\otimes 1_{C(\T)}: c\in {\cal G}_c\}\cup \{1, 1_{\tilde C}\otimes z, 1_{\tilde C}\otimes z^*\}$ and 
${\cal H}_2=\{c\otimes 1_{C(\T)}: c\in {\cal H}_{c,1}\}\cup \{1\otimes b: b\in {\cal H}_T\},$
where ${\cal H}_T\subset C(\T)_{s.a.}.$ We may assume 
that $1_{\tilde C}\in {\cal G}_c,$ $1_{\tilde C}\in {\cal H}_{c,1}$ and $1_{C(\T)}\in {\cal H}_T.$
We may also assume that $\|a\|\le 1$ for all $a\in {\cal G}_c\cup {\cal H}_2.$
Put 
$$
{\cal G}_0=\{cd\otimes gf: c, d\in {\cal G}_c\cup {\cal H}_{c,1},\,\, g,f\in \{z, z^*\}\cup {\cal H}_T\}.
$$

Fix $\dt, \dt_c>0, \sigma_1, \sigma_2>0.$
Put $\ep=\min\{\dt/4, \dt_c/4, \sigma_1/4, \sigma_2/4\}.$

Let $\dt_1'>0$ and ${\cal G}_{0m}\subset {\tilde C}$ be a finite subset 
such that there is a ${\cal G}_0$-$\ep$-multiplicative \cpc\, $L': {\tilde C}\otimes C(\T)\to D,$
for any \CA\, $D$ and any ${\cal G}_{0m}$-$\dt_1'$-multiplicative \cpc\, $L'': {\tilde C}\to D,$ 
such that
\beq
\|L'(g\otimes 1_{C(\T)})-L''(g)\|<\ep\rforal g\in {\cal G}_0.
\eneq

Let ${\cal G}_1={\cal G}_0\cup  {\cal G}_{0m}$ and 
$\dt_1=\min\{\dt_1'/4, \ep/4\}.$ 

Now suppose that $L: {\tilde C} \to {\tilde A}$ is a ${\cal G}_1$-$\dt_1$-multiplicative \cpc\, and 
$u\in {\tilde A}$ is a unitary. 
\Wlog, we may assume that there are  positive elements 
$e_1, e\in A$ with $\|e_1\|=1=\|e\|$ such that
\beq
e_1L(g)=L(g)e=L(g)\rforal g\in C, ee_1=e_1e=e_1\andeqn \tau_U(e_1)>1-\ep.
\eneq
Furthermore, \wilog, we may assume that $L(c)=e_1L(c)e_1$ for all $c\in C.$
Let $h\in U$ be as in  \ref{densesp}. 
Let $v=\exp (i e\otimes h).$   Note that $\tau_U(e^j)>1-\ep$ for all $j\in \N.$ 
We can choose $e$ so that both \eqref{densesp-1} and \eqref{densesp-2} hold. This lemma then follows from an easy  application 
of \ref{densesp} and  
Lemma 2.8 of \cite{LnHomtp} (with $L_1=\Phi_{v_1, L}$ and 
$L_2=\Phi_{v_2, L},$ where $v_1=u(\exp(ie\otimes h))$ and $v_2=\exp(ie\otimes h).$

\end{proof}

\begin{lem}\label{BHK00}
Let $A\in {\cal M}_0$ with continuous scale.
For any $1>\ep>0$ and any finite subset ${\cal F}\subset A,$ there exist  $\dt>0,$ $\sigma>0$, a finite subset
${\cal G}\subset A$ 
satisfying the following:

Let $B=B_1\otimes U,$
where $B_1\in {\cal M}_0$  with continuous scale which satisfies the UCT and  $U$ is UHF-algebras of infinite type.
Suppose that $\phi: A\to B$ is a  \hm.

If $u\in U({\tilde B})$  is a unitary such that
\beq\label{BHfull-1}
&&\|[\phi(x),\, u]\|<\dt\tforal x\in {\cal G},
\eneq
there exists a continuous path of unitaries $\{u(t): t\in [0,1]\}\subset U({\tilde B})$ such
that
\beq\label{BHTL-3}
&&u(0)=u,\,\,\, u(1)=1_{\tilde B},\\
&&\|[\phi(a),\, u(t)]\|<\ep\tforal a\in {\cal F}\tand for\,\, all\,\, t\in [0,1].
\eneq
\end{lem}

\begin{proof}
Fix a finite subset ${\cal G}'\subset A^{\bf 1}$ and $\ep'>0,$  there exists a 
positive element $e', e'', e'''\in B'\setminus \{0\}$ with $\|e'\|=1=\|e''\|=\|e'''\|,$ 
$e'e''=e''e'=e'$  such that $e'''e'=e'e'''=0.$ 
\beq\label{BHTL-nn0}
\|\phi(g)e'-\phi(g)\|<\ep'/2\andeqn \|e'\phi(g)-\phi(g)\|<\ep'/2\rforal g\in {\cal G}'.
\eneq
Let $\pi^{B\sim}: {\tilde B}\to \C$ be the quotient map.
\Wlog, we may assume that $\pi^{B\sim}(u)=1_\C.$ 
Since $U({\tilde B})=U_0({\tilde B}),$ we may 
write $u=\prod_{i=1}^m\exp(i h_{j0})$ for some 
$h_{j0}\in {\tilde B}_{s.a.}.$  Write $h_{j,0}=r_j+h_{j0}',$ where 
$r_j\in \R$ and $h_{j0}'\in B_{s.a.}.$  Note that 
$\sum_{j=1}^m r_j=2\pi k_u$ for  $k_u\in \Z.$  Therefore $u=\prod_{j=1}^m\exp(i h_{j0}').$
We may also assume, \wilog, that $u=1+x,$ where $x\in \overline{e''Be''}.$ 
It is easy to find an element $h_0\in \overline{e'''Be''}$ such that
$\tau(h_0)=\sum_{j=1}^m h_{j0}'$ for $\tau\in T(B).$  
Let $u'(t)=\exp(i th_0)$ for all $t\in [0,1].$ 
Note that
\beq\label{BHTL-nn1}
uu'(0)=u\andeqn uu'(1)\in CU({\tilde B}).
\eneq
Moreover, by \eqref{BHTL-nn0},
\vspace{-0.1in}\beq
\|\phi(g)u'(t)-u'(t)\phi(g)\|<2\ep'\rforal g\in {\cal G}'\andeqn \rforal t\in [0,1].
\eneq
In other words, we have just reduced 
the general case to the case that $u\in CU({\tilde B}).$ 

Let 
\beq
\Delta_0(\hat{h})=\inf\{\tau(\phi(h)); \tau\in T(B)\}\rforal h\in {\tilde A}_+^{\bf 1}\setminus \{0\}.
\eneq
Since both $A$ and $B$ are simple, $\Delta_0(h)>0$ for all $h\in A_+\setminus \{0\}.$
If $h\in {\tilde A}_+\setminus \{0\},$ then $ehe\ge 0$ for all $e\in A_+.$ It is then easy to see 
that $ehe\not=0$ for some $e\in A_+.$ Therefore $\Delta_0(\hat{h})>0$ for all $h\in {\tilde A}_+\setminus \{0\}.$ 
Put $\Delta=\Delta_0/2.$
Since $U$ is strongly self absorbing, we may reduce the general 
case to the case that 
$\phi:A\to B_1\otimes 1_U$ and $u\in {\tilde B_1}.$

Define
$\Delta_1: (A\otimes C(\T))_+^{q, {\bf 1}}\setminus\{0\}\to (0,1)$ by
\beq\label{Ddel-1}
\hspace{-0.3in}\Delta_1(\hat{h})=\sup\{ {\Delta(h_1)\tau_m(h_2)\over{4}}:\,
\hat{h}\ge \widehat{h_1\otimes h_2}\andeqn h_1\in A_+\setminus \{0\},\,\,\, h_2\in C(\T)_+\setminus\{0\}\}.
\eneq

Let ${\cal F}_1=\{ a\otimes b: a\in {\cal F}\andeqn b\in \{1, z, z^*\}\}.$ 
Put $A_1={\tilde A}\otimes C(\T).$ 
Let $\gamma_1>0, \gamma_2>0, \dt>0$ be positive numbers, and,  ${\cal H}_1\subset (A_1)_+^{\bf 1}\setminus \{0\},$
${\cal H}_2\subset A_1$ and ${\cal G}\subset A_1$ be finite subsets  given 
by \ref{CUniqN1} (referring to  \ref{UniqN1}) for $\ep/16$ (instead of $\ep$) and ${\cal F}_1$ (instead of ${\cal F}$).
\Wlog, we may assume that ${\cal G}=\{a\otimes g:  a\in {\cal G}_a\andeqn g\in \{1, z, z^*\}\},$
where ${\cal G}_a\subset A$ is a finite subset containing $1_{\tilde A},$ and 
${\cal H}_1=\{a\otimes g: a\in {\cal H}_{1a}\andeqn g\in {\cal H}_{1T}\},$ 
where ${\cal H}_a\subset A_+^{\bf 1}\setminus \{0\}$ and ${\cal H}_{1T}\subset C(\T)_+^{\bf 1}\setminus \{0\}$ 
are finite subsets containing identities.  We may also assume 
that ${\cal H}_2=\{a\otimes g: a\in {\cal H}_{2a}\andeqn g\in {\cal H}_{2T}\},$
where ${\cal H}_{2a}$ and ${\cal H}_{2T}$ are finite subsets containing identities. 
We may also assume that every element in ${\cal H}_{1T}$ and ${\cal H}_{2T}$ is 
a polynomial of $z, z^*$ with degree no more than an integer $N.$ Furthermore, all the coefficients 
have absolute values no more than $M$ for some $M\ge 1.$ 

Choose 
\beq
\sigma_1=\min\{\gamma_1/2, \inf\{\Delta_1(\hat{h}): h\in {\cal H}_1\}/4MN\andeqn\\
\sigma_2=1-\gamma_2/16MN.
\eneq
Choose $\dt_c=\min\{\ep/16, \sigma_1/4, \sigma_2/4\}/4MN$ and 
${\cal G}_c={\cal F}\cup {\cal G}_a\cup {\cal H}_{1a}\cup {\cal H}_{2a}.$ 
Let ${\cal G}_c'\subset A$  be a finite subset such that 
every element $a\in {\cal G}_c$ has the form $a=\lambda +b$ for some $\lambda\in \C$ and 
$b\in {\cal G}_c'.$ 

Let $\dt_1>0$ and ${\cal G}_1\subset A$ be the finite subset required by \ref{densesp3}
for the above given $\dt,$ $\dt_c,$ $\sigma_1,$ $\sigma_2,$  ${\cal H}_1,$ ${\cal H}_2$ and ${\cal G}.$
\Wlog, we may assume that ${\cal G}_1\supset {\cal G}_c'.$
Furthermore, 

Now assume that $u\in CU({\tilde A})$ such that
\beq
\|[u,\, \phi(a)]\|<\dt_1\rforal a\in {\cal G}_1.
\eneq

Applying \ref{densesp3},  we obtain $e\in (B_1)_+$ with $\|e\|=1$ and $h\in U_{s.a.}$ satisfying 
the conclusions of \ref{densesp3}.  Note that we may assume 
that 
\beq\label{BHTL-nn9}
e\phi(g)=\phi(g)e\rforal g\in {\cal G}_1.
\eneq


Put $v_1=u\exp(ie\otimes h)$ and $v_2=\exp(ie\otimes h).$  
Note that $\tau(e\otimes h)=0$ so $v_2\in CU({\tilde B}).$ 

Let $L_1$ and $L_2$ be given by \ref{densesp3}.
Note by \eqref{densesp3-n11},  \eqref{densesp3-n12}  and \eqref{densesp3-n13},
we may write $L_1=\Phi_{v_1, \phi}$ and $L_2=\Phi_{v_2, \phi}.$ 
Then, by \eqref{densesp3-2},  \eqref{densesp-1} and the choice of $\dt_c,$ 
we have 
\beq\label{BHTL-nn10}
\tau(L_i(h))\ge \Delta_1(\hat{h})\rforal h\in {\cal H}_1\,\,\,i=1,2.
\eneq
We also have ${\rm dist}({\bar v_1}, {\bar v_2})=0.$ 
Then, by \eqref{densesp3-1}  and applying \ref{CUniqN1}, we obtain a unitary $W\in {\tilde B}$ such that
\beq\label{BHTL-nn11}
\|W^*L_2(f)W-L_1(f)\|<\ep/16\rforal f\in {\cal F}_1.
\eneq
Therefore
\beq\label{BHTL-nn12}
&&\hspace{-0.2in}\|[L(a),\,\,\, W^*v_2W]\|<\ep/8\andeqn \|L(a)-W^*L(a)W\|<\ep/8 \rforal a\in {\cal F}\\
&& \andeqn \|v_1-W^*v_2W\|<\ep/8.
\eneq
Let $v_1^*W^*v_2W=\exp (ih_1)$ for some 
$h_1\in {\tilde B}_{s.a.}$ such that $\|h_1\|\le 2\arcsin (\ep/16).$ 
Now define $u(t)=u\exp(i3t(e\otimes h))$ for $t\in [0,1/3],$  $u(t)=u(1/3)\exp(i3(t-1/3)h_1)$
for $t\in (1/3, 2/3]$ and $u(t)=u(2/3)W^*\exp (i(3(1-t))(e\otimes h)W.$  So $\{u(t): t\in [0,1]\}$ is a continuous path of 
unitaries in ${\tilde B}$ such that $u(0)=u$ and $u(1)=1_{\tilde B}.$
Moreover, we estimates, by \eqref{BHTL-nn9} and \eqref{BHTL-nn12}, that
\beq
\|[\phi(a),\,\, u(t)]\|<\ep\rforal a\in {\cal F}.
\eneq

\end{proof}

\begin{cor}\label{Casympu}
Let $A, B\in {\cal D}$ be two separable \CA\, with  continuous scale and with $K_i(A)=K_i(B)=\{0\}$ ($i=0,1$) which satisfy the UCT. 
Suppose that $\phi_1, \phi_2: A\to B$ are two \hm s such that
\vspace{-0.1in}\beq
(\phi_1)_T=(\phi_2)_T.
\eneq
Then there is a continuous path of unitaries $\{u(t): t\in [0, 1)\}\subset {\tilde B}$ with $u(0)=1$
such that
\beq
\lim_{t\to\infty} {\rm Ad}\, u(t)\circ \phi_1(a)=\phi_2(a)\rforal a\in A.
\eneq
\end{cor}

\begin{proof}
The proof of \ref{TTMW} (see also \ref{Lrluniq}) shows that there exists a sequence of unitaries 
$\{u_n\}\subset {\tilde B}$ such that
\beq
\lim_{t\to\infty}{\rm Ad}\, u_n\circ \phi_1(a)=\phi_2(a)\rforal a\in A.
\eneq
Note that $U({\tilde B})=U_0({\tilde B}).$ 
It is then standard, by applying \ref{BHK00} repeatedly, one can obtain a required continuous path of unitaries. 
\end{proof}

\begin{thm}\label{LFNZ}
Let $A$ be a non-untal separable 
finite simple \CA\,  with finite nuclear dimension and with $KK(A,D)=0$ for all separable \CA s $D$ which satisfies the UCT.
Then $A\in {\cal D}_0.$ Without assuming continuous scale, $gTR(A)\le 1.$
\end{thm}

\begin{proof}
It suffices to consider the case that $A$ has continuous scale. 
By \cite{Wnz}, $A\cong A\otimes {\cal Z}.$ 
It follows from  \ref{TTTAD} that it suffices to show that $A\otimes {\cal Z}$ is tracially approximate divisible 
in the sense of \ref{Dappdiv}. 
Let $B=A\otimes Q.$ Then,  by \ref{TTTAD}, $B\in {\cal D}_0.$ Note, by \ref{TTMW},
$B\cong C,$ where $C\in {\cal M}_0.$ Moreover $C\cong A\otimes U$ for any 
UHF-algebra $U.$ 
To simplify notation, we write $B=C.$ 
Pick a pair of relatively prime supernatural numbers $\p$ and $\q.$ Let 
\beq
{\cal Z}_{\p, \q}=\{f\in C([0,1], Q): f(0)\in M_\p\otimes 1\andeqn f(1)\in M_\q\otimes 1\}\andeqn\\
D\otimes {\cal Z}_{\p,\q}=\{f\in C([0,1], D\otimes Q): f(0)\in D\otimes M_\p\otimes 1\andeqn f(1)\in D\otimes M_\q\otimes 1\}
\eneq
for  any \CA\, $D.$ Note, by \cite{RW}, ${\cal Z}$ is a stationary inductive limit of ${\cal Z}_{\p, \q}.$
Therefore, it suffices to show that $A\otimes {\cal Z}_{\p, \q}$ is (tracially) approximately divisible.
Since $B\cong B\otimes Q,$ it suffices, then, to show the following, 
for any finite subset ${\cal F}\subset A\otimes {\cal Z}_{\p, \q}$ and any $\ep>0,$ 
there exists a \hm\, $\Phi: B\otimes {\cal Z}_{\p, \q}\to A\otimes {\cal Z}_{\p, \q}$ such that
\beq
{\rm dist}(f, \Phi(B\otimes {\cal Z}_{\p, \q}))<\ep\rforal f\in {\cal F}.
\eneq

Let $\phi_\p: B\otimes M_\p\to A\otimes M_\p$  and 
$\phi_\q: B\otimes M_\q\to   A\otimes M_\q$  be  isomorphisms given by \ref{TTMW}.
Moreover, since $T(B)=T(A)=T(A\otimes M_\q)=T(A\otimes M_q),$
we may assume that $(\phi_\q)_T=(\phi_\q)_T,$ by
\ref{TTMW}.  Let $\psi_\p: B\otimes M_\p\otimes M_\q=B\otimes Q \to A\otimes M_\p\otimes M_\q=A\otimes Q$
given by $\psi_\p=\phi_\p\otimes {\rm id}_{M_\q}$ and 
let $\psi_\q=\phi_\q\otimes {\rm id}_{M_\p}: B\otimes Q\to A\otimes Q.$ 
Now $(\psi_\p)_T=(\psi_\q)_T.$ It follows from 
\ref{Casympu} that there exists a continuous path of unitaries $\{u(t): t\in [0,1)\}\subset A\otimes Q$
with $u(0)=1$ and 
$\lim_{t\to 1}u^*(t)\psi_\p(a)u^*(t)=\psi_\q(a)$ for all $a\in B\otimes M_\p\otimes M_\q.$
Define, for each $a(t)\in B\otimes {\cal Z}_{\p, \q},$ 
\beq
\Phi(a(t))=u^*(t)\psi_\p(a(t))u(t)\rforal t\in [0,1)\andeqn \Phi(a(1))=\psi_\q(a(1)).
\eneq
One then checks that $\Phi$ defines an isomorphism from $B\otimes Z_{\p, \q}$ onto 
$A\otimes {\cal Z}_{\p, \q}.$

\end{proof}

{\bf The proof of Theorem \ref{TMpartI}}:

\begin{proof}
This immediately follows
 from \ref{TTCLM}  and the above.
\end{proof}





{\bf The proof of Corollary \ref{CWiso}}:

\begin{proof}
If $A$ has finite nuclear dimension, then by \cite{aTz},
$A\otimes W$ is ${\cal Z}$-stable. If ${\tilde T}(A)=\{0\},$ then
$A\otimes W$ is purely infinite and non-unital.  Since $K_i(A\otimes W)=\{0\}$ ($i=0,1$),
$A\otimes {\cal K}\cong {\cal O}_2\otimes {\cal K}$ by \cite{Pclass} (see also \cite{KP}).
In the case  that ${\tilde T}(A)\not=\emptyset,$  then $gTR(A\otimes W)\le 1$ by  \ref{LFNZ}.
Thus \ref{CWiso} follows from \ref{TMpartI}.
\end{proof}


%


\providecommand{\href}[2]{#2}




\end{document}